\documentclass[pdf]{article}

\usepackage{pdfpages}
\usepackage[
	bookmarksopen=true,
	bookmarksdepth=2
]{hyperref}
\usepackage{xcolor}
\usepackage[utf8]{inputenc}
\usepackage{amsmath}
\usepackage{amsfonts}
\usepackage{amssymb}
\usepackage{amsthm}
\usepackage{array}
\usepackage{marginnote}
\usepackage{enumitem}
\usepackage{epsfig}
\usepackage{graphicx}
\usepackage[normalem]{ulem}
\usepackage{fancyhdr}

\usepackage{pst-node}
\usepackage{tikz-cd}
\usetikzlibrary{arrows.meta, positioning}
\usepackage{ulem}
\usepackage{cleveref}

\usepackage{accessibility}

\newcommand{\defstyle}[1]{\textbf{#1}}
\newcommand{\myprob}[1]{\mathbb P \left[ #1 \right]}
\newcommand{\probPalm}[2]{\mathbb P_{#1} \left[ #2 \right]}
\newcommand{\probCond}[2]{\mathbb P \left[ #1 \left| #2 \right. \right]}
\newcommand{\probPalmC}[3]{\mathbb P_{#1} \left[ #2 \left| #3 \right. \right]}
\newcommand{\probhat}[1]{\widehat{\mathbb P} \left[ #1 \right]}
\newcommand{\omidhat}[1]{\widehat{\mathbb E} \left[ #1 \right]}
\newcommand{\probhatC}[2]{\widehat{\mathbb P} \left[ #1 \left| #2\right. \right]}
\newcommand{\omidhatC}[2]{\widehat{\mathbb E} \left[ #1 \left| #2\right. \right]}

\newcommand{\omid}[1]{\mathbb E \left[ #1 \right]}
\newcommand{\omidPalm}[2]{\mathbb E_{#1} \left[ #2 \right]}
\newcommand{\omidCond}[2]{\mathbb E \left[ #1 \left| #2 \right. \right]}

\newcommand{\omidPalmC}[3]{\mathbb E_{#1} \left[ #2 \left| #3 \right. \right]}

\newcommand{\norm}[1]{\left| #1 \right|}

\newcommand{\tv}[1]{||#1||_{\mathrm{TV}}}
\newcommand{\identity}[1]{1_{#1}}
\newcommand{\bs}[1]{\boldsymbol{#1}}
\newcommand{\card}[1]{\left|#1\right|}
\DeclareMathOperator{\ave}{average}
\DeclareMathOperator{\wusf}{WUSF}
\DeclareMathOperator{\owusf}{OWUSF}

\newcommand{\del}[1]{}
\newcommand{\mar}[1]{\marginpar{\scriptsize  #1}}
\newcommand{\unwritten}[1]{}

\newcommand{\supp}{\mathrm{supp}}

\newcommand{\oball}[2]{B_{#1}(#2)}

\newcommand{\restrict}[2]{{
		\left.\kern-\nulldelimiterspace 
		#1 
		\vphantom{\big|} 
		\right|_{#2} 
}}

\DeclareMathOperator{\anc}{anc}

\DeclareMathOperator{\freq}{freq}

\DeclareMathOperator{\groperator}{Gr}
\newcommand{\gr}[1]{\groperator(#1)}

\newcommand{\sigfield}[3]{\mathcal{#2}^{\mathrm{#1}}_{\mathrm{#3}}}

\theoremstyle{theorem}
\newtheorem{theorem}{Theorem}[section]
\newtheorem{lemma}[theorem]{Lemma}
\newtheorem{proposition}[theorem]{Proposition}
\newtheorem{corollary}[theorem]{Corollary}

\newtheorem{problem}[theorem]{Problem}

\theoremstyle{definition}
\newtheorem{definition}[theorem]{Definition}

\newtheorem{assumption}[theorem]{Assumption}
\newtheorem{example}[theorem]{Example}
\theoremstyle{definition}
\newtheorem{remark}[theorem]{Remark}
\newtheorem{convention}[theorem]{Convention}

\newtheorem{model}{Model}

\crefname{model}{Model}{Models}
\crefname{submodel}{Model}{Models}
\crefname{algorithm}{Algorithm}{Algorithms}
\crefname{equation}{}{}
\crefname{problem}{Problem}{Problems}
\crefname{theorem}{Theorem}{Theorems}
\crefname{corollary}{Corollary}{Corollaries}

\theoremstyle{theorem}

\numberwithin{equation}{section}

\makeatletter
\let\orgdescriptionlabel\descriptionlabel
\renewcommand*{\descriptionlabel}[1]{%
	\let\orglabel\label
	\let\label\@gobble
	\phantomsection
	\edef\@currentlabel{#1}%
	\let\label\orglabel
	\orgdescriptionlabel{#1}%
}

\newlist{stepC}{enumerate}{1}
\setlist[stepC,1]{
	label=\textbf{Step C\arabic*},
	ref=C\arabic*,
	leftmargin=* 
}
\crefname{stepCi}{Step}{Steps}
\Crefname{stepCi}{Step}{Steps}

\newlist{stepCC}{enumerate}{1}
\setlist[stepCC,1]{
	label=\textbf{Step C\arabic{stepCi}.\arabic*},
	ref=C\arabic{stepCi}.\arabic*,
	leftmargin=* 
}
\crefname{stepCCi}{Step}{Steps}
\Crefname{stepCCi}{Step}{Steps}

\newlist{stepCCprime}{enumerate}{1}
\setlist[stepCCprime,1]{
	label=\textbf{Step C\arabic{stepCi}.\arabic*$'$},
	ref=C\arabic{stepCi}.\arabic*$'$,
	leftmargin=* 
}
\crefname{stepCCprimei}{Step}{Steps}
\Crefname{stepCCprimei}{Step}{Steps}

\newlist{stepL}{enumerate}{1}
\setlist[stepL,1]{
	label=\textbf{Step L\arabic*},
	ref=L\arabic*,
	leftmargin=* 
}
\crefname{stepLi}{Step}{Steps}
\Crefname{stepLi}{Step}{Steps}

\newlist{stepLL}{enumerate}{1}
\setlist[stepLL,1]{
	label=\textbf{Step L\arabic{stepLi}.\arabic*},
	ref=L\arabic{stepLi}.\arabic*,
	leftmargin=* 
}
\crefname{stepLLi}{Step}{Steps}
\Crefname{stepLLi}{Step}{Steps}

\makeatother

\begin{document}

\title{Indistinguishability in One-or-Two-Ended Forests on Unimodular Random Graphs}

\author{Francois Baccelli\footnote{Inria Paris, francois.baccelli@ens.fr}, Ali Khezeli \footnote{School of Mathematics, Institute for Research in Fundamental Sciences, alikhezeli@ipm.ir}}



\maketitle

\begin{abstract}
	Indistinguishability is a form of ergodicity, introduced by Lyons and Schramm for percolation clusters, that has become a fundamental notion in the theory of random infinite graphs. We prove the indistinguishability of connected components for a broad family of random one-ended or two-ended oriented forests on unimodular graphs using a new approach. We unify these models by introducing `coalescing Markov trajectories' (CMTs), which encompass a wide range of classical coalescing models, including river models and coalescing random walks. We also establish the indistinguishability of level-sets, a problem that has not previously been studied and lies beyond the scope of existing techniques. Using the latter, we prove that the clusters of the stationary voter model on a unimodular graph are indistinguishable.
	
	The core of the proof approach is the reduction of the indistinguishability of the components (respectively, level-sets) to the ergodicity (respectively, tail triviality) of the ancestry chain of the root. This shows a structural property that the ancestry chain is the fundamental object governing indistinguishability. For CMTs, the proof is completed by proving the tail triviality of Markov chains on unimodular random graphs, which is of independent interest. The flexibility of the approach is illustrated by further applications: It yields indistinguishability results for some point-map models on Bernoulli and Poisson point processes, including Howard's model and the strip point-map. It also yields a new and substantially simpler proof of indistinguishability for the wired uniform spanning forest, which is the only one-ended model for which indistinguishability was previously studied in the literature.
\end{abstract}

\noindent\textbf{Mathematics Subject Classification (2020).}
60B99,60G55,37A20,37A50.

\smallskip

\noindent\textbf{Keywords.}
Indistinguishability, coalescing Markov trajectories, unimodular random graph, uniform spanning forest, point process, coupling, ergodic theory, invariant sigma-field, tail sigma-field.

\setcounter{tocdepth}{1}
\tableofcontents

\section{Introduction}
\label{intro}

\subsection{Indistinguishability}
\label{intro:indistinguishability}

Following the seminal paper~\cite{LySc99}, the notion of indistinguishability naturally appears in models involving random graphs with more than one connected components. Informally, the general question is as follows: Can one find a property $A$ of the connected components (belonging to a specific class of invariant properties) such that, with positive probability, there exist some components that satisfy $A$, and some others that don't? If so, the components are called \textit{distinguishable}. 
For instance, is it possible that some components are transient while some others are recurrent? 
In~\cite{LySc99}, indistinguishability is proved for the infinite clusters of the Bernoulli percolation (or more generally, any insertion-tolerant percolation) on a unimodular transitive graph. More precisely, \cite{LySc99} proves that the properties which are measurable and automorphism-invariant cannot distinguish the infinite clusters of Bernoulli percolation (such properties should be defined without knowing where the origin is; e.g., transience or recurrence of the components).
The more recent paper~\cite{HuNa17indistinguishability} proves a similar statement for the \textit{(free or wired) uniform spanning forest} of a unimodular random graph or network, which is the only one-ended model whose indistinguishability is proved in the literature. The claim for the free uniform spanning forest was also proved independently in~\cite{Ti2018indistinguishability}, assuming that the free and wired versions are not identical.

The notion of unimodular random graphs has deep connections to stationary point processes and measure-preserving actions; see e.g., \cite{processes}. Also, when the underlying graph is ergodic, indistinguishability can be rephrased as ergodicity in various manners; e.g., the ergodicity of the lazy simple random walk on a connected component, or the ergodicity of the \textit{root-changing equivalence relation}. So, a proof of indistinguishability can be considered as a proof of ergodicity, which shows the importance of the subject from a dynamical-systems viewpoint. This will be discussed further in \Cref{subsec:indistinguishabilityGeneral}.

The approach of~\cite{LySc99} involves many novel ideas and results which are valuable on their own. In particular, it is proved that there are one or infinitely many infinite components, and in the second case, each component is transient and has infinitely many ends. Also, the proof uses \textit{pivotal edges}, the insertion of which changes the type of some components (assuming that the components are distinguishable), and leverages the random walk on a connected component. 
A similar idea is used in~\cite{HuNa17indistinguishability}, but due to the lack of insertion-tolerance, pivotal edges are replaced by \textit{pivotal updates}, where the latter are defined using the \textit{cycle-breaking dynamic}. 
However, the difficulty of one-ended models is that pivotal updates do not exist for \textit{tail component-properties}, which are discussed further in \Cref{intro:method}. Also, by a finite modification, one cannot merge two one-ended components into a single component. For such tail properties, \cite{HuNa17indistinguishability} uses another idea based on, roughly speaking, resampling $\wusf$ inside a large ball conditionally on what is seen outside the ball.\del{\footnote{More precisely, conditionally on the union of the ancestral lines of all points outside the ball.}} 
The recent work~\cite{AlPeTi25indistinguishability}, which was published in arXiv shortly before the completion of the present work, extends the idea of pivotal updates and proves indistinguishability for various two-ended and $\infty$-ended models.


\subsection{Models of Interest in this Work}
\label{intro:models}

\begin{figure}
	\begin{center}
		\includegraphics[height=.21\textheight,trim={4.5cm 5cm 8cm 5.5cm},clip]{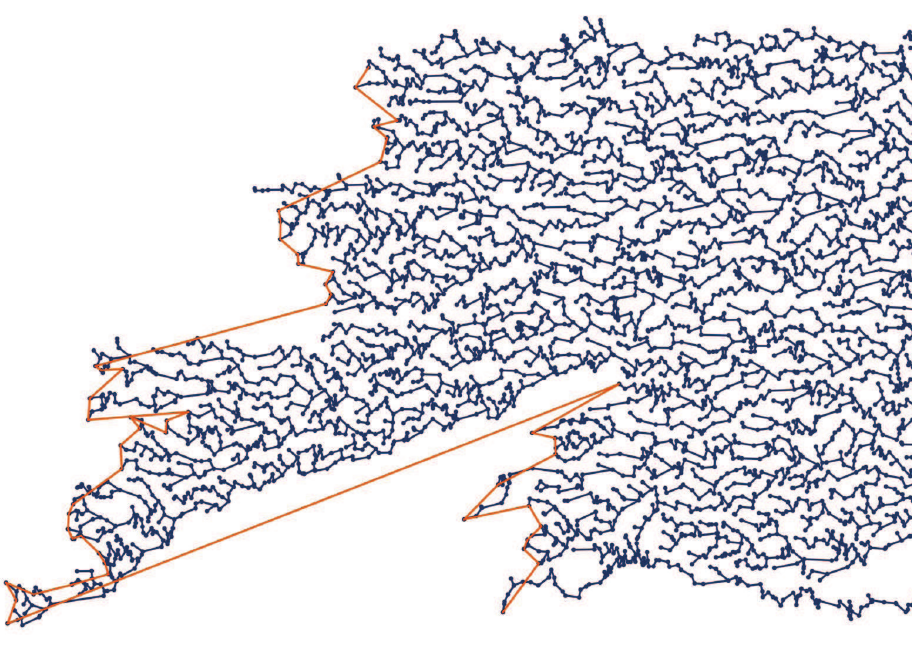}
		\includegraphics[height=.21\textheight,trim={2cm 8cm 1cm 7cm},clip]{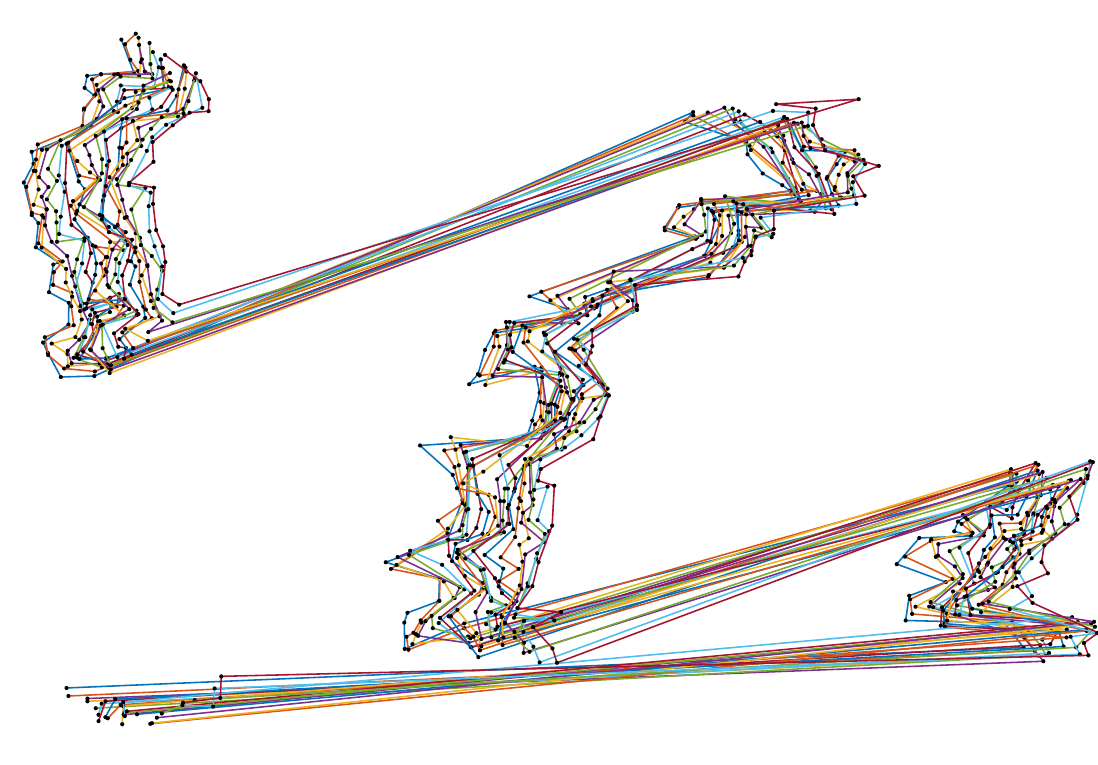}
		\alt{The level-sets in the strip point-map}
		\caption{The level-sets in the strip point-map on the Poisson point process in the plane (\Cref{ex:strip}). On the left, part of the graph is shown together with one level-set. On the right, only 20 consecutive level-sets are shown.}
		\label{fig:strip}
	\end{center}
\end{figure}

In this work, we present a novel approach and prove indistinguishability for various \textit{oriented forests}, where each vertex has exactly one out-going edge. Such forests are considered on a base space which is either a unimodular random graph, a unimodular discrete space, or the set of points of the Bernoulli/Poisson point process (we will also discuss deterministic base graphs later; see \Cref{main:cmt}). This includes all undirected forests with one-ended components (since there is a unique way to orient the edges towards the ends), and hence, such models are ubiquitous.
By considering the out-going edges, oriented forests are equivalent to cycle-free \defstyle{point-maps}, where the latter means a mapping $\bs F$ from the base space to itself, possibly using extra randomness (see \Cref{subsec:point-map}). Unimodularity implies that the connected components of such forests are automatically one-ended or two-ended almost surely (see \Cref{thm:classification}), which is stressed in the title of the paper.


The general approach is introduced in \Cref{main:method,intro:method} below. While the approach is intended for arbitrary oriented forest models, some of its steps are model-dependent. The main framework for which we establish indistinguishability is \defstyle{Coalescing Markov Trajectories (CMT)} on a unimodular random graph or discrete space.
Such models are defined by, roughly speaking, choosing the out-going edges of the vertices randomly and independently, under some conditions that will be specified in \Cref{model1,model2} in \Cref{subsec:cmt}. This unifies various models defined in the literature, e.g., river models, renewal forests defined on $\mathbb Z$, coalescing random walks, and coalescing Markov chains (see \Cref{subsec:examples}). This also implies that the clusters of the stationary voter model on unimodular graphs are indistinguishable (see \Cref{subsec:voter}). 
Beyond indistinguishability, we will also provide general criteria for connectedness and one-endedness of CMTs, which will be introduced in \Cref{main:qualitative}.
To show the flexibility of the approach, we extend \Cref{model1,model2}  and prove the indistinguishability in some models defined on the support of Bernoulli or Poisson point processes, including Howard's model, the strip point-map and some variations (see \Cref{model3,model4,model5,model6} in \Cref{sec:BernoulliPoisson}). 
As another application, we provide an alternative and simpler proof of the indistinguishability of connected components of $\wusf$.

The general approach is also used to prove the indistinguishability of \defstyle{level-sets}, where the latter correspond to generations if one regards each connected component of $\bs F$ as a family tree, where $\bs F(x)$ is considered as the parent of $x$ (see Figure~\ref{fig:strip} and also \Cref{subsec:classification} for the formal definition). This problem is novel and is beyond the scope of the previous techniques. The reason is that pivotal updates do not exist for tail component-properties (since a finite modification cannot merge two level-sets), and also the specific method of~\cite{HuNa17indistinguishability} cannot be applied to level-sets.
	We should note that the method of~\cite{HuNa17indistinguishability} can be used to prove the indistinguishability of the connected components of CMTs (although the proof is long), but not that of the level-sets. However, the reduction to the ancestral chain is new even for the indistinguishability of connected components, which might be useful for other models in the future.
	
Our initial motivation for studying the indistinguishability of level-sets was in the study of equivariant bijections between the level-sets in point-maps, particularly, in the strip point-map (see Figure~\ref{fig:strip}), stated as follows (see \Cref{subsec:bijection} for more formal definitions): 
\begin{proposition}[Bijection Between Level-Sets]
	\label{prop:bijection0}
	If $\bs F$ has indistinguishable level-sets, then there exists an equivariant bijection between consecutive level-sets of $\bs F$.
\end{proposition}
\Cref{prop:bijection} proves an extended version of this result using the ideas of equivariant matching between point processes, which have been of great interest in the last two decades.
An example where such a bijection does not exist is the canopy tree (see~\cite{processes}), where, intuitively, each level-set has \textit{twice more points} than the next level-set. This is formalized in \Cref{subsec:bijection} as \textit{relative intensity}, and is used in \Cref{prop:bijection}.

As a second application, we define a discrete-time version of the voter model in \Cref{ex:voter-discrete} and show that its clusters are just the level-sets of the dual coalescing Markov chain model. o, our results imply:

\begin{proposition}[Voter Model]
	\label{prop:voter}
	The clusters of the finest stationary multi-type voter model on a unimodular graph or discrete space (under the conditions mentioned in \Cref{ex:voter-discrete,ex:voter-continuous}) are indistinguishable.
\end{proposition}
The discrete-time and continuous-time cases are proved in \Cref{ex:voter-discrete,ex:voter-continuous} respectively.

As another motivation, the indistinguishability of level-sets implies that the probability of every event can be approximated by Birkhoff-type averages on a level-set, stated as follows, although the latter is very thin and occupies an asymptotically zero fraction of the base graph. Below, the \textit{$n$-cousins} of $\bs o$ are the points of the level-set of $\bs o$ that share the same $n$'th ancestor with $\bs o$.

\begin{lemma}
	\label{lem:average}
	If $[\bs G, \bs o]$ is ergodic and the level-sets of $\bs F$ are indistinguishable, then for every event $A$, $\myprob{A}$ is equal to the a.s. limit of the empirical average of $\identity{A}([\bs G, x; \bs F])$, where $x$ ranges over the $n$-cousins of $\bs o$, 
	as $n\to\infty$.
\end{lemma}
The proof is given in \Cref{subsec:bijection}.

\subsection{Main Results}
\label{intro:results}

\del{(later: mention subsections in order) This subsection introduces the main results of the paper. We state the indistinguishability results on components and level-sets in various models in \Cref{main:cmt,main:wusf} using the general approach mentioned in the previous subsection. Also, \Cref{main:deterministic} states analogous results on deterministic base graphs. \Cref{main:markov} provides a result of independent interest regarding Markov chains on unimodular graphs. Also, \Cref{main:qualitative} provides further qualitative properties of the connected components of CMTs.}

\subsubsection{A General Approach for Proving Indistinguishability in Oriented Forests on Unimodular Graphs}
\label{main:method}

The general approach consists of three steps. To avoid technicality at this stage, we describe these steps roughly here, and postpone the motivation and details to \Cref{intro:method} below. In the first step (Steps~\ref{step1comp} and~\ref{step1foil}), roughly speaking, we prove that every property of components/level-sets is equivalent to some \textit{tail} component/level-set property. 
In fact, tail component-properties should be replaced by \textit{tail branch-properties}, see \Cref{intro:method}.
The core of the approach is the second step (Steps~\ref{step2comp} and~\ref{step2foil}), which reduces the indistinguishability of components (resp. level-sets) to the ergodicity (resp. tail-triviality) of the \textit{ancestry chain} of the root. This step is proved for CMTs, WUSF and some other models later, thereby, it shows that the ancestry chain governs the indistinguishability of components/level-sets in these models. The last step (Steps~\ref{step3comp} and~\ref{step3foil}) is proving the ergodicity (resp. tail-triviality) of the ancestry chain.

This approach does not rely on pivotal updates. It leverages some measure-theoretic lemmas on the completion of sub-sigma-fields (applied to various types of invariant or tail sigma-fields), which are developed in \Cref{ap:measurability}, are of independent interest, and play important roles in the new proof method. Many of the results are also proved with coupling techniques.

The following theorem summarizes the general approach discussed above:

\begin{theorem}[Indistinguishability in General Point-Maps]
	\label{thm:general}
	Let $\bs F$ be an equivariant point-map on a unimodular graph or discrete space $[\bs G, \bs o]$.
	\begin{enumerate}[label=(\roman*)]
		\item \label{thm:general:comp} If $\bs F$ is cycle-free a.s., then it satisfies the claim of Step~\ref{step1comp} , which rules out non-tail properties 
		(see \Cref{thm:I in T 2} for the formal statement). If $\bs F$ satisfies the claims of Steps~\ref{step2comp} and~\ref{step3comp} as well, then it has indistinguishable connected components. 
		\item \label{thm:general:level} If every component of $\bs F$ is one-ended a.s., then it satisfies the claim of Step~\ref{step1foil} (see \Cref{thm:I in T 2} for the formal statement). If $\bs F$ satisfies the claims of Steps~\ref{step2foil} and~\ref{step3foil} as well, then it has indistinguishable level-sets. 
	\end{enumerate}
\end{theorem}

It should be noted that the claims of the three steps are necessary for indistinguishability.
The full generality of Steps~\ref{step1comp} and~\ref{step1foil} rules out non-tail properties in any model. 
We will also prove Steps~\ref{step2comp} and~\ref{step3comp} for CMTs and WUSF using coupling techniques, as well as Steps~\ref{step2foil} and~\ref{step3foil} for CMTs. The formal results in these cases are presented in the next subsection.

The proof idea is roughly described in \Cref{intro:method}, but the formal proof is postponed to the end of \Cref{subsec:drainage-ergodic}, where the required notations are developed and Steps~\ref{step1comp} and~\ref{step1foil} are proved (\Cref{thm:I in T 2}).


\subsubsection{Indistinguishability in CMTs on Unimodular Graphs or on Bernoulli/Poisson Point Processes}
\label{main:cmt}

As a main application of the general approach, we apply this method to various models as follows:

\begin{theorem}[Indistinguishability in CMTs and Similar Models]
	\label{thm:cmt-all}
	One has:
	\begin{enumerate}[label=(\roman*)]
		\item In the following models, which are alluded to in \Cref{intro:models} (see also \Cref{subsec:examples,subsec:BernoulliExamples} for precise definitions), the connected components are indistinguishable: Nguyen's river model and its variants, renewal forests on $\mathbb Z$, coalescing random walks and coalescing Markov chains on unimodular graphs, Howard's model, strip point map and its variants.
		\item In these models, in the cases where the components are one-ended a.s., the level-sets are also indistinguishable
		\item More generally, these claims hold for general cases of CMTs under the conditions in \Cref{model1,model2,model3,model4,model5,model6}, which unify the last examples (see \Cref{thm:indistinguishability} for a more precise statement).
	\end{enumerate}
\end{theorem}
\begin{proof}
	The last part will be proved in \Cref{thm:indistinguishability} and in \Cref{subsec:bernoulli,subsec:poisson}. The first two parts are derived as special cases in \Cref{subsec:examples,sec:BernoulliPoisson}.
\end{proof}

We will also prove the one-endedness of the models mentioned in this theorem (under mild conditions) in \Cref{subsec:examples,sec:BernoulliPoisson} based on the qualitative results of \Cref{main:qualitative}. As a result, we deduce the existence of equivariant bijections between consecutive level-sets (\Cref{prop:bijection}).

In the next paragraphs, we briefly describe \Cref{model1,model2,model3,model4,model5,model6}, alluded to in \Cref{thm:cmt-all}. A variant of the last theorem is also provided on deterministic base graphs, which will be presented in \Cref{thm:fixed-comp,thm:fixed-foil}.

\textbf{CMTs on Unimodular Graphs.}
First, CMTs on unimodular graphs are defined in \Cref{def:cmt}. Roughly speaking, $\bs F:\bs G\to\bs G$ is a random function such that, given $\bs G$, the values $\bs F(v)$ for $v\in \bs G$ are chosen randomly and independently. \Cref{model1,model2}, defined in \Cref{subsec:cmt}, impose some general conditions on such CMTs, which can be summarized as follows: In \Cref{model1}, it is assumed that the CMT is cycle-free, satisfies a \defstyle{balance} condition and a \defstyle{weak irreducibility} condition. The latter is a relaxation of the notion of irreducibility for Markov chains. In \Cref{model2}, for the indistinguishability of level-sets, it is additionally assumed that the components are one-ended and that a \defstyle{weak aperiodicity} condition is satisfied. The latter is a relaxation of aperiodicity of Markov chains (i.e., the gcd of the cycle lenghts is 1), but since no directed cycle is assumed to exist, another type of cycles is considered. See also \Cref{rem:model1} regarding these conditions.

\textbf{Models on Bernoulli/Poisson Point Processes.}
\Cref{model3,model4}, defined in \Cref{sec:BernoulliPoisson}, are analogous models on the Bernoulli point process on the vertices (i.e., each vertex of the base space is retained or deleted randomly and independently). Here, the inherent randomness of the point process may guarantee variants of the weak irreducibility and weak aperiodicity conditions, even if no randomness is used in the definition of $\bs F(\cdot)$ (except maybe for breaking ties). This includes Howard's model. Some extra conditions are needed, in particular, $\bs F(\cdot)$ should depend only on a \textit{stopping set} with some conditions. This is needed to guarantee that the ancestral line is a Markov chain; see \Cref{sec:BernoulliPoisson} for more details. Also, \Cref{model5,model6} define analogous models on the Poisson point process, which includes the strip point-map. Here, the underlying Poisson point process is assumed to be defined on $\mathbb R^d$, $\mathbb H^d$, or more generally, on a \textit{unimodular random measured metric space}.

\textbf{Models with Non-Independent Jumps.}
\Cref{model3,model4} also extend \Cref{model1,model2} in another way by letting the jumps be non-independent (even if the Bernoulli point process contains all vertices; i.e., has parameter 1).
As mentioned, we need that $\bs F$ is a factor of i.i.d. marks and that, for all samples $(G,o)$, the value of $\bs F(o)$‌ is determined by the marks in a finite {stopping set}. In most examples, this is guaranteed by constructing $\bs F(o)$‌ by revealing the marks one by one (based on the previously revealed marks), and deciding whether to stop or not by a stopping rule (i.e., by a decision tree), which gives a rich flexibility in defining the point-map. Under the further assumptions of \Cref{model3,model4}, indistinguishability of components/level-sets hold in these models as well.


\subsubsection{Ergodicity and Tail Triviality of Markov Chains on Unimodular Graphs}
\label{main:markov}

We prove the ergodicity and tail triviality for Markov chains on unimodular graphs, which are results of independent interest. This establishes the last step in the proof of indistinguishability in CMTs (see  Steps~\ref{step3comp} and~\ref{step3foil} in \Cref{intro:method}), but it should be noted that no CMT is assumed in these results. The ergodicity result is a slight generalization of Theorem~4.6 of~\cite{processes}, but the tail triviality result is new to the best of our knowledge.  

Let $[\bs G, \bs o]$ be a unimodular random graph. Assume that, for each realization $G$ of the graph and for every $x\in G$, a probability measure $K(x,\cdot)$ on $G$ is given. Let $(\bs X_n)_{n\geq 0}$ denote the Markov chain on $G$ with kernel $K$. If $K$ satisfies some isomorphism-invariance and measurability condition, then $[\bs G, \bs o; (\bs X_n)_n]$ is a well defined random object; see \Cref{subsec:stationary} for the formal definition.
Note that there is no CMT at this stage, but if we define a CMT with the same kernel $K$, then $(\bs X_n)_n$ has the same distribution as $\anc(\bs o)$, before the first self-intersection (given $\bs G$ and $\bs o$). Also, the conditions of \Cref{model1} (namely, cycle-free, balance, and weak irreducibility) are still meaningful for $(\bs X_n)_n$. 

For $n\geq 0$, by choosing $\bs X_n$ as the new root and forgetting $\bs X_0,\ldots, \bs X_{n-1}$, a random object $[\bs G, \bs X_n; (\bs X_{n+i})_{i\geq 0}]$ is obtained, which defines another Markov chain; see \Cref{subsec:stationary} for more details (in fact, this is a dynamical system since the shift requires no additional randomness).
Note that two different Markov chains, $(\bs X_n)_n$ and $\left([\bs G, \bs X_n; (\bs X_{n+i})_{i\geq 0}]\right)_n$, are considered here. We will clarify which Markov chains is meant each time.

\begin{theorem}[Steps~\ref{step3comp} and~\ref{step3foil}: Ergodicity and Tail Triviality of Markov Chains on Unimodular Graphs]
	\label{thm:pathIntro}
	Let $[\bs G, \bs o]$ be an ergodic unimodular graph or discrete space and $(\bs X_n)_n$ be a Markov chain on $\bs G$ started from $\bs o$. 
	\begin{enumerate}[label=(\roman*)]
		\item \label{thm:pathIntro:ergodic} If the conditions of balance and weak irreducibility of \Cref{model1} hold, then the Markov chain $\left([\bs G, \bs X_n; (\bs X_{n+i})_{i\geq 0}]\right)_n$ is ergodic.
		\item \label{thm:pathIntro:tail} If, in addition, the weak aperiodicity condition of \Cref{model2} is satisfied, then the Markov chain of Part~\ref{thm:pathIntro:ergodic} is also tail trivial.
	\end{enumerate}
\end{theorem}
Note that the shift-invariant or tail events considered in this theorem are invariant under rooted-graph isomorphisms as well. For instance, it is known that one cannot select an end of $\bs G$ in a measurable way, and hence, one cannot define an invariant event based on the convergence of the random walk to a specific end.

\Cref{thm:pathErgodic-poisson} provides a similar result for Markov chains on unimodular continuum spaces. Also, \Cref{thm:pathI=T-intro} provides a version of this theorem for deterministic graphs. It states that, on a typical sample of the graph, the tail sigma-field of the Markov chain is equivalent to the invariant sigma-field.

It should be noted that $\left([\bs G, \bs X_n; (\bs X_{n+i})_{i\geq 0}]\right)_n$ is not necessarily stationary. 
Nevertheless, by the balance condition, it becomes stationary after \textit{biasing} the probability measure suitably; see \Cref{lem:stationary}. This might not be true for general point-maps, but ergodicity still makes sense (see \Cref{rem:quasi}).

Part~\ref{thm:pathIntro:ergodic} is in fact a slight generalization of Theorem~4.6 of~\cite{processes}, which in turn generalizes a result of~\cite{LySc99} (we just weaken the condition $K>0$ to weak irreducibility). We provide a restatement of this claim in \Cref{thm:pathErgodic} in \Cref{subsec:pathErgodic}, mentioning the connection with harmonic functions. We will also provide a self-contained proof. This serves as \Cref{step3comp} in the proof of the indistinguishability of components of CMTs.

Part~\ref{thm:pathIntro:tail} is a new result. Here, as in the classical theory of Markov chains, tail triviality means that the probability of every \textit{tail event} is either 0 or 1, where being a tail event means, roughly speaking, not depending on the first $n$ elements of the chain for every $n$.
We will restate and prove this claim in \Cref{thm:pathTail} which serves as \Cref{step3foil} in the proof of the indistinguishability of level-sets of CMTs.

\subsubsection{Qualitative Properties of the Connected Components in CMTs}
\label{main:qualitative}

We prove several qualitative properties of CMTs regarding connectedness and one-endedness, which are described below. 
In addition, as an application of indistinguishability, we prove a \textit{connectivity decay} property in \Cref{prop:decay}.
These results are useful for verifying one-endedness in order to use \Cref{model2}. In particular, they are used in \Cref{subsec:examples} to prove easily that all models alluded to in \Cref{thm:cmt-all} have one-ended components under mild conditions, and mostly have infinitely many connected components. The proofs of all of these results are given in \Cref{sec:one-ended} using coupling techniques.

We start by showing that the cycle-free assumption is necessary in the study of indistinguishability:

\begin{proposition}
	\label{prop:finiteComps}
	Assume $[\bs G, \bs o; \bs F]$ satisfies the balance and weak irreducibility conditions of \Cref{model1}. Then, almost surely, if the cycle-free condition fails, then all components of $\bs F$ are finite. 
\end{proposition}

Also, we show that the balance condition cannot be dropped; see \Cref{subsec:distinguishable} for examples with distinguishable components or level-sets. It is also easy to see that the weak irreducibility condition cannot be dropped (see coalescing Markov chains in \Cref{subsec:crw}).

In the rest of the subsection, we assume that $\bs F$ is a CMT on a unimodular graph or discrete space $[\bs G, \bs o]$ that satisfies the conditions of \Cref{model1}. For a realization $G$ of the graph, we let $\mathbb P_G$ be the distribution of the oriented forest on the base graph $G$. We also let $K(x,y):=\probPalm{G}{\bs F(x)=y}$.
The indistinguishability of components (\Cref{thm:cmt-all}) implies that either all components are one-ended or all are two-ended.  We strengthen this in the following results, which is of independent interest and is proved in without using indistinguishability.

\begin{theorem}[Disconnectedness Implies One-Endedness]
	\label{thm:one-ended-or-connected}
	In \Cref{model1}, almost surely, if the graph of $\bs F$ is disconnected, then all components are one-ended. In other words, the following dichotomy holds for almost every sample of $[\bs G, \bs o; \bs F]$:
	\begin{enumerate}[label=(\roman*)]
		\item Either the graph of $\bs F$ is connected (and it is either one-ended or two-ended),
		\item Or the graph of $\bs F$ is disconnected and all its components are one-ended.
	\end{enumerate}
\end{theorem}

It is easy to see that the connectedness of $\bs F$ is equivalent to the condition that, starting from any two points of $\bs G$, the trajectories of two independent instances of the Markov chain on $\bs G$ intersect a.s. We weaken this condition as follows, and also show that there are either one or infinitely many components. 

\begin{definition}[Intersection Property]
	\label{def:colision}
	In \Cref{model1}, we say that a sample $(G,o)$ satisfies the \defstyle{path-intersection property} (resp. \defstyle{infinite path-intersection property}) if the paths of two independent copies $(\bs X_n)_n$ and $(\bs X'_n)_n$ of the Markov chain on $G$ started from $o$ and with kernel $K$ intersect (resp. intersect infinitely many times) a.s., where an intersection means $(m,n)\in\mathbb N^2$ such that $\bs X_m=\bs X'_n$ (note that the starting points are excluded).
\end{definition}

\begin{theorem}[Number of Components]
	\label{thm:infComps}
	In \Cref{model1}, the number of connected components of $\bs F$ is either 1 or $\infty$ a.s. Additionally, if $[\bs G, \bs o]$ is ergodic, then:
	\begin{enumerate}[label=(\roman*)]
		\item \label{thm:infComps:1} If almost every sample of $[\bs G, \bs o]$ has the path-intersection property, then $\bs F$ is connected a.s.
		\item \label{thm:infComps:2}Otherwise, $\bs F$ has infinitely many connected components a.s.
	\end{enumerate}
\end{theorem}

As an application of indistinguishability, we leverage the last theorem to prove the following result, which is similar to Theorem~4.1 of~\cite{LySc99}. The proof uses indistinguishability and also uses \textit{cluster frequencies}, which will be discussed in \Cref{lem:freq}.

\begin{proposition}[Connectivity Decay]
	\label{prop:decay}
	In \Cref{model1}, almost every sample $(G,o)$ satisfies the following connectivity decay property: If $\bs F$ is disconnected, then 
	\[\inf \{\probPalm{G}{x\in C(o)}: x\in G \}=0,\]
	where $\probPalm{G}{x\in C(o)}$ is the probability that $x$ and $o$ are in the same connected component of $\bs F$.
\end{proposition}

Whether or not $\bs F$ is connected, one still needs one-endedness (as assumed in \Cref{model2}) in order to study the indistinguishability of level-sets (see the examples of \Cref{subsec:examples}).
The next theorem characterizes one-endedness using the \defstyle{Green function} of $K$, which is defined by the following classical formula: 
\begin{equation}
	\label{eq:green}
	g_K(G,x,y):=\sum_{n=0}^{\infty} K^n(x,y) = \probPalm{G}{y\in \anc(x)},
\end{equation}
where the last equality follows from the cycle-free condition.

\begin{theorem}[One-Endedness]
	\label{thm:one-ended}
	In \Cref{model1}, the connected components of $\bs F$ are one-ended a.s. if and only if 
	\begin{equation}
		\label{eq:greendecay}
		\lim_n g_K(\bs G,\bs o, \bs X_n)=0 \text{ in probability},
	\end{equation}
	where $(\bs X_n)_n$ is the Markov chain on $\bs G$ starting from $\bs o$ and with kernel $K$. In particular, this holds if, almost surely, $\lim_{y}g_K(\bs G, \bs o,y)=0$, where the limit is over $y\in \bs G$ when $d(\bs o, y)\to \infty$.
\end{theorem}

Both criteria of convergence in this theorem will be used in the examples of \Cref{subsec:examples}.

\subsubsection{Alternative Proof for WUSF}
\label{main:wusf}

As an application of the general approach, we provide in \Cref{sec:wusf} an alternative simpler proof for the following result of~\cite{HuNa17indistinguishability}. As already mentioned, when $\wusf$ is disconnected, the graph is transient and Wilson's algorithm provides an oriented version of $\wusf$ (see \Cref{sec:wusf}). In this case, $\wusf$ can be regarded as a point-map (which is not a CMT), and hence, \Cref{step1comp,step2comp,step3comp} make sense.

\begin{theorem}[\cite{HuNa17indistinguishability}, Theorem~1.1]
	\label{thm:usf-indistinguishability}
	In every unimodular graph $[\bs G, \bs o]$ in which $\omid{\mathrm{deg}(\bs o)}<\infty$, the connected components of $\wusf(\bs G)$ are indistinguishable.
\end{theorem}


We will also prove the analogous result on a deterministic base graph (\Cref{thm:fixed-comp}).
Also, our method gives a simple proof of the ergodicity of $\wusf$ on unimodular graphs and the tail triviality of $\wusf$ on deterministic graphs, which is provided in \Cref{thm:usf-ergodic}.

The proofs are given in \Cref{sec:wusf} and do not rely on the \textit{cycle-breaking dynamics} of~\cite{HuNa17indistinguishability}. Verifying \Cref{step2comp} is achieved by a coupling method based on a generalization of the \textit{stochastic covering property} of $\wusf$~\cite{PePe14}, which is provided in \Cref{lem:usf-covering}.

However, we could not establish \Cref{step2foil,step3foil} for $\wusf$ (we prove \Cref{step2foil} in the transient case).
A natural problem is whether the level-sets of $\wusf$ on a unimodular non-bipartite graph are indistinguishable or not (\Cref{prob:lerw}). In the transient case, only \Cref{step3foil} is missing, which is the tail-triviality of the loop-erased random walk.  

\subsubsection{Results on Deterministic Base Graphs}
\label{main:deterministic}

Theorem~1.2 of~\cite{Hu20indistinguishability} provides an analogous indistinguishability result for $\wusf$ on a deterministic base graph, assuming one-endedness and the \textit{Liouville property} (the latter means that every bounded harmonic function is constant). This result extends Theorem~4.5 of~\cite{BeKePeSc11}, which deals with the special case $\mathbb Z^d$. The result of~\cite{Hu20indistinguishability} also deals with $k$-tuples of components. We provide alternative proofs of these results and similar results for CMTs, and extend them to non-Liouville or non-one-ended cases.

In contrast to the setting of unimodular random base graphs, when the base graph is deterministic, we neither fix a root nor consider equivalence under isomorphisms. This gives more flexibility to distinguish the components; e.g., one can color the components based on the inclusion of any given arbitrary point. In~\cite{HuNa17indistinguishability}, indistinguishability is proved only for tail component-properties, assuming one-endedness and the Liouville property. In fact, the latter is necessary (otherwise, the \textit{Poisson boundary} is nontrivial and the \textit{boundary points} of the components distinguish them). We restate this result as follows and study the non-Liouville and non-one-ended cases as well. For the latter, we use the \textit{oriented} version of $\wusf$ given by Wilson's algorithm rooted at infinity (see \Cref{subsec:usf-defs}).

\begin{theorem}[Components in a Deterministic Graph]
	\label{thm:fixed-comp}
	Let $\bs F$ be the oriented wired uniform spanning forest on a transient locally-finite graph $G$ (resp. a CMT on a countable set $G$).
	\begin{enumerate}[label=(\roman*)]
		\item If $G$ (resp. the Markov kernel corresponding to the CMT) has the Liouville property, then the components of $\bs F$ are indistinguishable by tail branch-properties.
		\item Otherwise, any tail branch-property is equivalent to some measurable subset of the Poisson boundary; in other words, the Poisson boundary gives all the ways to distinguish the components by tail branch-properties. In particular, for every tail branch-property $A$, the components of any two points such that their ancestral lines have the same boundary points  (given an arbitrary realization of the Poisson boundary) are not distinguished by $A$ a.s.
	\end{enumerate}
	In addition, similar claims hold for $k$-tuples of distinct components.
\end{theorem}

The required definitions and the proof are given in \Cref{subsec:tail-fixed,sec:wusf}. See also \Cref{ex:nonLiouville} for a non-Liouville example. The core of the proof is that \Cref{step2comp} for CMTs, proved in \Cref{thm:tailDrainage}, does not require unimodularity and works for deterministic graphs as well. \Cref{step1comp} is not needed since we already consider tail branch-properties. \Cref{step3comp} is also equivalent to the Liouville property in the setting of deterministic base graphs.

For the level-sets of CMTs, we first prove a variant of \Cref{step3foil} in \Cref{thm:pathI=T-intro} below. 
Despite \Cref{thm:pathIntro}, the invariant and tail sigma-fields of a Markov chain on $G$ might be nontrivial. 
The equality or inequality of these two sigma-fields is of great importance in the literature:
\begin{definition}
	A Markov chain with initial distribution $\mathbb P$ is called \defstyle{steady} if its invariant and tail sigma-fields agree mod $\mathbb P$.
\end{definition}
We prove steadiness
under the following conditions, which is of independent interest:

\begin{theorem}[Steadiness of Markov Chains on Deterministic Graphs]
	\label{thm:pathI=T-intro}
	In Part~\ref{thm:pathIntro:tail} of \Cref{thm:pathIntro}, almost every sample $(G,o)$ of $[\bs G, \bs o]$ satisfies the following property:
	For every probability measure $\mu$ on $G$, the Markov chain $(\bs X_i)_{i\geq 0}$ on $G$ with initial distribution $\mu$ is steady.
	In particular, in the Liouville case, $(\bs X_n)_n$ is tail-trivial.
	\\
	In addition, the same holds for any deterministic graph or discrete space $G$ that satisfies the following:
	\begin{equation}
		\label{eq:aperiodic2}
		\exists l\in\mathbb N, \exists \epsilon>0: \forall x\in G: \tv{K^l(x,\cdot)\wedge K^{l+1}(x,\cdot)}>\epsilon.
	\end{equation}
\end{theorem}
Here, $\tv{\cdot}$ denotes the total variation norm of measures. Note that~\eqref{eq:aperiodic2} does not necessarily hold in \Cref{model2}.
This theorem is restated and strengthened in \Cref{thm:pathI=T} and is proved in \Cref{subsec:pathTail}. 


\begin{theorem}[Level-Sets in a Deterministic Graph]
	\label{thm:fixed-foil}
	Consider a cycle-free CMT on a deterministic base graph $G$ such that all of its level-sets are infinite a.s. 
	Then, the tail properties of the ancestral line are the only properties (up to equivalence) that can distinguish the level-sets. In particular, if the Markov chain on $G$ is steady 
	(e.g., for typical realizations of the graph in \Cref{model2}, see \Cref{thm:pathI=T-intro}), then:
	\begin{enumerate}[label=(\roman*)]
		\item \label{thm:fixed-foil-1} If the Liouville property holds, then the level-sets of $\bs F$ are indistinguishable by tail level-set-properties.
		\item \label{thm:fixed-foil-2} Otherwise, the Poisson boundary gives all the ways to distinguish the level-sets by tail level-set-properties. In particular, in each component, the level-sets are indistinguishable by tail level-set-properties.
	\end{enumerate}
\end{theorem}

This will be proved in \Cref{subsec:tail-fixed}. In fact, every tail level-set-property is equivalent to some tail property of the ancestral line (see \Cref{prop:tail-fixed}). Also, the first paragraph of the theorem means that the \textit{tail boundary} (see, e.g., \cite{Ka92boundary}) gives all the ways to distinguish the level-sets by tail level-set-properties.



\subsection{Introduction to the General Approach}
\label{intro:method}

In this subsection, we outline the general proof approach. Two variants are provided, one for connected components and one for level-sets. 
We also describe how this method can be applied to CMTs and $\wusf$. 
Note that $\wusf$ can be regarded as an oriented forest by Wilson's algorithm rooted at infinity if the base graph is transient (see \Cref{subsec:usf-defs}).

\subsubsection{Proving Indistinguishability of Connected Components}
\label{intro:method-comp}
Assume we have a random triple $[\bs G, \bs o;\bs F]$, where $[\bs G, \bs o]$ is a unimodular graph, $\bs o$ is the root and $\bs F$ is an equivariant random oriented forest on $\bs G$. As mentioned in \Cref{intro:models}, one may regard $\bs F$ as a function on $\bs G$ by letting $\bs F(v)$ be the endpoint of the unique out-going edge from $v$.
Let $\anc(v)$ be the infinite path in $\bs F$ started from $v$ using only out-going edges, which we call the \textit{ancestral line} of $v$ {(it is called the \textit{future} of $v$ in~\cite{HuNa17indistinguishability})}.\del{\footnote{The ancestral line of $v$ is called the \textit{future} of $v$ in~\cite{HuNa17indistinguishability}. One can imagine two opposite directions of time; one by imagining moving particles (where the particles coalesce), and another by the \textit{family tree} viewpoint (where $\bs F(v)$ is regarded as the parent of $v$). We will use the family tree viewpoint throughout the paper.}} 
As in the previous works, a component-property is formalized as an event $A$, on the space of all samples $[G,o;f]$, that is invariant under changing the root in the same connected component (in the forest defined by $f$). Then, indistinguishability is equivalent to $\myprob{A}=\myprob{[\bs G, \bs o; \bs F]\in A}\in\{0,1\}$ for all component-properties $A$, assuming that $[\bs G, \bs o]$ is ergodic {(see \Cref{subsec:ergodic,subsec:indistinguishabilityGeneral} for more details).}

Our method is based on the following simple but essential observation (see \Cref{lem:basicInclusion}):

\begin{equation}
	\label{obs:comp}
	\begin{minipage}{0.75\textwidth}
		\textit{Every property $B$ which depends only on $[\bs G, \bs o; \anc(\bs o)]$ and is shift-invariant, is a component-property. In addition, it is a tail component-property.}
	\end{minipage}
\end{equation}

Here, as in Markov chain theory, the shift means moving one step forward in the path and forgetting the previous point; i.e., $\sigma[G,x_0;(x_0,x_1,\ldots)]:=[G,x_1;(x_1,x_2,\ldots)]$ (see \Cref{eq:shift}). 
The first part of the observation is clear: For a forest $f$, if $u,v\in G$ belong to the same component, then $\anc(u)$ and $\anc(v)$ shift-couple (i.e., are identical after removing some initial segments of possibly different lengths from each path). Thus, for $B$ as above, $\identity{B}[G,u;f]=\identity{B}[G,v;f]$. That is, $B$ is a component-property.
Also, $f$ satisfies the definition of \textit{tail component-properties} (see \Cref{subsec:drainage-events} for the precise definition); i.e., if we modify $f$ at finitely many points in such a way that the component of $u$ in $f$ is changed at finitely many points, then $\identity{B}[G, u;f]$ is not changed since, under these assumptions, the new ancestral line of $u$ shift-couples with the initial one. 

Moreover, in the last paragraph, it is enough that the new component of $u$ and the initial one share an infinite \textit{branch} that contains the ancestry chain of $u$, where a branch means a connected component obtained after removing finitely many edges. Hence, we call $B$ a \defstyle{tail branch-property}, see \Cref{def:super-tail-comp} for the formal definition (the latter is not needed in the one-ended cases and, in this case, the reader might replace it with tail component-properties throughout the paper).

In fact, we need a converse to Observation~\eqref{obs:comp}. More precisely, a proof of indistinguishability of components can be obtained by the following steps:

\begin{stepC}
	\item 
	\label{step1comp} 
	Proving that every component-property is equivalent to some  tail branch-property, where two events $A_1$ and $A_2$ are called equivalent if $\myprob{A_1\Delta A_2}=0$.
	\item 
	\label{step2comp} 
	Proving that every  tail branch-property of $[\bs G, \bs o; \bs F]$ is equivalent to some shift-invariant property of $[\bs G, \bs o; \anc(\bs o)]$.
	\item 
	\label{step3comp} 
	Proving that, if $[\bs G, \bs o]$ is ergodic, then every shift-invariant property of $[\bs G, \bs o; \anc(\bs o)]$ is trivial {(i.e., has probability 0 or 1)}; i.e, $[\bs G, \bs o; \anc(\bs o)]$ is ergodic.
\end{stepC}
Note that in the last step, $[\bs G, \bs o; \anc(\bs o)]$ is not necessarily stationary under shift, but ergodicity still makes sense. This will be discussed in \Cref{rem:quasi}.

\Cref{step1comp} rules out non-tail properties and reduces the problem to indistinguishability by tail branch-properties.
We prove this step in full generality in \Cref{thm:I in T 2}. 
This was proved in~\cite{HuNa17indistinguishability} in the special case of the $\wusf$, leveraging the ideas of~\cite{LySc99} including \textit{pivotal updates}. With the help of some measure theoretic lemmas that we provide in \Cref{ap:measurability}, we give a proof for the general case, which is much simpler, involves almost no $\epsilon,\delta$ calculations and does not use pivotal updates.

\Cref{step2comp} is the converse to \Cref{obs:comp} alluded above, and is the core idea of this proof method. It connects the properties of $\bs F$ to those of the path $\anc(\bs o)$ (in the latter, the forest is deleted and only $\anc(\bs o)$ is kept).
We will prove \Cref{step2comp} for CMTs and $\wusf$ (in \Cref{thm:tailDrainage} and \Cref{cor:usf-sigfield} respectively) without requiring unimodularity, by a coupling method briefly described below.
In~\cite{HuNa17indistinguishability}, tail component-properties are studied by, roughly speaking, resampling $\bs F$ conditionally on what is seen outside a large ball. In the present work, we just condition on $\bs G, \bs o$ and $\anc(\bs o)$, which results in a much simpler proof of indistinguishability. 
\Cref{step2comp} can be proved for a given model by showing the following claims (which are implicit in the proof of \Cref{thm:tailDrainage} for CMTs):
\begin{stepCC}
	\setcounter{stepCi}{2}
	\item 
	\label{step2comp-1}
	Proving that every  tail branch-property $A$ of $[\bs G, \bs o; \bs F]$ is equivalent to some event $A'$ that depends only on $[\bs G, \bs o; \anc(\bs o)]$.
	\item 
	\label{step2comp-2}
	Proving that $A'$ can be chosen to be a shift-invariant event.
\end{stepCC}
To prove \Cref{step2comp-1}, we will use the following simple observation:
\begin{equation}
	\label{obs:knowingAnc}
	\begin{minipage}{0.75\textwidth}
		\textit{{If $A$ is a tail branch-property, then $\identity{A}[G,o;f]$ is invariant by all finite modifications of $f$ which do not modify $\anc_f(o)$.}}
	\end{minipage}
\end{equation}
This is true for general models because, by such a finite modification of $f$, the new component of $o$ shares an infinite branch with the initial one.
By this observation, \Cref{step2comp-1} is reduced to the following, where \textit{tail triviality} has the classical meaning that every tail event has probability zero or one:
\begin{stepCCprime}
	\setcounter{stepCi}{2} 
	\item  \label{step2comp-1-second}
	Proving the tail triviality of $\bs F$ conditionally on knowing $\bs G, \bs o$ and $\anc(\bs o)$ (for a suitable version of the conditional distribution).
\end{stepCCprime}
{We will prove this step for CMTs and $\wusf$ (\Cref{thm:tailDrainage} and \Cref{lem:usf-conditionalIndependence})} using a coupling method to show that $A$ is independent of $\restrict{\bs F}{\oball{n}{\bs o}}$ conditionally on $\bs G, \bs o$ and $\anc(\bs o)$ (where $\oball{n}{\bs o}$ is the ball of radius $n$ centered at $\bs o$). This also gives an intuitive proof of the tail triviality of the (wired or free) uniform spanning forest of a given graph, which seems to be a new proof (\Cref{thm:usf-ergodic}).
We will then prove \Cref{step2comp-2} {(\Cref{thm:tailDrainage} and \Cref{lem:usf-measurable})} using another coupling method for CMTs and for $\wusf$. 

%
%

\Cref{step3comp} should be proved for each model separately. For CMTs, $\anc(\bs o)$ is just a Markov chain on a unimodular graph. {Such Markov chains satisfy an ergodicity property, which} is implied by slightly generalizing a result of~\cite{processes}, which we prove in \Cref{thm:pathIntro,thm:pathErgodic}. For $\wusf$, in the transient case, $\anc(\bs o)$ is the loop-erased simple random walk and the claim is quickly implied (in \Cref{lem:usf-trivial I}) by the ergodicity of the simple random walk on unimodular graphs.


\subsubsection{Proving Indistinguishability of Level-Sets}
\label{intro:method-foil}

As mentioned in \Cref{intro:models}, one may regard each component of $\bs F$ as a family tree by regarding $\bs F(v)$ as the parent of $v$. We define the \textit{level-sets} of $\bs F$ as the generations in this family tree (see \Cref{subsec:classification}).
In~\cite{eft}, it is proved that, in the one-ended case, all level-sets are infinite a.s. (see \Cref{subsec:classification}). Level-set-properties and  tail level-set-properties are defined similarly in the present paper (see \Cref{def:invariantDrainage,def:tailDrainage}). It seems that the method of~\cite{HuNa17indistinguishability} cannot be applied for the indistinguishability of level-sets. 
Our method is based on the following variant of Observation~\eqref{obs:comp}, which is again model-free:

\begin{equation}
	\label{obs:foil}
	\begin{minipage}{0.75\textwidth}
		\textit{On the event that $\bs F$ is locally-finite, one-ended, and has infinite level-sets, every property $B$ which depends only on $[\bs G, \bs o; \anc(\bs o)]$ and is a tail event (under shifts) is a level-set-property. In addition, it is a tail level-set-property.}
	\end{minipage}
\end{equation}
In this observation, the notion of tail properties of $[\bs G, \bs o; \anc(\bs o)]$ is a special case of tail properties of Markov chains (see \Cref{def:path-tail}). 
The observation holds
because, if two ancestral lines join after an equal number of steps, then the property $B$ holds for either both or none of them (See \Cref{lem:basicInclusion}). With the hint of Observation~\eqref{obs:foil} and the need for its converse, a proof of indistinguishability of level-sets can be obtained in a one-ended model by the following steps:
\begin{stepL}
	\item 
	\label{step1foil}
	Proving that every level-set-property is equivalent to some  tail level-set-property. 
	\item 
	\label{step2foil}
	Proving that every  tail level-set-property is equivalent to some tail property (under shift) of $[\bs G, \bs o; \anc(\bs o)]$.
	\item 
	\label{step3foil}
	Proving that $[\bs G, \bs o; \anc(\bs o)]$ is tail-trivial (under shift) if $[\bs G, \bs o]$ is ergodic.
\end{stepL}
\Cref{step1foil} is again proved in full generality in \Cref{thm:I in T 2}. 
For proving \Cref{step2foil}, the idea is that Observation~\eqref{obs:knowingAnc} is valid for any tail level-set-property $A$ as well. So, if \Cref{step2comp-1-second} is proved, then $A$ is equivalent to an event $A'$ that depends only on $[\bs G, \bs o; \anc(\bs o)]$. Then, it remains to show that $A'$ can be chosen as a tail property under shift. We will prove the latter for CMTs and for $\wusf$ on transitive graphs by a coupling method (see \Cref{thm:tailDrainage,cor:usf-sigfield}).


To prove \Cref{step3foil} for CMTs, we prove that Markov chains on unimodular graphs are tail trivial under mild conditions, which is a novel result (\Cref{thm:pathIntro,thm:pathTail}). We stress that everything is considered up to graph-isomorphisms here, which makes it difficult to define tail events. Recall also that \Cref{thm:pathI=T-intro,thm:pathI=T} prove a version of this result on deterministic base graphs. 

\subsection{Structure of the Paper}
\label{intro:structure}

\begin{figure}
	\begin{center}
		\begin{tikzpicture}[
			node distance=.5cm and 1cm,
			chapter/.style={rectangle, draw, rounded corners, thick, minimum width=1cm, minimum height=.5cm},
			arrow/.style={-Stealth, thick},
			dashedarrow/.style={-Stealth, thick, dotted}
			]
			
			\node[chapter] (def) {\ref{sec:basic} Definitions};
			\node[chapter, below right=of def] (main) {\ref{sec:cmt} CMTs};
			\node[chapter, below=of def] (warmup) {\ref{sec:warmUp} Lattice CMTs};
			\node[chapter, right=of def] (c3) {\ref{sec:Markov} C3, L3};
			\node[chapter, above=of def] (intro) {\ref{intro} Introduction};
			\node[chapter, right=of c3] (c1) {\ref{sec:drainage} C1, L1};
			\node[chapter, below=of c1] (c2) {\ref{sec:cmt-proof} C2, L2};
			\node[chapter, above=of c3] (wusf) {\ref{sec:wusf} WUSF};
			\node[chapter, below right=of c1] (poisson) {\ref{sec:BernoulliPoisson} Bernoulli, Poisson};
			\node[chapter, above right=of c3] (comps) {\ref{sec:one-ended} Qualitative Properties};
			
			\draw[arrow] (intro) -- (def);
			\draw[arrow] (def) -- (main);
			\draw[arrow] (main) -- (warmup);
			\draw[arrow] (main) -- (c3);
			\draw[arrow] (main) -- (c2);
			\draw[dashedarrow] (c3) -- (c2);
			\draw[arrow] (c3) -- (c1);
			\draw[arrow] (c3) -- (comps);
			\draw[dashedarrow] (c1) -- (wusf);
			\draw[arrow] (c1) -- (poisson);
			\draw[dashedarrow] (c3) -- (wusf);
			\draw[arrow] (c2) -- (poisson);
			\draw[dashedarrow] (c1) -- (c2);
			\draw[arrow] (def) -- (wusf);
		\end{tikzpicture}
	\end{center}
	\alt{The dependency structure of the sections}
	\caption{The dependency structure of the sections. A dotted arrow means that only the definitions of the sigma-fields (see Table~\ref{tab:symbols}) are used from the preceding section.}
	\label{fig:dependency}
\end{figure}

\begin{table}[htbp]
	\centering
	\caption{\textbf{Table of Symbols.}}
	\label{tab:symbols}
	\begin{tabular}{l @{\quad} l @{\quad} l}
		{Symbol} & {Description} & {Defined in}  \\
		\hline
		$\mathcal G_*$ & The space  of rooted graphs or discrete spaces & Def \ref{def:gstar}\\
		& $\quad$ $[G,o]$, possibly having marks.&\\
		$\mathcal G'_*$ & The space  of $[G,o;f]$, where $f:G\to G$. & Def \ref{def:gprime}\\
		$\mathcal G_{\infty}$ & The space of $[G,x_0;(x_i)_{i\geq 0}]$, where $x_i\in G$. & Sec \ref{subsec:stationary}\\
		$\bs F$ & A point-map or CMT. & Def~\ref{def:pointmap}, \ref{def:cmt}\\
		$K$ & a factor Markov kernel. & Sec \ref{subsec:cmt}\\
		$b$ & The function used in the balance condition.& \Cref{model1}\\
		$C(o)$ & The connected component containing $o$. & Def \ref{def:component}\\
		$L(o)$ & The level-set containing $o$. & Def \ref{def:component}\\
		$\anc(o)$ & The ancestor line of $o$. &‌Sec\ref{subsec:classification}\\
		$D(o), D_n(o)$ & The descendants of $o$. & Sec \ref{subsec:classification}\\
		$\sigfield{}{I}{}$  & Invariant sigma-field on $\mathcal G_*$. & \Cref{subsec:ergodic}\\
		$\sigfield{pm}{I}{}$  & Sigma-field of invariant point-map-properties. & Def \ref{def:invariantDrainage}\\
		$\sigfield{path}{I}{}$  & Sigma-field of invariant path properties. & Def \ref{def:path-invariant}\\
		$\sigfield{comp}{I}{}$ & Sigma-field of component-properties. & Def \ref{def:invariantDrainage}\\ 
		$\sigfield{level}{I}{}$ & Sigma-field of level-set-properties. & Def \ref{def:invariantDrainage}\\
		$\sigfield{pm}{T}{}$  & Sigma-field of tail point-map-properties. & Def \ref{def:tailDrainage}\\
		$\sigfield{path}{T}{}$  & Sigma-field of tail path properties. & Def \ref{def:path-tail}\\
		$\sigfield{branch}{T}{}$ & Sigma-field of tail branch-properties. & Def \ref{def:super-tail-comp}\\ 
		$\sigfield{level}{T}{}$ & Sigma-field of tail level-set-properties. & Def \ref{def:tailDrainage}\\ 
		$\sigfield{path}{I}{G},\ldots$ & Analogous sigma-fields when a deterministic & Sec~\ref{subsec:pathErgodic}, \ref{subsec:tail-fixed}\\
		& \quad base graph $G$ is fixed. & 
	\end{tabular}
\end{table}

\begin{figure}
	\begin{center}
		\begin{tikzcd}[column sep= small, row sep=scriptsize]
			\hyperref[def:invariantDrainage]{\sigfield{pm}{I}{}}\arrow[rrr]\arrow[ddd, "{\ref{thm:I in T 2}}", shift left] &&& \hyperref[def:invariantDrainage]{\sigfield{comp}{I}{}}\arrow[rrr]\arrow[ddd, "{\ref{thm:I in T 2}}", shift left] &&& \hyperref[def:invariantDrainage]{\sigfield{level}{I}{}}\arrow[ddd, "{\ref{thm:I in T 2}}", shift left]\\
			&&&&&&\\
			&&&&&&\\
			%
			\hyperref[def:tailDrainage]{\sigfield{pm}{T}{}}\arrow[uuu,shift left]\arrow[rrr, "\ref{lem:basicInclusion0}"]\arrow[ddd, "{\ref{thm:tailDrainage}}", shift left, dashrightarrow] &&& \hyperref[def:super-tail-comp]{\sigfield{branch}{T}{}}\arrow[uuu,shift left]\arrow[rrr, "\ref{lem:basicInclusion0}"]\arrow[ddd, "{\ref{thm:tailDrainage}}", shift left, dashrightarrow] &&& \hyperref[def:tailDrainage]{\sigfield{level}{T}{}}\arrow[uuu,shift left]\arrow[ddd, "{\ref{thm:tailDrainage}}", shift left, dashrightarrow]\\
			&&&&&&\\
			&&&&&&\\
			%
			\hyperref[subsec:ergodic]{\sigfield{}{I}{}}\arrow[uuu,shift left]\arrow[rrr,shift left] &&& \hyperref[def:path-invariant]{\sigfield{path}{I}{}}\arrow[lll, "{\ref{thm:pathIntro},\ref{thm:pathErgodic}}", shift left, dashrightarrow]\arrow[uuu, "\ref{lem:basicInclusion}", shift left]\arrow[rrr,shift left] &&& \hyperref[def:path-tail]{\sigfield{path}{T}{}}\arrow[lll,"{\ref{thm:pathIntro}, \ref{thm:pathTail}}", shift left, dashrightarrow]\arrow[uuu, "\ref{lem:basicInclusion}",shift left]
			%
		\end{tikzcd}
		\alt{The relations between the sigma-fields used in this paper}
		\caption{The relations between the sigma-fields mentioned in \Cref{tab:symbols} (see \Cref{subsec:pathErgodic,subsec:tail-fixed} for analogous sigma-fields for a fixed underlying graph). An arrow, say from $\mathcal F$ to $\mathcal F'$, means $\mathcal F\subseteq\mathcal F'$ mod $\mathbb P$ (\Cref{def:augmentation}), possibly with some additional assumptions. Only $\sigfield{comp}{I}{}, \sigfield{branch}{T}{}, \sigfield{path}{I}{}$ and $\sigfield{}{I}{}$ are needed for the proof of indistinguishability of components. 
			The label of an arrow navigates to the result which proves the inclusion (those without label are straightforward). 
			Solid arrows are valid for all models, but the dashed arrows are proved only for Markov chains or CMTs in \Cref{sec:Markov,sec:drainage,sec:cmt-proof} (some are also proved for $\wusf$ in \Cref{sec:wusf}).}
		\label{fig:sigfields}
	\end{center}
\end{figure}

To help the reader, Figure~\ref{fig:dependency} describes the dependency structure of the sections. Also, Table~\ref{tab:symbols} lists the frequently used symbols, in particular, the various invariant and tail sigma-fields used throughout the paper. The relations between these sigma-fields are also depicted in Figure~\ref{fig:sigfields}.

In \Cref{main:method,intro:method}, we introduced the general approach of proving indistinguishability in oriented forests on unimodular graphs. We also summarized the main results in \Cref{intro:results}. Basic definitions are provided in \Cref{sec:basic}, including unimodularity, point-maps, ergodicity and indistinguishability. \Cref{sec:cmt} provides the formal definition of CMTs and specifies \Cref{model1,model2}. In particular, \Cref{subsec:examples} discusses how these models unify various models of the literature and studies their indistinguishability properties. The three steps of proving indistinguishability for CMTs are given in \Cref{sec:Markov,sec:drainage,sec:cmt-proof}, where \Cref{sec:Markov} proves ergodicity and tail triviality of Markov chains on unimodular graphs.
Before that, \Cref{sec:warmUp} provides a significantly simpler proof of indistinguishability for CMTs on Euclidean lattices and similar models. This proof does not cover general CMTs, but contains many important ingredients of the ideas.
\Cref{sec:one-ended} studies the qualitative properties of CMTs regarding connectedness and one-endedness.
The alternative proof for the indistinguishability in $\wusf$ is given in \Cref{sec:wusf}. The reader may jump to this section right after \Cref{sec:basic}, but the arguments for $\wusf$ are more compressed than those of CMTs.  \Cref{sec:BernoulliPoisson} provides the analogous versions of CMTs on Bernoulli or Poisson point processes (\Cref{model3,model4,model5,model6}), that generalize Howard's model and the strip point-map, and proves their indistinguishability properties. Open problems are listed in \Cref{sec:problems}.
\Cref{ap:measurability} includes measure-theoretic lemmas on the completion of sub-sigma-fields, which are essential ingredients in the proofs.
Finally, some further developments of the notion of stopping sets are also given in \Cref{ap:stopping}, which are used in \Cref{sec:BernoulliPoisson}.

\section{Definitions}
\label{sec:basic}

\subsection{Notation}
\label{sec:notation}

If $G$ is a graph or a metric space, the distance function on $G$ is always denoted by $d$. For $r\geq 0$ and $v\in G$, let $\oball{r}{v}$ denote the closed ball $\oball{r}{v}:=\{w\in G: d(v,w)\leq r\}$.  
In this paper, we regard graphs merely as metric spaces, and the edges of $G$ are considered as a marking of pairs of points. More generally:
\begin{convention}[Marks]
	\label{conv:marks}
	When we talk about a graph or a discrete space $G$, we presume that $G$ has some given marking of the points and/or the pairs of points, where a \defstyle{marking} is a function $m:G\to M$ or $m:G\times G\to M$ for some given complete separable mark space $M$ (this generalizes the notion of \textit{networks} \cite{processes}). Note that we may assume multiple marks as well.
	\\
	We will sometimes distinguish one or more markings of $G$ by using semicolon. For instance, when studying (random) forests on $G$, we deal with objects of the form $(G;f)$, where $f$ is a forest on $G$, while allowing $G$ to have other markings as well (e.g., edges or i.i.d. marks described later in \Cref{subsec:marking}).
\end{convention}

 A graph is \defstyle{locally-finite} if the degree of every vertex is finite. More generally, a discrete space $G$ is \defstyle{boundedly-finite} if every bounded subset of $G$ is finite.
A \defstyle{rooted graph or discrete space} is a pair $(G,o)$, where $G$ is a graph or discrete (metric) space and $o\in G$. The notion of \defstyle{doubly-rooted} spaces is defined similarly and is denoted by $(G,o_1,o_2)$. We say that $(G_1,o_1)$ is \defstyle{isomorphic} to $(G_2,o_2)$ if there is a mark-preserving invertible isometry $\rho:G_1\to G_2$ such that $\rho(o_1)=o_2$. This is an equivalence relation and the equivalence class of $(G,o)$ is denoted by $[G,o]$.

\begin{definition}
	\label{def:gstar}
	Let $\mathcal G_*$ be the set of equivalence classes of rooted boundedly-finite graphs or discrete spaces (recall that graphs or discrete spaces are allowed to have marks as above, but we do not use a symbol for the presumed mark space to avoid too long notations). Define $\mathcal G_{**}$ similarly for doubly-rooted graphs or discrete spaces. 
\end{definition}

For the case of graphs, the Benjamini-Schramm metric is defined on $\mathcal G_*$ and $\mathcal G_{**}$ and makes them Polish spaces (see~\cite{processes}).
For the more general case of discrete metric spaces, \cite{I} defines GHP-type metrics on $\mathcal G_*$ and $\mathcal G_{**}$ and it is shown that they are Borel subsets of some Polish spaces. 
\begin{convention}
	\label{conv:forgettingG}
	If $v\in G$ and $g$ is a function on $\mathcal G_*$, we use the notation $g(v)$ rather than $g([G,v])$ by an abuse of notation, if $G$ is clearly known. Similarly, $h(u,v)$ stands for $h([G,u,v])$ for a function $h$ on $\mathcal G_{**}$.
\end{convention}

The symbol $A\Delta B$ denotes the symmetric difference of the sets $A$ and $B$.
If $G$ is a graph, a \defstyle{branch} of $G$ is a connected component of $G\setminus U$ for some finite set $U\subseteq G$ {(equipped with the induced graph structure on $G\setminus U$)}. Two branches $B_1$ and $B_2$ are called equivalent if $\card{B_1\Delta B_2}<\infty$, where $\card{\cdot}$ denotes the cardinal of a set. The \defstyle{ends} of $G$ are the equivalence classes of the infinite branches of $G$.

We usually denote the random objects by bold symbols for the ease of reading. If $\bs X$ and $\bs Y$ are random objects on the same state space $E$, then $\bs X\sim\bs Y$ means that $\bs X$ and $\bs Y$ have the same distribution (we mostly do not refer to the probability space here). By a \defstyle{sample} or a \defstyle{realization} of $\bs X$, we mean a point $x\in E$ in the state space of $\bs X$. The random object $\bs X$ is \defstyle{essentially constant} if $\bs X=x$, a.s., for some sample $x$. If $b:E\to\mathbb R^{\geq 0}$ is a measurable function such that $\lambda:=\omid{b(\bs X)}\in (0,\infty)$, we let $\widehat{\mathbb P}$ be the probability measure obtained by \defstyle{biasing by $b(\bs X)$} {(defined on the same probability space as $\mathbb P$)}; i.e., $\probhat{A}:=\frac 1 {\lambda}\omid{b(\bs X)\identity{A}}$ for every event $A$. 


We say that two sequences $(x_i)_i$ and $(x'_i)_i$ \defstyle{eventually coincide} if $\exists N: \forall i\geq N: x_i=x'_i$. Also, we say that $(x_i)$ and $(x'_i)_i$ \defstyle{shift-couple} if $\exists N,N': \forall i\geq 0: x_{N+i} = x'_{N'+i}$. 

\begin{definition}[Null-Event-Augmentation]
	\label{def:augmentation}
	In a probability space $(\Omega,\mathcal F, \mathbb P)$, two events $A$ and $B$ are \defstyle{equivalent mod $\mathbb P$} if $\myprob{A\Delta B}=0$. Given a sub-sigma-field $\mathcal F_1\subseteq\mathcal F$, the \defstyle{null-event-augmentation} of $\mathcal F_1$ is the sigma-field
	\begin{equation*}
		\hat{\mathcal F}_1:=\{A\in\mathcal F: \exists A_1\in\mathcal F_1: \myprob{A\Delta A_1}=0 \}.
	\end{equation*}
	If $\mathcal F_2\subseteq\mathcal F$ is another sub-sigma-field, we say 
	\begin{eqnarray*}
		\text{$\mathcal F_1\subseteq\mathcal F_2$ mod $\mathbb P$ (resp., $\mathcal F_1=\mathcal F_2$ mod $\mathbb P$)}
	\end{eqnarray*}
	 if $\mathcal F_1\subseteq\hat{\mathcal F}_2$, or equivalently, $\hat{\mathcal F_1}\subseteq \hat{\mathcal F}_2$ (resp. $\hat{\mathcal F_1}=\hat{\mathcal F}_2$). In particular, $\mathcal F_1$ is \defstyle{trivial} if it is equivalent to $\{\emptyset,\Omega\}$ mod $\mathbb P$; i.e., $\forall A\in\mathcal F_1: \myprob{A}\in\{0,1\}$.
\end{definition}
It is straightforward to see that $\hat{\mathcal F}_1$ is the sigma-field generated by $\mathcal F_1$ and all null events of $\mathcal F$. 
\Cref{ap:measurability} provides various lemmas regarding null-event-augmentation, which are essential ingredients in the proofs of indistinguishability in this work.

\subsection{Unimodular Graphs and Discrete Spaces}

A \defstyle{random rooted graph or discrete space} is a random object $[\bs G, \bs o]$ in $\mathcal G_*$ (recall that we allow $\bs G$ to have a marking). Its distribution is a probability measure on $\mathcal G_*$. We used brackets here because every sample of $[\bs G, \bs o]$ is an equivalence class of rooted spaces. However, we sometimes say \textit{`let $(G,o)$ be a sample of $[\bs G, \bs o]$'}, meaning that it is a representative of an equivalence class.

In addition, $[\bs G, \bs o]$ is called \defstyle{unimodular} if it satisfies the following \defstyle{mass transport principle (MTP)}:
\begin{equation}
	\label{eq:mtp}
	\omid{\sum_{x\in \bs G} g(\bs G, \bs o, x)} = \omid{\sum_{x\in \bs G} g(\bs G, x, \bs o)},
\end{equation}
for all measurable functions $g:\mathcal G_{**}\to\mathbb R^{\geq 0}$. See~\cite{processes,I}.

The following lemma is well known about unimodular graphs (see e.g., Lemma~2.15	of~\cite{I}).

\begin{lemma}[Everything Happens at the Root]
	\label{lem:happensAtRoot}
	Let $[\bs G, \bs o]$ be a unimodular graph or discrete space and $A\subseteq\mathcal G_*$ be an event. The \defstyle{equivariant {(factor)} subset} $\bs S:=\bs S(\bs G)\subseteq \bs G$ defined by $\bs S(G):=\{v\in G: [G,v]\in A\}$ satisfies the following: $\myprob{\bs o\in \bs S}>0$ if and only if $\myprob{\bs S\neq\emptyset}>0$. Also, $\myprob{\bs o\in \bs S}=1$ if and only if $\myprob{\bs S=\bs G}=1$.
\end{lemma}

\subsection{Point Processes}
\label{subsec:pointProcess}

Roughly speaking, a \defstyle{point process} $\Phi$ on $\mathbb R^d$ (or $\mathbb Z^d$) is a random discrete subset of $\mathbb R^d$. The point process $\Phi$ is \defstyle{stationary} if $\Phi-t\sim \Phi, \forall t\in\mathbb R^d$. One might regard $\Phi$ as a random discrete metric space. It is well known that the \defstyle{Palm version} of $\Phi$, which is heuristically obtained by conditioning on $0\in \Phi$, is a unimodular discrete space (see e.g., Example~2.7 of~\cite{I}). 

Since we do not assume rotation invariance, it is convenient to add marks to $\Phi$ that determine the direction of the primary axes. For instance, one might let $m(u,v):=v-u\in\mathbb R^d$ with mark space $\mathbb R^d$. If this marking is used, then every measurable property of $\Phi$ corresponds to some measurable subset of $\mathcal G_*$.


\subsection{Random Forests and Point-Maps}

\subsubsection{Factor Marking and Equivariant Marking}
\label{subsec:marking}
Let $[\bs G, \bs o]$ be a unimodular random graph or discrete space (possibly having some marking). One might assign additional marks to the points or pairs of points of $\bs G$. For instance, if $m_0:\mathcal G_{**}\to\Xi$ is a measurable function, where $\Xi$ is a complete separable (new) mark space, then one can let $m(u,v):=m_0[\bs G, u, v]$ for $u,v\in \bs G$ (for marking single vertices, one can restrict attention to $m(u,u)$ or repeat the definition for $m_0:\mathcal G_*\to\Xi$). This is called a \defstyle{factor marking}. One obtains that $[\bs G, \bs o; m]$ is unimodular.

To allow extra randomness, we use the term \defstyle{equivariant marking} or \defstyle{equivariant process} defined in Definition~2.9 of~\cite{I} (the term \textit{equivariant} is usually used in the literature without allowing extra randomness, but we use the term \textit{factor} in that case). Formally, for every sample $G$, the space $\Xi^{G\times G}$ is the state space of the markings on $G$. An equivariant marking is a map that assigns to every sample $G$ a random marking $\bs m_G(\cdot,\cdot)$ of $G$ (or equivalently, a probability measure on $\Xi^{G\times G}$) such that: (1) For every isomorphism $\rho:G\to G'$, the random marking $\bs m_G(\rho^{-1}(\cdot), \rho^{-1}(\cdot))$ of $G'$ has the same distribution as $\bs m_{G'}$, (2) For every event $A$ depending on $G,o$ and a marking of $G$, the function $[G,o]\mapsto \myprob{[G,o; \bs m_G]\in A}$ is measurable. By Lemma~2.12 of~\cite{I}, the random object $[\bs G, \bs o; \bs m_{\bs G}]$ is unimodular. 

As an example, one might assign i.i.d. marks with the uniform distribution on $[0,1]$ to the points (or pairs of points) of $\bs G$.

\subsubsection{Point-Maps}
\label{subsec:point-map}

Let $[\bs G, \bs o]$ be a random rooted graph or discrete space. 
Roughly speaking, a point-map on $\bs G$ is a random function $\bs F:\bs G\to\bs G$. This is defined below using the notion of equivariant marking mentioned above.

\begin{definition}[$\mathcal G'_*$]
	\label{def:gprime}
	Let $\mathcal G'_*$ be the set of equivalence classes of tuples of the form $(G,o;f)$, where $(G,o)$ is a rooted graph or discrete space and $f:G\to G$ 
	(recall that $G$ is allowed to have some other markings by \Cref{conv:marks}).
	The equivalence class of $(G,o;f)$ is denoted by $[G,o,f]$.
\end{definition}

We are interested in the \textit{graph of $f$} obtained by connecting $u$ to $f(u)$ for every $u\in G$. Note that the edges of $G$ (if any) are not used in defining connected components and may only be used in defining $f$. Also, the graph of $f$ is not necessarily a subgraph of $G$.

Note that a function $f:G\to G$ can be considered as a marking of $G$ by $(u,v)\mapsto\identity{\{f(u)=v\}}$. Therefore, $\mathcal G'_*$ is a subset of another instance of $\mathcal G_*$ for a larger mark space. In particular, $\mathcal G'_*$ is a Borel subset of some Polish space. Note also that there is a natural projection from $\mathcal G'_*$ to $\mathcal G_*$ defined by $[G,o;f]\mapsto [G,o]$.


\begin{definition}[Point-Map]
	\label{def:pointmap}
	An \defstyle{(equivariant) point-map}\del{\footnote{Point-maps are called \textit{point-shifts} in~\cite{eft}, but we preferred to use the term point-map here because it is more reminiscent of a map from $G$ to itself.}} is a map that assigns to every graph or discrete space $G$ a probability measure $\nu_G$ on $G^G$ such that it is equivariant under isomorphisms (i.e., for every isomorphism $\rho:G\to G'$, the map $f\mapsto \rho\circ f\circ \rho^{-1}$ pushes $\nu_G$ forward to $\nu_{G'}$) and 
	for every Borel set $A\subseteq \mathcal G'_*$, the function $[G,o]\mapsto \nu_G(\{f: [G,o;f]\in A\})$ is measurable on $\mathcal G_*$. If $\bs F=\bs F(G,\cdot):G\to G$ is a random function chosen with distribution $\nu_G$ (defined for all $G$), then $\bs F$ is also called an (equivariant) point-map.
%
\end{definition}

Note that in the definition of $\bs F(G,\cdot)$, no common probability space is assumed for different graphs $G$. A special case of the definition is when $\bs F$ is defined as a function of an i.i.d. marking of $G$ (in addition to the marks already on $G$, if any). In this case, $\bs F$ is called a \textit{factor of i.i.d.} in the literature.

\Cref{def:pointmap} extends the definitions of point-maps and vertex-maps defined in~\cite{HeLa05characterization} and~\cite{eft} respectively, by allowing extra randomness {(they are called \textit{point-shifts} in these references, but we preferred to use \textit{point-maps})}. This is also a special case of \textit{equivariant processes} defined in~\cite{I}. As mentioned in \Cref{subsec:marking}, if $[\bs G, \bs o]$ is unimodular, then $[\bs G, \bs o; \bs F]$ (which is a well defined random element of $\mathcal G'_*$) is also unimodular.



\begin{example}[Random Forests as Point-Maps]
	Let $[\bs G, \bs o]$ be a unimodular graph and let $\bs T$ be an equivariant random subgraph of $\bs G$ which is a forest a.s. If every connected component of $\bs T$ has at most two ends, then $\bs T$ might be considered as the set of connected components of a point-map $\bs F$ as follows. 
	To define $\bs F$, it is enough to orient the edges such that the out-degrees of all nodes are at most one (let $\bs F(v)=v$ if the out-degree of $v$ is zero). In every one-ended component $C$ of $\bs T$, orient the edges towards the end of $C$. In every two-ended component $C$ of $\bs T$, choose one of the two ends by tossing a coin and orient the edges towards that end. In every finite component $C$, choose a random point $v_C\in C$ uniformly, and orient the edges towards $v_C$.
\end{example}

\subsubsection{Classification of Components and Level-Sets}
\label{subsec:classification}

The following definition and theorem are borrowed from~\cite{eft}.

\begin{definition}[Components and Level-Sets]
	\label{def:component}
	Let $G$ be a graph or discrete space and $f:G\to G$. 
	The graph of $f$ is the graph $\gr{f}$ with vertex set $G$ and the directed edges $(v,f(v))$ for all $v\in G$. 
	The \defstyle{component of $v$ in $f$} and the \defstyle{level-set} of $v$ in $f$ are defined by
	\begin{align*}
		C(v):=C_f(v)&:= \text{ the connected component of $\gr{f}$ that contains $v$},\\
		L(v):= L_f(v)&:= \{w\in G: \exists n\geq 0: f^{(n)}(w)=f^{(n)}(v)\}\subseteq C(v).
	\end{align*}
\end{definition}

In each cycle-free connected component, one can define a \defstyle{height function} $h$, which is defined (uniquely up to addition by a constant) by the condition $h(f(v))=h(v)-1$. The level-sets of the height function are exactly the level-sets of $f$.


\begin{theorem}[Classification of Point-Maps, \cite{eft}]
	\label{thm:classification}
	Let $[\bs G, \bs o]$ be a unimodular graph of discrete space and $\bs F$ be a point-map. Then, almost surely, the graph of $\bs F$ is locally-finite and every connected component $C$ of $\bs F$ on $\bs G$ is of one of the following types:
	\begin{enumerate}[label=(\roman*)]
		\item $C$ is finite and has a unique cycle,
		\item $C$ is two-ended and every level-set in $C$ is finite,
		\item $C$ is one-ended and every level-set in $C$ is infinite.
	\end{enumerate}
\end{theorem}

In $\gr{f}$, a cycle-free component can be regarded as a (unisex) family tree by regarding $f(v)$ as the parent of $v$ for all $v\in G$. These trees are called \textit{eternal family trees} in~\cite{eft}, since there is no node which is the ancestor of all other nodes. The {\defstyle{ancestor line}} and the set of \defstyle{descendants} of $v$ are defined by ($f$ is removed from the notation whenever it is clear what $f$ is):
\begin{align*}
	\anc(v)&:=\anc_f(v):= (f^{(n)}(v))_{n\geq 0},\\
	D_n(v)&:= D_{n,f}(v) := \{u\in G: f^{(n)}(u)=v\},\\
	D(v)&:=D_f(v):= \{u\in G: v\in\anc(u)\} = \cup_n D_n(v).
\end{align*}
So, one has $L(v)=\cup_n D_n(f^{(n)}(v))$.

Note that the distribution of $[\bs G, \bs F(\bs o); \bs F]$ (obtained by moving the root to its parent) is not equivalent to that of $[\bs G, \bs o; \bs F]$ in general (for instance, the new root $\bs F(\bs o)$ always has a child, but this does not necessarily hold for $\bs o$). However, we can prove the following weaker result: 
\begin{lemma}
	\label{lem:shiftToF(o)}
	If $\bs F$ is an equivariant point-map on a unimodular graph or discrete space $[\bs G, \bs o]$, then the distribution of $[\bs G, \bs F(\bs o); \bs F]$ is obtained by biasing that of $[\bs G, \bs o; \bs F]$ by $d_1(\bs o):=\card{\bs F^{-1}(\bs o)}$. So, the distribution of $\sigma\left([\bs G, \bs o; \anc(\bs o)]\right)$ is obtained by biasing that of $[\bs G, \bs o; \anc(\bs o)]$ by $\omidCond{d_1(\bs o)}{[\bs G, \bs o; \anc(\bs o)]}$.
\end{lemma}
\begin{proof}
	The mass transport principle implies $\myprob{[\bs G, \bs F(\bs o);\bs F]\in A} = \omid{d_1(\bs o)\identity{A}}$ for every event $A$. This implies the first claim. The second claim is obtained by conditioning on $[\bs G, \bs o; \anc(\bs o)]$.
\end{proof}

\subsection{Ergodicity}
\label{subsec:ergodic}

A unimodular graph or discrete space $[\bs G, \bs o]$ is \defstyle{ergodic} if the \defstyle{invariant sigma-field} $\sigfield{}{I}{}$ is trivial. The latter is the set of events $A\subseteq \mathcal G_*$ which are invariant under changing the root; i.e., $\identity{A}[G,o]=\identity{A}[G,v]$ for all $(G,o)$ and $v\in G$. If $\bs F$ is a point-map, the ergodicity of $[\bs G, \bs o; \bs F]$ is defined similarly using the sigma-field of \textit{invariant point-map properties} defined in \Cref{def:invariantDrainage} below. 

\del{The\mar{\ali{Papers usually cite ergodic decomposition without such an explanation. This is adapted to a book.}} \textit{ergodic decomposition theorem} states that every unimodular probability measure $\mathbb P$ on $\mathcal G_*$ can be written as a mixture (informally, a convex combination in the integral form) of ergodic unimodular probability measures, {which are called the \textit{ergodic components} of $\mathbb P$}. For proving statements in which the assumptions and the claims are of the form `\textit{some invariant property holds a.s.}' (which is the case in this paper; see, e.g., \Cref{thm:indistinguishability}), one may use the ergodic decomposition and assume safely that $[\bs G, \bs o]$ is ergodic from the beginning.}

\subsection{Indistinguishability in General Oriented Forests and Point-Maps}
\label{subsec:indistinguishabilityGeneral}

In order to define  indistinguishability formally, we first provide the definition of invariant properties (which is classical), component-properties (as in~\cite{HuNa17indistinguishability}) and level-set-properties, in what follows.

Recall that $\mathcal G'_*$ is the space of all $[G,o;f]$, where $G$ is a graph or discrete space, $o\in G$ and $f:G\to G$.
In the following definitions, we use the notation $[G,o;f]$ for a variable element of $\mathcal G'_*$.

\begin{definition}[Component-Properties and Level-Set-Properties]
	\label{def:invariantDrainage}
	Let $A\subseteq\mathcal G'_*$ be an event. 
	\begin{enumerate}[label=(\roman*)]
		\item $A$ is called an \defstyle{invariant point-map-property} if it is invariant under changing the root; i.e., if $[G,o;f]\in A$, then $[G,v;f]\in A$ for all $v\in G$. 
		Let $\sigfield{pm}{I}{}$ denote the sigma-field of invariant point-map properties (not to be confused with the sigma-field $\mathcal I$ of invariant properties defined in \Cref{subsec:ergodic}).
		
		\item $A$ is called a \defstyle{component-property} if it is invariant under changing the root in its own component; i.e., if $[G,o;f]\in A$, then $[G,v;f]\in A$ for all $v\in C_f(o)$ (see \Cref{def:component}). Let $\sigfield{comp}{I}{}$ denote the sigma-field of  component-properties.
		
		\item $A$ is called a \defstyle{level-set-property} if it is invariant under changing the root in its own level-set; i.e., if $[G,o;f]\in A$, then $[G,v;f]\in A$ for all $v\in L_f(o)$. Let $\sigfield{level}{I}{}$ denote the sigma-field of  level-set-properties.
	\end{enumerate}
	Given a component-property (resp. level-set-property) $A$, a connected component (resp. level-set) $C$ of the graph of $f$ is called of \defstyle{type} $A$ if $\forall x\in C: [G,x;f]\in A$. Otherwise, $C$ is of type $\neg A$.
\end{definition}

We clearly have the following inclusions: $\sigfield{pm}{I}{}\subseteq \sigfield{comp}{I}{}\subseteq \sigfield{level}{I}{}$.

\begin{definition}[Indistinguishability]
	\label{def:indistinguishability}
	A point-map $\bs F$ on a random rooted graph $[\bs G, \bs o]$ has \defstyle{indistinguishable connected components (resp. level-sets)} if, for every {component-property} (resp. {level-set-property}) $A$, almost surely, either all components (resp. level-sets) are of type $A$ or all are of type $\neg A$.
	In other words, every {factor coloring} of the points of $\bs G$ with two colors (where the coloring may depend on $\bs F$ as well) satisfies the following condition almost surely: If every connected component (resp. level-set) of $\bs F$ is uni-color, then all points of $\bs G$ have the same color.
\end{definition}



Note that ergodicity is not needed in \Cref{def:indistinguishability}. However, to prove or disprove the indistinguishability of a model, by ergodic decomposition, one may safely assume that $[\bs G, \bs o; \bs F]$ is ergodic (we will assume this frequently in the proofs). Indeed, 
by ergodic decomposition and \Cref{lem:happensAtRoot}, one can rewrite indistinguishability in the following equivalent form:

\begin{lemma}[Indistinguishability in the Ergodic Case]
	\label{lem:indistinguishability}
	Let $[\bs G, \bs o]$ be a unimodular graph or discrete space and $\bs F$ be a point-map.
	\begin{enumerate}[label=(\roman*)]
		\item $\bs F$ has indistinguishable components (resp. level-sets) if and only if the same holds in almost every ergodic component of $[\bs G, \bs o; \bs F]$\del{ (see \Cref{subsec:ergodic})}.
		\item In the case where $[\bs G, \bs o; \bs F]$ is ergodic, $\bs F$ has indistinguishable components (resp. level-sets) if and only if every component-property (resp. level-set-property) $A$ is trivial. In the non-ergodic case, indistinguishability is equivalent to $\sigfield{comp}{I}{}=\sigfield{}{I}{}$ mod $\mathbb P$ (resp. $\sigfield{level}{I}{}=\sigfield{}{I}{}$ mod $\mathbb P$).
	\end{enumerate}
\end{lemma}
In the case where $\myprob{A}=1$, all components (resp. level-sets) are of type $A$ a.s., and in the case where $\myprob{A}=0$, all are of type $\neg A$ a.s.

Note that all of the steps mentioned in \Cref{intro:method} are necessary for indistinguishability. This immediately implies the next corollary, which is of independent interest.
	
	\begin{corollary}[Ergodicity of the Ancestral Line]
		Let $\bs F$ be an equivariant point-map on a unimodular graph $[\bs G, \bs o]$ that either has a single connected component, or has indistinguishably components. Then, if $(\bs X_n)_{n\geq 0}$ denotes the ancestral chain of $\bs o$, then $[\bs G, \bs X_0; (\bs X_n)_{n\geq 0}]$ is ergodic under shift (see \Cref{rem:quasi} below). 
	\end{corollary}
	In particular, the loop-erased random walk on a unimodular transient random graph with $\omid{d(\bs o)}<\infty$ is ergodic (see also \Cref{prob:lerw} for its tail-triviality). We will prove this in \Cref{lem:usf-trivial I} and use it in the proof of \Cref{thm:usf-indistinguishability}.
	
	\begin{remark}
		\label{rem:quasi}
		The shift $\sigma$ does not necessarily preserve the distribution of $[\bs G, \bs o; \anc(\bs o)]$. However, one may still define ergodicity since the distribution of $\sigma[\bs G, \bs o; \anc(\bs o)]$ is absolutely continuous with respect to that of $[\bs G, \bs o; \anc(\bs o)]$ (see \Cref{lem:shiftToF(o)}). If the latter is also absolutely continuous with respect to the former, then this notion is a special case of ergodicity for measure-class-preserving dynamical systems.
	\end{remark}
	
As alluded to in the introduction, when the underlying graph is ergodic, indistinguishability can be rephrased as ergodicity in various manners. In~\cite{LySc99}, it is shown that indistinguishability is equivalent to the ergodicity of the lazy simple random walk on $C(\bs o)$.\del{As a result, the probability of every event can be approximated by Birkhoff averages on $C(\bs o)$, although $C(\bs o)$ may be very thin and may occupy an asymptotically zero fraction of the base graph (see \Cref{subsec:bijection}).} In more recent language, if one regards $C(\bs o)$ as a unimodular random graph or discrete space (while keeping the information of the rest of the forest as a decoration), indistinguishability is equivalent to the ergodicity of $C(\bs o)$, see Proposition~5 of~\cite{GaLy09}.\del{(the term \textit{discrete space} is used here because the notion of indistinguishability still makes sense even if $C(\bs o)$ is not equipped with a graph structure and just comes from a partitioning of the vertices; e.g., the partition by level-sets of a CMT)}
Also, indistinguishability can be rephrased as the ergodicity of a certain \textit{measured equivalence relation}: We say that two configurations are equivalent if the second one can be obtained from the first one by changing the root to a new root in the same component. These interpretations are valid for the indistinguishability of \textit{level-sets} as well, as discussed in \Cref{intro:models}.

\subsection{Applications of the Indistinguishability of Level-Sets}
\label{subsec:bijection}

In this subsection, we prove \Cref{lem:average} and an extended version of \Cref{prop:bijection0}.
Let $\bs F$ be an equivariant point-map on a unimodular graph or discrete space $[\bs G, \bs o]$. 
As mentioned in \Cref{intro:indistinguishability}, when the base graph or discrete space is ergodic, the indistinguishability of components (resp. level-sets) can be interpreted as the ergodicity of $C(\bs o)$ (resp. $L(\bs o)$) in some sense. This is a direct consequence of \Cref{lem:indistinguishability}. As a result, the probability of every event $A$ can be approximated by ergodic averages on $C(\bs o)$ (resp. $L(\bs o)$). For instance, if one performs a random walk $(\bs X_n)_{n\geq 0}$ on $C(\bs o)$, one has $\myprob{A} = \lim_n \frac 1n \sum_{i=1}^n \identity{A}([\bs G, \bs X_i; \bs F])$ (see stationarity and ergodicity of random walks in \Cref{lem:stationary,thm:pathIntro}). For level-sets, \Cref{lem:average} states a more intrinsic result. To prove this result, we need the following lemma, where $\ave\{u(x):x\in V\}:=\left(\sum_{x\in V}u(x)\right)/\card{V}$. 
\begin{lemma}[Nested Partitions, \cite{Kh23unimodular}]
	\label{lem:nested}
	Let $[\bs G, \bs o]$ be a unimodular graph or discrete space and let $(\Pi_n)_n$ be a factor nested sequence of finite partitions, where \emph{finite} and \emph{nested} mean that $\card{\Pi_n(u)}<\infty, \Pi_n(u)\subseteq \Pi_{n+1}(u), \forall n\in\mathbb N, \forall u\in \bs G$ a.s. Let $\Pi(u):=\cup_n \Pi_n(u)$. Then, for an arbitrary function $g\in L_1(\mathcal G_*)$, the limit of the averages over $\Pi_n(\bs o)$ defined by
	\[
	\tilde g:=\lim_n \ave\left\{g[\bs G, x]: x\in \Pi_n(\bs o) \right\}
	\]
	exists and is equal to $\omidCond{g}{\sigfield{\Pi}{I}{}}$ a.s. The convergence also holds in $L_1(\mathcal G_*)$.
\end{lemma}
The claim is a special case of Lemma~8.1 of~\cite{Kh23unimodular}, which is included for self-containedness.

\begin{proof}
	Let $g_n$ be the average of $g[\bs G, x]$ over $x\in \Pi_n(\bs o)$. We claim that $g_n=\omidCond{g}{\sigfield{\Pi_n}{I}{}}$ a.s. Assuming this, the nested property implies that $(g_n)_n$ is a backward martingale. Then, the claim is implied by the backward martingale convergence theorem. So, it remains to prove that $g_n=\omidCond{g}{\sigfield{\Pi_n}{I}{}}$ a.s.	
	It is clear that $g_n$ is $\sigfield{\Pi_n}{I}{}$-measurable. Also, for every event $B\in\sigfield{\Pi_n}{I}{}$, one has
	\begin{align*}
		\omid{g_n\identity{B}} &= \omid{\frac 1{\card{\Pi_n(\bs o)}}\sum_{x\in \Pi_n(\bs o)} g(x)\identity{B}(\bs o)}\\
		&= \omid{\sum_{x\in \bs G} \frac 1{\card{\Pi_n(\bs o)}} g(x)\identity{B}(x)\identity{\{x\in \Pi_n(\bs o)\}}}\\
		&= \omid{\sum_{x\in \bs G} \frac 1{\card{\Pi_n(x)}} g(\bs o)\identity{B}(\bs o)\identity{\{\bs o\in \Pi_n(x)\}}}\\
		&= \omid{g(\bs o)\identity{B}(\bs o)},
	\end{align*}
	where we have used the facts $\identity{B}(x)=\identity{B}(\bs o)$ and $\Pi_n(x)=\Pi_n(\bs o)$ for all $x\in\Pi_n(\bs o)$ and the mass transport principle. Now, the claim is implied by the definition of conditional probability.
\end{proof}

\begin{proof}[Proof of \Cref{lem:average}]
	By letting $\Pi_n$ be the partition $\{D_n(x): x\in \bs G\}$,
	the claim is directly implied by \Cref{lem:indistinguishability,lem:nested}.
\end{proof}

To extend and prove \Cref{prop:bijection0}, we need to define
an \defstyle{equivariant bijection between consecutive level-sets of $\bs F$}. By such a bijection, we mean another equivariant point-map $\bs F'$ (which may depend on $\bs F$ as well) such that $\bs F'$ is injective a.s. (as a mapping on $\bs G$) and that $\bs F'(\bs o)\in L_{\bs F}(\bs F(\bs o))$ a.s. By the mass transport principle, this implies that $\bs F'$ is surjective as well, and $\forall x\in \bs G: \bs F'(x)\in L_{\bs F}(\bs F(x))$ a.s. As a result, $\bs F'$ defines a bijection between consecutive level-sets of $\bs F$ indeed.

\begin{proposition} [Bijection Between Level-Sets (Extended Version)]
	\label{prop:bijection}
	Let $\bs F$ be an equivariant point-map on a unimodular graph or discrete space $[\bs G, \bs o]$. 
	Then, there exists an equivariant bijection between consecutive level-sets of $\bs F$ if and only if $\omidCond{d_1(\bs o)}{\sigfield{level}{I}{}}=1$ a.s., where $d_1(\bs o):=\card{\bs F^{-1}(\bs o)}$. In addition, this holds if $\bs F$ has indistinguishable level-sets.
\end{proposition}
Here, by \Cref{lem:nested}, if $\bs S$ denote the previous level-set of $L_{\bs F}(\bs o)$, then $\omidCond{d_1(\bs o)}{\sigfield{level}{I}{}}$ can be interpreted as the \defstyle{relative intensity} of $\bs S$ with respect to $L_{\bs F}(\bs o)$; i.e., the average number of points of $\bs S$ per points of $L_{\bs F}(\bs o)$. For instance, for the canopy tree, this quantity is either 0 or 2 depending on whether $\bs o$ is in the \textit{first} level-set or not. 

\begin{proof}[Proof of \Cref{prop:bijection}]
	First, assume that $\bs F'$ exists. To deduce the equality $\omidCond{d_1(\bs o)}{\sigfield{level}{I}{}}=1$, we should prove that $\omid{d_1(\bs o)\identity{A}}=\myprob{A}$ for every $A\in\sigfield{level}{I}{}$. One has
	\begin{align*}
		\myprob{d_1(\bs o)\identity{A}}&= \omid{\sum_{x\in \bs G} \identity{A} \identity{\{\bs F(x)=\bs o\}}}\\
		&= \myprob{[\bs G, \bs F(\bs o); \bs F]\in A}\\
		&= \myprob{[\bs G, \bs F'(\bs o); \bs F]\in A}\\
		&= \myprob{A}.
	\end{align*}
	In this equation, the second equality holds by the mass transport principle, the third equality holds since $A\in\sigfield{level}{I}{}$, and the last equality holds because $\bs F'$ is bijective, which implies that moving the root to $\bs F'(\bs o)$ preserves the distribution (\Cref{lem:shiftToF(o)}). So, we proved that $\omidCond{d_1(\bs o)}{\sigfield{level}{I}{}}=1$.
	
	Conversely, assume $\omidCond{d_1(\bs o)}{\sigfield{level}{I}{}}=1$. For every $x\in \bs G$, choose a total order on $\bs F^{-1}(x)$, randomly and uniformly (if $[\bs G, \bs o; \bs F]$ has no nontrivial automorphisms a.s., one can choose this order without extra randomness as well). By the forest structure of $\bs F$, this induces a unique total order on every level-set of $\bs F$ such that $x\leq y$ if and only if $\bs F(x)\leq \bs F(y)$ (except when $\bs F(x)=\bs F(y)$). For every $i\in\mathbb Z$, let $h_i(x)$ be the $i$-th point after $x$ in this total order. Let $\tau(x)$ be the smallest $i> 0$ that satisfies $\bs F(h_i(x))\geq h_i(\bs F(x))$. Finally, let $\bs F'(x):=h_{\tau(x)}(\bs F(x))$. This is similar to the right-stable balancing allocation studied in~\cite{LaMoTh14} (between the counting measure on $L_{\bs F}(\bs F(x))$ and the measure $\xi(x):=\card{\bs F^{-1}(x)}$). So, by Theorem~5.6 of~\cite{LaMoTh14} and leveraging unimodularity of the level-set containing the origin, one can show that $\bs F'$ satisfies the claim.
	
	Finally, assuming that the level-sets of $\bs F$ are indistinguishable, we should prove that $\omidCond{d_1(\bs o)}{\sigfield{level}{I}{}}=1$ a.s. By ergodic decomposition, we may assume that $[\bs G, \bs o; \bs F]$ is ergodic. So, \Cref{lem:indistinguishability} implies that $\sigfield{level}{I}{}$ is trivial. Since $\omidCond{d_1(\bs o)}{\sigfield{level}{I}{}}$ is $\sigfield{level}{I}{}$-measurable, it is essentially constant. Since $\omid{d_1(\bs o)}=1$, the constant is equal to 1, and the claim is proved.
\end{proof}

\section{Coalescing Markov Trajectories (CMT) on Unimodular Graphs}
\label{sec:cmt}

In this section, we define the notion of coalescing Markov trajectories (in \Cref{subsec:cmt}), alluded to in \Cref{intro:models}. We define our main models of interest of CMTs (\Cref{model1,model2}) and state the indistinguishability result for them (\Cref{thm:indistinguishability}). The proof of the theorem will be given in \Cref{sec:warmUp,sec:Markov,sec:drainage,sec:cmt-proof}. Before that, \Cref{subsec:examples} discusses how CMTs generalize various models in the literature and studies their indistinguishability and qualitative properties.


\subsection{Definition and Main Models of CMTs}
\label{subsec:cmt}

Let $K:\mathcal G_{**}\to\mathbb R^{\geq 0}$ be a measurable function which satisfies $\sum_v K(G,u,v)=1$ for all $(G,u)$. For every graph or discrete space $G$ and every $u\in G$, one obtains a probability measure $K(G,u,\cdot)$ on $G$, which will be called \textit{the jump distribution} in what follows. So, $K(G,\cdot,\cdot)$ is a Markov kernel on $G$ for every $G$. Note that $K$ is invariant under the isomorphisms of doubly-rooted graphs (since it is well-defined on $\mathcal G_{**}$). We call $K$ a \defstyle{factor Markov kernel}. 

\begin{definition}
	\label{def:cmt}
	Let $K$ be a factor Markov kernel, as defined in the previous paragraph. Define the corresponding \defstyle{Coalescing Markov Trajectories (CMT)} model $\bs F$ as follows: For every graph or discrete space $G$, define a random function $\bs F:=\bs F(G,\cdot): G\to G$ such that $\bs F(u):=\bs F(G,u)$ has distribution $K(G,u,\cdot)$, for every $u\in G$, and $(\bs F(u))_{u\in G}$ are independent.
\end{definition}

The terminology of CMT is chosen because, for every $x\in G$, the path $(\bs F^{(n)}(x))_{n\geq 0}$ until the first self-intersection (if any) is the trajectory of a Markov chain on $G$. Also, two such trajectories on $G$ coalesce and continue together as soon as they intersect ({not necessarily after the same number of steps}). This is not to be confused with coalescing Markov \textit{chains}, which are defined on space-time (see \Cref{subsec:crw}).

It can be seen that a CMT $\bs F$ satisfies the measurability condition in \Cref{def:pointmap}, and is hence a point-map. Conversely, a point map is a CMT if and only if, for every graph or discrete space $G$, the random points $(\bs F(G,u))_{u\in G}$  of $G$ are independent.
We use the notation
\begin{equation}
	K^n(G,u,v):=\myprob{\bs F^{(n)}(u)=v}
\end{equation}
for the $n$-fold iterate of $K$. 

In this paper, we consider CMT‌ models together with a unimodular graph or discrete space $[\bs G, \bs o]$.
Recall from \Cref{subsec:point-map} that, if $[\bs G, \bs o]$ is unimodular, then so is $[\bs G, \bs o; \bs F]$. 
In addition, if $[\bs G, \bs o]$ is ergodic, then $[\bs G, \bs o; \bs F]$ is also ergodic (we leave it to the reader to build $\bs F$ using i.i.d. marks on $\bs G\times \bs G$; i.e., $\bs F$ is a factor of i.i.d.). A formal proof of this fact will also be included in \Cref{thm:tailDrainage}.

At first reading, the reader might focus on the case where $\bs G$ is a deterministic Cayley graph or a transitive graph with a unimodular automorphism group. In particular,
throughout this subsection, we use the following simple-to-state but important example to illustrate the definitions (see e.g., \Cref{ex:nguyen_variant} for an explicit special case). This example generalizes several river models (also called \textit{drainage networks}) and renewal forests in the literature; see \Cref{subsec:examples} for discussion and further examples. 

\begin{example}[Lattice CMTs]
	\label{ex:lattice}
	Let $\Phi$ be a lattice in $\mathbb R^d$ (i.e., a discrete subgroup of $\mathbb R^d$) and let $\mu$ be an arbitrary probability distribution on $\Phi$ (called \textit{the jump distribution}). For every $x\in\Phi$, let $\bs F(x):=x+U(x)$, where $U(x)\in \Phi$ is a random point with distribution $\mu$ chosen independently for all $x$. In this case, a CMT is obtained.
\end{example}

To prove indistinguishability of components, we will require the following assumptions, which constitute one of the main models studied in this paper. 

\begin{model}
	\label{model1}
	Let $\bs F$ be a CMT on a unimodular graph or discrete space $[\bs G, \bs o]$, as in \Cref{def:cmt}, and let $K$ be the underlying factor Markov kernel. For every graph or discrete space $G$, let $\mathbb P_{G}$ denote the distribution of $(\bs F(G,x))_{x\in G}$, which is a probability measure on $G^G$. Let also $b:\mathcal G_*\to\mathbb R^{\geq 0}$ be a given measurable function such that $\omid{b(\bs o)}<\infty$, where $b(\bs o):=b[\bs G, \bs o]$ by \Cref{conv:forgettingG}. In this model, it is also assumed that almost every sample $(G,o)$ of $[\bs G, \bs o]$ satisfies the following conditions (we call $(G,o)$ a \textit{typical sample} in this case):
	\begin{enumerate}[label=(\roman*)]
		\item \label{model1-cycle} (Cycle-free) There is no cycle in the graph of $\bs F$ on $G$, $\mathbb P_G$-a.s.
		\item \label{model1-collision} ({Weak Irreducibility}) The graph with vertex set $G$ and edge set $\{(x,y): K(x,y)+K(y,x)>0 \}$ is connected.
		\item \label{model1-balance} ({Balance}) For the function $b$ assumed above, one has {$b(G,\cdot)>0$} and $\sum_{x\in G} b(x)K(x,y)=b(y), \forall y\in G$. In other words, $b(G,\cdot)$ is an invariant measure for the kernel $K(G,\cdot,\cdot)$.
	\end{enumerate}
\end{model}

For instance, in \Cref{ex:lattice}, the balance condition holds for $b\equiv 1$. In the examples in this paper, we usually have $b\equiv 1$, except in the coalescing simple random walks and Markov chains (\Cref{ex:crw,ex:cmc}). The lattice CMT is cycle-free if and only if $\mu$ is supported in an {open} side of some linear hyperplane. Also, one can show that weak irreducibility holds if and only if the lattice generated by the support of $\mu$ is equal to $\Phi$. See \Cref{subsec:lattice}.

It should be noted that \textit{detailed balance} cannot hold in \Cref{model1}, since the cycle-free condition and $K(x,y)>0$ imply $K(y,x)=0$. Also, in general, $b(G,\cdot)$ is a measure with infinite mass. Indeed, the Markov kernel $K(G,\cdot,\cdot)$ is transient because of the cycle-free condition.

For indistinguishability of level-sets, we will need the following additional assumptions.

\begin{model}
	\label{model2}
	Assume $[\bs G, \bs o; \bs F], K$ and $b$ satisfy the assumptions of \Cref{model1}. Assume in addition that almost every sample $(G,o)$ of $[\bs G, \bs o]$ satisfies the following conditions: 
	\begin{enumerate}[label=(\roman*)]
		\setcounter{enumi}{3}
		\item\label{model2-aper} ({Weak aperiodicity})
		The set  
		\begin{equation}
			\label{eq:aperiodic}
			\{d\in\mathbb N:\; \exists n\in\mathbb N, \exists x,y\in G: K^n(x,y)>0, K^{n+d}(x,y)>0\}
		\end{equation}
		is nonempty and its elements have no common divisor larger than 1.
		\item (One-Endedness) All components of $\bs F$ on $G$ are one-ended and have infinite level-sets, $\mathbb P_G$-a.s. (see also \Cref{thm:classification}).
	\end{enumerate}
\end{model}

One-endedness and weak aperiodicity of \Cref{ex:lattice} will be studied in \Cref{subsec:lattice}.
Note that the conditions of \Cref{model1,model2} are meaningful for a deterministic base graph $G$. 

The following is the main indistinguishability result for CMTs.

\begin{theorem}[Indistinguishability in CMTs on Unimodular Graphs]
	\label{thm:indistinguishability}
	One has
	\begin{enumerate}[label=(\roman*)]
		\item \label{thm:indistinguishability-comp} In \Cref{model1}, $\bs F$ has indistinguishable connected components. 
		\item \label{thm:indistinguishability-foil} In \Cref{model2}, $\bs F$ has indistinguishable level-sets.
	\end{enumerate}
\end{theorem}

\begin{remark}
	\label{rem:model1}
	Here are some remarks about the conditions of \Cref{model1,model2}.
	\begin{enumerate}[label=(\roman*)]
		\item The cycle-free condition can be solely rephrased  in terms of $K$: The graph with edge set $\{(u,v): K(u,v)>0\}$ contains no directed cycle a.s. We will use this fact in \Cref{sec:Markov}.
		
		\item In all the examples of \Cref{model1} given in this paper, the cycle-free condition is ensured by the fact that the jumps $x\mapsto \bs F(x)$ decrease some function $h=h[G,o,x]$ (see \Cref{prob:lyapunov} for the converse). This implies the cycle-free condition if $h$ is a \defstyle{Lyapunov cocycle} which we define as follows: $h$ is a measurable function on $\mathcal G_{**}$ such that $h[G,x,y]+h[G,y,z]=h[G,x,z]$. The value of $h[G,x,y]$ may be regarded as the \textit{height} of $y$ relative to $x$. For instance, in \Cref{ex:lattice}, one might let $h[\Phi,x,y]:=v\cdot(y-x)$ for some unit vector $v\in\mathbb R^d$. Such a function exists naturally when $\bs F$ describes the movement of particles in space-time in a coalescing particle system (see \Cref{ex:crw,ex:cmc}). See \Cref{subsec:examples} for more examples.
		
		\item 
		Weak irreducibility and the balance condition {are not needed for Steps~\ref{step1comp}, \ref{step2comp}, \ref{step1foil} and~\ref{step2foil}. They are only used in Steps~\ref{step3comp} and~\ref{step3foil} and in \Cref{sec:one-ended}. The same applies to the weak aperiodicity condition. So, for indistinguishability results,} one may remove these assumptions and assume that Steps~\ref{step3comp} and/or~\ref{step3foil} hold. {This allows one to have more than one two-ended component; e.g., in \Cref{ex:lattice} (see also \Cref{subsec:lattice}).}
		
		\item 
		The Markov chain with kernel $K(G,\cdot,\cdot)$ on $G$ cannot be irreducible in the classical sense since no two states can be mutually reachable.
		Also, without weak irreducibility, the model can have two completely different behaviors in two or more unrelated parts of the graph (e.g., in the checkerboard lattice, one can consider a point-map on the even sublattice and another one on the odd sublattice).
		
		\item
		A sufficient condition for weak irreducibility is the following variant of the path-intersection property {(compare with \Cref{def:colision})}: $\forall (G,x,y):\probPalm{G}{\anc(x)\cap\anc(y)\neq\emptyset}>0$. But this condition is too strong, see e.g., \Cref{ex:free,ex:dl}. 
		
		\item 
		The balance condition {will be used to show} that $[\bs G, \bs o; \anc(\bs o)]$ (after forgetting the rest of $\bs F$) satisfies a stationarity property spelled out in \Cref{lem:stationary}. This is not just a proof technique, as we will provide examples where indistinguishability fails without the balance condition (see \Cref{ex:canopy,ex:canopy2}). See also \Cref{prob:bias}. 
		
		\item 
		The balance condition for Markov kernels on unimodular graphs is actually studied in~\cite{processes}, where $K$ is called an \emph{environment} and $b$ is called an \emph{initial bias}. The balance condition is called \textit{stationarity} in~\cite{processes}. {Further discussion is provided in \Cref{sec:Markov}.}
		
		
		\item Condition~\ref{model2-aper} is called \textit{weak} aperiodicity since classical aperiodicity makes sense only for irreducible Markov chains. 
		We will strengthen this condition later in \Cref{lem:aperiodicity} and show that the set~\eqref{eq:aperiodic} contains all natural numbers a.s. if all assumptions of \Cref{model2} are satisfied.	
		
		\item Some natural examples do not satisfy the weak aperiodicity condition; e.g., \Cref{ex:nguyen,ex:crw,ex:cmc,ex:free,ex:dl}. 
		In some of these examples, we will prove the indistinguishability of level-sets with another idea. 
		Without weak aperiodicity, the only missing part of the proof is \Cref{step3foil} (see \Cref{rem:periodic2}).
		See also \Cref{rem:periodic,prob:aperiodic,prob:aperiodic2} for the general cases where this condition fails or some weaker conditions hold.
		
		\item Without the one-endedness condition, \Cref{thm:one-ended-or-connected} implies that all level-sets of $\bs F$ are finite a.s., and hence, there is nothing to prove regarding their indistinguishability. Recall that \Cref{thm:one-ended} provides some criteria for one-endedness.
		
	\end{enumerate}
\end{remark}

\subsection{Examples of CMTs}
\label{subsec:examples}

In this subsection, we show that CMTs generalize various models in the literature. We study their indistibuishability properties by \Cref{thm:indistinguishability} and their qualitative properties using the results of \Cref{main:qualitative}.

\subsubsection{Lattice CMTs}
	\label{subsec:lattice}

	To apply \Cref{thm:indistinguishability} to lattice CMTs, we first verify the conditions of \Cref{model1,model2} in the following lemma. The main result is \Cref{prop:latticeIndistinguishability} below, whose proof is postponed to \Cref{sec:warmUp}.
	\begin{lemma}
		\label{lem:lattice}
		Let $\bs F$ be a lattice CMT as described in \Cref{ex:lattice}.
		\begin{enumerate}[label=(\roman*)]
			\item \label{lem:lattice:1} $\bs F$ satisfies the assumptions of \Cref{model1} for $b\equiv 1$ if and only if the support of $\mu$ generates the lattice and is included in an open side of some linear hyperplane.
			\item \label{lem:lattice:2} Assuming the last condition, if either $d\geq 2$ or the jump distribution has infinite mean, then every connected component of $\bs F$ is one-ended.
			\item \label{lem:lattice:3} Assuming the condition mentioned in~\ref{lem:lattice:1}, the weak aperiodicity condition is equivalent to the condition that $\{u-v: u,v\in\supp(\mu)\}$ generates $\Phi$.
		\end{enumerate}
	\end{lemma}
	For $d=1$, $\bs F$ can be either one-ended or two-ended; see e.g., \Cref{ex:renewalEFT}.
	\begin{proof}
		\ref{lem:lattice:1} is already shown after the definition of \Cref{model1}. 
		
		\ref{lem:lattice:2}. By \Cref{thm:one-ended}, it is enough to show that the Green function $g(0,y)$ vanishes as $y$ goes to infinity. This holds more generally for transient random walks on $\mathbb Z^d$; see e.g., Theorem~T2 in Section~24 of~\cite{bookSpitzer76} (Blackwell's theorem also proves the 1-dimensional case with infinite mean). This proves the claim.
				
		\ref{lem:lattice:3}. Let $\Phi'$ be the lattice generated by $\{u-v: u,v\in\supp(\mu)\}$. If $\Phi'=\Phi$, then for arbitrary $w\in\supp(\mu)$, one can write $w=\sum_i c_i(u_i-v_i)$, where $c_i\in \mathbb Z$ and $u_i,v_i\in\supp(\mu)$. This implies that the set~\eqref{eq:aperiodic} contains 1, and weak aperiodicity holds. Conversely, assume that weak aperiodicity holds. For $d$ in the set~\eqref{eq:aperiodic}, one can find $n\in\mathbb N$ and sequences $u_1,\ldots,u_n$ and $v_1,\ldots,v_{n+d}$ in $\supp(\mu)$ such that $\sum_i u_i = \sum_j v_j$. For every $w_1,\ldots,w_d\in\supp(\mu)$, one obtains $\sum_k w_k = (\sum_k w_k + \sum_i u_i) - \sum_j v_j \in \Phi'$. In words, the sum of any $d$ elements of $\supp(\mu)$ is in $\Phi'$. Since such $d$'s have largest common divisor 1, one can obtain that $\supp(\mu)\subseteq\Phi$, and hence, $\Phi=\Phi'$.
	\end{proof}
	
	If the conditions of the last lemma hold, then \Cref{thm:indistinguishability} implies that $\bs F$ has indistinguishable components and level-sets. More generally, the next result shows that some of the assumptions may be dropped. 
	
	\begin{proposition}[Indistinguishability in Lattice CMTs]
		\label{prop:latticeIndistinguishability}
		For lattice CMTs of \Cref{ex:lattice}, the claims of \Cref{thm:indistinguishability} hold even without the weak irreducibility and weak aperiodicity assumptions. More precisely,
		\begin{enumerate}[label=(\roman*)]
			\item If the cycle-free condition holds, then the components are indistinguishable.
			\item If, in addition, the components are one-ended a.s. (see \Cref{lem:lattice}), then the level-sets are indistinguishable.
		\end{enumerate}
	\end{proposition}
	
	The proof is postponed to \Cref{sec:warmUp}.
	In \Cref{subsec:nguyen,subsec:renewal,subsec:crw} below, we will see some instances of lattice CMTs with different characteristics.
	
\subsubsection{Nguyen's River Model and its Variant}
\label{subsec:nguyen}

\begin{example}[Nguyen's River Model]
	\label{ex:nguyen}
	Let $\Phi$ be the set of points $x=(x_1,\ldots,x_d)$ of $\mathbb Z^{d}$ such that $\sum_i x_i$ is even. Let $\bs F(x):=(\bs f_0(x),x_d - 1)$, where $\bs f_0(x)$ is one of the neighbors of $(x_1,\ldots,x_{d-1})$ in the lattice $\mathbb Z^{d-1}$, chosen uniformly at random and independently for all $x\in\Phi$. 
\end{example}

This is an instance of \Cref{model1} with $b\equiv 1$. It generalizes Nguyen's model in~\cite{Ng90}, which is the case $d=2$. For any $d\geq 2$, \Cref{thm:one-ended} easily implies that the connected components are one-ended a.s.
In the case $d=2$ (resp. $d=3$), there is only one connected component and the level-sets are just the horizontal lines (resp. planes) in the lattice. So, for $d\in\{2,3\}$, the indistinguishability of level-sets (\Cref{thm:indistinguishability}) easily follows  from the mixing property of the lattice. 

For $d\geq 4$, there are infinitely many components. This example is an instance of \Cref{model1}, and hence, has indistinguishable components by \Cref{thm:indistinguishability}. 
Despite the fact that it does not fit into \Cref{model2} due to the lack of weak aperiodicity, \Cref{prop:latticeIndistinguishability} implies that the level-sets of this example are indistinguishable.

For following the proof of \Cref{thm:indistinguishability}, a nice example to keep in mind is the variant of Nguyen's model, defined below. This example directly satisfies the assumptions of \Cref{model2} and no further trick is needed.


\begin{example}[Variant of Nguyen's Model]
	\label{ex:nguyen_variant}
	Consider the even lattice $\Phi\subseteq\mathbb Z^2$, as in \Cref{ex:nguyen}. For each $x\in \Phi$, let $\bs F(x)$ be one of $x+(-1,-1), x+(1,-1)$ and $x+(0,-2)$, chosen uniformly at random and independently for all $x$ (one can generalize this to higher dimensions straightforwardly). This model satisfies the conditions of \Cref{model2}. Hence, it has indistinguishable components and level-sets by \Cref{thm:indistinguishability}.
\end{example}

\subsubsection{Renewal Model}
\label{subsec:renewal}

\begin{example}[Renewal model on $\mathbb Z$]
	\label{ex:renewalEFT}
	Let $\Phi=\mathbb Z$ and $\mu$ be a probability measure on $\mathbb N$. For each $x\in\mathbb Z$, let $\bs F(x):=x+\bs U(x)$, where $\bs U(x)>0$ is a random number with distribution $\mu$ chosen independently for all $x\in\mathbb Z$. 
\end{example}

This is an instance of lattice CMTs. By \Cref{lem:lattice}, weak irreducibility holds if and only if $\supp(\mu)$ is non-lattice with respect to $\mathbb Z$ (i.e., not contained
in any $d\mathbb{Z}$, with $d>1$). In this case, weak aperiodicity holds if and only if $\supp(\mu)-z$ is non-lattice for some $z\in\supp(\mu)$.
However, even without weak irreducibility and weak aperiodicity, \Cref{prop:latticeIndistinguishability} implies that $\bs F$ has indistinguishable components. Also, if the components are one-ended, then the level-sets are indistinguishable. 

The one-endedness of the model is characterized by the following results of~\cite{BaSo19renewal,BaMiKh24coupling}: If $\omid{\bs U(x)}<\infty$ and $\supp(\mu)$ is non-lattice, there exists a single component, which is two-ended. Connectedness is also implied by the path-intersection property (see \Cref{thm:infComps}), and two-endedness is also implied by \Cref{thm:one-ended} and Blackwell's theorem, which shows that the Green function tends to a positive constant as distance tends to infinity. If $\omid{\bs U(x)}<\infty$ and $\supp(\mu)\subseteq d\mathbb N$, with $d$ minimal, then there are $d$ two-ended components.

If $\omid{\bs U(x)}=\infty$, then this model has one-ended components by \Cref{lem:lattice} (and has either one or infinitely many components in the non-lattice case). 
Also, \cite{BaMiKh24coupling} shows that when $\mu$ is the Pareto distribution with exponent $\beta$, there is a phase transition for connectedness: If $\beta<\frac 12$, there are infinitely many one-ended components, and if $\frac 12\leq \beta<1$, there is a single one-ended component.

%

\subsubsection{Coalescing Random Walks and Markov Chains}
\label{subsec:crw}

Another fundamental example of CMTs is \textit{coalescing Markov chains} (\Cref{ex:crw,ex:cmc} below), which include \textit{coalescing random walks}. This is not to be confused with coalescing Markov \textit{trajectories}. Here, the model is space-time; i.e., $\bs G\times \mathbb Z$, whereas the model of \Cref{def:cmt} is on $\bs G$.

\begin{example}[Coalescing Random Walks]
	\label{ex:crw}
	Let $[\bs G, \bs o]$ be a unimodular graph such that $\omid{\deg(\bs o)}<\infty$. At each time $t\in\mathbb Z$ and each vertex $x\in\bs G$, start an independent simple random walk from $x$ under the condition  that the particles coalesce when they meet at the same space-time point. To state this in the setting of \Cref{model1}, let $\bs G':=\bs G\times\mathbb Z$ and $\bs o':=(\bs o, 0)$ and add marks to $\bs G'$ for distinguishing the time direction. For each $(x,t)\in \bs G'$ and $y\sim x$, let $K((x,t),(y,t+1)):=1/\deg(x)$. Then, $K$ is a factor Markov kernel. Let $\bs F$ be the resulting CMT.
\end{example}

First, we study which conditions of \Cref{model1,model2} hold:
\begin{lemma}
	\label{lem:crw}
	In \Cref{ex:crw},
	\begin{enumerate}[label=(\roman*)]
		\item $\bs F$ is cycle-free and the balance condition is satisfied with $b(x,t):=\deg_G(x)$.
		\item  Weak irreducibility holds if and only if $\bs G$ is non-bipartite a.s.
		\item Weak aperiodicity \emph{does not} hold.
		\item If $\bs G$ is infinite a.s. and has bounded degrees, then all components of $\bs F$ are one-ended a.s. (see \Cref{prob:crw}).
	\end{enumerate}
\end{lemma}
Note that the connectedness or disconnectedness of $\bs F$ depends on the path-intersection property; see \Cref{thm:infComps}.
\begin{proof}
	The first three parts are clear (note that the simple random walk has period 2 in bipartite graphs). For the last part, we use \Cref{thm:one-ended}. Denote $\bs F^{(n)}(\bs o')$ by $\bs X_n=:(\bs Y_n,n)$, where $(\bs Y_n)_n$ is distributed as the simple random walk on $\bs G$. 
	Let $p$ be the kernel of $(\bs Y_n)_n$. By the definition of the Green function $g_K$ and the fact that $\mathrm{deg}(x)p^n(x,y) = \mathrm{deg}(y)p^n(y,x)$, one obtains
	\begin{eqnarray*}
		\omid{\frac{\mathrm{deg}(\bs o)}{\mathrm{deg}(\bs Y_n)}g_K(\bs o, \bs X_n)} &=& \omid{\frac{\mathrm{deg}(\bs o)}{\mathrm{deg}(\bs Y_n)} p^n(\bs o, \bs Y_n)}\\
		&=& \omid{p^n(\bs Y_n,\bs o)}\\
		&=& \omid{p^{2n}(\bs o, \bs o)},
	\end{eqnarray*}
	where the last equality holds because going $n$ steps further from $\bs Y_n$ is equivalent to going $2n$ steps from $\bs o$.
	The last term converges to zero by Blackwell's renewal theorem (see e.g., \cite{Li77probabilistic}) for the return times to the root. Therefore, by the boundedness of ${\mathrm{deg}(\bs o)}/{\mathrm{deg}(\bs Y_n)}$, one obtains \Cref{eq:greendecay}. Hence, \Cref{thm:one-ended} implies that $\bs F$ has one-ended components a.s. 
\end{proof}

Now, we study indistinguishability of components and level-sets:

\begin{proposition}[Indistinguishability in Coalescing Random Walks]
	\label{prop:crw}
	In \Cref{ex:crw}, 
	\begin{enumerate}[label=(\roman*)]
		\item The connected components are indistinguishable.
		\item If the level-sets are infinite a.s., then there are at most two indistinguishability classes for the level-sets. If, in addition, $\bs G$ is non-bipartite a.s., then all level-sets are indistinguishable.
	\end{enumerate}
\end{proposition}
Note that, due to the lack of weak aperiodicity, our results do not directly imply the indistinguishability of level-sets in this example. Nevertheless, we may prove this with another idea. The reader might postpone reading the rest of these arguments to after seeing the proof of \Cref{thm:indistinguishability}.

\begin{proof}
	If $\bs G$ is non-bipartite a.s., then the first claim follows from \Cref{thm:indistinguishability} and \Cref{lem:crw}. Otherwise, one may partition $\bs G'$ into two isomorphic weakly irreducible parts (each part is a discrete metric space although not a graph). So, the assumptions of \Cref{model1} hold and $\bs F$ has indistinguishable components by \Cref{thm:indistinguishability}. 
	
	For the second part, we cannot directly apply \Cref{thm:indistinguishability} due to the lack of weak aperiodicity.
	The only remaining part of the proof is showing the tail triviality of the ancestor line in \Cref{thm:pathIntro} (see \Cref{prob:aperiodic} and the discussion around it). One can show that the tail triviality of $[\bs G', \bs X_n; (\bs X_{n+i})_{i\geq 0}]$ is equivalent to the tail triviality of $[\bs G, \bs Y_n; (\bs Y_{n+i})_{i\geq 0}]$ (since there is no additional information in the second coordinates of $((\bs Y_n,n))_n$). The latter is directly implied by \Cref{thm:pathIntro} in the non-bipartite case. This implies that the level-sets of $\bs F$ are indeed indistinguishable. However, if $\bs G$ is bipartite and the two parts are distinguishable, then the level-sets of $\bs F$ are distinguishable by parity (and there are exactly two indistinguishability classes of level-sets).
\end{proof}

\del{
\begin{remark}
	Given $k\geq 2$, the $k$-tuples of components might or might not be indistinguishable by tail component-properties in $\mathcal G'_*$. For instance, if $\bs G$ is the 3-regular tree, then 3-tuples are indistinguishable but 4-tuples are distinguishable. This can be seen by \Cref{thm:fixed-comp}.
\end{remark}
}

\begin{example}[Coalescing Markov Chains]
	\label{ex:cmc}
	Let $[\bs G, \bs o]$ be a unimodular graph or discrete space and let $b_0\geq 0$ be an equivariant function on $\bs G$ such that $\omid{b_0(\bs o)}<\infty$. Assume $b_0(\bs G, \cdot)$ is an invariant measure for a given factor Markov kernel $p$ on $\bs G$ and satisfies the detailed balance condition $b_0(x)p(x,y)=b_0(y)p(y,x)$. Then, the model of \defstyle{coalescing Markov chains} is defined on $\bs G':=\bs G\times \mathbb Z$ similarly to \Cref{ex:crw}, namely, $K((x,t),(y,t+1)):=p(x,y)$.
\end{example}
This generalizes \Cref{ex:crw} (let $b_0(x):=\deg_G(x)$ and $p$ be the kernel of the simple random walk).
The last proof can be modified slightly to prove the following (see \Cref{thm:pathIntro} and \Cref{rem:periodic}).
\begin{proposition}[Indistinguishability in Coalescing Markov Chains]
	\label{prop:cmc}
	In \Cref{ex:cmc}, assume in addition that the Markov chain on $\bs G$ satisfies weak irreducibility and the relaxed version of weak aperiodicity by allowing the gcd of~\eqref{eq:aperiodic} be $d<\infty$. Then, the components are indistinguishable and the number of indistinguishability classes of level-sets is a divisor of $d$.
\end{proposition}


\subsubsection{Discrete-Time Voter Model}
\label{subsec:voter}

The voter model describes the evolution of opinions in a graph. To the best of our knowledge, on infinite graphs, this model is discussed only in continuous-time in the literature. Here, we provide a discrete-time version and discuss its dual representation and stationary version. We also prove the indistinguishability of clusters, which is new. 

\begin{example}[Discrete-Time Voter Model]
	\label{ex:voter-discrete}
	Let $G$ be a graph or discrete space and $r:G\times G\to[0,1]$ be the rate function of opinion spread. Each $(x,y,i)\in G\times G\times \mathbb N$ is called a potential trigger at time $i$. Assume each trigger is activated with probability $r(x,y)$ independently from all other triggers. 
	Start from a partition of $G$ at time $t=0$, whose clusters represent the vertices with the same opinions (at the beginning, all vertices have different opinions and the values of the opinions are irrelevant for us, but it is straightforward to modify the definitions if one is interested in specifying the opinions by marks in $(0,1)$). 
	Inductively, proceed as follows. At time $t+1$, each active trigger $(x,y,t+1)$ tries to push the opinion of $x$ to $y$; i.e., to cut $y$ from its own cluster and attach it to the cluster of $x$. If there are more than one active triggers of the form $(\cdot,y,t+1)$, choose one of them randomly and uniformly (or with a prescribed distribution defined measurably and not depending on time). So, at each time $t\in \mathbb N\cup\{0\}$, a partition of $G$ is obtained. We call this the \defstyle{discrete-time multi-type voter model}.
	
	This model has a dual representation by coalescing random walks. For each $(x,t)\in G\times \mathbb N$, if we trace backwards the opinions leading to $(x,t)$, one can see that a backward Markov chain is obtained. Doing this for all $(x,t)$, one obtains a coalescing Markov chain model as in \Cref{ex:cmc}. Every level-set of $\bs F$ represents a set of vertices with the same opinion at a given time. This immediately gives a \defstyle{stationary version} of the voter model: If $\bs F$ is the backward coalescing Markov chain model on $G\times \mathbb Z$, consider the partition obtained by the level-sets of $\bs F$. In addition, one can prove that this is the \textit{finest} stationary distribution; i.e., the partitions of every other stationary distribution of the model are coarser than this specific one (in the coupling given by the dual representation).
	
	Now, let $[\bs G, \bs o]$ be a unimodular graph or discrete space. One may define a voter model as above, assuming that all parameters ($r$ and the choice at the ties) are factors. We need to assume that the kernel of the backward chains satisfies the balance condition of \Cref{ex:cmc}. For example, this holds if $\bs G$ is a graph, $r$ is constant on the edges and $\omid{\mathrm{deg}(\bs o)}<\infty$ (hence, the backward chains are just lazy simple random walks). In this case, \Cref{prop:cmc} implies \Cref{prop:voter}.
	
	The analogous result for the continuous-time voter model is discussed in \Cref{ex:voter-continuous}. Also,
	similar indistinguishability results may also be given for the voter model on deterministic graphs in the light of \Cref{subsec:tail-fixed}. 
\end{example}

\subsubsection{CMTs on Free Groups and Diestel-Leader Graphs}

The following examples show CMTs where the forward paths from two vertices may have zero chance to meet, and also where the weak aperiodicity condition is not satisfied.

\begin{example}[Free Groups]
	\label{ex:free}
	Let $G$ be the free group generated by $k$ elements, namely, $e_1,\ldots,e_k$. To distinguish the generators, put mark $i$ on every pair $(x,xe_i)$. Then, for each $x\in G$, let $\bs F(x)$ be one of $xe_1,\ldots,xe_k$ chosen uniformly at random. This satisfies the assumptions \Cref{model1} with $b\equiv 1$, and hence, the components of $\bs F$ are indistinguishable (note that there are infinitely many components because the path-intersection property is not satisfied, see \Cref{thm:infComps}). In fact, the descendants of a given $x\in G$ form a critical Galton-Watson tree and $C(x)$ is an \textit{Eternal Galton-Watson tree} \cite{eft}. So, the components are one-ended a.s. However, the weak irreducibility condition does not hold. We guess that the level-sets are distinguishable as well, but this is not implied by our results (see \Cref{prob:aperiodic2}).
	

\end{example}

\begin{example}[Diestel-Leader Graph]
	\label{ex:dl}
	Given $2\leq q\in\mathbb N$, let $T_q$ be the $(q+1)$-regular tree. Let $h$ be the height function with respect to a distinguished end of $T_q$. The (unimodular) Diestel-Leader graph $DL(q,q)$ is defined on the vertex set $\{(x,y)\in T_q^2: h(x)+h(y)=0\}$. Two vertices $(x,y)$ and $(x',y')$ are adjacent if and only if $x'$ and $y'$ are adjacent to $x$ and $y$ respectively. In this case, let $K((x,y),(x',y'))=1/q$ only if $h(x')=h(x)-1$. This model satisfies all of the assumptions except weak aperiodicity. However, similarly to \Cref{ex:crw}, one can prove that the components and level-sets are indistinguishable.\\
	As a variant, assume the vertices have i.i.d. marks $\bs I(t)\in\{1,2\}$ for $t\in\mathbb Z$. Let $K((x,y),(x',y'))=q^{-\bs I(h(x))}$ if $d((x',y'),(x,y))=\bs I(h(x))$ and $h(x')=h(x)-\bs I(h(x))$. Does this model have indistinguishable level-sets? (see \Cref{prob:aperiodic2}).\\
	The reader is invited to define similar models on lamp-lighter groups.
\end{example}

\subsubsection{Distinguishability Without the Balance Condition}
\label{subsec:distinguishable}

The following are examples in which the balance condition are not satisfied and level-sets or components are distinguishable. These examples are inspired from Example~14.15 of~\cite{bookLyPe16}, which is an example of a Markov chain on $\mathbb Z$ with a nontrivial tail sigma-field. The canopy tree is used to obtain a unimodular analogue of the example. We recall from \cite{processes} that the canopy tree consists of a sequence of layers, say $H_0, H_1,\ldots$ with the following properties: For all $i$, each point $x\in H_i$ has exactly one neighbor in $H_{i+1}$, which is called the \textit{parent} of $x$. Also, $x$ has exactly $2$ neighbors in $H_{i-1}$ if $i>0$. Denote the parent of $x$ by $p(x)$. Then, if the root $\bs o$ is chosen such that $\myprob{\bs o\in H_i}=2^{-i-1}$, a unimodular tree is obtained.


\begin{example}[Distinguishable Level-Sets]
	\label{ex:canopy}
	Consider the following kernel on the canopy tree (defined above): Let $l(x)\geq 0$ denote the index of the layer containing $x$. Let $K(x,p(x)):=1-2^{-l(x)-1}$, $K(x,p^{(2)}(x)):=2^{-l(x)-1}$. It is easy to check that the graph of the resulting CMT $\bs F$ is cycle-free and connected. Weak irreducibility holds, but the balance condition is not satisfied for any function $b$ (such a $b$ should be zero on $H_0$, and inductively be zero on all $H_i$'s). We claim that $\bs F$ has distinguishable level-sets.
	Indeed, if $n(x)$ is the number of $n\geq 0$ in which $\bs F^{(n+1)}(x)\neq p(\bs F^{(n)}(x))$, then $l+n$ is an invariant of the level-sets of $\bs F$ and distinguishes all of them.
\end{example}

\begin{example}[Distinguishable Components]
	\label{ex:canopy2}
	Let $\bs T$ be the canopy tree and $\bs G:=\bs T\times \{0,1\}$. We may put the mark 0 on $\bs T\times \{0\}$ and the mark 1 on $\bs T\times \{1\}$. So, $\bs G$ contains two distinguishable copies of $\bs T$ (alternatively, one might append a leaf to every vertex of $\bs T\times \{0\}$ in order to distinguishing the two copies of $\bs T$).
	Define the kernel $K$ for $v\in\bs T$ and $i\in\{0,1\}$ by $K((v,i),(p(v),i)):=1-2^{-i-1}$ and $K((v,i),(p(v),1-i)):=2^{-i-1}$. This model is cycle-free and weakly-irreducible, but the balance condition does not hold. We claim that $\bs F$ has distinguishable components.
	By the Borel-Cantelli lemma, one can check that $\anc(\bs o)$ contains finitely many \textit{jumps} between $\bs T\times\{0\}$ and $\bs T\times \{1\}$ a.s. For $i=0,1$, let $\bs S_i$ be the set of vertices $v$ in which $\anc(v)$ eventually lies in $\bs T\times \{i\}$. One can check that $\bs S_0$ and $\bs S_1$ are precisely the connected components of $\bs F$ a.s. One has $0<\myprob{\bs o\in \bs S_i}<1$. Hence, ergodicity implies that $\bs S_0$ and $\bs S_1$ are nonempty a.s. Now, we have distinguished the two connected components.
\end{example}

\section{A Simple Proof for Indistinguishability in Lattice CMTs (and More Models)}
\label{sec:warmUp}

In this section, we prove \Cref{thm:indistinguishability} in the special case of the lattice models of \Cref{ex:lattice} and \Cref{subsec:lattice}, assuming that Steps~\ref{step1comp} and~\ref{step1foil} are proved (which will be done in \Cref{sec:drainage}).
The benefit is that the proof is much simpler than the general case of \Cref{thm:indistinguishability} (and simpler than other indistinguishability proofs in the literature), but shows many underlying proof ideas. 
In addition, we will show in \Cref{prop:latticeIndistinguishability} that the assumptions of weak irreducibility and weak aperiodicity are not necessary for lattice CMTs.

The main proof technique consists in constructing a coupling of $\bs F$ with another copy $\bs F'$ such that $\restrict{\bs F'}{\oball{n}{o}}$ is independent from $\bs F$, but the two ancestor lines of $o$ shift-couple. The proof is provided in \Cref{prop:simpler,prop:latticeIndistinguishability} below, and is  based on the following lemma. In \Cref{subsec:wusf-indistinguishability}, a similar coupling will be used to obtain a simple proof of indistinguishability for the wired uniform spanning forest on $\mathbb Z^d$.

\begin{lemma}
	\label{lem:coupling}
	In the lattice CMTs of \Cref{ex:lattice}, given $x,y\in \Phi$, let $(\bs X_n)_n$ and $(\bs Y_n)_n$ be Markov chains on $\Phi$ with kernel $K$ and starting from $x$‌ and $y$ respectively.
	\begin{enumerate}[label=(\roman*)]
		\item \label{lem:coupling:shift} (Shift-Coupling). If $\supp(\mu)$ generates the lattice, then, for every $x,y\in \Phi$, there exists a coupling of $(\bs X_n)_n$ and $(\bs Y_n)_n$ such that, $\exists k,N: \forall n\geq N: \bs X_{n}=\bs Y_{n+k}$ a.s.
		\item \label{lem:coupling:successful} (Successful Coupling). If, in addition, $\{u-v: u,v\in\supp(\mu)\}$ generates $\Phi$, then, for every $x,y\in \Phi$, there exists a coupling of $(\bs X_n)_n$‌ and $(\bs Y_n)_n$ such that $\exists N: \forall n\geq N:\bs X_n=\bs Y_n$ a.s.
	\end{enumerate}
\end{lemma}

\del{Note that the converses of these claims also hold.}
In dimension 1, The second statement is a slight generalization of Ornstein's coupling; see Example~5.7.7 of~\cite{bookDu19durrett} and note that the assumptions of irreducibility and aperiodicity in the proof of~\cite{bookDu19durrett} can be weakened to weak irreducibility and weak aperiodicity. A generalization to higher dimensional lattices (and more general Abelian groups) is proved in \cite{Mu19coupling} as mentioned in the following proof.

\begin{proof}
	Let $\Phi'$ be the sub-lattice generated by $\{u-v: u,v\in \supp(\mu)\}$. By Theorem~1.3 of~\cite{Mu19coupling}, $\Phi'$ is equal to the set of $y\in \Phi$ such that there exists a successful coupling of $(\bs X_n)_n$ and $(\bs Y_n)_n$ starting from $\bs X_0=0$‌ and $\bs Y_0=y$. This immediately implies \ref{lem:coupling:successful}. Now, we prove \ref{lem:coupling:shift}. One has $\supp(\mu)\subseteq \Phi'+w$ for some $w\in \Phi$. Since $\supp(\mu)$ generates $\Phi$, one obtains that $\Phi'+w\mathbb Z=\Phi$. Hence, given any $x,y\in \Phi$, there exists $m\in\mathbb Z$ such that $y-x\in \Phi'+ mw$. If $m\geq 0$, construct $(\bs X_i)_{i\leq m}$ starting from $\bs X_0:=x$. For all $i\geq 0$, one has $\bs X_i\in x+ \Phi' + iw$, and thus, $y-\bs X_m\in\Phi'$. So, there exists a successful coupling for the Markov chains starting from $\bs X_m$‌ and $y$. This gives a shift-coupling of $(\bs X_n)_n$ and $(\bs Y_n)_n$. If $m<0$, construct $(\bs Y_i)_{i\leq -m}$ starting from $\bs Y_0:=y$, which implies that $\bs Y_{-m}-x\in\Phi'$. So, there exists a successful coupling for the Markov chains starting from $x$‌ and $\bs Y_{-m}$. By the same argument as above, this gives a shift-coupling of $(\bs X_n)_n$ and $(\bs Y_n)_n$ and the claim is proved.
\end{proof}

In the case of one-dimensional Markov chains, the successful coupling of \Cref{lem:coupling} is used in Theorem~5.7.6 and Example~5.7.7 of~\cite{bookDu19durrett} to prove the tail triviality of the Markov chain. This establishes \Cref{step3foil} in this case. In the following proposition, we return to general CMTs and show that the existence of a shift-coupling (resp. successful coupling) can be used to deduce indistinguishability directly from \Cref{step1comp} (resp. \Cref{step1foil}). Note that there exist examples which do not have the shift-coupling property; e.g., the coalescing simple random walk on the regular tree (\Cref{ex:crw}).

\begin{proposition}[Simpler Case for Indistinguishability]
	\label{prop:simpler}
	One has
	\begin{enumerate}[label=(\roman*)]
		\item \label{prop:simpler-comp} In \Cref{model1}, if $[\bs G, \bs o]$ has the shift-coupling property a.s. (i.e., satisfies the claim of part~\ref{lem:coupling:shift} of \Cref{lem:coupling}), then $\bs F$ has indistinguishable connected components. 
		\item \label{prop:simpler-foil} In \Cref{model2}, if $[\bs G, \bs o]$ has the successful coupling property a.s. (i.e., satisfies the claim of part~\ref{lem:coupling:successful} of \Cref{lem:coupling}), then  $\bs F$ has indistinguishable level-sets.
	\end{enumerate}
	In particular, these claims hold for lattice CMTs under the assumptions of \Cref{lem:coupling}.
\end{proposition}

This is a special case of \Cref{thm:indistinguishability}, but we present a separate proof which is significantly simpler. The goal is to highlight the differences with the proof of the general case of \Cref{thm:indistinguishability}. So, we will make further assumptions in the proof (assuming one-endedness and that \Cref{step1comp,step1foil} hold) for simplicity of arguments.

\begin{proof}
	We may assume that $[\bs G, \bs o]$ is ergodic.
	Let $A$ be a component-property (resp. level-set-property). It is enough to prove that $\myprob{A}\in\{0,1\}$.
	As mentioned before the proposition, we assume that \Cref{step1comp} (resp. \Cref{step1foil}) holds (this will be proved for general point-maps in \Cref{sec:drainage}, but can also be proved for CMTs using the method of \textit{pivotal updates} of~\cite{LySc99,HuNa17indistinguishability}). Therefore, it is enough to assume that $A$ is a tail component-property (resp. tail level-set-property). We also assume that $\bs F$‌ is one-ended a.s. for simplicity (otherwise, for part~\ref{prop:simpler-comp}, one needs to assume that $A$‌ is a \textit{tail branch-property}; see \Cref{sec:drainage}).
	
	The idea of the proof is to show that $A$ is independent from the restriction of $\bs F$ to $\oball{n}{\bs o}$ (conditionally on $[\bs G, \bs o]$) for all $n$, and to deduce that $A$ is independent of itself.
	Let $n\geq 0$ be arbitrary. For every sample $(G,o)$ of $[\bs G, \bs o]$, construct another copy $\bs F'$ of $\bs F$ on $G$ as follows. Inside $\oball{n}{o}$, construct $\bs F'$ as an independent copy of $\bs F$ (given $(G, o)$). This constructs a part of $\anc_{\bs F'}(o)$. 
	By assumption, one can continue the ancestor line of $o$ in $\bs F'$ by coupling it with $\bs F$ in such a way that $\anc_{\bs F'}(o)$ shift-couples (resp. successfully couples) with $\anc_{\bs F}(o)$. The possible issue is that $\anc_{\bs F'}(o)$ hits $\oball{n}{o}$ before the coupling time. If this occurs, say at $x=\bs F'^{(T)}(o)$‌ for the first time, forget the coupling after time $T$ since $\bs F'(x)$ is already defined. Instead, after time $T+1$, use a similar coupling starting from $\bs F'(x)$ and repeat if $\oball{n}{o}$ is hit again. Using the fact that every point of $\oball{n}{o}$ appears at most once in $\anc_{\bs F'}(o)$ (by the cycle-free condition), this procedure repeats at most finitely many times. So, at the end, $\anc_{\bs F'}(o)$ is constructed such that it shift-couples (resp. successfully couples) with $\anc_{\bs F}(o)$ a.s. Let $E$‌ denote the event where this last property holds, which satisfies $\myprob{E}=1$.
	
	Finally, for each $x\not\in\oball{n}{o}$ where $\bs F'(x)$ is not already defined, let $\bs F'(x):=\bs F(x)$. The definition implies that $\bs F'$ has the same distribution as $\bs F$ given $(G, o)$ and given $\restrict{\bs F}{\oball{n}{o}}$).
	
	Let $A'$ be the event $[\bs G, \bs o; \bs F']\in A$. By construction, $A$ is independent from $\restrict{\bs F'}{\oball{n}{o}}$ (conditionally on $[\bs G, \bs o]$). Also, the construction shows that $\bs F'$ is a finite modification of $\bs F$ and the component (resp.  level-set) of $\bs o$ is modified at only finitely many points (one the event $E$). Therefore, since $A$ is a tail component-property (resp. level-set-property), one obtains that $A\cap E = A'\cap E$. Hence, $\myprob{A\Delta A'}=0$. This shows that $A'$ is independent from $\restrict{\bs F'}{\oball{n}{o}}$ (conditionally on $[\bs G, \bs o]$). Since this holds for all $n$, one obtains that $A'$ is independent from $\bs F'$ (conditionally on $[\bs G, \bs o]$). Therefore, $A'$ is independent from itself (conditionally on $[\bs G, \bs o]$). So, $\probCond{A'}{\bs G, \bs o}\in\{0,1\}$, and hence, $\probCond{A}{\bs G, \bs o}\in\{0,1\}$.
	
	For every sample $(G,o)$ of $[\bs G, \bs o]$, let $h(G,o):=\probPalm{G}{[G,o;\bs F]\in A}$, where $\mathbb P_{G}$ is defined in \Cref{model1}. By the last paragraph, $h\in\{0,1\}$ for almost every sample $(G,o)$. Note that for $u,v\in G$, on the event $u\in C(v)$ (resp.  $u\in L(v)$), one has $\identity{A}[G,u]=\identity{A}[G,v]$ by the definition of $A$. Therefore, if $\probPalm{G}{u\in C(v)}>0$ (resp. $\probPalm{G}{u\in L(v)}>0$), then $h(G,u)=h(G,v)$. 
	So, $h$ is constant if we show that the graph with vertex set $G$ and edge set $E_C:=\{(u,v): \probPalm{G}{u\in C(v)}>0\}$ (resp. $E_L:=\{(u,v): \probPalm{G}{u\in L(v)}>0\}$) is connected. The connectedness of $E_C$ is directly implied by the weak irreducibility condition. That of $E_L$ can also be proved by the weak aperiodicity condition, but this is a little more involved and we prove it in \Cref{lem:foil-connectivity}.
	
	The last paragraph shows that, in almost every sample $(G,o)$, either $\forall u\in G: \probPalm{G}{[G,u;\bs F]\in A}=0$ or $\forall u\in G: \probPalm{G}{[G,u;\bs F]\in A}=1$. This implies indistinguishability and the claim is proved.
%
%
\end{proof}

The last result implies \Cref{prop:latticeIndistinguishability} as follows.

\begin{proof}[Proof of \Cref{prop:latticeIndistinguishability}]
	%
	As mentioned after the definition of \Cref{model1}, the balance condition holds for $b\equiv 1$.
	Let $\Phi'$ and $w$ be as in the proof of \Cref{lem:coupling}. 
	The idea is that $\bs F$ has similar behaviors in all translates of $\Phi'$. This is formalized as follows.
	
	One has $\bs F(\Phi+mw)\subseteq \Phi+(m+1)w$ for all $m\in\mathbb Z$. One can repeat the proof of \Cref{prop:simpler} until deducing that $h\in\{0,1\}$, but the graphs with edge sets $\{(u,v): \probPalm{G}{u\in C(v)}>0\}$ or $\{(u,v): \probPalm{G}{u\in L(v)}>0\}$ are possibly disconnected. 
	Repeating the proof is possible since, in the proof, the constructed parts of $\anc_{\bs F}(\bs o)$ and $\anc_{\bs F'}(\bs o)$ start from the same vertex. This implies that $\forall n: \bs F^{(n)}(\bs o)-\bs F'^{(n)}(\bs o)\in \Phi'$, and hence, the successful coupling in the proof can always be constructed for lattice CMTs with the minimal assumptions of this proposition. Now, since $\Phi$‌ is transitive, $h$ should be constant. So, either $h\equiv 0$ or $h\equiv 1$. This implies the claim.
\end{proof}

\section{Steps C3 and L3: Ergodicity and Tail Triviality of Markov Chains on Unimodular Spaces}
\label{sec:Markov}

In this section, as announced in \Cref{main:markov}, we study Markov chains on unimodular graphs or discrete spaces. We prove the results stated in \Cref{main:markov}; i.e., stationarity (after a suitable biasing), ergodicity and tail-triviality for such Markov chains. These results establish \Cref{step3comp} and \Cref{step3foil}, but recall that there is no CMT in these results.

%
%

\subsection{Stationarity}
\label{subsec:stationary}

In this subsection, we generalize Theorem~4.1 of~\cite{processes} regarding the stationarity of Markov chains on unimodular graphs.
Fix a factor Markov kernel $K$ as in \Cref{subsec:cmt}. 
Given a sample $(G,o)$, let $(\bs X_n)_{n\geq 0}$ be a Markov chain on $G$ with kernel $K$ and starting from $o$. Assume that the balance condition of \Cref{model1} is satisfied by $K$ for some initial bias $b:\mathcal G_*\to\mathbb R^{\geq 0}$ (note again that no CMT‌ is required for this assumption). So, the measure $b(G,\cdot)$ is an invariant measure for the Markov chain on $G$. In general, $b(G,\cdot)$ is an infinite measure and the Markov chain on $G$ might be transient.

Below, we define a different Markov chain on the space $\mathcal G_{\infty}$ of all objects $[G,x_0;(x_i)_{i\geq 0}]$, where $G$ is a graph or discrete space and $(x_i)_i$ is a sequence of points in $G$ . 
Consider the shift operator $\sigma$ on $\mathcal G_{\infty}$ defined as follows:
\begin{equation}
	\label{eq:shift}
	\begin{split}
		\sigma[G,x_0;(x_i)_{i\geq 0}]&:= [G,x_1;(x_{i+1})_{i\geq 0}].
	\end{split}
\end{equation}
It can be proved that the map that assigns, to every $[G,o]$, the distribution of $[G,o;(\bs X_n)_{n\geq 0}]$ (which is a probability measure on $\mathcal G_{\infty}$), is measurable. So, given a random graph or discrete space $[\bs G, \bs o]$, the random element $[\bs G, \bs o; (\bs X_n)_{n\geq 0}]$ of $\mathcal G_{\infty}$ is well defined. \del{The\mar{\ali{Since \cite{processes} already discusses stationarity, we may move further explanations to the book.}} next lemma studies when the distribution of $[\bs G, \bs o; (\bs X_n)_n]$ is stationary (i.e., invariant) under the shift $\sigma$, possibly after a biasing.
Note that the Markov chain $(\bs X_n)_n$ on $G$ might be transient (e.g., in \Cref{model1}) and might not have any stationary probability measure, in contrast to the Markov chain $([\bs G, \bs X_n; (\bs X_{n+i})_{i\geq 0}])_n$ on $\mathcal G_{\infty}$.}

Recall that $\omid{b(\bs o)}<\infty$ in \Cref{model1} and that biasing by $b(\bs o)$ means considering the probability measure $\hat{\mathbb P}[B]:=\frac{1}{\omid{b(\bs o)}} \omid{b(\bs o)\identity{B}}$ for every event $B$. One might bias the distribution of $[\bs G, \bs o; (\bs X_n)_{n\geq 0}]$ by $b(\bs o)$ as well. We denote the resulting distribution by $\hat{\mathbb P}$ again, by an abuse of notation.

\begin{lemma}[Stationarity of Markov Chains on Unimodular Graphs \cite{processes}]
	\label{lem:stationary}
	If $K$ satisfies the balance condition of \Cref{model1}, then the probability measure $\hat{\mathbb P}$ on $\mathcal G_{\infty}$ is invariant under the shift $\sigma$. In other words, after biasing by $b(\bs o)$, the Markov chain $\left([\bs G, \bs X_n; (\bs X_{n+i})_{i\geq 0}]\right)_{n\geq 0}$ (on $\mathcal G_{\infty}$) is stationary.
\end{lemma}
This lemma generalizes Theorem~4.1 of~\cite{processes}, which proves that $\left([\bs G, \bs X_n]\right)_n$ is a stationary Markov chain when biasing by $b(\bs o)$ (reversibility is also studied in \cite{processes}). 
In fact, the special case of random walks is already considered in Equation~(1.1) of~\cite{HuNa17indistinguishability}.
An important consequence of this result is that it enables the use of Birkhoff's pointwise ergodic theorem.

\begin{proof}
	Fix an event $B\subseteq\mathcal G_{\infty}$. 
	Let $\mathbb P_{(G,v)}$ be the distribution of the Markov chain $(\bs X_i)_{i\geq 0}$ on $G$ starting from $v$. 
	Define the function $g:\mathcal G_{*}\to\mathbb R^{\geq 0}$ by $g[G,v]:= \probPalm{(G,v)}{[G,v;(\bs X_i)_{i\geq 0}]\in B}$. One has
	\begin{eqnarray*}
			\probhat{[\bs G, \bs X_1; (\bs X_{i+1})_{i\geq 0}]\in B} &=& \frac 1{\omid{b(\bs o)}} \omid{b(\bs o)\sum_{v\in\bs G} K(\bs o,v) g(v)}\\
			&=& \frac 1{\omid{b(\bs o)}} \omid{\sum_{v\in\bs G} b(v) K(v, \bs o) g(\bs o)}\\
			&=& \frac 1{\omid{b(\bs o)}} \omid{b(\bs o)g(\bs o)}\\
			&=& \probhat{[\bs G, \bs X_0; (\bs X_i)_{i\geq 0}]\in B},
		\end{eqnarray*}
	where the second and third equations hold by the MTP~\eqref{eq:mtp} and the balance condition respectively. So, the claim is proved.
\end{proof}


Note that if $K$ is cycle-free and one considers a CMT $\bs F$‌ with the same kernel $K$, then $\anc(\bs o)=(\bs  F^{(n)}(\bs o))_n$ has the same distribution as $(\bs X_n)_n$‌ given $\bs G$ and $\bs o$, but the sequence $[\bs G, \bs F^{(n)}(\bs o); \bs F]$ is not stationary in general, as discussed before \Cref{lem:shiftToF(o)}.




\subsection{\Cref{step3comp}: Invariant Path Events and Ergodicity}
\label{subsec:pathErgodic}

In this subsection, we prove the part of \Cref{thm:pathIntro} regarding ergodicity, which is restated in \Cref{thm:pathErgodic} below.

	

For a deterministic graph or discrete space $G$, the space $G^{\mathbb N\cup\{0\}}$ is the state space of Markov chains on $G$. We 
recall the following classical definition.
\begin{definition}
	\label{def:path-invariant-fixed}
	For deterministic $G$, an event $A\subseteq G^{\mathbb N\cup\{0\}}$ is called an \defstyle{invariant path property on $G$} if it is invariant under the shift $(x_i)_{i\geq 0}\mapsto (x_{i+1})_{i\geq 0}$. Let $\sigfield{path}{I}{G}$ denote the sigma-field of invariant path properties of $G$.
\end{definition}

Allover the paper, the subscript $G$ is used for the sigma-fields when a deterministic graph $G$ is fixed.
The following similar definition allows one to consider all graphs, but it is important to note that graphs are considered up to isomorphism, despite the last definition.


\begin{definition}[Invariant Path Properties]
	\label{def:path-invariant}
	Consider the space $\mathcal G_{\infty}$  defined in \Cref{subsec:stationary}.
		An event $A\subseteq \mathcal G_{\infty}$ is called a (shift-) \defstyle{invariant path property} (on $\mathcal G_{\infty}$) if $\sigma^{-1}A=A$. Equivalently, for every $G$ and every pair of sequences $(x_i)_{i\geq 0}$ and $(x'_i)_{i\geq 0}$ in $G$ that shift-couple {(see \Cref{sec:notation})}, one has $\identity{A}[G,x_0;(x_i)_i]=\identity{A}[G,x'_0;(x'_i)_i]$. Let $\sigfield{path}{I}{}$ be the sigma-field of invariant path events on $\mathcal G_{\infty}$.
\end{definition}

\Cref{ex:pathProperty-3reg} shows an example of an invariant event. 
Note also that, given $G$, the map $(x_i)_i\mapsto [G,x_0;(x_i)_i]$ induces a function $\sigfield{path}{I}{}\to\sigfield{path}{I}{G}$.

This following theorem is a restatement of Part~\ref{thm:pathIntro:ergodic} of \Cref{thm:pathIntro}. As mentioned in \Cref{main:markov}, it is a slight generalization of Theorem~4.6 of~\cite{processes}.

\begin{theorem}[\Cref{step3comp}: Ergodicity of Markov Chains on Unimodular Graphs, \cite{processes}]
	\label{thm:pathErgodic}
	Let $[\bs G, \bs o]$ be an ergodic unimodular graph and $(\bs X_i)_i$ be a Markov chain on $\bs G$, starting from $\bs o$, whose kernel satisfies the balance condition and the weak irreducibility of \Cref{model1} (but not necessarily the cycle-free condition). Then, 
	\begin{enumerate}[label=(\roman*)]
		\item \label{thm:pathErgodic-trivial} The invariant path sigma-field $\sigfield{path}{I}{}$ is trivial.
		\item \label{thm:pathErgodic-ergodic} The distribution of $[\bs G, \bs o; (\bs X_i)_i]$, biased by $b(\bs o)$ (which is invariant under the shift $\sigma$ by \Cref{lem:stationary}), is ergodic.
		\item \label{thm:pathErgodic-harmonic} Every bounded $K$-harmonic function on $\mathcal G_*$ is essentially constant.
	\end{enumerate}
	In addition, {if $[\bs G, \bs o]$ is not ergodic, then} $\sigfield{path}{I}{}=\sigfield{}{I}{}$ mod $\mathbb P$.
\end{theorem}
In this theorem, we say that a measurable function $h:\mathcal G_*\to\mathbb R$ is \defstyle{$K$-harmonic} if \[h[G,o]=\sum_{x\in G} K(o,x) h[G,x], \quad \forall [G,o]\in\mathcal G_*.\]

\begin{proof}
	\ref{thm:pathErgodic-trivial}. 
	We proceed similarly to the proof of Theorem~5.1 of~\cite{LySc99}.
	Consider the space of all $[G,o; (x_i)_{i\in\mathbb Z}]$; i.e., we allow for negative indices. Let $\rho$ be the shift on this space defined similarly to $\sigma$.
	Since $[\bs G, \bs o; (\bs X_n)_{n\geq 0}]$ is $\hat{\mathbb P}$-stationary by \Cref{lem:stationary}, one can extend it to negative times to obtain a $\hat{\mathbb P}$-stationary process $[\bs G, \bs o; (\bs X_n)_{n\in\mathbb Z}]$ (consider $[\bs G, \bs X_m; (\bs X_{m+n})_{n\geq -m}]$ and take weak limit using the consistency of this sequence). One can check that the process $(\bs X_{-n})_{n\geq 0}$ is the time-reversal {(with respect to the infinite invariant measure $b$)}; i.e., a random walk with kernel $K'(x,y):=b(y)K(y,x)/b(x)$, and is independent from $(\bs X_n)_{n\geq 0}$ (conditionally on $(\bs G, \bs o)$). 
	
	Let $A\in\sigfield{path}{I}{}$. It makes sense to consider $\identity{A}[\bs G, \bs o; (\bs X_i)_{i\in\mathbb Z}]$. Let $\epsilon>0$ be arbitrary. {By standard tools of measure theory, one can approximate $A$ as follows:} There exists $N\in\mathbb N$ and an event $A'\subseteq\mathcal G_{\infty}$ that depends only on $[\bs G, \bs o; (\bs X_i)_{0\leq i\leq N}]$, such that $\probhat{A\Delta A'}<\epsilon$.
	By the stationarity (\Cref{lem:stationary}) and the relation $\sigma^{-1}(A)=A$, one obtains $\probhat{A\Delta \rho^N(A')}<\epsilon$. The key point is that $\rho^N(A')$ depends only on $[\bs G, \bs o; (\bs X_i)_{i\leq 0}]$, while $A'$ depends only on $[\bs G, \bs o; (\bs X_i)_{i\geq 0}]$. Since the two sides of $(\bs X_i)_i$ are independent, this implies that $\rho^N(A')$ is independent from $A'$, conditionally on $[\bs G, \bs o]$. Since both events are arbitrarily close to $A$, it should be the case that $A$ is independent from itself, conditionally on $[\bs G, \bs o]$. This is made precise in the next paragraph.
	
	Let $Y:=\probhatC{A}{\bs G, \bs o}$, $Z:=\probhatC{A'}{\bs G, \bs o}$ and $W:=\probhatC{\rho^{N}(A')}{\bs G, \bs o}$. The previous paragraph implies that $\omidhat{(Z-Y)^+}\leq \probhat{A'\setminus A}<\epsilon$ and $\omidhat{(W-Y)^+}\leq \probhat{\rho^N(A')\setminus A}<\epsilon$. Hence, using the facts $Y\geq Z-(Z-Y)^+\geq 0$ and $Z,W\in[0,1]$, one obtains
	\begin{align*}
		\omidhat{Y^2} & \geq \omidhat{(Z-(Z-Y)^+)(W-(W-Y)^+)}\\
		&\geq \omidhat{ZW}-\omidhat{(Z-Y)^+} - \omidhat{(W-Y)^+}\\
		&\geq \omidhat{\omidhatC{\identity{A'}}{\bs G, \bs o}\cdot\omidhatC{\identity{\rho^N(A')}}{\bs G, \bs o}} -2\epsilon\\
		&= \omidhat{\omidhatC{\identity{A'}\identity{\rho^N(A')}}{\bs G, \bs o}}-2\epsilon\\
		&= \probhat{A'\cap \rho^N(A')} -2\epsilon\\
		&\geq \probhat{A} -4\epsilon = \omidhat{Y}-4\epsilon. 
	\end{align*}
	Since $\epsilon$ can be arbitrarily small and $0\leq Y\leq 1$, this implies that $Y\in\{0,1\}$ a.s.; i.e., $\probhatC{A}{\bs G, \bs o}\in\{0,1\}$ a.s.
	
	
	We now show that $\probhatC{A}{\bs G, \bs o}$ does not depend on $\bs o$ (note that the explicit distribution of $(\bs X_n)_n$, given a sample of $(\bs G,\bs o)$, provides an explicit version of the last conditional distribution).
	For $j\in\{0,1\}$, let $E_j$ be the event $\probhatC{A}{\bs G, \bs o}=j$.
	By the previous paragraph and \Cref{lem:happensAtRoot}, almost all samples $(G,o)$ of $[\bs G, \bs o]$ satisfy $[G,u]\in E_0\cup E_1, \forall u\in G$. Consider such a sample.
	The definition of $E_j$ implies that, if $[G,u]\in E_j$, then $\myprob{[G,u;(\bs X_i)_i]\in A}=j$, where $(\bs X_i)_i$ is the Markov chain started from $u$. Assume $K(u,u')>0$ and let $(\bs X'_i)_i$ be the Markov chain started from $u'$. Since $\myprob{\bs X_1=u'}>0$, we can couple the two chains in such a way that $\myprob{\bs X_{i+1}=\bs X'_i, \forall i}>0$. In this case, the invariance of $A$ gives $\identity{A}[G,u; (\bs X_i)_i]=\identity{A}[G,u'; (\bs X'_i)_i]$. Since this happens with positive probability, one obtains that $\identity{E_j}[G,u]=\identity{E_j}[G,u']$. Hence, the assumption of weak irreducibility of \Cref{model1} implies that either $[G,u]\in E_0,\forall u\in G$ or $[G,u]\in E_1, \forall u\in G$. This implies that $E_0$ and $E_1$ are invariant events up to null modifications (note that we have considered a generic sample in the last argument). So, the ergodicity of $[\bs G, \bs o]$ implies that $\myprob{E_j}\in\{0,1\}$. This implies that $\myprob{A}\in\{0,1\}$ and the claim is proved.
	
	\ref{thm:pathErgodic-ergodic}. The claim is already proved in part~\ref{thm:pathErgodic-trivial} by the definition of ergodicity.
	
	\ref{thm:pathErgodic-harmonic}. 
	The claim is proved similarly to the classical bijection between invariant functions and harmonic functions for Markov chains (see e.g, \cite{bookLyPe16}). Let $h$ be a bounded $K$-harmonic function on $\mathcal G_*$. For every deterministic $(G,o)$ satisfying the balance condition of \Cref{model1}, the process $h[G, \bs X_n]$ is a bounded martingale. Hence, the limit $\tilde h[G,o,(x_i)_i]:=\lim_i h[G, x_i]$ is well defined for almost every sample $(x_i)_i$ of the random walk. It can be seen that $\tilde h$ is an invariant {path function}. So, the triviality of $\sigfield{path}{I}{}$ implies that $\tilde h[\bs G, \bs o; (\bs X_i)_i]$ is essentially constant. In addition, the martingale property implies that $h[G,o] = \omid{\tilde{h}[G,o;(\bs X_i)_i]}$. This implies that $h[\bs G, \bs o]$ is also essentially constant.
\end{proof}

\subsection{Strengthening Weak Aperiodicity}
\label{subsec:aperiodic}

This subsection is only needed for the indistinguishability of level-sets and the reader might skip it at first reading.
As a first application of \Cref{thm:pathErgodic}, we strengthen the weak aperiodicity property of \Cref{model2} in the following lemma, which will be used later in the proof of \Cref{thm:pathI=T}. 

\begin{lemma}
	\label{lem:aperiodicity}
	In \Cref{model2}, almost surely,
	\[
	\forall u\in \bs G, \forall d_0\in\mathbb N: \exists n\in\mathbb N, \exists y\in \bs G: K^n(u,y)>0 \text{ and } K^{n+d_0}(u, y)>0.
	\]
	The cycle-free and one-endedness conditions are not needed here.
\end{lemma}

\begin{proof}
	By \Cref{lem:happensAtRoot}, it is enough to prove the claim for $u=\bs o$.
	We may assume that $[\bs G, \bs o]$ is ergodic, without loss of generality. Let $(\bs X_n)_n$ be the Markov chain with kernel $K$, starting from $\bs o$. Let $d_0\in \mathbb N$ be arbitrary.
	For each $d\in \mathbb N$, let $\bs S_d:=\{x\in \bs G: \exists n, y: K^n(x,y)>0, K^{n+d}(x,y)>0\}$. The weak aperiodicity of \Cref{model2} implies that $\bs S_d\neq\emptyset$ for sufficiently many values of $d$, more precisely, for $d_1, \ldots, d_k$ such that $\sum_i a_i d_i = d_0$ for some $a_1,\ldots,a_k\in \mathbb Z$. By ergodicity, we may choose $d_1,\ldots, d_k$ to be constant. Also, \Cref{lem:happensAtRoot} implies that $\myprob{\bs o\in \bs S_{d_i}}>0, \forall i$. Therefore, ergodicity (\Cref{thm:pathErgodic}) implies that $(\bs X_n)_n$ intersects each $\bs S_{d_i}$ infinitely often a.s. In addition, at each intersection with $\bs S_{d_i}$, say at time $t$, there is a positive probability that $\exists n: K^{n+d_i}(\bs X_t, \bs X_{t+n})>0$. By ergodicity again, for every $i$, the last event happens for infinitely many $t$ a.s. Therefore, for each $i$, one can almost surely find $\norm{a_i}$ intervals of the form $[t_{ij}, t_{ij}+n_{ij}]\subseteq\mathbb N$ (for $1\leq j\leq \norm{a_i}$), such that $K^{n_{ij}+d_i}(\bs X_{t_{ij}}, \bs X_{t_{ij}+n_{ij}})>0$ and the intervals are disjoint. Note that, in each of these intervals $[t_{ij},t_{ij}+n_{ij}]$, we have two choices for going from $\bs X_{t_{ij}}$ to $\bs X_{t_{ij}+n_{ij}}$: One by a path of length $n_{ij}$ and another by a path of length $n_{ij}+d_{i}$. Now, for $T\geq \max_{i,j}(t_{ij}+n_{ij})$, one can show that $K^{T+c}(\bs o, \bs X_T)>0$ and $K^{T+c+d_0}(\bs o, \bs X_T)>0$, where $c=-\sum_i a_i d_i \identity{\{a_i\leq 0\}}=\sum_i a_i d_i \identity{\{a_i\leq 0\}} -d_0 \geq 0$. This proves the claim.
\end{proof}

\begin{lemma}
	\label{lem:foil-connectivity}
	In almost every sample $(G,o)$ of \Cref{model2}, the graph with vertex set $G$ and edge set $\{(u,v): \probPalm{G}{v\in L(u)}>0\}$ is connected.
\end{lemma}

This lemma was used in the proof of \Cref{prop:simpler}. We stress that the claim may fail if the level-sets are finite.

\begin{proof}
	For $v\neq u$, the condition $\probPalm{G}{v\in L(u)}>0$ is equivalent to the existence of $k>0$ and \defstyle{admissible sequences} $u=u_0,u_1,\ldots,u_k$ and $v=v_0,v_1,\ldots,v_k$ that intersect only at $u_k=v_k$, where admissible means $K(u_i,u_{i+1})>0, \forall i$.


	Assume $(G,o)$ satisfies the conditions of \Cref{model1,model2} and the claim of \Cref{lem:aperiodicity}. Assume also that $\bs F$ is locally-finite and all level-sets of $\bs F$ are infinite, $\mathbb P_G$-a.s. 
	Let $u,v\in G$ be such that $K(u,v)>0$. We will show that $u$ and $v$ are connected in the graph mentioned in the lemma. The claim of the lemma then follows from weak irreducibility.
	
	By \Cref{lem:aperiodicity}, there exist $n\in\mathbb N$ and $w\in G$ such that $K^n(v,w)>0$ and $K^{n+1}(v,w)>0$. Let $v_0,\ldots,v_n$ and $v'_0,\ldots,v'_{n+1}$ be two admissible sequences from $v$ to $w$. 
	We call a point $u_0\in G$ \textit{good} if there exist $k\geq 0$ and admissible sequences $(u_i)_{i=0}^{n+k+1}$ and $(w_i)_{i=0}^k$ such that $w_0=w$, $u_{n+k+1}=w_k$ and $(u_i)_i$ does not intersect any of the sequences $(w_i)_i, (v_i)_i$ and $(v'_i)_i$ except at $w_k$. 
	\begin{center}
		\begin{tikzpicture}[x=0.01mm, y=0.01mm]
			\filldraw[black] (20705,-13362) circle(1.5pt) node[left] {{\scriptsize $w_k$}};
			\filldraw[black] (23985,-13980) circle(1.5pt) node[left] {{\scriptsize $w$}};
			\filldraw[black] (22450,-14205)  node {{\scriptsize$(w_i)_i$}};
			\filldraw[black] (24400,-13600)  node {{\scriptsize$(v_i)_i$}};
			\filldraw[black] (23800,-12850)  node {{\scriptsize$(u_i)_i$}};
			\filldraw[black] (27450,-15150) circle(1.5pt) node[right] {{\scriptsize $u$}};
			\filldraw[black] (27360,-12630) circle(1.5pt) node[right] {{\scriptsize $u_0$}};
			\filldraw[black] (27270,-14925) circle(1.5pt) node[right] {{\scriptsize $v$}};
			\filldraw[black] (24400,-14750)  node {{\scriptsize$(v'_i)_i$}};
			\draw[thick] (20700,-13350)--(20701,-13351)--(20703,-13354)--(20708,-13359)--(20714,-13368)--(20724,-13380)
			--(20737,-13396)--(20753,-13415)--(20772,-13439)--(20795,-13467)--(20821,-13499)
			--(20851,-13535)--(20883,-13573)--(20917,-13614)--(20953,-13657)--(20991,-13702)
			--(21030,-13748)--(21070,-13793)--(21109,-13838)--(21149,-13883)--(21188,-13926)
			--(21226,-13967)--(21264,-14006)--(21300,-14043)--(21335,-14077)--(21369,-14108)
			--(21401,-14136)--(21433,-14161)--(21462,-14183)--(21491,-14203)--(21519,-14219)
			--(21547,-14232)--(21573,-14242)--(21600,-14250)--(21630,-14256)--(21660,-14259)
			--(21690,-14258)--(21720,-14254)--(21750,-14246)--(21780,-14235)--(21810,-14221)
			--(21840,-14204)--(21870,-14184)--(21900,-14161)--(21930,-14137)--(21960,-14110)
			--(21990,-14083)--(22020,-14054)--(22050,-14025)--(22080,-13996)--(22110,-13967)
			--(22140,-13940)--(22170,-13913)--(22200,-13889)--(22230,-13866)--(22260,-13846)
			--(22290,-13829)--(22320,-13815)--(22350,-13804)--(22380,-13796)--(22410,-13792)
			--(22440,-13791)--(22470,-13794)--(22500,-13800)--(22526,-13808)--(22553,-13818)
			--(22579,-13831)--(22606,-13847)--(22632,-13866)--(22659,-13888)--(22685,-13913)
			--(22712,-13940)--(22738,-13970)--(22765,-14002)--(22791,-14036)--(22818,-14072)
			--(22844,-14110)--(22871,-14149)--(22897,-14189)--(22924,-14230)--(22950,-14271)
			--(22976,-14312)--(23003,-14352)--(23029,-14392)--(23056,-14430)--(23082,-14467)
			--(23109,-14502)--(23135,-14535)--(23162,-14566)--(23188,-14593)--(23215,-14618)
			--(23241,-14640)--(23268,-14658)--(23294,-14674)--(23321,-14685)--(23347,-14694)
			--(23374,-14698)--(23400,-14700)--(23426,-14698)--(23453,-14693)--(23479,-14685)
			--(23505,-14672)--(23531,-14656)--(23557,-14636)--(23582,-14613)--(23608,-14585)
			--(23633,-14555)--(23658,-14521)--(23683,-14485)--(23707,-14446)--(23732,-14405)
			--(23756,-14363)--(23780,-14319)--(23805,-14274)--(23829,-14229)--(23853,-14184)
			--(23878,-14140)--(23903,-14096)--(23928,-14054)--(23953,-14014)--(23979,-13977)
			--(24005,-13942)--(24032,-13910)--(24059,-13881)--(24087,-13857)--(24115,-13836)
			--(24144,-13819)--(24174,-13807)--(24204,-13799)--(24235,-13795)--(24267,-13795)
			--(24300,-13800)--(24327,-13807)--(24355,-13816)--(24383,-13828)--(24412,-13843)
			--(24442,-13861)--(24473,-13882)--(24504,-13906)--(24537,-13933)--(24570,-13963)
			--(24604,-13995)--(24639,-14030)--(24675,-14067)--(24711,-14107)--(24748,-14148)
			--(24786,-14191)--(24824,-14236)--(24862,-14282)--(24901,-14330)--(24939,-14378)
			--(24978,-14426)--(25017,-14475)--(25056,-14524)--(25094,-14572)--(25133,-14620)
			--(25170,-14668)--(25207,-14714)--(25244,-14759)--(25280,-14802)--(25314,-14843)
			--(25348,-14883)--(25381,-14920)--(25413,-14955)--(25443,-14987)--(25472,-15017)
			--(25499,-15044)--(25525,-15068)--(25550,-15089)--(25573,-15107)--(25595,-15122)
			--(25615,-15134)--(25633,-15143)--(25650,-15150)--(25671,-15155)--(25689,-15155)
			--(25703,-15149)--(25715,-15139)--(25723,-15123)--(25728,-15103)--(25730,-15078)
			--(25729,-15049)--(25726,-15016)--(25720,-14980)--(25713,-14940)--(25704,-14898)
			--(25694,-14855)--(25683,-14809)--(25671,-14763)--(25659,-14717)--(25648,-14670)
			--(25637,-14625)--(25627,-14580)--(25619,-14537)--(25611,-14496)--(25606,-14457)
			--(25603,-14421)--(25602,-14388)--(25604,-14357)--(25608,-14330)--(25615,-14306)
			--(25624,-14285)--(25636,-14266)--(25650,-14250)--(25666,-14236)--(25684,-14224)
			--(25704,-14213)--(25726,-14204)--(25749,-14196)--(25775,-14190)--(25802,-14184)
			--(25830,-14179)--(25860,-14175)--(25891,-14172)--(25923,-14170)--(25956,-14168)
			--(25989,-14167)--(26023,-14166)--(26058,-14166)--(26093,-14166)--(26127,-14167)
			--(26162,-14168)--(26197,-14170)--(26231,-14172)--(26266,-14175)--(26299,-14179)
			--(26332,-14184)--(26365,-14190)--(26397,-14196)--(26428,-14204)--(26459,-14213)
			--(26490,-14224)--(26520,-14236)--(26550,-14250)--(26577,-14264)--(26603,-14280)
			--(26631,-14298)--(26659,-14318)--(26688,-14340)--(26717,-14365)--(26749,-14392)
			--(26781,-14422)--(26815,-14454)--(26850,-14488)--(26886,-14525)--(26924,-14563)
			--(26962,-14603)--(27001,-14645)--(27041,-14688)--(27080,-14731)--(27120,-14775)
			--(27159,-14818)--(27197,-14861)--(27233,-14902)--(27267,-14940)--(27299,-14977)
			--(27329,-15010)--(27355,-15040)--(27378,-15066)--(27397,-15089)--(27413,-15107)
			--(27426,-15122)--(27436,-15133)--(27442,-15141)--(27447,-15146)--(27449,-15149)
			--(27450,-15150);
			\draw[thick] (20700,-13350)--(20701,-13349)--(20705,-13348)--(20710,-13345)--(20720,-13341)--(20733,-13335)
			--(20750,-13327)--(20771,-13317)--(20797,-13306)--(20827,-13292)--(20862,-13276)
			--(20900,-13260)--(20941,-13241)--(20984,-13222)--(21030,-13202)--(21077,-13182)
			--(21124,-13162)--(21172,-13142)--(21219,-13122)--(21265,-13104)--(21310,-13086)
			--(21354,-13070)--(21396,-13054)--(21436,-13041)--(21473,-13029)--(21509,-13018)
			--(21543,-13009)--(21575,-13002)--(21606,-12997)--(21635,-12993)--(21663,-12991)
			--(21690,-12990)--(21720,-12991)--(21750,-12994)--(21779,-12999)--(21807,-13007)
			--(21834,-13016)--(21861,-13028)--(21888,-13042)--(21914,-13058)--(21939,-13075)
			--(21965,-13094)--(21989,-13114)--(22014,-13135)--(22038,-13157)--(22063,-13178)
			--(22087,-13200)--(22111,-13222)--(22135,-13242)--(22159,-13262)--(22184,-13280)
			--(22208,-13297)--(22233,-13312)--(22258,-13325)--(22283,-13335)--(22308,-13343)
			--(22333,-13349)--(22359,-13352)--(22384,-13352)--(22410,-13350)--(22434,-13346)
			--(22458,-13340)--(22482,-13331)--(22506,-13321)--(22530,-13307)--(22554,-13292)
			--(22578,-13275)--(22603,-13255)--(22627,-13234)--(22651,-13211)--(22675,-13186)
			--(22699,-13161)--(22724,-13134)--(22748,-13107)--(22772,-13080)--(22796,-13053)
			--(22821,-13026)--(22845,-12999)--(22869,-12974)--(22893,-12949)--(22917,-12926)
			--(22941,-12905)--(22965,-12885)--(22989,-12868)--(23013,-12853)--(23036,-12839)
			--(23060,-12829)--(23083,-12820)--(23107,-12814)--(23130,-12810)--(23157,-12808)
			--(23183,-12809)--(23210,-12812)--(23235,-12819)--(23261,-12829)--(23286,-12842)
			--(23311,-12857)--(23336,-12875)--(23361,-12894)--(23385,-12916)--(23409,-12938)
			--(23433,-12962)--(23457,-12986)--(23481,-13010)--(23505,-13034)--(23530,-13056)
			--(23555,-13078)--(23580,-13098)--(23606,-13116)--(23632,-13132)--(23659,-13145)
			--(23687,-13156)--(23715,-13163)--(23744,-13168)--(23774,-13170)--(23805,-13170)
			--(23831,-13168)--(23857,-13164)--(23883,-13158)--(23910,-13150)--(23938,-13140)
			--(23966,-13127)--(23995,-13113)--(24024,-13097)--(24053,-13079)--(24083,-13059)
			--(24113,-13038)--(24144,-13016)--(24174,-12993)--(24205,-12969)--(24236,-12944)
			--(24267,-12920)--(24298,-12895)--(24329,-12871)--(24360,-12848)--(24391,-12825)
			--(24422,-12804)--(24453,-12785)--(24484,-12767)--(24514,-12752)--(24544,-12739)
			--(24574,-12728)--(24604,-12719)--(24634,-12714)--(24663,-12711)--(24692,-12711)
			--(24721,-12714)--(24750,-12720)--(24777,-12728)--(24804,-12738)--(24831,-12751)
			--(24857,-12767)--(24883,-12786)--(24908,-12808)--(24933,-12833)--(24958,-12861)
			--(24982,-12891)--(25005,-12923)--(25029,-12958)--(25052,-12995)--(25074,-13033)
			--(25097,-13072)--(25119,-13113)--(25142,-13154)--(25164,-13195)--(25187,-13237)
			--(25210,-13277)--(25233,-13317)--(25257,-13356)--(25281,-13393)--(25306,-13428)
			--(25332,-13461)--(25358,-13492)--(25386,-13519)--(25415,-13544)--(25444,-13565)
			--(25475,-13583)--(25508,-13598)--(25541,-13609)--(25576,-13616)--(25612,-13620)
			--(25650,-13620)--(25680,-13618)--(25711,-13614)--(25744,-13608)--(25778,-13599)
			--(25813,-13588)--(25850,-13575)--(25890,-13560)--(25931,-13542)--(25974,-13522)
			--(26019,-13499)--(26066,-13474)--(26116,-13447)--(26167,-13418)--(26221,-13387)
			--(26276,-13354)--(26332,-13319)--(26390,-13283)--(26450,-13246)--(26510,-13207)
			--(26571,-13168)--(26632,-13128)--(26693,-13087)--(26753,-13047)--(26812,-13008)
			--(26870,-12968)--(26926,-12930)--(26980,-12894)--(27031,-12859)--(27078,-12826)
			--(27123,-12795)--(27164,-12767)--(27200,-12742)--(27233,-12719)--(27262,-12699)
			--(27286,-12682)--(27307,-12667)--(27323,-12656)--(27336,-12647)--(27346,-12640)
			--(27353,-12635)--(27357,-12632)--(27359,-12631)--(27360,-12630);
			\draw[thick,dashed] (23985,-13980)--(23986,-13982)--(23988,-13987)--(23991,-13995)--(23997,-14008)--(24004,-14025)
			--(24014,-14048)--(24025,-14075)--(24039,-14106)--(24055,-14140)--(24072,-14176)
			--(24090,-14213)--(24109,-14251)--(24129,-14287)--(24149,-14322)--(24169,-14354)
			--(24189,-14383)--(24210,-14408)--(24230,-14430)--(24251,-14447)--(24273,-14460)
			--(24296,-14469)--(24320,-14474)--(24345,-14475)--(24365,-14473)--(24386,-14469)
			--(24409,-14463)--(24432,-14454)--(24456,-14442)--(24481,-14427)--(24507,-14411)
			--(24533,-14391)--(24561,-14369)--(24589,-14346)--(24617,-14320)--(24646,-14293)
			--(24675,-14265)--(24705,-14235)--(24735,-14206)--(24765,-14175)--(24796,-14146)
			--(24826,-14116)--(24857,-14088)--(24887,-14060)--(24918,-14035)--(24948,-14011)
			--(24978,-13990)--(25009,-13971)--(25038,-13955)--(25068,-13942)--(25098,-13932)
			--(25127,-13926)--(25157,-13923)--(25186,-13923)--(25216,-13928)--(25245,-13935)
			--(25270,-13944)--(25295,-13955)--(25320,-13969)--(25345,-13985)--(25371,-14004)
			--(25396,-14026)--(25422,-14049)--(25447,-14076)--(25473,-14104)--(25499,-14135)
			--(25525,-14168)--(25551,-14202)--(25577,-14239)--(25603,-14276)--(25629,-14315)
			--(25655,-14355)--(25681,-14396)--(25707,-14437)--(25734,-14479)--(25760,-14520)
			--(25786,-14562)--(25812,-14603)--(25839,-14643)--(25865,-14683)--(25891,-14721)
			--(25918,-14758)--(25944,-14794)--(25970,-14828)--(25997,-14861)--(26023,-14891)
			--(26050,-14920)--(26076,-14946)--(26102,-14970)--(26129,-14992)--(26155,-15012)
			--(26182,-15030)--(26208,-15046)--(26235,-15060)--(26267,-15074)--(26299,-15086)
			--(26332,-15095)--(26366,-15102)--(26401,-15107)--(26438,-15109)--(26477,-15109)
			--(26517,-15107)--(26559,-15104)--(26603,-15098)--(26648,-15091)--(26695,-15082)
			--(26742,-15073)--(26791,-15062)--(26840,-15050)--(26888,-15038)--(26936,-15025)
			--(26982,-15012)--(27027,-14999)--(27068,-14987)--(27107,-14976)--(27141,-14966)
			--(27172,-14956)--(27198,-14948)--(27220,-14941)--(27237,-14936)--(27250,-14931)
			--(27260,-14928)--(27265,-14927)--(27269,-14925)--(27270,-14925);
		\end{tikzpicture}
	\end{center}
	We claim that there exists a good point. Indeed, consider the event that $\forall 0<i\leq n+1: \bs F^{(i)}(u)=v_i$, $\bs F$ is locally-finite and $L_{\bs F}(u)$ is infinite. On this event, the set $L_{\bs F}(u)\setminus \cup_i D(v'_i)$ is infinite and consists only of good points a.s. So, there exists a good point $u_0$.
	Now, both $u$ and $v$ have admissible sequences to $u_{n+k+1}$, of length $n+k+1$, and disjoint from $(u_i)_i$. So, $\probPalm{G}{u_0\in L(u)}>0$  and $\probPalm{G}{u_0\in L(v)}>0$. So, $u$ and $v$ are connected by the two edges $(u,u_0)$ and $(u_0,v)$, and the claim is proved.
	%
%
\end{proof}

\subsection{\Cref{step3foil}: Tail Events and Tail-Triviality}
\label{subsec:pathTail}

In this subsection, we prove Part~\ref{thm:pathIntro:tail} of \Cref{thm:pathIntro} and \Cref{thm:pathI=T-intro}, which are restated in Theorems~\ref{thm:pathI=T} and~\ref{thm:pathTail} respectively.
This subsection is only needed for the indistinguishability of level-sets and the reader might skip it at first reading. We first recall the tail sigma-field on a deterministic space in \Cref{def:path-tail-fixed} below. Then, the analogous notion on $\mathcal G_*$‌ is given in \Cref{def:path-tail}.
\begin{definition}
	\label{def:path-tail-fixed}
	For deterministic $G$, an event $A\subseteq G^{\mathbb N\cup\{0\}}$ is called a \defstyle{tail path property on $G$} if, for every $n\in\mathbb N$, it is contained in the sigma-field generated by $x_n,x_{n+1},\ldots$; i.e., it does not depend on $x_0,\ldots,x_{n-1}$. Let $\sigfield{path}{T}{G}$ denote the sigma-field of invariant path properties of $G$.
\end{definition}

One has $\sigfield{path}{I}{G}\subseteq \sigfield{path}{T}{G}$. Our first result is the following, which is a restatement of \Cref{thm:pathI=T-intro}. Here, given a deterministic $(G,v)$, we let $\mathbb P_{(G,v)}$ be the distribution of the Markov chain on $G$ with kernel $K$ starting from $v$. Note that this Markov chain is steady if $\sigfield{path}{T}{G}=\sigfield{path}{I}{G}$ mod $\mathbb P_{(G,v)}$.

\begin{theorem}[Steadiness of Markov Chains on Deterministic Graphs]
	\label{thm:pathI=T}
	Let $[\bs G, \bs o]$ be an ergodic unimodular graph and $(\bs X_i)_i$ be a Markov chain on $\bs G$, starting from $\bs o$, whose kernel satisfies the conditions of balance, weak irreducibility and weak aperiodicity of \Cref{model2} (but not necessarily the cycle-free condition).
	Then, in almost every sample $(G,o)$ of $[\bs G, \bs o]$, the Markov chain $(\bs X_i)$ is steady. In addition,
	\begin{equation}
		\label{eq:pathI=T}
		\sigfield{path}{T}{G}=\sigfield{path}{I}{G} \text{ mod }\left(\sum_{v\in G}\mathbb P_{(G,v)}\right).
	\end{equation}
	Also, the same holds for any deterministic $G$ that satisfies~\eqref{eq:aperiodic2}.
\end{theorem}

Note that equivalence mod $\sum_{v\in G}\mathbb P_{(G,v)}$ is stronger than equivalence mod $\mathbb P_{(G,v)}$ for all $v\in G$. For the former, one should approximate the events in a way not depending on the choice of $v\in G$. We will need this in \Cref{subsec:tail-fixed} because \Cref{lem:happensAtRoot} no longer holds and we need to consider all possible roots simultaneously.



This theorem generalizes a similar claim on transitive Markov chains on a deterministic transitive environment
(see e.g., Theorem~14.47 of~\cite{bookLyPe16}). We prove this result by a slight generalization of the proof of~\cite{bookLyPe16}, but by leveraging unimodularity. 
Before giving the proof, we present the analogous claim on tail triviality when $G$ is not fixed (\Cref{thm:pathTail}).

\begin{definition}[Tail Path Properties on $\mathcal G_{\infty}$]
	\label{def:path-tail}
	Consider the space $\mathcal G_{\infty}$ defined in \Cref{subsec:stationary}.
		An event $A\subseteq \mathcal G_{\infty}$ is called a \defstyle{tail path property} (on $\mathcal G_{\infty}$) if, for every $n\in\mathbb N$, $A$ is in the sigma-field generated by $[G,x_n; (x_{n+i})_{i\geq 0}]$; i.e., is not changed by modifying $(x_0,\ldots,x_{n-1})$. Let $\sigfield{path}{T}{}$ denote the sigma-field of tail path events on $\mathcal G_{\infty}$.
\end{definition}

\begin{example}
	\label{ex:pathProperty-3reg}
	Let $[\bs G, \bs o]$ be the 3-regular tree in which each vertex $v$ is equipped with two marks $\bs m_1(v)$ and $\bs m_2(v)$. Assume that $\bs m_1(\cdot)$ are i.i.d. uniform numbers in $[0,1]$ and $\bs m_2(\cdot)$ are periodic marks in $\{0,1\}$ (i.e., $\bs m_2(u)\neq \bs m_2(v)$ if $u$‌ is adjacent to $v$) such that $\myprob{\bs m_2(\bs o)=1}=1/2$. Also, let $(\bs X_i)_i$ be the simple random walk starting from the root. 
	\begin{enumerate}[label=(\roman*)]
		\item The event that $\limsup_{i\to\infty} \bs m_1(\bs X_i)>1/2$ is a shift-invariant event, but it has probability one. In fact, by \Cref{thm:pathErgodic}, there is no nontrivial shift-invariant event here.
		\item The events $\bs m_2(\bs o)=0$ and $\bs m_2(\bs o)=1$ are nontrivial tail events. We will see that these are essentially the only nontrivial tail events (see \Cref{rem:periodic}).
	\end{enumerate}
\end{example}

The following is the next tail triviality result, which is an expansion of Part~\ref{thm:pathIntro:tail} of \Cref{thm:pathIntro}.

\begin{theorem}[\Cref{step3foil}: Tail Triviality for Markov Chains on Unimodular Graphs]
	\label{thm:pathTail}
	Under the assumptions of Part~\ref{thm:pathIntro:tail} of \Cref{thm:pathIntro},
	\begin{enumerate}[label=(\roman*)]
		\item The tail path sigma-field $\sigfield{path}{T}{}$ is trivial.
		\item $[\bs G, \bs o; (\bs X_n)_n]$ is tail-trivial under the shift $\sigma$.
		\item Every bounded space-time $K$-harmonic function on $\mathcal G_*$ is essentially constant.
	\end{enumerate}
	In addition, in the non-ergodic case, $\sigfield{path}{T}{}=\sigfield{path}{I}{}=\sigfield{}{I}{}$ mod $\mathbb P$.
\end{theorem}

Here, a \defstyle{space-time $K$-harmonic} function is defined as a measurable function $h:\mathcal G_*\times (\mathbb N\cup\{0\})\to\mathbb R$ such that \[h([G,o],n)=\sum_{x\in G} K(o,x) h([G,x],n+1), \quad \forall [G,o]\in\mathcal G_*, \forall n\geq 0.\]
Note that $\sigfield{path}{I}{}\subseteq \sigfield{path}{T}{}$. Hence, \Cref{thm:pathTail} extends \Cref{thm:pathErgodic}. In fact, \Cref{thm:pathErgodic} will be used in the proof.


\begin{remark}[Periodic Case]
	\label{rem:periodic}
	If the weak aperiodicity condition is modified by allowing the largest common divisor of the set \Cref{eq:aperiodic} to be $d_0\in\mathbb N$, then a slight modification of the proofs gives that the tail sigma-field is equivalent to the sigma-field generated by at most $d_0$ disjoint tail events (assuming ergodicity). Indeed, if $E$ is the graph on $\bs G$ where $u$ is connected to $v$ when $K^{d_0}(u,v)>0$, then $E$ has at most $d_0$ connected components and indistinguishability holds in each of its components. 	
	Therefore, the level-sets of $\bs F$ can be partitioned into $d_0$ classes in a periodic way and each class contains indistinguishable level-sets. Hence, the number of indistinguishability classes of level-sets is a divisor of $d_0$. Note that, when $d_0=\infty$, the proof fails; e.g., in \Cref{ex:free,ex:dl}. This was resolved by another idea in coalescing Markov chains of \Cref{subsec:crw}.
\end{remark}

We now start to prove \Cref{thm:pathI=T,thm:pathTail}. Adapting the strategy of Theorem~14.47 of~\cite{bookLyPe16}, we start with the following lemma.

\begin{lemma}
	\label{lem:pathTV}
	Assume $[\bs G, \bs o]$ is unimodular and $K$ is a factor Markov kernel that satisfies the balance and weak irreducibility conditions of \Cref{model1} (but not necessarily the cycle-free condition). 
	Then, for every $k\in\mathbb N$ and almost every sample $(G,o)$ of $[\bs G, \bs o]$, if there exist $x$ and $l$ such that $K^{l+k}(o,x)>0$ and $K^l(o,x)>0$, then
	\begin{equation}
		\label{eq:pathTV}
		\forall v\in G: \lim_{n\to\infty} \tv{K^{n+k}(v,\cdot)-K^n(v,\cdot)} = 0.
	\end{equation}
	The same holds for every deterministic graph or discrete space $G$ that satisfies~\eqref{eq:aperiodic2}.
\end{lemma}
The reader might compare the assumption of this lemma with~\eqref{eq:aperiodic}.
\begin{proof}
	We may let $k$ be fixed. We may also assume that $[\bs G, \bs o]$ is ergodic.
	For $l\in\mathbb N$ and $\epsilon\geq 0$, let
	\[
	S_{l,\epsilon}:=S_{l,\epsilon}(G):=\{v\in G: \tv{K^l(v,\cdot)\wedge K^{l+k}(v,\cdot)}\geq\epsilon \}.
	\]	 
	Note that the assumption of the lemma on $(G,o)$ is equivalent to $o\in \cup_{l,m} S_{l,1/m}$. If $\forall l,m:\myprob{\bs o\in S_{l,1/m}}=0$, the assumption does not hold. Hence, we can fix $l$ and $\epsilon>0$ such that $\myprob{\bs o\in S_{l,\epsilon}}>0$. Since $\tv{K^l(v,\cdot)\wedge K^{l+k}(v,\cdot)}$ is nondecreasing in $l$, we may choose $l$ such that it is coprime with $k$. Then, the ergodicity of $[\bs G, \bs o]$ and~\Cref{lem:happensAtRoot} imply that $S_{l,\epsilon}\neq\emptyset$ a.s. By \Cref{lem:stationary} (stationarity after biasing by $b(\bs o)$), one obtains that $\hat{\mathbb P}$-almost surely, $(\bs X_i)_i$ lies in $S_{l,\epsilon}$ either infinitely often or never. Thus, this happens $\mathbb P$-almost surely as well. 
	
	The event  $\card{(\bs X_i)_i\cap S_{l,\epsilon}}=\infty$ is an invariant path property. So, \Cref{thm:pathErgodic} and the previous paragraph imply that this event happens almost surely (even if $\bs o\not\in S_{l,\epsilon}$). 
	Independently from $(\bs X_i)_i$, let $(\bs t_i)_i$ be a sequence of times such that $\bs t_0:=0$ and $\bs t_{i+1}-\bs t_i$ is chosen randomly and uniformly in $\{l,l+k\}$. Since $(\bs X_{\bs t_i})_i$ is a Markov chain with kernel $(K^l+K^{l+k})/2$, one obtains similarly that $(\bs X_{\bs t_i})_i$ also visits $S_{l,\epsilon}$ infinitely often (note that, since $l$ and $k$ are coprime, this Markov chain is also weakly irreducible). Fix a sample $(G,o)$ such that, starting the Markov chain $(\bs X_i)_i$ from every point $v\in G$, the last property holds a.s.
	Observe that almost every sample is of this type. We will prove that this sample has the desired property. Let $v\in G$ be arbitrary.
	
	We construct a coupling of two copies of $(\bs X_i)_i$ on $G$ starting from $v$ as follows. We do this by constructing two random sequences $(\bs W_0,\bs W_1,\ldots)$ and $(\bs W'_0,\bs W'_1,\ldots)$, where each $\bs W_i$ or $\bs W'_i$ is a finite path in $G$, such that the end points of $\bs W_i$ and $\bs W'_i$ are the same for all $i$ and each $\bs W_i$ and $\bs W'_i$ is generated by going either $l$ or $l+k$ steps of the kernel $K$. To do this, assume that $\bs W_i$ and $\bs W'_i$ are constructed and $x$ is their common end point. If $x\not\in S_{l,\epsilon}$, choose $\bs W_{i+1}=\bs W'_{i+1}$ by going either $l$ or $l+k$ steps from $x$ (by tossing a fair coin). 
	If $x\in S_{l,\epsilon}$, do the following. Let $\nu_j$ be the distribution of a path of length $j$ starting from $x$ and $q_j:=K^j(x,\cdot)$ be the distribution of its end point. We would like to choose $\bs W_{i+1}$ and $\bs W'_{i+1}$ with distribution $(\mu_l+\mu_{l+k})/2$. First, choose the end point $\bs y$ with distribution $(q_l+q_{l+k})/2$. Then, choose $\bs W_{i+1}$ and $\bs W'_{i+1}$ randomly with distribution $(\mu_l+\mu_{l+k})/2$ independently, conditionally on ending in $\bs y$. Note that the lengths of these paths differ by $\pm k$ with probability $\sum_y q_l(y)q_{l+k}(y)/(q_l(y)+q_{l+k}(y)) \geq \sum_y (q_l(y)\wedge q_{l+k}(y))/2\geq \epsilon/2$. Otherwise, the lengths are equal. Now, $(\bs W, \bs W')$ is constructed.
	
	Since $S_{l,\epsilon}$ is visited by the endpoints of $(\bs W_i)_i$ infinitely often, there exists a first $i$ such that the sum of the lengths of $\bs W_0, \ldots, \bs W_i$ is exactly that of $\bs W'_0,\ldots,\bs W'_i$ plus $k$. Let $\bs W''_j:=\bs W'_j$ if $j\leq i$ and $\bs W''_j:=\bs W_j$ if $j\geq i$. Now, by concatenating the paths, $(\bs W, \bs W'')$ gives a coupling of two copies of $(\bs X_j)_j$, namely $(\bs X_j, \bs X''_j)_j$, such that $\bs X''_j = \bs X_{j+k}$ for large enough $j$, namely, for $j\geq \tau$. Hence, $\lim_j\myprob{\bs X''_j=\bs X_{j+k}}\geq \lim_j\myprob{\tau\leq j}=1$. This implies the claim.
\end{proof}

\begin{proof}[Proof of \Cref{thm:pathI=T-intro,thm:pathI=T}]
	The stronger form of weak aperiodicity, proved in \Cref{lem:aperiodicity}, implies that the assumption of \Cref{lem:pathTV} holds for every $k>0$.
	Let $(G,o)$ be a sample that satisfies~\eqref{eq:pathTV} for every $k$ (almost every sample satisfies these conditions).
	We now show that $(G,o)$ satisfies~\eqref{eq:pathI=T}. This claim follows from~\eqref{eq:pathTV} and Theorem~2.2 of~\cite{Ka92boundary}, but we include a direct proof which is a modification of that of Theorem~14.47 of~\cite{bookLyPe16} and {utilizes~\eqref{eq:pathTV}}. 
	
	Let $f$ be a bounded $\sigfield{path}{T}{G}$-measurable function. To prove~\eqref{eq:pathI=T}, it is enough to construct an $\sigfield{path}{I}{G}$-measurable function $\tilde f$ such that $\forall x\in G: \probPalm{(G,x)}{f=\tilde f}=1$. 
	For every $x\in G$ and $k\geq 0$, define 
	$\Lambda(G,x,k):=\omidPalm{(G,x)}{f\big((\bs X^{(k)}_i)_{i\geq 0}\big)}$, where $\bs X^{(k)}_{i}:=\bs X_{i-k}$ for $i\geq k$‌, $\bs X^{(k)}_i:=y_i$ for $i<k$, and $(y_i)_{0\leq i<k}$‌ is arbitrary. Since $f$ is a tail path function, the definition of $\Lambda(G,x,k)$ does not depend on the choice of $y_0,\ldots,y_{k-1}$.
	%
	%
	Note that $\Lambda(G,\cdot,\cdot)$ is space-time $K$-harmonic (as a function on $G\times (\mathbb N\cup\{0\})$). Hence, for all $n,k\geq 0$,
	\begin{align*}
		\Lambda(G,x,0) & = \omidPalm{(G,x)}{\Lambda(G,\bs X_{n+k},n+k)},\\
		\Lambda(G,x,k) & = \omidPalm{(G,x)}{\Lambda(G,\bs X_{n},n+k)}.
	\end{align*}
	So,
	\[
	\norm{\Lambda(G,x,0)-\Lambda(G,x,k)} \leq 2\ \mathrm{sup}(\Lambda) \tv{K^{n+k}(x,\cdot)-K^n(x,\cdot)} \to 0,
	\]
	where the last convergence holds as $n\to\infty$ by~\eqref{eq:pathTV} and the fact that $\Lambda$ is bounded (since $f$ is bounded). Therefore, $\Lambda(G,x,k)=\Lambda(G,x,0)$. Heuristically, this implies that $f$ is invariant (up to a negligible modification) by reversing the procedure, which is made precise as follows. Space-time $K$-harmonicity of $\Lambda$ implies that the function $\lambda(G,x):=\Lambda(G,x,0)$ is $K$-harmonic. Therefore, the martingale convergence theorem implies that, for every $x_0\in G$ and $\mathbb P_{(G,x_0)}$-a.e. sequence $(x_i)_{i\geq 0}$, the limit $\tilde f((x_i)_i):=\lim_{n\to\infty} \lambda(G,x_n)$ exists. Note that $\tilde f$ is invariant; i.e., $\sigfield{path}{I}{G}$-measurable. 
	Since $\Lambda(G,\bs X_n,n)$ is almost surely defined, the arguments above show that it is equal to $\Lambda(G,\bs X_n,0)$. So, by the definition of $\Lambda$ and the Markov property of $(\bs X_n)_n$, one almost surely has 
	\begin{align*}
		\tilde f((\bs X_i)_i) & =  \lim_n \lambda(G,\bs X_n) = \lim_n \Lambda(G,\bs X_n,n)\\
		&= \lim_n \omidPalmC{(G,o)}{f((\bs X_i)_i)}{\bs X_n}\\
		&= \lim_n \omidPalmC{(G,o)}{f((\bs X_i)_i)}{\bs X_0,\ldots, \bs X_n}.
	\end{align*}
	By L\'evy zero-one law, the last limit is equal to $f((\bs X_i)_i)$ a.s., and hence, $\tilde f = f$, $\mathbb P_{(G,o)}$-a.s. So, the claim is proved.
	%
\end{proof}

\begin{proof}[Proof of \Cref{thm:pathIntro,thm:pathTail}]
	Theorem~2.2 of~\cite{Ka92boundary} and~\eqref{eq:pathTV} imply that the Markov chain $([\bs G, \bs X_n])_n$ is steady. So, \Cref{thm:pathErgodic} implies that it has trivial tail. If $(\bs G, \bs o)$ has no nontrivial automorphisms a.s., then it follows that the Markov chain $([\bs G, \bs X_n; (\bs X_{i+n})_{i\geq 0}])_n$ is also tail-trivial. For the general case, we prove the claim similarly to that of \Cref{thm:pathI=T} mentioned above.
	Let $g$ be a bounded $\sigfield{path}{T}{}$-measurable function. It is enough to show that $g$ is essentially constant. Note that for every sample $(G,o)$, the function $f:=f_{(G,o)}:= g[G,o;\cdot]:G^{\mathbb N\cup\{0\}}\to\mathbb R$ is a tail function, i.e., is $\sigfield{path}{T}{G}$-measurable. The arguments in the proof of \Cref{thm:pathI=T} can be repeated word by word to show that for almost all samples $(G,o)$, the function $\Lambda(G,x,n)$ does not depend on $n$, 
	the function $\lambda(G,x):=\Lambda(G,x,0)$ is $K$-harmonic, and the function $\tilde f=\tilde f_{(G,o)}$ defined therein is equal to $f$, $\mathbb P_{(G,o)}$-a.s. Define the function $\tilde g:\mathcal G_{\infty}\to\mathbb R$ by $\tilde g[G,o;(x_i)_i]:=\tilde f_{(G,o)}((x_i)_i)$ and note that it is well defined (i.e., does not depend on the representative) and measurable. One also obtains $\tilde g = g$, $\mathbb P$- a.s. One can check that $\tilde g$ is invariant; i.e., $\sigfield{path}{I}{}$-measurable. Therefore, \Cref{thm:pathErgodic} implies that $\tilde g$ is essentially constant. Hence, so is $g$ and the claim is proved.
\end{proof}

\section{Steps C1 and L1: Ruling Out Non-Tail Properties in General Models}
\label{sec:drainage}

In this section, we prove Steps~\ref{step1comp} and~\ref{step1foil} for general point-maps (not limited to CMTs) in \Cref{thm:I in T 2} and provide the proof of \Cref{thm:general}. Before that, \Cref{subsec:drainage-events} provides the definitions and basic properties of tail branch/level-set properties.

\subsection{Tail Component/Level-Set Properties}
\label{subsec:drainage-events}

Recall that $\mathcal G'_*$ is the space of all $[G,o;f]$, where $G$ is a graph or discrete space and $f:G\to G$. Recall also the sigma-fields $\sigfield{pm}{I}{}$, $\sigfield{comp}{I}{}$ and $\sigfield{level}{I}{}$ of invariant point-map-properties, component-properties and level-set-properties from \Cref{def:invariantDrainage}.
In this subsection, we define tail point-map-properties, branch-properties and tail level-set-properties on $\mathcal G'_*$ for general point-maps. 

The path sigma-fields, defined in \Cref{sec:Markov}, apply to point-maps as well:
If $A$ is an event in $\mathcal G_{\infty}$, one might consider the event $[G,o;\anc_f(o)]\in A$ on $\mathcal G'_*$. More precisely, we identify $A$ with the event $\pi^{-1}(A)\subseteq\mathcal G'_*$, where $\pi:\mathcal G'_*\to\mathcal G_{\infty}$ is the projection defined by $\pi[G,o;f]:=[G,o;\anc_f(o)]$. 
So, by an abuse of notation, one may regard the sigma-fields $\sigfield{path}{I}{}$ and $\sigfield{path}{T}{}$, defined in \Cref{sec:Markov}, as sigma-fields on $\mathcal G'_*$. 

In the following definitions, we use $[G,o;f]$ as a variable element of $\mathcal G'_*$.

%
%
%

\begin{definition}[Tail Point-Map/Component/Level-Set Properties]
	\label{def:tailDrainage}
	$\quad$
	\begin{enumerate}[label=(\roman*)]
		\item An event $A\subseteq\mathcal G'_*$ is called a \defstyle{tail {point-map}-property} if 
		it is an invariant point-map property (that is, $A\in\sigfield{pm}{I}{}$, see \Cref{def:invariantDrainage}) and 
		it is invariant under all finite modifications of $f$; i.e., if $f':G\to G$ and $\card{\{f\neq f'\}}<\infty$, then $\identity{A}[G,o;f]= \identity{A}[G,o;f']$. Let $\sigfield{pm}{T}{}$ be the sigma-field of invariant tail point-map-properties.
		
		\item An event $A\subseteq\mathcal G'_*$ is called a \defstyle{tail component-property} if it is a component-property (that is, $A\in\sigfield{comp}{I}{}$) and it is invariant under {those} finite modifications of $f$ such that $C_f(o)$ is also modified at finitely many points. More precisely, if $f':G\to G$ is such that $\card{\{f\neq f'\}}<\infty$ and $\card{C_f(o)\Delta C_{f'}(o)}<\infty$, then $\identity{A}[G,o;f]= \identity{A}[G,o;f']$.
		Let $\sigfield{comp}{T}{}$ be the sigma-field of tail component-properties. 
		
		\item An event $A\subseteq\mathcal G'_*$ is called a \defstyle{tail level-set-property} if it is a level-set-property (that is, $A\in\sigfield{level}{I}{}$) and it is invariant under {those} finite modifications of $f$ such that $L_f(o)$ is also modified at finitely many points. More precisely, if $f':G\to G$ is such that $\card{\{f\neq f'\}}<\infty$ and $\card{L_f(o)\Delta L_{f'}(o)}<\infty$, then $\identity{A}[G,o;f]= \identity{A}[G,o;f']$.
		Let $\sigfield{level}{T}{}$ be the sigma-field of tail level-set-properties. 
	\end{enumerate}
\end{definition}

We assumed invariance in the definition of tail point-map properties for coherence with the definitions of tail level-set-properties and tail branch-properties defined next. Note that, when the base graph is deterministic, tail properties can be defined for events in $G^G$, which is classical, and no invariance assumption is needed (see also \Cref{thm:usf-ergodic}).

As mentioned in \Cref{intro:method}, tail component-properties are not well suited for two-ended models, which are studied in some of the results of this section. For this reason, we use the following stronger notion of tail branch-properties. For models with one-ended components, the reader might keep using tail component-properties in all of the following discussions; see \Cref{rem:branch}.


\begin{definition}[Tail Branch-Properties]
	\label{def:super-tail-comp}
	Given $(G,o;f)$, an \textbf{ancestral branch} of $C_f(o)$ is a branch $B$ of $C_f(o)$ (see \Cref{sec:notation}) such that $f^n(o)$ is eventually in $B$ as $n\to\infty$.	
	An event $A\subseteq\mathcal G'_*$ is called a \defstyle{tail ancestral-branch-property} (or for short, \defstyle{tail branch-property}) if 
	it is a component-property (that is, $A\in\sigfield{comp}{I}{}$) and,
	when $[G,o;f]$ and $[G,o;f']$ are such that $\card{\{f\neq f'\}}<\infty$ and $C_{f'}(o)$ shares an ancestral branch with $C_f(o)$ (and the graph structures of $f$ and $f'$ agree on the common branch), then $\identity{A}[G,o;f']=\identity{A}[G,o;f]$.
	Let $\sigfield{branch}{T}{}$ be the sigma-field of tail branch-properties. 
\end{definition}

In other words, a component-property $A$ is an ancestral-branch-property if
$\identity{A}[G,o;f]$ is determined by knowing $(G,o)$, knowing $f$ at all but finitely many points, and knowing an {ancestral branch} of $C_f(o)$.



\begin{remark}
	\label{rem:branch}
	It is straightforward to check that $\sigfield{branch}{T}{}\subseteq\sigfield{comp}{T}{}$ mod $\mathbb P$, assuming that the components are cycle-free a.s.
	Equality might fail in two-ended models (note that if $C_f(o)$ is two-ended, then a finite modification of $f$ might preserve an ancestral branch of $o$ but result in an infinite modification of $C_f(o)$).
	If all components are one-ended a.s., then $\sigfield{branch}{T}{}=\sigfield{comp}{T}{}$ mod $\mathbb P$. In this case, the reader can replace $\sigfield{branch}{T}{}$ with $\sigfield{comp}{T}{}$ and the arguments become simpler. This is safe for the indistinguishability theorem since indistinguishability is trivial in the two-ended case of CMTs by \Cref{thm:one-ended-or-connected}. However, we use $\sigfield{branch}{T}{}$ since some of the next results do not require one-endedness; e.g., \Cref{step1comp,step2comp} (\Cref{thm:I in T 2,thm:tailDrainage}). 
\end{remark}

The definitions imply that $\sigfield{pm}{T}{}\subseteq\sigfield{pm}{I}{}$, $\sigfield{branch}{T}{}\subseteq\sigfield{comp}{I}{}$ and $\sigfield{level}{T}{}\subseteq\sigfield{level}{I}{}$. We also have the following basic inclusions (see Figure~\ref{fig:sigfields}).

\begin{lemma}[Connections Between Path/Branch/Level-Set Sigma-Fields]
	\label{lem:basicInclusion0}
	One has $\sigfield{pm}{T}{}\subseteq \sigfield{branch}{T}{}$. Also, if $[\bs G, \bs o; \bs F]$ is a random element of $\mathcal G'_*$ such that $C(\bs o)$ is locally finite and cycle-free a.s., and $L(\bs o)$ is infinite a.s., then 
	$\sigfield{branch}{T}{}  \subseteq \sigfield{level}{T}{}$ mod $\mathbb P$.
\end{lemma}
\begin{proof}
	We only prove the second assertion and the rest is left to the reader. 
	Let $A\in\sigfield{branch}{T}{}$ and let $E$ be the set of $[G,o;f]$ in which $C_f(o)$ is locally finite and cycle-free, and $\card{L_f(o)}=\infty$. Note that $E\in\sigfield{level}{T}{}$.
	Assume $[G,o;f]\in E$, $[G,o';f']\in E$, $\card{\{f\neq f'\}}<\infty$ and $\card{L_f(o)\Delta L_{f'}(o')}<\infty$. These facts imply that there exists $z\in L_f(o)\cap L_{f'}(o')$, far enough from $o$, such that $\anc_f(z)=\anc_{f'}(z)$. As a result, $\anc_f(o)$ and $\anc_{f'}(o')$ both shift-couple with $\anc(z)$. Now, $C_f(o)$ and $C_{f'}(o')$ share an ancestral branch which is obtained by removing the set $\{f\neq f'\}$. Hence, the definition of $A\in\sigfield{branch}{T}{}$ implies that $\identity{A}[G,o';f']=\identity{A}[G,o;f]$. This proves that $A\cap E\in\sigfield{level}{T}{}$. The claim then follows since $A$ is equivalent to $A\cap E$.
	%
	%
	%
\end{proof}

As mentioned in Observations~\eqref{obs:comp} and~\eqref{obs:foil}, tail branch/level-set properties are connected to the invariant/tail path-sigma-fields as follows. Recall that Steps~\ref{step2comp} and~\ref{step2foil} are aimed at proving the converse of this lemma.

\begin{lemma}[Comparison of Point-Map/Path Sigma-Fields]
	\label{lem:basicInclusion}
	Let $[\bs G, \bs o; \bs F]$ be a random element of $\mathcal G'_*$. 
	\begin{enumerate}[label=(\roman*)]
		\item \label{lem:basicInclusion-comp} One has $\sigfield{path}{I}{} \subseteq \sigfield{branch}{T}{}$.
		\item \label{lem:basicInclusion-foil}
		If $C(\bs o)$ is locally finite and cycle-free a.s. and $L(\bs o)$ is infinite a.s., then $\sigfield{path}{T}{} \subseteq \sigfield{level}{T}{}$ mod  $\mathbb P$.
	\end{enumerate}
\end{lemma}

\begin{proof}
	\ref{lem:basicInclusion-comp}.
	Let $A\in\sigfield{path}{I}{}$. 
	We have $\identity{A}[G,o;f]=\identity{A}[G,o';f]$ for all $o'\in C_f(o)$, since $\anc(o)$ shift-couples (see \Cref{sec:notation}) with $\anc(o')$. This implies that $A\in\sigfield{comp}{I}{}$.
	Assume $(G,o;f)$ and $(G,o;f')$ are given such that $\card{\{f\neq f'\}}<\infty$ and $B$ is a common ancestral branch of $C_f(o)$ and $C_{f'}(o)$. This implies that $\anc_f(o)$ shift-couples with $\anc_{f'}(o)$, and one can hence find $i$ such that $v:=f^i(o)$ satisfies $\anc_f(v)=\anc_{f'}(v)\subseteq B$.  Since $A\in\sigfield{path}{I}{}$, one obtains that $\identity{A}[G,o;f] = \identity{A}[G,v;f]=\identity{A}[G,v;f']=\identity{A}[G,o;f']$. 
	 Hence, $A\in\sigfield{branch}{T}{}$ and the claim is proved.

	
	\ref{lem:basicInclusion-foil}.
	Let $E'$ be the event that $C(\bs o)$ is locally finite and cycle-free, and that $L(\bs o)$ is infinite. Note that $E'\in\sigfield{level}{T}{}$.
	Let $A\in\sigfield{path}{T}{}$.
	We have $\identity{A}[G,o;f]=\identity{A}[G,o';f]$ for all $[G,o;f]\in E'$ and $o'\in L_f(o)$ (since $\anc(o)$ eventually coincides with $\anc(o')$). 
	This implies that $A\cap E'\in\sigfield{level}{I}{}$.
	Assume $[G,o;f],[G,o;f']\in E'$ are such that $\card{\{f\neq f'\}}<\infty$ and $\card{L_f(o)\Delta L_{f'}(o)}<\infty$. These imply the existence of $z\in L_f(o)\cap L_{f'}(o)$ such that $\anc_f(z)=\anc_{f'}(z)$. Now, $\anc(z)$ eventually coincides with both $\anc_f(o)$ and $\anc_{f'}(o)$.
	Since $A\in\sigfield{path}{T}{}$, one obtains that $\identity{A}[G,o;f] = \identity{A}[G,z;f]=\identity{A}[G,z;f']=\identity{A}[G,o;f']$. 
	Therefore, $A\cap E'\in\sigfield{level}{T}{}$.
	 Since $A$ is equivalent to $A\cap E'$, one obtains that $A$ is in the null-event-augmentation of $\sigfield{level}{T}{}$ (\Cref{def:augmentation}).
\end{proof}


\subsection{Steps~\ref{step1comp} and~\ref{step1foil}}
\label{subsec:drainage-ergodic}

The following theorem establishes Steps~\ref{step1comp} and~\ref{step1foil} for general point-maps on unimodular graphs (not limited to CMTs).


\begin{theorem}[Steps~\ref{step1comp} and~\ref{step1foil}: Invariant Properties are Almost Tail]
	\label{thm:I in T 2}
	Let $[\bs G, \bs o]$ be a unimodular graph or discrete space and let $\bs F$ be a point-map on $\bs G$.
	\begin{enumerate}[label=(\roman*)]
		\item \label{thm:I in T 2-all}If $\bs G$ is infinite a.s.,
		then
		$
		\sigfield{pm}{I}{}= \sigfield{pm}{T}{} \text{ mod } \mathbb P.
		$
		More generally, this claim holds if $\bs F$ is any equivariant random graph on $\bs G$.
		\item \label{thm:I in T 2-comp} If $\bs F$ is cycle-free a.s., then 
		$
		\sigfield{comp}{I}{}=\sigfield{branch}{T}{} \text{ mod } \mathbb P.
		$
		\item \label{thm:I in T 2-foil} If the level-sets of $\bs F$ are infinite a.s., then 
		$
		\sigfield{level}{I}{}=\sigfield{level}{T}{} \text{ mod } \mathbb P.
		$
	\end{enumerate}
\end{theorem}
Part~\ref{thm:I in T 2-all} generalizes the classical fact in point process theory that the invariant sigma-field is included in the tail sigma-field modulo zero measure events (see e.g., Exercise 12.3.2 of~\cite{bookDaVe03II}). Part~\ref{thm:I in T 2-comp} establishes \Cref{step1comp}. As mentioned in \Cref{intro:method}, this generalizes a result of~\cite{HuNa17indistinguishability}, which proves the claim for $\wusf$. We prove the general case with a simpler proof that does not involve pivotal updates or $\epsilon,\delta$ calculations. Part~\ref{thm:I in T 2-foil} also establishes \Cref{step1foil}.

\begin{proof}
	Let $\sigfield{pm}{\hat T}{}$ denote the null-event-augmentation of $\sigfield{pm}{T}{}$. Define $\sigfield{branch}{\hat T}{}$ and $\sigfield{level}{\hat T}{}$ similarly.
	Fix $\epsilon>0$ and let $A$ be an event in $\sigfield{pm}{I}{}$, $\sigfield{comp}{I}{}$ or $\sigfield{level}{I}{}$, respectively in each part of the theorem. In order to prove that $A$ belongs to $\sigfield{pm}{\hat T}{}$, $\sigfield{branch}{\hat T}{}$ or $\sigfield{level}{\hat T}{}$ respectively, by \Cref{lem:completion-g}, it is enough to find a function $\alpha$, which is measurable with respect to $\sigfield{pm}{\hat T}{}$, $\sigfield{branch}{\hat T}{}$ or $\sigfield{level}{\hat T}{}$ respectively, and satisfies $\omid{\norm{\identity{A}-\alpha}}<\epsilon$. Since $A$ is a Borel set, there exist $R<\infty$ and an event $A'$ that depends only on $[\bs G, \bs o; \restrict{\bs F}{\oball{R}{\bs o}}]$, such that $\myprob{A\Delta A'}<\epsilon$. For part~\ref{thm:I in T 2-all}, take $\alpha:=\probCond{A'}{\sigfield{pm}{I}{}}$, which heuristically means averaging $\identity{A'}[\bs G,\cdot; \bs F]$ over all points of $\bs G$. One has 
	\begin{align*}
		\omid{\norm{\identity{A}-\alpha}} &= \omid{\norm{\omidCond{\identity{A}-\identity{A'}}{\sigfield{pm}{I}{}}}}\\
		& \leq \omid{\omidCond{\norm{\identity{A}-\identity{A'}}}{\sigfield{pm}{I}{}}}
		= \omid{\norm{\identity{A}-\identity{A'}}} = \myprob{A\Delta A'}<\epsilon.
	\end{align*}
	By taking similarly $\alpha:= \probCond{A'}{\sigfield{comp}{I}{}}$ in part~\ref{thm:I in T 2-comp} and $\alpha:=\probCond{A'}{\sigfield{level}{I}{}}$ in part~\ref{thm:I in T 2-foil}, one obtains similarly that $\omid{\norm{\identity{A}-\alpha}}<\epsilon$. So, it remains to prove that $\alpha$ is measurable with respect to $\sigfield{pm}{\hat T}{}$, $\sigfield{branch}{\hat T}{}$ or $\sigfield{level}{\hat T}{}$ respectively, in each part of the proof. In order to prove this, the definition of conditional expectation does not seem to help directly. Instead, we will construct an explicit version of $\alpha$ using a random walk or a suitable partition, as described below.	
	We prove the three statements in the order \ref{thm:I in T 2-all}, \ref{thm:I in T 2-foil}, \ref{thm:I in T 2-comp}.
	\vspace{2mm}
	
	\ref{thm:I in T 2-all}.
	%
	By Lemma~3.10 of~\cite{Kh23unimodular}, one can construct a Markov kernel $h$ on $\bs G$ that does not depend on $\bs F$, such that $h(\cdot, \cdot)>0$ and it is symmetric (i.e., $\forall x,y\in\bs G: h(x,y)=h(y,x)$) a.s. 	
	Let $(\bs Y_i)_{i\geq 0}$ be a random walk on $\bs G$ with kernel $h$, starting from $\bs o$ and independent from $\bs F$. 
	Let $\tilde{\mathbb P}$ be the distribution of $[\bs G, \bs o; \bs F, (\bs Y_i)_i]$.
	Since $h$ is symmetric, it balances the counting measure on $\bs G$. Hence, \Cref{lem:stationary} implies that the sequence $([\bs G, \bs Y_i; \bs F])_{i\geq 0}$ is stationary. So, Birkhoff's pointwise ergodic theorem implies that the following limit is well defined $\tilde{\mathbb P}$-a.s.
	\begin{equation}
		\label{eq:lem:I in T}
		g[\bs G, \bs o; \bs F, (\bs Y_i)_i]:= \lim_{N\to\infty} \frac 1 N \sum_{i=1}^N \identity{A'}[\bs G, \bs Y_i; \bs F].
	\end{equation}
	In addition, the pointwise ergodic theorem implies that $g=\probCond{A'}{\mathcal J}$, $\tilde{\mathbb P}$-a.s., where $\mathcal J$ is the sigma-field of events that are invariant under the shift $(\bs Y_i)_{i\geq 0}\mapsto (\bs Y_{i+1})_{i\geq 0}$. So, we may apply \Cref{thm:pathErgodic} for $(\bs Y_i)_i$ (and keep $\bs F$ as a decoration) and deduce that $\mathcal J=\sigfield{pm}{I}{}$ mod $\tilde{\mathbb P}$. Hence, $g=\probCond{A'}{\sigfield{pm}{I}{}}=\alpha$, $\tilde{\mathbb P}$-a.s.
	
	Let $M'$ be the set of $[G,o;f]$ such that:
	\begin{enumerate}[label=(\arabic*)]
		\item $G$ is infinite and, for all $x,y\in G$, one has $h(G,x,y)=h(G,y,x)>0$ and $\sum_z h(G,x,z)=1$.
		\item For all $x\in G$, the random variable $g[G,x;f,(\bs Y_i)_i]$ (with respect to the randomness of $(\bs Y_i)_i$ started from $x$) is essentially constant.
		\item If $g'[G,x,f]$ is the constant such that $g[G,x;f,(\bs Y_i)_i]=g'[G,x,f]$ a.s., then $g'[G,x,f]$‌ does not depend on $x$.
	\end{enumerate}
	 Observe that $\myprob{M'}=1$ and $g'=\alpha$, a.s. So, it is enough to show that $g'$‌ is $\sigfield{pm}{\hat T}{}$-measurable. 
	 Assume $[G,o;f]\in M'$, $o'\in G$ and $f'\in G^G$‌ such that $\card{\{f\neq f'\}}<\infty$. 
	 By \Cref{lem:cber,lem:completion-I(R)}, to prove the $\sigfield{pm}{\hat T}{}$-measurability of $g'$, it is enough to show that $[G,o';f']\in M'$ and $g'[G,o;f] = g'[G,o';f']$ (this also implies that $M'\in\sigfield{pm}{T}{}$, as required in \Cref{lem:completion-I(R)}).
	 Let $x\in G$‌ be arbitrary and let 
	 $(\bs Y_i)_i$ be the random walk on $G$ with kernel $h$, starting from $x$ (note that it does not depend on $f$ nor $f'$). Since the kernel $h$ on $G$ is positive, this Markov chain on $G$ is irreducible. Therefore, the counting measure on $G$ is the unique stationary measure up to multiplicity by constants. Since the counting measure has infinite total mass (since $G$ is infinite by assumption), this chain cannot be positive recurrent. So, either it is transient or null recurrent. In either case, the frequency of visits of $(\bs Y_i)_i$ to the set $\oball{n+R}{o}$ is zero a.s., where $n$ is chosen such that $\{f\neq f'\}\subseteq \oball{n}{o}$. On the other hand, the definition of $A'$ implies that $\identity{A'}[G,y; f] = \identity{A'}[G,y;f']$ for all $y\in G\setminus\oball{n+R}{o}$. These facts imply that $g[G,x;f,(\bs Y_i)_i] = g[G,x; f', (\bs Y_i)_i]$ a.s. This in turn implies that $g[G,x;f,(\bs Y_i)_i]$ is essentially constant and the constant does not depend on $x$, and hence, the claim is proved.
	\vspace{2mm}
	
	\ref{thm:I in T 2-foil}. Instead of using a random walk with kernel $h$, we use asymptotic averaging on the level-set of the origin as follows. Define
	\[
	g[\bs G, \bs o; \bs F]:= \lim_{i\to\infty}    \mathrm{average}\{\identity{A'}[\bs G, x; \bs F]: x\in D_i(\bs F^{(i)}(\bs o)) \}, 
	\]
	if the limit exists, (see the definition of $\mathrm{average\{\cdot\}}$ just after \Cref{lem:nested}). If the limit does not exist, let $g=0$. Indeed, since for each $i$, the partition $\{D_i(\bs F^i(v)): v\in L(\bs o)\}$ of $L(\bs o)$ is equivariant and these partitions are nested,
	Lemma~8.1 of~\cite{Kh23unimodular} (see \Cref{lem:nested}) implies that the limit exists and is equal to $\probCond{A'}{\sigfield{level}{I}{}}=\alpha$ a.s.
	So, it is enough to prove that $g$ is $\sigfield{level}{\hat T}{}$-measurable. We will prove this by \Cref{lem:completion-I(R)} again.
	
	First, note that $g[\bs G, \bs o; \bs F]=g[\bs G, y; \bs F]$ for all $y\in L(\bs o)$, since $D_i(\bs F^{(i)}(\bs o))=D_i(\bs F^{(i)}(y))$ for large enough $i$.
	Let $M'$ be the set of $[G,o;f]$ such that the graph of $f$ is locally-finite, all level-sets of $f$ are infinite, and $f$ is one-ended. Note that $M'\in\sigfield{level}{T}{}$ and $\myprob{M'}=1$ by \Cref{thm:classification}. Assume $[G,o;f]\in M', [G,o;f']\in M', \card{\{f\neq f'\}}<\infty$ and $\card{L_f(o)\Delta L_{f'}(o)}<\infty$.
	This implies that, for all but finitely many $y\in L_f(o)\cap L_{f'}(o)$, one has $\anc_f(y)=\anc_{f'}(y)=:\anc(y)$ and $\anc(y)$ does not intersect $\oball{n+R}{o}$, where $n$ is chosen such that $\{f\neq f'\}\subseteq \oball{n}{o}$. Note that $\identity{A'}[G,x;f]=\identity{A'}[G,x;f']$ for all $x\in G\setminus\oball{n+R}{o}$. This fact, together with the infiniteness of $L_f(o)$, imply $g[G,o;f]=g[G,o;f']$. 
	Since  $M'\in\sigfield{level}{T}{}$, \Cref{lem:cber,lem:completion-I(R)} imply that $g$ is $\sigfield{level}{\hat T}{}$-measurable, and the claim is proved. 
	\vspace{2mm}
	
	\ref{thm:I in T 2-comp}. 
	We will use percolation in the following steps.
	\vspace{2mm}
	
	\textbf{Defining $\hat g$}.
	Let $[G, o; f]$ be a sample such that the graph of $f$ is locally-finite and $C_f(o)$ is either one-ended or two-ended. By \Cref{thm:classification}, almost every sample is of this type. We will define $\hat{g}[G,o;f]$ in the following cases.
	
	First, consider the case where $C(o)$ is one-ended. Let $\Phi_i$ be the Bernoulli bond percolation on $C(o)$ with parameter $1-2^{-i}$; i.e., every edge of $C(o)$ belongs to $\Phi_i$ with probability $1-2^{-i}$ independently from the other edges. Let $\bs S_i = \bs S_i(o)$ be the connected component of $o$ in $\Phi_i$. One might also couple the percolations for different values of $i$ such that $\Phi_i\subseteq\Phi_{i+1}$, and hence, $\bs S_i\subseteq \bs S_{i+1}$. Then, $\bs S_i$ is finite a.s. and converges to $C(o)$ as $i\to\infty$. 
	Define 
	\[g_i[G,o;f, (\Phi_i)_i]:=\mathrm{average}\{\identity{A'}[G, x; f]: x\in \bs S_i \}.\] 
	Since the percolation clusters form an equivariant nested sequence of partitions of $C(o)$ that converge to the trivial partition of $C(o)$, Lemma~8.1 of~\cite{Kh23unimodular} (see \Cref{lem:nested}) implies that $g:=\lim_i g_i$ exists in almost every sample of $[G,o;f, (\Phi_i)_i]$. 
	Note that $g$ might depend on the $\Phi_i$'s as well (in fact, one can see at the end of the proof that $g$ does not depend on $\Phi_i$'s a.s. since adding Poisson randomness does not break ergodicity). Define $\hat g[G, o; f]:=\omid{g}$, where the expectation is with respect to the randomness of the $\Phi_i$'s. 
	
	
	Second, consider the case where $C(o)$ is two-ended. In this case, let $H$ be the union of the bi-infinite paths in the graph of $f$ (note that there might be more than one bi-infinite path since we consider general point-maps in this theorem and \Cref{thm:one-ended-or-connected} is not available). For $u\in G$, let $\tau(u)=\tau_f(u)$ be the first element of $\anc(u)$ that lies in $H$ (if well defined). Let $Y_0:=\tau(o)$, $Y_j:=f^j(Y_0)$ for $j> 0$ and let $Y_{-j}$ be the unique point in $D_j(Y_0)\cap H$. Note that $\left(\tau^{-1}(Y_j)\right)_j$ is the partition of $C(o)$ into the branches hanging from a bi-infinite path. Let $g'_i$ be the average of $\identity{A'}[G,x;f]$ over $x\in\bigcup_{j=0}^i \tau_f^{-1}(Y_j)$. Define $g''_i$ similarly by the overage over $j\in [-i,i]$. Let $\hat g:=\lim_i g'_i$ if the limit exists. We will show that this is the case in almost every sample.
	\vspace{2mm}
	
	\textbf{Defining $M''$}.
	Let $M''$ be the set of $[G,o;f]$ that:
	\begin{enumerate}[label=(\arabic*)]
		\item $[G,o;f]$ satisfies the claims of \Cref{thm:classification} and each component of $f$ is either one-ended or two-ended.
		\item If $C_f(o)$‌ is one-ended, then $g=\lim_i g_i[G,o;f,(\Phi_i)_i]$ exists for almost every sample of $(\Phi_i)_i$.
		\item If $C_f(o)$ is two-ended, then $\hat g = \lim_i g'_i$ exists.
	\end{enumerate} 
	
	\textbf{Proving $\myprob{M''}=1$}.
	To prove this claim, it remains to show that, on the event $E$ in which $C(\bs o)$ is two-ended, $\hat g$ exists a.s. 
	We may assume $\myprob{E}>0$.
	By the mass transport principle, one obtains $\omid{\card{\tau^{-1}(\bs o)}} = \myprob{E}$. Hence, $\tau^{-1}(\bs o)$ is finite a.s.
	One can easily show that the sequence $\left([\bs G, Y_i; \bs F]\right)_i$, biased by $1/\card{\tau^{-1}(Y_0)}\identity{E}$, is stationary (this is a general property of unimodular two-ended trees and can be proved by using the mass transport principle directly, see Theorem~4.17 of~\cite{eft}). Therefore, Birkhoff's pointwise ergodic theorem implies that the following limits exist a.s. on $E$:
	\begin{align*}
		\lim_i \frac 1 i\sum_{j=0}^{i-1} \card{\tau^{-1}(Y_j)},\quad
		\lim_i \frac 1 i \sum_{j=0}^{i-1} \sum_{x\in \tau^{-1}(Y_j)} \identity{A'}[\bs G, x; \bs F].
	\end{align*}
	Hence, $\lim_i g'_i$ exists and the claim is proved. In addition, one can make the sum over $j\in [-i,i]$ similarly (and divide by $2i+1$) and the value of the limits does not change a.s. Therefore, $\lim_i g'_i=\lim_i g''_i$ a.s. on $E$. 
	\vspace{2mm}
	
	\textbf{Proving that $\hat g=\alpha$ a.s.} Consider the percolation $\Phi$ defined above in the one-ended case. When $C(o)$ is two-ended, replace $\Phi_i$ by the percolation on the unique bi-infinite path of $C(o)$. Let the rest of the edges of $C(o)$ be always in $\Phi_i$ (i.e., be \textit{open}). Define $\bs S_i, g_i$ and $g$ similarly to the one-ended case. Since $\bs S_i$ is a union of the sets $\tau^{-1}(Y_j)$ over $j$ in a large interval, one deduces that $g=\lim_i g''_i=\lim_i g'_i=\hat g$ a.s. on $E$. Now, \Cref{lem:nested} for the partition $\bs S_i$ (both in the one-ended and two-ended cases) implies that $g=\alpha$, $\tilde{\mathbb P}$-a.s., where $\tilde{\mathbb P}$ is the distribution of $[\bs G, \bs o; \bs F, (\Phi_i)_i]$. This implies that $\hat g=\alpha$, $\mathbb P$-a.s.

	\vspace{2mm}
	
	\textbf{Measurability of $\hat g$}.
	We now show that $\hat g$ is $\sigfield{branch}{\hat T}{}$-measurable.
	Assume that $[G,o;f]\in M'', o'\in G, f'\in G^G, \card{\{f\neq f'\}}<\infty$ and there exists a common ancestral branch $B$ of $C_f(o)$ and $C_{f'}(o')$. By \Cref{lem:cber,lem:completion-I(R)}, we need to prove that $[G,o';f']\in M''$‌ and $\hat g[G,o;f]=\hat g[G,o';f']$ (this also shows that $M''\in\sigfield{branch}{T}{}$, as required in \Cref{lem:completion-I(R)}). We may assume $\{f\neq f'\}\subseteq \oball{n}{o}$. Note that $\restrict{f}{B}=\restrict{f'}{B}$. So, by \Cref{thm:classification}, either both $C_f(o)$ and $C_{f'}(o')$ are one-ended or both are two-ended. We consider these cases separately.
	
	First, assume $C_f(o)$ and $C_{f'}(o')$ are one-ended. In this case, $C_f(o)\setminus B$ and $C_{f'}(o')\setminus B$ are finite. 
	Denote by $(\Phi'_i)_i$  the sequence of percolations used in the definition of $\hat{g}[G,o';f']$ and define $(\bs S'_i)_i$ similarly. One can couple $(\Phi'_i)_i$ with $(\Phi_i)_i$ in such a way that $\restrict{\Phi_i}{B} = \restrict{\Phi'_i}{B}, \forall i$ a.s. (since these are Bernoulli percolations). 
	This implies that $\bs S_i\cap B = \bs S'_i\cap B$ for large enough $i$. Note that $\bs S_i\setminus B$ and $\bs S'_i\setminus B$ are included in a finite set which does not depend on $i$, and also $\identity{A'}[G,x;f]=\identity{A'}[G,x;f]$ if $x\in G\setminus \oball{n+R}{o}$. 
	So, in the averages made in $g_i[G,o;f, (\Phi_i)_i]$ and $g_i[G,o';f',(\Phi'_i)_i]$, all of the terms are equal, except finitely many of them.
	This implies that $\hat{g}[G,o;f]=\hat{g}[G,o';f']$, as desired.
	
	Second, assume $C_f(o)$ and $C_{f'}(o')$ are two-ended. Define $H'$, $(Y'_i)_i$ and $\tau_{f'}$ similarly for $f'$. Since $C_f(o)$ and $C_{f'}(o')$ share an ancestral branch, one can easily deduce that 
	\[
	\exists k\in \mathbb Z, \exists J\geq 0: \forall j\geq J: Y'_j=Y_{j+k}, \quad \tau_{f'}^{-1}(Y'_j)= \tau_f^{-1}(Y_{j+k}).
	\]
	As a result, the averaging sets in the definitions of $g'_i[G,o;f]$ and $g'_i[G,o';f']$ are almost equal. 
	Since $\identity{A'}[G,x;f]=\identity{A'}[G,x;f']$ for all $x\in G\setminus\oball{n+R}{o}$, one obtains that $\lim_i g'_i[G,o;f]=\lim_i g'_i[G,o';f']$; i.e., $\hat g[G,o;f]=\hat g[G,o';f']$, as desired.
	
	In both cases, we showed that $\hat g[G,o;f]=\hat g[G,o';f']$. 
	So, \Cref{lem:completion-I(R)} for the event $M''$ implies that $\hat g$ is $\sigfield{branch}{\hat T}{}$-measurable. So, $\alpha$ is also $\sigfield{branch}{\hat T}{}$-measurable and the proof is completed.
	%
	%
\end{proof}

\begin{proof}[Proof of \Cref{thm:general}]
	We prove \ref{thm:general:comp}. The proof of~\ref{thm:general:level} is similar and is omitted. In \Cref{thm:I in T 2}, we proved \Cref{step1comp}; i.e., $\sigfield{comp}{I}{}=\sigfield{branch}{T}{}$ mod $\mathbb P$. If the claim of \Cref{step2comp} holds, one has $\sigfield{branch}{T}{}=\sigfield{path}{I}{}$ mod $\mathbb P$. If the claim of \Cref{step3comp} holds, then $\sigfield{path}{I}{}=\sigfield{}{I}{}$ mod $\mathbb P$. If both hold, one obtains that $\sigfield{comp}{I}{}=\sigfield{}{I}{}$ mod $\mathbb P$. This implies indistinguishability of components by \Cref{lem:indistinguishability}.
\end{proof}

\section{Steps C2 and L2: Reduction to Path Properties and Proof of Indistinguishability in CMTs}
\label{sec:cmt-proof}

In this section, we establish Steps~\ref{step2comp} and~\ref{step2foil} for CMTs (we do not require unimodularity of the base space). Then, using the results of \Cref{sec:Markov,sec:drainage}, the proof of indistinguishability for CMTs (\Cref{thm:indistinguishability}) is completed in \Cref{subsec:proofOfIndistinguishability}. Finally, indistinguishability for CMTs on a deterministic base space (see \Cref{main:deterministic}) is discussed in \Cref{subsec:tail-fixed}.

\del{
Note that the triviality of $\sigfield{branch}{T}{}$ and $\sigfield{level}{T}{}$ is used in proving the triviality of $\sigfield{comp}{I}{}$ and $\sigfield{level}{I}{}$, unlike the proof for Markov chains (where the triviality of $\sigfield{path}{I}{}$ is proved before that of $\sigfield{path}{T}{}$).
}

\subsection{Steps~\ref{step2comp} and~\ref{step2foil} for CMTs}
\label{subsec:drainageTail}

\del{It\mar{\ali{The second half of this paragraph is already discussed many times. Maybe a 1-line introductory sentence is enough.}} is a classical fact that an i.i.d. marking of a deterministic base space is tail trivial. Part~\ref{thm:tailDrainage-all} of \Cref{thm:tailDrainage} below is an analogous claim for CMTs on random base spaces (not necessarily unimodular). In the next parts of the theorem, analogous claims are proved for tail branch/level-set properties of CMTs, which reduce such properties to invariant/tail properties of $\anc(\bs o)$ and establish Steps~\ref{step2comp} and~\ref{step2foil}.}

\begin{theorem}[Steps~\ref{step2comp} and~\ref{step2foil} for CMTs]
	\label{thm:tailDrainage}
	Let $\bs F$ be a CMT on any random rooted graph or discrete space $[\bs G, \bs o]$. Then, 
	\begin{enumerate}[label=(\roman*)]
		\item \label{thm:tailDrainage-all} $\sigfield{pm}{T}{} = \sigfield{}{I}{}$ mod  $\mathbb P$, where $\sigfield{}{I}{}$ is the invariant sigma-field (not to be confused with $\sigfield{pm}{I}{}$, see \Cref{thm:I in T 2}).
		\item \label{thm:tailDrainage-comp} $\sigfield{branch}{T}{} = \sigfield{path}{I}{}$ mod $\mathbb P$.
		\item \label{thm:tailDrainage-foil} If $C(\bs o)$ is locally finite and cycle-free a.s. and $L(\bs o)$ is infinite a.s., then $\sigfield{level}{T}{} = \sigfield{path}{T}{}$ mod $\mathbb P$.
	\end{enumerate}	
%
\end{theorem}
The first part is a  version of the classical fact that an i.i.d. marking of a deterministic space is tail-trivial. Note that unimodularity is not assumed in this result.
Note also that, in contrast with the proof of \Cref{sec:warmUp} for lattice models, we do not require any shift-coupling or successful coupling property. The other conditions of \Cref{model1,model2} are not required as well. 

\begin{proof}
	Let $K$ be the Markov kernel of $\bs F$, defined in \Cref{subsec:cmt}.
	The proof of the first statement uses a method similar to that of \Cref{sec:warmUp}. The same applies to the other statements, but heavily relies on the Markov chain sigma-fields. The idea is to show that the event under consideration is independent from itself, conditionally on $[\bs G, \bs o; \anc(\bs o)]$ (this establishes \Cref{step2comp-1}), and to show that its conditional probability is measurable with respect to the desired path sigma-field (this establishes \Cref{step2comp-2}). The proof relies on \Cref{lem:completion-I(R),lem:completion-cond}.
	\vspace{0.2cm}
	
	\ref{thm:tailDrainage-all}.
	Since events in $\sigfield{}{I}{}$ depend only on the base graph (not on the CMT nor on the root), one has $\sigfield{}{I}{}\subseteq\sigfield{pm}{T}{}$. So, we should prove that $\sigfield{pm}{T}{}\subseteq\sigfield{}{I}{}$.	
	Let $A\in\sigfield{pm}{T}{}$ and $n\in\mathbb N$ be arbitrary. Construct a copy $\bs F'$ of $\bs F$ coupled with $\bs F$ as follows. For every sample $(G,o)$ of $[\bs G, \bs o]$, construct $\bs F$ on $G$ and construct $\bs F'$ inside $\oball{n}{o}$ as an independent copy of $\restrict{\bs F}{\oball{n}{o}}$. Outside $\oball{n}{o}$, let $\bs F'$ be identical to $\bs F$. Then, $\bs F'$ has the same distribution as $\bs F$ given $(G,o)$. Let $A'$ be the event $[\bs G, \bs o; \bs F']\in A$. By construction, $A$ is independent from $\restrict{\bs F'}{\oball{n}{\bs o}}$ (conditionally on $[\bs G, \bs o]$). Also, since $A$ is a tail point-map-property, one has $\identity{A}=\identity{A'}$ a.s. This shows that $A'$ is independent from $\restrict{\bs F'}{\oball{n}{\bs o}}$ (conditionally on $[\bs G, \bs o]$). Since this holds for all $n$, one obtains that $A'$ is independent from $\bs F'$ (conditionally on $[\bs G, \bs o]$). Therefore, $A'$ is independent from itself (conditionally on $[\bs G, \bs o]$), and hence, the same holds for $A$ as well. So, $g:=\probCond{A}{[\bs G, \bs o]}\in\{0,1\}$ a.s. So, Lemma~\ref{lem:completion-cond} implies that $A$ is in the null-event-augmentation of the sigma-field generated by $[\bs G, \bs o]$. 
	
	To complete the proof, 
	by Lemma~\ref{lem:completion-cond}, it is enough to construct a version of $\probCond{A}{[\bs G, \bs o]}$ that depends only on $\bs G$. Let $G$ be deterministic. For $x\in G$, define $g'[G,x]:=\probPalm{G}{[G,x;\bs F]\in A}$. It is clear that $g'$ is a version of $\probCond{A}{[\bs G, \bs o]}$. By using the same instance of $\bs F$ for different $x,y\in G$, and the fact $A\in\sigfield{pm}{I}{}$, one obtains that $g'[G,x]=g'[G,y]$. So, $g'$ depends only on $G$ and the claim is proved.
	\vspace{0.2cm}	
	
	\ref{thm:tailDrainage-comp}.
	\Cref{lem:basicInclusion} implies that $\sigfield{path}{I}{}\subseteq\sigfield{branch}{T}{}$. Conversely,
	let $A\in\sigfield{branch}{T}{}$. We will use \Cref{lem:completion-cond} for the sigma-fields $\mathcal F$ and $\mathcal F'$ with $\mathcal F=\sigfield{path}{I}{}$ and $\mathcal F'$ the sigma-field of events in $\mathcal G'_*$ that depend only on $[\bs G, \bs o; \anc(\bs o)]$. To use this lemma, we first show that $\probCond{A}{[\bs G, \bs o; \anc(\bs o)]}$ is $\hat{\mathcal F}$-measurable (see the notation used in \Cref{lem:completion-cond}). For this, we use the following explicit version of the conditional probability. Let $M'$ be the set of $[G,x_0;(x_i)_{i\geq 0}]$ such that $x_N,x_{N+1},\ldots$ are distinct for some $N$. 
	Note that $M'\in\sigfield{path}{I}{}$ (while the event that $x_0,x_1,\ldots$ are distinct is not in $\sigfield{path}{I}{}$).
	Note also that $\myprob{[\bs G, \bs o; \anc(\bs o)]\in M'}=1$. Given $[G,x_0;(x_i)_{i\geq 0}]\in M'$, define a random element $\bs F_1\in G^G$ using $(x_i)_i$ as follows: 
	Let $\bs F$ be the CMT with kernel $K$ on $G$.
	If $N$ is the smallest number such that $(x_i)_{i\geq N}$ are distinct, let $\bs F_1(x_{j}):=x_{j+1}, \forall j\geq N$. 
	For other points $x\in G$, let $\bs F_1(x):=\bs F(x)$. Then, on $M'$, define
	\begin{equation}
		\label{eq:thm:tailDrainage-g}
		g[G,x_0;(x_i)_i]:= \probPalm{G}{[G,x_0;\bs F_1]\in A}.
	\end{equation}
	Outside $M'$, let $g[G,x_0;(x_i)_i]:=0$.
	Note that $\anc(x_0)$ is a simple path a.s. and that $\bs F_1$ is a regular conditional version of $\bs F$ conditionally on $\anc_{\bs F}(x_0)=(x_i)_i$.
	
	We now prove that $g$ is $\hat{\mathcal F}$-measurable. For this, we use Lemma~\ref{lem:completion-I(R)} for the event $M'$ and the countable Borel equivalence relation corresponding to $\sigfield{path}{I}{}$; see \Cref{lem:cber} (in fact, since $g$ is zero outside $M'$, it is $\mathcal F$-measurable). Assume $[G,x_0;(x_i)_i]\in M'$, $[G,x'_0;(x'_i)_i]\in M'$ and that $(x'_i)_i$ shift-couples with $(x_i)_i$ (but maybe $x_0\neq x'_0$). Assume $x_{m+i}=x'_{m'+i}$ for all $i\geq 0$. 
	Define $\bs F'_1$ similarly to $\bs F_1$, but for $(x'_i)_i$, using the same source of randomness $\bs F$.
	One has
	\begin{equation}
		\label{eq:tailDrainage-coupling}
		\begin{split}
			g[G,x_0;(x_i)_i]&= \probPalm{G}{[G,x_0;\bs F_1]\in A},\\
			g[G,x'_0;(x'_i)_i]&= \probPalm{G}{[G,x'_0;\bs F'_1]\in A}.
		\end{split}
	\end{equation}
	Now, the constructed coupling between $\bs F_1$‌ and $\bs F'_1$ implies that $[G,x'_0;\bs F'_1]$ is obtained from $[G,x_0;\bs F_1]$ by changing the root to $x_m$, which is in the same component as $x_0$, then applying a finite modification without changing the ancestors of the new root, and then moving the root again in its component (to $x'_0$). So, the fact $A\in\sigfield{branch}{T}{}$ implies that $\identity{A}[G,x_0;\bs F_1] = \identity{A}[G,x'_0; \bs F'_1]$ a.s., and hence, $g[G,x_0;(x_i)_i] = g[G,x'_0,(x'_i)_i]$. Then, \Cref{lem:completion-I(R),lem:cber} imply that $g$ is measurable with respect to $\hat{\mathcal F}$.
	
	
	Second, we show that $g\in\{0,1\}$ a.s.
	Let $n\in\mathbb N$ be arbitrary. Given any sample $(G,o)$, construct a copy $\bs F'$ of $\bs F$ coupled with $\bs F$ as follows. Given $\bs F$, construct $\bs F'$ inside $\oball{n}{o}\setminus\anc(o)$ as an independent copy of $\bs F$. In $\oball{n}{o}^c\cup \anc(o)$, let $\bs F'$ be identical to $\bs F$. Then, $\bs F'$ has the same distribution as $\bs F$ given $(G, o)$. Let $A'$ be the event $[G, o; \bs F']\in A$. By construction, $A'$ is independent from the restriction of $\bs F$ to $\oball{n}{o}$, conditionally on $(G, o; \anc(o))$. Also, since $A\in\sigfield{branch}{T}{}$, one has $\identity{A}=\identity{A'}$ a.s. in this coupling. This shows that $A$ is independent from the restriction of $\bs F$ to $\oball{n}{o}$ conditionally on $(G, o; \anc(o))$. Since this holds for all $n$, one obtains that $A$ is independent from $\bs F$ conditionally on $(G, o; \anc(o))$. Therefore, $A$ is independent from itself conditionally on $(G, o; \anc(o))$). This implies that $g\in\{0,1\}$ a.s.
	
	By the above facts, Lemma~\ref{lem:completion-cond} implies that $A\in\hat{\mathcal F}$, and the claim is proved.
	
	
	\vspace{0.2cm}
	\ref{thm:tailDrainage-foil}.
	The proof is similar to the previous case, so, we sketch the proof and only highlight the differences.
	By \Cref{lem:basicInclusion}, one has $\sigfield{path}{T}{}\subseteq\sigfield{level}{T}{}$ mod $\mathbb P$. Conversely,
	let $A\in\sigfield{level}{T}{}$.
	We claim that $\probCond{A}{[\bs G, \bs o; \anc(\bs o)]}$ is $\sigfield{path}{\hat T}{}$-measurable, where the latter is the null-event-augmentation of $\sigfield{path}{T}{}$. Consider the regular version $g$ of this conditional probability defined in~\Cref{eq:thm:tailDrainage-g}.
	We prove that $g$ is $\sigfield{path}{\hat T}{}$-measurable. 
	For this, we use \Cref{lem:completion-I(R)} for the set $M''$ of all $[G,x_0;(x_i)_i]$ that satisfy the following properties: 
	$(x_i)_{i\geq N}$ are distinct for some $N$ and if $\bs F_1$ is the corresponding modification of $\bs F$ (defined before~\Cref{eq:thm:tailDrainage-g}), then $\bs F_1$ is locally-finite and $L_{\bs F_1}(x_0)$ is infinite, $\mathbb P_G$-a.s.
	Note that $M''$‌ is a tail path property and $\myprob{[\bs G, \bs o; \anc(\bs o)]\in M''}=1$.
	Assume $[G,x_0;(x_i)_i]\in M''$ and $[G,x'_0;(x'_i)_i]\in M''$ such that $(x'_i)_i$ eventually coincides with $(x_i)_i$.
	We should prove that $g[G,x_0;(x_i)_i]=g[G,x_0;(x'_i)_i]$. 
	By defining $\bs F_1$ and $\bs F'_1$ similarly to the previous case, \eqref{eq:tailDrainage-coupling} is still valid. 
	One can see that $\card{L_{\bs F_1}(x_0)\Delta L_{\bs F'_1}(x'_0)}<\infty$ in this coupling.
	Now, $[G,x'_0;\bs F'_1]$ is obtained from $[G,x_0;\bs F_1]$ by three steps: By moving the root $x_0$ to a point $z\in L_{\bs F_1}(x_0)$ which is far enough, then applying a finite modification to $\bs F_1$ {(which is a finite modification of $L(z)$ as well)}, and then moving the root again to $x'_0$, which is in the level-set of $z$ if $z$ is far enough.
	Therefore, since $A\in\sigfield{level}{T}{}$, one obtains that $\identity{A}[G,x_0;\bs F_1] = \identity{A}[G,x'_0; \bs F'_1]$ a.s. So, \eqref{eq:tailDrainage-coupling} implies that $g[G,x_0;(x_i)_i]=g[G,x'_0;(x'_i)_i]$. Hence, as mentioned above, \Cref{lem:cber,lem:completion-I(R)} imply that $g$ is $\sigfield{path}{\hat T}{}$-measurable, and the claim is proved.
	
	It is proved in the previous part that $g\in\{0,1\}$ a.s. (the proof does not rely on $A\in\sigfield{branch}{T}{}$). 
	So, Lemma~\ref{lem:completion-cond} implies that $A$ is in the null-event-augmentation of $\sigfield{path}{T}{}$ and the proof is completed.
%
\end{proof}

\subsection{Proof of Indistinguishability for CMTs}
\label{subsec:proofOfIndistinguishability}
	
We now conclude the proof of \Cref{thm:indistinguishability}.

\begin{proof}[Proof of \Cref{thm:indistinguishability}]
	For the indistinguishability of components, Steps~\ref{step1comp}, \ref{step2comp} and~\ref{step3comp} are established in Theorems~\ref{thm:I in T 2}, \ref{thm:tailDrainage} and~\ref{thm:pathErgodic} respectively. 
	For the indistinguishability of level-sets, Steps~\ref{step1foil}, \ref{step2foil} and~\ref{step3foil} are established in Theorems~\ref{thm:I in T 2}, \ref{thm:tailDrainage} and~\ref{thm:pathTail} respectively. 
	So, the claims are implied by \Cref{thm:general}.
%
%
%
\end{proof}

\subsection{Indistinguishability in CMTs on Deterministic Graphs}
\label{subsec:tail-fixed}

In this subsection, we prove the analogous indistinguishability results for CMTs on deterministic graphs; i.e., \Cref{thm:fixed-comp,thm:fixed-foil}. Let $G$ be a deterministic graph or discrete space and $K$ be a Markov kernel on $G$. One can define a CMT $\bs F$ with kernel $K$, similarly to \Cref{def:cmt}, on the space $G^G$ (or $\{0,1\}^{G\times G}$). Component-properties and tail component-properties are defined in~\cite{Hu20indistinguishability}. We recall this definition and modify it to tail branch-properties as follows. A Borel subset $A$ of $G\times G^G$ (rather than $\mathcal G'_*$) is called a \defstyle{component-property} if $\identity{A}(v,f)=\identity{A}(v',f)$ for every $f\in G^G$, $v\in G$ and $v'\in C_f(v)$. If so, $A$ is called a \defstyle{tail branch-property} if $\identity{A}(v,f)=\identity{A}(v,f')$ for every finite modification $f'$ of $f$ such that $C_f(v)$ and $C_{f'}(v)$ share a common ancestral branch. We extend these definitions to level-set-properties as follows. We call $A$ a \defstyle{level-set-property} if $\identity{A}(v,f)=\identity{A}(v',f)$ for every $f\in G^G$, $v\in G$ and $v'\in L_f(v)$. If so, we call $A$ a \defstyle{tail level-set-property} if $\identity{A}(v,f)=\identity{A}(v,f')$ for every finite modification $f'$ of $f$ in which $\card{L_f(v)\Delta L_{f'}(v)}<\infty$.

As mentioned in \Cref{main:deterministic},
we restrict attention to variants of tail properties as follows: Let $\sigfield{pm}{T}{G}, \sigfield{branch}{T}{G}$ and $\sigfield{level}{T}{G}$ denote the sigma-fields of tail point-map-properties, tail branch-properties and tail level-set-properties respectively, when the underlying graph $G$ is fixed. By focusing on the ancestral line of $v$, one can also regard $\sigfield{path}{I}{G}$ and $\sigfield{path}{T}{G}$, defined in \Cref{sec:Markov}, as sigma-fields on $G\times G^G$ (by identifying them with their inverse images under the projection $(v,f)\mapsto \anc_f(v)$).

Let $\bs F$ be a CMT on $G$. Denote the distribution of $\bs F$ by $\mathbb P_G$ (on $G^G$). For considering all possible roots in $G$ simultaneously, we consider the following $\sigma$-finite measure on $G\times G^G$: 
\[
	Q_G:=\left(\sum_{v\in G}\delta_v\right)\otimes \mathbb P_G.
\]


\begin{proposition}[Steps~\ref{step2comp} and~\ref{step2foil} on a Deterministic Graph]
	\label{prop:tail-fixed}
	Let $\bs F$ be a CMT on a deterministic graph or discrete space $G$ and define $Q_G$ as above.
	\begin{enumerate}[label=(\roman*)]
		\item \label{prop:tail-fixed:all} $\sigfield{pm}{T}{G}$ is trivial under $\mathbb P_G$.
		\item \label{prop:tail-fixed:comp} One has $\sigfield{branch}{T}{G}=\sigfield{path}{I}{G}$ mod $Q_G$.
		\item \label{prop:tail-fixed:foil} If $\bs F$ is locally finite and cycle-free a.s. and all level-sets of $\bs F$ are infinite a.s., then $\sigfield{level}{T}{G}=\sigfield{path}{T}{G}$ mod $Q_G$.
	\end{enumerate}
\end{proposition}


\begin{proof}
	The proof is similar to that of \Cref{thm:tailDrainage} and we only highlight the differences. Given $o\in G$ and $A\in\sigfield{branch}{T}{G}$ (resp. $A\in\sigfield{level}{T}{G}$), the same proof as that of \Cref{thm:tailDrainage} constructs another event $A'\in \sigfield{path}{I}{G}$ (resp. $A'\in\sigfield{path}{T}{G}$) that is equivalent to $A$ under $\mathbb P_{(G,o)}$, where $\mathbb P_{(G,o)}:=\delta_o\otimes\mathbb P_G$. It can be seen that the construction of $A'$ does not depend on the choice of $o\in G$ (one needs to use $Q_G$ instead of $\mathbb P$ in \Cref{lem:completion-I(R)} and extend this lemma accordingly). This implies that $Q_G(A\Delta A')=0$, which implies the claim.
\end{proof}

By combining with \Cref{thm:pathI=T-intro}, one obtains:
\begin{corollary}
	In \Cref{model2}, almost every sample $G$ of $\bs G$ satisfies
	\[\sigfield{level}{T}{G} = \sigfield{branch}{T}{G}=\sigfield{path}{T}{G}=\sigfield{path}{I}{G} \text{ mod } Q_G.\]
\end{corollary}

To deduce \Cref{thm:fixed-comp,thm:fixed-foil}, we first need to define the Poisson boundary and the Liouville property. 
These definitions are slight generalizations of the definitions in Section~14.3 of~\cite{bookLyPe16} by removing the assumption of irreducibility. See also \cite{Ka92boundary} for a more general treatment.

The \defstyle{Liouville property} holds if  there is no nonconstant bounded $K$-harmonic function on $G$; equivalently, if $\sigfield{path}{I}{G}$ is trivial mod $Q_G$. In this case, the Poisson boundary can be defined to have only one point. More generally:

\begin{definition}[Poisson Boundary]
	Let $K$ be a transient Markov kernel on a countable set $G$. Let $\mu$ be an arbitrary probability measure on $G$ with full support. Fix a finite or countable dense sequence of events $A_1,A_2,\ldots$  in $\sigfield{path}{I}{G}$. Here, \textit{dense} is meant with respect to the pseudo-metric $d(A,B):=\probPalm{\mu}{A\Delta B}$, where $\mathbb P_{\mu}:=\sum_{v\in G}\mu(v)\mathbb P_v$ and $\mathbb P_v$ is the distribution of the Markov chain with kernel $K$ started from $v$ (the topology of this metric does not depend on the choice of $\mu$). The \defstyle{Poisson boundary} is the quotient of $G^{\mathbb N\cup\{0\}}$ under the equivalence relation $(x_n)_n\sim (y_n)_n\iff \forall i: \identity{A_i}((x_n)_n)=\identity{A_i}((y_n)_n)$, and equipped with the quotient sigma-field. 
\end{definition}
The choice of the dense set $A_1,A_2,\ldots$ results in various realizations of the Poisson boundary, but all such realizations are \textit{equivalent} in some sense, see~\cite{bookLyPe16}. Note that, for every initial point $v\in G$, the probability measure $\mathbb P_v$ induces a probability measure on the Poisson boundary. These probability measures might not be equivalent in the non-irreducible case. However, $\mathbb P_{\mu}$ induces a  probability measure on the Poisson boundary whose equivalence class does not depend on the choice of $\mu$.
Also, after adding and removing some null sets of boundary points, one may regard the Poisson boundary as the topological boundary of $G$ in some compactification of $G$; see e.g., Remark~14.32 of~\cite{bookLyPe16}. Then, the Markov chain converges to a point of the boundary a.s. in the topological sense.

As an example, note that the 3-regular tree has a nontrivial Poisson boundary. So, the corresponding coalescing random walks model (\Cref{ex:crw}) has distinguishable components (see the next example). Here, it is important that configurations are not considered up to graph-isomorphisms. Otherwise, the components would be indistinguishable as discussed in \Cref{subsec:crw}.

\begin{example}[A Non-Liouville Example]
	\label{ex:nonLiouville}
	Let $G$ be deterministic and consider the coalescing simple random walk model on $G':=G\times \mathbb Z$ (\Cref{ex:crw}). It can be seen that the invariant sigma-field of the space-time walk is determined by the tail sigma-field of the walk in $G$. For instance, let $k\geq 3$ and $G:=G_1\cup G_2$, where $G_1:=\mathbb Z^{\{1,\ldots,k\}}$ and $G_2:=\mathbb Z^{\{4,\ldots,k+3\}}$, regarded as subgraphs of the lattice $\mathbb Z^{k+3}$. It can be seen that the simple random walk lies eventually in one of $G_1$ and $G_2$ a.s. It is also known that the tail sigma-field of the simple random walk in $G_i$ is determined by the parity of the starting point. Hence, a version of the Poisson boundary of the space-time walk in $G'$ has exactly 4 points (determined by parity and by whether the walk eventually lies in $G_1$ or $G_2$). These 4 points distinguish the components of the coalescing simple random walks, which has infinitely many components..
\end{example}

\begin{proof}[Proof of \Cref{thm:fixed-comp} (For CMTs)]
	According to the construction of the Poisson boundary, mentioned above, every event in $\sigfield{path}{I}{G}$ is equivalent to (the inverse image of) some measurable subset of the Poisson boundary (this is analogous to \Cref{step3comp}). 
	So, the claim is directly implied by \Cref{prop:tail-fixed}.
	
	For $k$-tuples of components, the proof is similar and we avoid repeating all arguments (see Theorem~1.2 of~\cite{Hu20indistinguishability}). We only remark that the definition of component-properties and tail branch-properties can be extended to $k$-tuples straightforwardly. The sigma-fields $\sigfield{path}{I}{G}$ and $\sigfield{path}{T}{G}$ can also be extended to $k$-tuples of path. Finally, \Cref{prop:tail-fixed} can be extended with the same proof, which implies the claim.
\end{proof}

\begin{proof}[Proof of \Cref{thm:fixed-foil}]
	Assume that $\sigfield{path}{T}{G}=\sigfield{path}{I}{G}$ (by \Cref{thm:pathI=T-intro}, almost every sample $(G,o)$ of \Cref{model2} satisfies this property). If $(G,o)$ is such a sample and $A\in\sigfield{level}{T}{G}$ is arbitrary, then \Cref{prop:tail-fixed} implies that $A$ is equivalent to some $A'\in\sigfield{path}{I}{G}$ mod $Q_G$. This implies~\ref{thm:fixed-foil-1}. The claim of~\ref{thm:fixed-foil-2} is also implied by the fact that all level-sets of a given component of $\bs F$ correspond to a common boundary point in the Poisson boundary.
\end{proof}

\section{Qualitative Properties of the Connected Components of CMTs}
\label{sec:one-ended}

In this section, we prove the results announced in \Cref{main:qualitative} regarding connectedness and one-endedness. The proofs rely on \Cref{sec:Markov}.

\begin{proof}[Proof of \Cref{prop:finiteComps}]
	We may assume that $[\bs G, \bs o]$ is ergodic and the cycle-free condition fails a.s. Given $l\in \mathbb N$ and $\epsilon>0$, let $\bs S:=\{y\in \bs G: K^l(y,y)\geq \epsilon\}$. One might choose $l$ and $\epsilon$ such that $\myprob{\bs S\neq\emptyset}>0$, and hence, $\bs S\neq \emptyset$ a.s. If $(\bs X_n)_n$ is a Markov chain on $\bs G$ starting from $\bs o$ and with kernel $K$, then \Cref{thm:pathErgodic} implies that $\bs X_n\in\bs S$ for infinitely many indices $n$ a.s. The idea is that, in each visit to $\bs S$, a cycle is created with probability at least $\epsilon$. This is made precise as follows.
	
	Let $\tau_0$ be the first $n$ such  that $\bs X_n\in \bs S$. Inductively, let $\tau_{i+1}$ be the first $n\geq \tau_i + l$ such that $\bs X_{n}\in\bs S$. One has $\forall i: \tau_i<\infty$ a.s. By the definition of $\bs S$, one obtains $\probCond{\bs X_{\tau_i+l}=\bs X_{\tau_i}}{[\bs G, \bs o; (\bs X_n)_{n\leq \tau_i}]}\geq \epsilon$. This implies that $\exists i: \bs X_{\tau_i+l}=\bs X_{\tau_i}$ a.s. Since $\anc(\bs o)$ has the same distribution as $(\bs X_n)_n$ until the first self-intersection, one obtains that $\anc(\bs F)$ contains a cycle a.s. So, the claim is implied by \Cref{thm:classification}.
\end{proof}

\begin{definition}
	Let $G$ be a sample that satisfies the conditions of \Cref{model1}.
	Define the \defstyle{reachability relation} $\leq$ on $G$ as follows: For $u,v\in G$, we say $u\leq v$ when $g_K(G,u,v)=\sum_{n\geq 0} K^n(u,v)=\probPalm{G}{v\in \anc(u)}>0$, where the equality is implied by the cycle-free condition. This condition also implies that $\leq$ is a partial order. Define the \defstyle{upper cones} and the \defstyle{lower cones} as $\{\cdot\geq y\}=\{z\in G: z\geq y\}$ and $\{\cdot\leq y\}=\{z\in G: z\leq y\}$. Define $\{\cdot>y\}$ and $\{\cdot<y\}$ similarly.
\end{definition}

We also need \textit{update-tolerance}, defined as follows. This is similar to insertion-tolerance (see e.g., \cite{LySc99}) and the update-tolerance of the uniform spanning forest defined in~\cite{HuNa17indistinguishability}, but we define update-tolerance using finite modifications and prove it by a coupling method.

\begin{definition}
	\label{def:admissible-drainage}
	Given a graph $G$, a function $f:G\to G$ is called ($K$-) \defstyle{admissible} if $K(G,v,f(v))>0$ for all $v\in G$. Also, an \defstyle{admissible finite modification} of $f$ means another admissible function $f':G\to G$ such that $\card{\{f\neq f'\}}<\infty$.
\end{definition}

{Obviously, almost every sample of $[\bs G, \bs o; \bs F]$ is $K$-admissible.}

\begin{lemma}[Update-Tolerance]
	\label{lem:update-tolerance}
	In \Cref{model1}, let $B\subseteq\mathcal G'_*$ be any event such that $\myprob{B}=0$. Then, almost all samples $(G,o;f)$ of $[\bs G, \bs o; \bs F]$ satisfy the following property: For all admissible finite modifications $f'$ of $f$, one has $[G, o; f']\not\in B$.
\end{lemma}
\begin{proof}
	Choose $n\in\mathbb N$ randomly and independently with the geometric distribution. Define $\bs F'$ by resampling $\bs F$ in $\oball{n}{\bs o}$; i.e., choose $\restrict{\bs F'}{\oball{n}{\bs o}}$ as a copy of $\restrict{\bs F}{\oball{n}{\bs o}}$ independently from $\bs F$ (conditionally on $[\bs G, \bs o]$) and, for all $x\not\in \oball{n}{\bs o}$, let $\bs F'(x):=\bs F(x)$. Then, $[\bs G, \bs o; \bs F']$ has the same distribution as $[\bs G, \bs o; \bs F]$. So, $\myprob{[\bs G, \bs o; \bs F']\in B}=0$. Thus, for almost every sample $(G,o;f)$ of $[\bs G, \bs o; \bs F]$, one has $\probPalm{(G,o;f)}{[G,o;\bs F']\in B}=0$, where $\mathbb P_{(G,o;f)}$ is the distribution of $\bs F'$ given $(G,o;f)$. Note that $\mathbb P_{(G,o;f)}$ is supported on a countable set, which is precisely the set of admissible finite modifications of $f$ (assuming that $f$ is also admissible). This implies the claim.
\end{proof}

\begin{proof}[Proof of \Cref{thm:one-ended-or-connected}]
	We may assume that $[\bs G, \bs o]$ is ergodic. Hence, so is $[\bs G, \bs o; \bs F]$ (\Cref{thm:tailDrainage}). 
	Assume the claim does not hold; i.e., with positive probability, there are at least two components, with at least one of them two-ended. By ergodicity, this happens almost surely. 
	By update-tolerance (\Cref{lem:update-tolerance}), almost surely, all admissible finite modifications of $\bs F$ satisfy this property; i.e., are cycle-free, have at least two components and have at least one two-ended component. Also, almost surely, all admissible finite modifications of $\bs F$ satisfy the claims of \Cref{thm:classification} and \Cref{model1}.
	
	Fix any sample $G$ of $\bs G$ on which $\bs F$ and all admissible finite modifications of $\bs F$ satisfy the properties in the last paragraph a.s. (almost every sample is of this kind).
	If $x$ is in a bi-infinite path of $\bs F$, we claim that $\{\cdot\geq x\}\subseteq C(x)$ a.s. To prove this claim, assume that it is not true. If $\{\cdot \geq x\}$ intersects another two-ended component, then there exists an admissible finite modification of $\bs F$ that results in a 3-ended component, which is a contradiction. If $\{\cdot\geq x\}$ intersects a one-ended component, then there exists a finite modification of $\bs F$ that results in a two-ended component with some infinite level-sets. This contradicts the claims of \Cref{thm:classification}. Hence, it is proved that if $x$ is in a bi-infinite path, then $\{\cdot\geq x\}\subseteq C(x)$. 
	Let $\bs H=\bs H_{(G; \bs F)}$ be the set of points $u\in G$ such that $\{\cdot\geq u\}$ is entirely a subset of a two-ended component of $\bs F$.
	Therefore, $\bs H$ includes all bi-infinite paths a.s.  
	Note that $\bs H$ is increasing under the reachability relation; i.e., if $u\in \bs H$ and $u\leq v$, then $v\in \bs H$.
	
	Returning to $[\bs G, \bs o]$, let $\bs X$ be a random point of $\bs G$ with distribution $K(\bs o, \cdot)$, chosen independently from $\bs F$. By the stationarity property of \Cref{lem:stationary}, one obtains that $\probhat{\bs X\in \bs H}=\probhat{\bs o\in \bs H}$, where $\hat{\mathbb P}$ is the probability measure obtained by biasing by $b(\bs o)$. Note that, by the monotonicity of $\bs H$,  
	the event $\bs o\in \bs H$ is included in the event $\bs X\in\bs H$ (after removing the null event $K(\bs o, \bs X)=0$). So, the assumption $b>0$ and the last equality imply that $\myprob{\bs o\not\in \bs H, \bs X\in \bs H}=0$. Therefore, \Cref{lem:happensAtRoot} gives that, almost surely, $\forall u\not\in \bs H, \forall v\in \bs H: K(u,v)=0$.
	Now, we see that any two-ended component is not connected to any other component in the graph with edge set $\{(x,y): K(x,y)+K(y,x)>0 \}$. This contradicts weak irreducibility, and the claim is proved.
\end{proof}


\begin{lemma}[Infinite Path-Intersection]
	\label{lem:infiniteCollision}
	In \Cref{model1}, if the path-intersection property (\Cref{def:colision}) holds a.s., then the infinite path-intersection property holds a.s. as well.
\end{lemma}
\begin{proof}
	Let $(G,o)$ be a sample of $[\bs G, \bs o]$. By \Cref{lem:happensAtRoot}, we may assume that two Markov chains on $G$ started from any point $o'\in \bs G$ intersect a.s.
	Let $(\bs X_n)_n$ and $(\bs Y_n)_n$ be two independent Markov chains on $G$ starting from $o$ and with kernel $K$. 
	Let $i$‌ be the first integer such that $z:=\bs X_i$ is in the image of $(\bs Y_n)_n$. For the index $j$ for which $z=\bs Y_j$, the cycle-free condition implies that $j$‌ is also the first index such that $\bs Y_j$ is in the image of $(\bs X_n)_n$.
	The strong Markov property implies that, conditionally on $(\bs z,i,j)$, the sequences $(\bs X_{n+i})_{n\geq 0}$ and $(\bs Y_{n+j})_{n\geq 0}$ are two independent Markov chains on $G$ started from $\bs z$. So, they intersect a.s. By repeating this process, we find infinitely many intersection points a.s.
\end{proof}

%

\begin{lemma}
	\label{lem:collisionIneq}
	Let $K$ be any cycle-free Markov kernel on a measurable space $G$. For $x,y\in G$, let $p(x,y)$ be the probability that the paths of two independent Markov chains on $G$ with kernel $K$ and starting from $x$ and $y$ do not intersect. Then, $q(x,y)\geq q(x,x)/3$.
\end{lemma}
\begin{proof}
	Let $\epsilon:=q(x,x)$ and $\delta:=q(x,y)$. Let $(\bs X_n)_n$, $(\bs X'_n)_n$ and $(\bs Y_n)_n$ be independent Markov chains starting from $x$, $x$ and $y$ respectively. 
	Let $A$ be the event that the path $(\bs X_n)_n$ does not intersect with $(\bs X'_n)_n$.
	Let $0<i\leq\infty$ (resp $i'$) be the smallest number such that $\bs Y_i\in(\bs X_n)_{{n>0}}$ (resp. $\bs Y_{i'}\in (\bs X'_n)_{{n>0}}$). Let $0<j\leq \infty$ be the smallest number such that $\bs X_j=\bs Y_i$. Then, conditionally on $(\bs X_n)_{n\leq j}$ and $(\bs Y_n)_{n\leq i}$, the futures $(\bs X_n)_{n\geq j}$ and $(\bs Y_n)_{n\geq i}$ of the paths are independent Markov chains with kernel $K$ (the cycle-free assumption is important for the last conclusion). So, by letting $\bs Y'_n:=\bs Y_n$ for $n\leq i$ and $\bs Y'_{i+n}:=\bs X_{j+n}$ for $n\geq 0$, one obtains that $(\bs Y'_n)_n$ has the same distribution as $(\bs Y_n)_n$, and is independent of $(\bs X'_n)_n$.
	One has $\myprob{i<\infty, A}\geq \myprob{A}-\myprob{i=\infty}=\epsilon-\delta$. By symmetry of $i$ and $i'$, one gets $\myprob{i<\infty, i\leq i', A} \geq (\epsilon-\delta)/2$. The last event is contained in the event that $(\bs Y'_n)_n$ does not intersect with $(\bs X'_n)_n$. Hence, $\delta\geq (\epsilon-\delta)/2$ and the claim is implied.
\end{proof}

\begin{lemma}[Finite Intersection]
	\label{lem:finiteCollision}
	In \Cref{model1}, assume that $[\bs G, \bs o]$ is ergodic, and that not almost all samples of $[\bs G, \bs o]$ satisfy the path-intersection property. Then, {in almost every sample of $\bs G$}, two Markov chains on $\bs G$ started from any two points intersect finitely often a.s.
\end{lemma}
\begin{proof}
	Given $\epsilon>0$, let $\bs S$ be the set of $x\in \bs G$ such that two chains started from $x$ path-intersect with probability at most $1-\epsilon$. By the assumption, we may fix $\epsilon$ such that $\bs S\neq\emptyset$ a.s. So, the triviality of $\sigfield{path}{I}{}$ (\Cref{thm:pathErgodic}) implies that the Markov chain on $\bs G$ started from $\bs o$ hits $\bs S$ infinitely many times. 
	
	Let $(G,o)$ be a sample of \Cref{model1} such that the Markov chain from any $o'\in G$ intersects $\bs S$ infinitely often (by \Cref{lem:happensAtRoot}, almost every sample is of this type). For $x,y\in G$, we will show that the chains $(\bs X_n)_n$ and $(\bs Y_n)_n$ on $G$, started from $x$ and $y$ respectively, path-intersect finitely often a.s. 
	Let $t_1,t_2,\ldots \in \mathbb N\cup\{\infty\}$ and $s_1,s_2,\ldots\in\mathbb N\cup\{\infty\}$ denote the intersection times, in increasing order, such that $\bs X_{t_i}=\bs Y_{s_i}$ for all $i$ for which $t_i,s_i<\infty$.
	Let $\tau_0$ be the first time such that $\bs X_{\tau_0}\in \bs S$. Inductively, let $\tau_{m}$ be the first time in $(\tau_{m-1},\infty]$ such that $\bs X_{\tau_{m}}\in\bs S$ and $(\tau_{m-1},\tau_{m}]\cap (t_n)_n\neq\emptyset$ (we let $\tau_{m}=\infty$ if no such time exists). Let $i_m\in\mathbb N\cup\{\infty\}$ denote the first number such that $t_{i_m}\in (\tau_{m-1},\tau_{m}]$ (we let $i_0:=0$). Conditionally on the event $\tau_m<\infty$ and on $(\bs X_n)_{n\leq \tau_m}$ and $(\bs Y_n)_{n\leq s_{i_m}}$, the futures of the paths are independent Markov chains with kernel $K$ (the proof of this fact is left to the reader). So, \Cref{lem:collisionIneq} for the points $\bs X_{\tau_m}$ and $\bs X_{t_{i_m}}=\bs Y_{s_{i_m}}$ implies that $\tau_{m+1}=\infty$ with conditional probability at least $\epsilon/3$. This implies that $\exists m: \tau_m=\infty$ a.s., and the claim is proved.
	%
\end{proof}

We are now ready to prove \Cref{thm:infComps}.
\begin{proof}[Proof of \Cref{thm:infComps}]
	\ref{thm:infComps:1}. 
	We should prove that two independent Markov chains started from any two points in $\bs G$ intersect a.s.
	By infinite path-intersection (\Cref{lem:infiniteCollision}) and \Cref{lem:happensAtRoot}, we may consider a sample $(G,o)$ of \Cref{model1} in which any two Markov chains with a common starting point $o'\in G$ intersect infinitely often a.s. 
	For $x,y\in G$, say $x\sim y$ if the chains started at $x$ and $y$ intersect infinitely often a.s. We claim that $\sim$ is an equivalence relation. Assume $x\sim y$ and $y\sim z$. Let $(\bs X_n)_n$ and $(\bs Y_n)_n$ be Markov chains started from $x$ and $y$ coupled after the first path-intersection. Independently, let $(\bs Y'_n)_n$ and $(\bs Z_n)_n$ be Markov chains started from $y$ and $z$ coupled after the first path-intersection. Since $(\bs Y_n)_n$ and $(\bs Y'_n)_n$ are independent chains from $y$, they intersect infinitely a.s. Hence, by the coupling assumption, so do $(\bs X_n)_n$ and $(\bs Z_n)_n$. Since $(\bs X_n)_n$ and $(\bs Z_n)_n$ are independent chains started from $x$ and $z$, we find that $x\sim z$. So, we have proved that $\sim$ is an equivalence relation.
	
	Note that if $\probPalm{G}{x\in \anc(a)}>0$ and $\probPalm{G}{y\in \anc(a)}>0$ for some $a\in G$, then $x\sim y$ (use infinite path-intersection for two chains started from $a$). So, the weak irreducibility condition of \Cref{model1} implies that all points of $G$ are equivalent. This implies the claim.
	
	\ref{thm:infComps:2}. Let $(G,o)$ be a sample of \Cref{model1} that satisfies the finite path-intersection property of \Cref{lem:finiteCollision}. We claim that, on this sample, $\bs F$ has infinitely many components a.s. It is enough to show that, for all $k\in\mathbb N$, $\bs F$ has at least $k$ components a.s.
	For this, we construct $\bs F$ by the following algorithm:
	\begin{itemize}
		\item Generate $k$ independent Markov chains $(\bs X^i_n)_n$, $i=1,\ldots, k$ from $o$.
		\item For $i=1,\ldots,k$:
		\begin{itemize}
			\item For all $n\geq 0$, set $\bs F(\bs X^i_n):=\bs X^i_{n+1}$. If $\bs F(\bs X^i_n)$ is already defined, the new definition of $\bs F(\bs X^i_n)$ replaces the previous one.
		\end{itemize}
		\item For all $x\not\in \bigcup_i (\bs X^i_n)_n$, choose $\bs F(x)$ with distribution $K(x,\cdot)$, independently.
	\end{itemize}
	It is easy to show that $\bs F$ has the same distribution as the CMT with kernel $K$.
	By finite path-intersection (\Cref{lem:finiteCollision}), there exists $N<\infty$ such that the paths $(\bs X^i_n)_{n\geq N}$ for $i=1,\ldots,k$ do not intersect. This implies that $\bs X^1_N,\ldots,\bs X^k_N$ are in different components of $\bs F$ and the claim is proved.
\end{proof}

\begin{proof}[Proof of \Cref{prop:decay}]
	The proof is similar to that of Theorem~4.1 of~\cite{LySc99}.
	We may assume that $[\bs G, \bs o]$ is ergodic and $\bs F$ has infinitely many components a.s. (by \Cref{thm:infComps}).
	Let $h$ be a factor Markov kernel such that, almost surely, $h[\bs G,\cdot,\cdot]$ is a strictly positive symmetric function on $\bs G\times \bs G$. See Lemma~3.10 of~\cite{Kh23unimodular} for the existence of $h$ (if $[\bs G, \bs o]$ is a graph with finite expected degree, one can use the simple random walk instead).
	For every component $C$ of $\bs F$, let $\alpha(C)$ be the \textit{frequency} of $C$ given by \Cref{lem:freq} applied to the kernel $h$ and the constant function $b'\equiv 1$ (which satisfies the balance condition by the symmetry of $h$). By indistinguishability (Part~\ref{thm:indistinguishability-comp} of \Cref{thm:indistinguishability}), all components have the same frequency. Since there are infinitely many components, one obtains that all components have zero frequency a.s.
	
	Assume the claim is false. By ergodicity, for some $\epsilon>0$, we have $\forall x\in\bs G: \probPalm{\bs G}{x\in C(\bs o)}\geq \epsilon$ a.s. Now, Let $(\bs X_n)_n$ be the Markov chain on $\bs G$ with kernel $h$ from $\bs o$ independent of $\bs F$. Then, $\probPalm{\bs G}{\bs X_n\in C(\bs o)}\geq \epsilon$.
	So, the definition of $\alpha$ and dominated convergence implies that $\omidPalm{\bs G}{\alpha[\bs G, \bs o; C(\bs o)]}\geq \epsilon$, which is a contradiction.
\end{proof}

The following lemma is an extension of Lemma~4.2 of~\cite{LySc99} and is used in the previous proof.

\begin{lemma}[Cluster Frequency]
	\label{lem:freq}
	Let $M$ be the space of all $[G,o;S]$, where $(G,o)$ is a rooted graph and $S\subseteq G$.
	\begin{enumerate}[label=(\roman*)]
		\item There exists a measurable function $\alpha:M\to [0,1]$ that satisfies the following property: For every unimodular graph $[\bs G, \bs o]$ in which $\omid{\mathrm{deg}(\bs o)}<\infty$, and every equivariant partition $\Pi$ of $\bs G$, if $(\bs X_n)_n$ is the simple random walk on $\bs G$ started from $\bs o$, then the \emph{frequency of visits} of every element $E$ of $\Pi$ is equal to $\alpha[\bs G, \bs o; E]$ a.s. More precisely, almost surely,
		\[
		\forall E\in \Pi: \alpha[\bs G, \bs o; E] = \lim_n \frac 1 n \sum_{i=1}^n \identity{E}(\bs X_i).
		\]
		In addition, $\alpha$ can be chosen such that it does not depend on the root.
		\item If the simple random walk is replaced by the Markov chain with a given factor Markov kernel $h=h(G,x,y)$, then the same claim holds for every unimodular graph or discrete space $[\bs G, \bs o]$ such that, for some function $b'$ for which $\omid{b'(\bs o)}<\infty$, almost every sample $(G,o)$ of $[\bs G, \bs o]$ satisfies the following properties: (1) the balance condition of \Cref{model1} for the kernel $h$ and the function $b'$, (2) the paths of two Markov chains on $G$ started from any two points of $G$ intersect with positive probability.
	\end{enumerate}	
\end{lemma}

The proof of this lemma is identical to that of Lemma~4.2 of~\cite{LySc99} and is omitted for brevity. 
Note that the claim does not follow from Birkhoff's ergodic theorem since the sequence $(\identity{E}(\bs X_i))_i$ is not stationary (since the element $E$ of $\Pi$ is arbitrary and is not chosen equivariantly).

\begin{proof}[Proof of \Cref{thm:one-ended}]
	In every sample $(G,o)$ of $[\bs G, \bs o]$, construct an instance $(\bs X_n)_{n\geq 0}$ of the Markov chain on $G$ with kernel $K$, starting from $o$, and independent from $\bs F$ (note that $(\bs X_n)_n$ and $\anc(\bs o)$ have the same distribution and are independent given $(G, o)$).
	Let $\widehat{\mathbb P}$ be the distribution obtained by biasing by $b(\bs o)$.
	By keeping $\bs F$ as a decoration and using \Cref{lem:stationary} for $(\bs X_n)_n$, under $\widehat{\mathbb P}$, for all $i$, $[\bs G, \bs X_i; (\bs X_{n+i})_n,\bs F]$ has the same distribution as $[\bs G, \bs o; (\bs X_n)_n, \bs F]$.
	
	By ergodic decomposition\del{ (\Cref{subsec:ergodic})}, we can assume that $[\bs G, \bs o]$ is ergodic. So, $[\bs G, \bs o; \bs F]$ is also ergodic. So, by \Cref{thm:one-ended-or-connected}, we have the following two cases.
	
	\vspace{2mm}
	\textbf{Case 1.} Assume that $\bs F$ is connected and two-ended a.s. 
	Let $\bs H$ be the unique bi-infinite path in the graph of $\bs F$. Then, $\bs H$ is an equivariant subset of $\bs G$. 
	So, by Birkhoff's ergodic theorem, the \textit{frequency of visits of $(\bs X_n)_n$ to $\bs H$} is $\widehat{\mathbb P}$-a.s. well defined, where the latter is defined by
	\[
	\freq(U):=\lim_n \frac 1 n \sum_{i=1}^n \identity{U}(\bs X_i) \quad \text{for $U\subseteq G$, if the limit exists.}
	\] 
	In addition, by ergodicity (\Cref{thm:pathErgodic}), $\freq(\bs H)$ is constant and is equal to its expectation; i.e., 
	\begin{equation*}
		\freq(\bs H)=\widehat{\mathbb P}[\bs o\in\bs H] =:\lambda \text{ a.s.}
	\end{equation*}
	Note that if $\bs X_j\in\bs H$, then $(\bs X_i)_{i\geq j}$ cannot intersect the half-path $\bs H\cap \{\cdot <\bs X_j\}$. Therefore, at most finitely many of the $\bs X_i$'s fall in $\bs H\setminus\anc(\bs o)$. So, the last equation implies that
	$
	\lim_n \frac 1 n \sum_{i=1}^n \identity{\anc(\bs o)}(\bs X_i) = \lambda \text{ a.s.}
	$
	The bounded convergence theorem then gives that
	$
	\lim_n \frac 1 n \sum_{i=1}^n \myprob{\bs X_i\in \anc(\bs o)} = \lambda.
	$
	In other words,
	\begin{equation}
		\label{eq:thm:one-ended}
		\lim_n \frac 1 n \sum_{i=1}^n \omid{g_K(\bs o, \bs X_i)} = \lambda.
	\end{equation}
	\Cref{lem:happensAtRoot} gives that $\myprob{\bs o\in \bs H}>0$. So, the facts $b>0$ and $\lambda = \widehat{\mathbb P}[\bs o\in\bs H] = \omid{b(\bs o)\identity{\{\bs o\in\bs H\}}}/\omid{b(\bs o)}$ imply that $\lambda>0$. Therefore, \eqref{eq:thm:one-ended} implies that $g_K(\bs o, \bs X_n)$ cannot converge to zero in probability, and the claim is proved.
	
	\vspace{2mm}
	\textbf{Case 2.} Assume that all components of $\bs F$ are one-ended a.s. As in the proof of \Cref{thm:pathErgodic}, one may extend $(\bs X_n)_{n\geq 0}$ to negative times and the sequence $\left([\bs G, \bs X_i; (\bs X_{n+i})_{n\in\mathbb Z}, \bs F]\right)_{i\in\mathbb Z}$ is stationary after biasing by $b (\bs o)$. As a result, 
	\[\omidhat{g_K(\bs o, \bs X_n)} = \omidhat{g_K(\bs X_{-n},\bs o)}.\] 
	By one-endedness, $D(\bs o)$ is finite a.s. This implies that $\lim_{n\to\infty} \identity{D(\bs o)}(\bs X_{-n}) = 0$ a.s.; i.e., $\lim_{n\to\infty} \identity{\{\bs o\in \anc(\bs X_{-n})\}} = 0$ a.s. By taking expectation and using bounded convergence, one obtains that $\lim_{n\to\infty} \omidhat{g_K(\bs X_{-n},\bs o)} = 0$. So, the last display implies that $\lim_{n\to\infty} \omidhat{g_K(\bs o, \bs X_n)} = 0$. Thus, since $g_K(\bs o, \bs X_n)$ is a random variable in $[0,1]$, it converges to zero in probability. So, the claim is proved.
\end{proof}

\section{Indistinguishability in Wired Uniform Spanning Forests}
\label{sec:wusf}

In \Cref{subsec:wusf-indistinguishability} below, we use the method of the present paper to provide an alternative proof for the indistinguishability of the connected components of the wired uniform spanning forest on unimodular graphs (\Cref{thm:usf-indistinguishability} stated in \Cref{main:wusf}). This method also gives a simple proof of tail-triviality of $\wusf$ on deterministic graphs and its ergodicity on unimodular graphs (\Cref{thm:usf-ergodic} in \Cref{subsec:wusf-ergodic}). The proof uses a coupling given by a generalization of the \textit{stochastic covering property}, which is discussed in \Cref{subsec:usf-defs}. See also \Cref{prob:lerw}, and the discussion after it, regarding the indistinguishability of level-sets of $\wusf$.

\subsection{Definitions and Basic Properties}
\label{subsec:usf-defs}

We briefly recall some definitions and results regarding the wired uniform spanning forest. The reader is advised to see the main references like~\cite{BLPS-usf,bookLyPe16} for a detailed treatment of this object and for the \textit{free} uniform spanning forest.

The 
\defstyle{wired uniform spanning forest ($\wusf$)} is defined for every infinite locally-finite graph $G$. In short, 
$\wusf(G)$ is the weak limit of the uniform spanning tree of $G/\oball{n}{o}^c$, where the last notation means contracting all of the vertices of $\oball{n}{o}^c$ to a single point, and where uniform spanning tree means the random subgraph obtained by choosing one of the spanning trees with the uniform distribution. It is proved that the limits exist and do not depend on the choice of $o\in G$. It is proved that $\wusf(G)$ 
is almost surely a forest, but is not necessarily connected. Also, each component is infinite a.s. 

The wired uniform spanning forest satisfies
the \defstyle{spatial Markov property}: Conditioned on containing a finite set $E$ of edges and avoiding a finite set $E'$ (assuming that this event has positive probability), it is distributed as 
$E\cup \wusf((G\setminus E')/E)$. In addition, it is known that $\wusf(G)$ is tail-trivial. 

The wired uniform spanning forest can be constructed by Wilson's algorithm.\del{If\mar{\ali{It is better to cite, rather than spending so much space. This would also help the reader focus.}} $v_0,v_1,\ldots$ is an arbitrary ordering of the vertices and $P$ is an initial subforest of $G$, in each step, choose the next vertex $v_i$, perform a simple random walk from $v_i$ until hitting the previously constructed forest, and then, erase the loops. If the initial configuration is a single point, the result of Wilson's algorithm has the same distribution as $\wusf(G)$. If $G$ is \textit{transient}, we may also let $P$ be empty. In this case, the algorithm is called \textit{Wilson's algorithm rooted at infinity}. Recall that transience of $G$ means that the simple random walk never returns to its starting point with positive probability. Otherwise, $G$ is called \textit{recurrent}.} 
In the transient case, Wilson's algorithm \textit{rooted at infinity} creates an oriented version of $\wusf(G)$, which is denoted by $\owusf(G)$. In this case, we let $\anc(o)$ be the directed infinite path in $\owusf(G)$ starting from $o$, which is distributed as the loop-erased random walk starting from $o$.

\begin{lemma}
	\label{lem:usf-anc}
	Let $(G,o)$ be a deterministic transient locally-finite rooted graph and $\bs F:=\owusf(G)$. Then, 
	the result of Wilson's algorithm with initial configuration $P$ (where $P$ ranges over infinite simple paths in $G$ starting from $o$) is a version of the conditional distribution of $\bs F$ conditionally on $\anc_{\bs F}(o)=P$.
\end{lemma}
\begin{proof}
	If we start Wilson's algorithm rooted at infinity from $o$, then $\anc(o)$ is the result of the first step of the algorithm. So, the claim follows by conditioning Wilson's algorithm on its first step.
\end{proof}

Another property that we use is the \defstyle{stochastic covering property}: For an edge $e$, if $\bs F_1$ (resp. $\bs F_2$) is a version of $\wusf(G)$ conditioned on containing $e$ (resp. avoiding $e$), then there exists a coupling of $\bs F_1$ and $\bs F_2$ such that $\card{\bs F_1\Delta \bs F_2}\leq 2$. This is implied by the results of~\cite{PePe14} (see the proof of the next lemma). We prove a variant of this result in \Cref{lem:usf-covering} below.

\begin{lemma}
	\label{lem:usf-covering}
	Let $(G,o)$ be a rooted graph and assume $S$ is a finite subset of the edges. Assume also that, for each $e\in S$, nonempty subsets $I_1(e),I_2(e)\subseteq\{0,1\}$ are given. For $i=1,2$, let $\bs F_i$ be distributed as $\wusf(G)$ conditioned on $\identity{\bs F_i}(e)\in I_i(e), \forall e\in S$ (assuming that this happens with positive probability).
	\begin{enumerate}[label=(\roman*)]
		\item \label{lem:usf-covering-1} There exists a coupling of $\bs F_1$ and $\bs F_2$ such that $\card{\bs F_1\Delta \bs F_2}\leq 2\card{S}$ a.s. 
		\item \label{lem:usf-covering-2} If $G$ is transient and  $P$ is an infinite simple path starting from $o$ that satisfies $1\in I_1(e)\cap I_2(e),\forall e\in S\cap P$, then there exists a coupling of the oriented version of $\bs F_i$ conditioned on $\anc_{\bs F_i}(o)=P$ (for each $i=1,2$) with the same property.
	\end{enumerate}
\end{lemma}

\begin{proof}
		Let $A_i$ be the event $\forall e\in S: \identity{\bs F_i}(e)\in I_i(e)$.
		We call the function $I_1$ on $S$ \textit{valid} if $\myprob{A_1}>0$. This is equivalent to the condition that $\{e: I_1(e)=\{1\}\}$ has no cycles and $G\setminus \{e:I_1(e)=\{0\}\}$ has no finite connected component.
	
	\ref{lem:usf-covering-1}. 
	By the spatial Markov property, $\bs F_i\setminus S$ conditioned on $A_i$ is distributed as $\wusf(G_i)$, where $G_i$ is obtained by deleting the edges $e\in S$ in which $I_i(e)=\{0\}$ and contracting the edges $e\in S$ in which $I_i(e)=\{1\}$. 
	
	Let $S':=\{e\in S: I_1(e)\neq I_2(e)\}$.
	By induction on $\card{S'}$, we prove that there exists a coupling such that $\card{\bs F_1\Delta \bs F_2}\leq 2\card{S'}$ a.s. 
	The claim is trivial if $S'=\emptyset$. Also, if $S'=\{e_0\}$ is a single edge, the claim is implied by Proposition~ 2.2 and Example~5.3 of~\cite{PePe14}, combined with a subsequential weak limit. For being self-contained, we provide a direct proof of this case using monotone coupling.
	Let $\hat G_n:=G/\oball{n}{o}^c$ and let $\bs F_i^n$ be the uniform spanning tree of $\hat G_n$ conditioned on $\forall e\in S:\identity{\bs F_i^n}(e)\in I_i(e)$ (for $i=1,2$). By Rayleigh's monotonicity principle, one can find a monotone coupling of $\bs F_1^n$ and $\bs F_2^n$; in the sense that one of $\bs F_i^n\setminus\{e_0\}$ contains the other one (depending on $I_1(e_0)$ and $I_2(e_0)$). Since $\card{\bs F_1^n}=\card{\bs F_2^n}$, one obtains that $\card{\bs F_1^n\Delta\bs F_2^n}\leq 2$. By taking a subsequential limit of this coupling as $n\to \infty$, one obtains the claim (note that since the marginals of this coupling are known to converge, one can deduce tightness and the existence of a convergent subsequence). So, the base of the induction is shown.

	Now, suppose $\card{S'}>1$. Choose $e_0\in S'$ and, for $j=1,2$, let $I'_j(e_0)=I_{3-j}(e_0)$ $I'_j(e):=I_j(e)$ for $e\neq e_0$. We show at the end of the proof that $e_0$ can be chosen such that at least one of $I'_1$ and $I'_2$ is \textit{valid} (defined above), except in some trivial cases.  Assume this holds. Assume also that $I'_1$ is valid without loss of generality. Let $\bs F'_1$ be $\wusf(G)$ conditioned on $\forall e: \identity{\bs F'_1}(e)\in I'_1(e)$.	
	Note that $I'_1$ differs from $I_1$ in a single edge and differs from $I_2$ in $\card{S'}-1$ edges. So, by the induction hypothesis, there exists a coupling of $\bs F_1$ and $\bs F'_1$ such that $\card{\bs F_1\Delta \bs F'_1}\leq 2$, and a coupling of $\bs F'_1$ and $\bs F_2$ such that $\card{\bs F'_1\Delta\bs F_2}\leq2\card{S'}-2$. By combining these two couplings (by first sampling $\bs F'_1$ and then sampling $\bs F_1$ and $\bs F_2$ independently conditioned on $\bs F'_1$), one finds a coupling of $\bs F_1,\bs F'_1$ and $\bs F_2$ such that $\card{\bs F_1\Delta \bs F_2}\leq 2\card{S'}$. This proves the induction claim.
	
	It remains to show that such an $e_0$ can be chosen except in trivial cases. If there exists $e\in S'$ such that $I_1(e)=\{0,1\}$ or $I_2(e)=\{0,1\}$, then $e_0:=e$ works. Otherwise, one may write $S'=S'_1\cup S'_2$, where $S'_i:=\{e\in S': I_i(e)=\{1\}, I_{3-i}(e)=\{0\}\}$. In particular, each of $S'_1$ and $S'_2$ is a forest. If $i\in\{1,2\}$ and $e\in S'_i$ can be chosen such that $S'_{3-i}\cup\{e\}$ is a forest, then such $e_0:=e$ works. If that this is not the case, $S'_1$ and $S'_2$ are edge-disjoint forests with the same vertex sets and the same connected components. Thus, for $j=1,2$, $\bs F_j$ is distributed as $\wusf(G)$ conditioned on containing $S'_j$ (the condition of avoiding $S'_{3-j}$ is redundant). Similarly to the case $\card{S'}=1$, there exists a coupling of $\bs F_j$ and $\wusf(G)$ such that their symmetric difference has at most $2\card{S'_j}$ edges. By combining these couplings, one finds a coupling of $\bs F_1$ and $\bs F_2$ such that $\card{\bs F_1\Delta\bs F_2}\leq 2\card{S'_1}+2\card{S'_2}=2\card{S'}$. This proves the claim.
	%

	
	\ref{lem:usf-covering-2}.
	Let $\hat G_n:=\backslash B_n(o)^c$ and let $z$ be its unique vertex out of $B_n(o)$.
	Since $G$ is transient, Wilson's algorithm in $G$ rooted at infinity is the limit of Wilson's algorithm in $\hat G_n$ rooted at $z_n$. So, \Cref{lem:usf-anc} implies that, if $P_n$ is the path $P$ stopped at the first exit from $B_n(o)$, then $\owusf(G)$ conditioned on $\anc(o)=P$ is the limit of $\wusf(G_n)$ conditioned on containing $P_n$. Now, the claim is implied by taking a subsequential limit of the coupling in $\hat G_n$ constructed in Part~\ref{lem:usf-covering-1}.
%
\end{proof}

If $[\bs G, \bs o]$ is unimodular (resp. ergodic), then $[\bs G, \bs o;\wusf(\bs G)]$ is also unimodular (resp. ergodic) 
(a proof of the last ergodicity is given in \Cref{thm:usf-ergodic} below). Also, the number of ends of $\wusf(\bs G)$ has been recently fully characterized recently as follows (see~\cite{DiHu-EndsOfWUSF}): If $\bs G$ is transient (e.g., if it has infinitely many ends), then every component of $\wusf(\bs G)$ is one-ended a.s. If $\bs G$ is recurrent, then $\wusf(\bs G)$ is connected. In this case, if $\bs G$ is two-ended, then $\wusf(\bs G)$ is connected and two-ended a.s. If $\bs G$ is recurrent and one-ended, which is the most difficult case, then $\wusf(\bs G)$ is also one-ended a.s.

\subsection{Ergodicity and Tail Triviality of $\wusf$}
\label{subsec:wusf-ergodic}

Before proving indistinguishability, we use the coupling methods used in this paper to give an easy proof of the following well-known result.


\begin{theorem}[Ergodicity and Tail Triviality of $\wusf$ \cite{BLPS-usf}]
	\label{thm:usf-ergodic}
	\ 
	\begin{enumerate}[label=(\roman*)]
		\item \label{thm:usf-ergodic-1} For every locally-finite graph $G$, $\wusf(G)$ is tail trivial.
		\item \label{thm:usf-ergodic-3} One has $\sigfield{pm}{I}{}=\sigfield{pm}{T}{}=\sigfield{}{I}{}$ mod $\mathbb P$. As a result, if $[\bs G, \bs o]$ is ergodic, then so is $[\bs G, \bs o; \wusf(\bs G)]$.
	\end{enumerate}
\end{theorem}
It should be noted that, in the first part, tail events are considered in the classical sense of being invariant under finite modifications (the invariance in \Cref{def:tailDrainage} is not assumed).
\begin{proof}
	The proof of \ref{thm:usf-ergodic-1} is similar to~\ref{thm:usf-ergodic-3} and is skipped for brevity. It is proved in part~\ref{thm:I in T 2-all} of \Cref{thm:I in T 2} that $\sigfield{pm}{I}{}= \sigfield{pm}{T}{}$ mod $\mathbb P$. So, we only need to prove that $\sigfield{pm}{T}{} = \sigfield{}{I}{}$ mod $\mathbb P$ (this is analogous to \Cref{thm:tailDrainage} regarding CMTs).
	
	Let $A\in\sigfield{pm}{T}{}$ be a tail property of $[\bs G, \bs o; \wusf(\bs G)]$.  We claim that $A$ is independent from itself conditionally on $[\bs G, \bs o]$; i.e., $\probCond{A}{\bs G, \bs o}\in\{0,1\}$ a.s. For this, it is enough to prove that for every $R<\infty$, $A$ is independent from $\restrict{\wusf(\bs G)}{\oball{R}{o}}$ conditionally on $[\bs G, \bs o]$. Equivalently, for every event $B$ that depends only on $\bs G, \bs o$ and $\restrict{\wusf(\bs G)}{\oball{R}{o}}$, we should prove
	\begin{equation*}
		\probCond{A\cap B}{\bs G, \bs o} = \probCond{A}{\bs G, \bs o}\cdot \probCond{B}{\bs G, \bs o}, \text{ a.s.}
	\end{equation*}
	Let $(G,o)$ be a sample of $[\bs G, \bs o]$. We construct a coupling of two copies $(\bs F_1,\bs F_2)$ of $\wusf(G)$ as follows. First, construct $\bs F_1$ and $\bs F_2$ independently. Then, keep $\restrict{\bs F_2}{\oball{R}{o}}$ and erase the rest of $\bs F_2$. By \Cref{lem:usf-covering}, one can resample the rest of $\bs F_2$ coupled with $\bs F_1$ such that $\bs F_1\Delta\bs F_2$ is finite a.s. (see \Cref{lem:usf-covering}). We denote the distribution of this coupling by the same notation $\mathbb P_{(G,o)}$.
	
	For $i=1,2$, let $A_i$ (resp. $B_i$) be the event that $[G,o; \bs F_i]\in A$ (resp. $\in B$). By construction, $A_1$ is independent from $B_2$ in the above coupling. Since $\bs F_1\Delta\bs F_2$ is finite a.s. and $A$ is a tail point-map-property, $\probPalm{(G,o)}{A_1\Delta A_2}=0$. This implies that $A_2$ is independent from $B_2$. Since $\bs F_2$ is distributed as $\wusf(G)$, the above claim is proved. So, $\probCond{A}{\bs G, \bs o}\in\{0,1\}$ a.s. 
	
	Since $A\in\sigfield{pm}{T}{}\subseteq\sigfield{pm}{I}{}$, one can show that there exists a version of $\probCond{A}{\bs G, \bs o}$ that does not depend on the root. For this, consider the version $g[G,x]:=\myprob{[G,x;\wusf(G)]\in A}$. By using the same instance of $\wusf(G)$ for different $x\in G$, one obtains that $g$ depends only on $G$. We have also proved that $g\in\{0,1\}$ a.s. Therefore, Lemma~\ref{lem:completion-cond} implies that $A$ is in the null-event-augmentation of $\mathcal I$ and the claim is proved.
	%
	%
\end{proof}

\subsection{Indistinguishability of Components}
\label{subsec:wusf-indistinguishability}

In this section, we provide an alternative proof of \Cref{thm:usf-indistinguishability} based on the method of \Cref{intro:method-comp}.
If $\bs G$ is recurrent, then $\wusf(\bs G)$ is connected and there is nothing to prove. So, we assume that $\bs G$ is transient.
Hence, \cite{Hu16cyclebreaking} implies that $\wusf(\bs G)$ is one-ended a.s. Thus, it is enough to prove \Cref{thm:usf-indistinguishability} for $\owusf$ (since there is a unique way to orient the edges of $\wusf$ towards the ends), which can be considered as a point-map as discussed in \Cref{intro:models}.
We will prove the theorem by conditioning on the infinite path from the origin (i.e., the ancestral line) according to the steps sketched in the introduction.
We need the following coupling, which is based on \Cref{lem:usf-covering}.


%
%

\begin{lemma}
	\label{lem:usf-coupling}
	In the setting of \Cref{lem:usf-anc}, given $0\leq R<\infty$, there exists a coupling $(\bs F_1,\bs F_2)$ of two copies of $\owusf(G)$ such that:
	\begin{enumerate}[label=(\roman*)]
		\item \label{lem:usf-coupling-marginal} Each $\bs F_i$ has the same law as $\owusf(G)$.
		\item \label{lem:usf-coupling-anc} $\anc_{\bs F_1}(o)=\anc_{\bs F_2}(o)$ a.s.  
		\item \label{lem:usf-coupling-independent} Conditioned on $\anc_{\bs F_1}(o)$, $\bs F_1$ is independent from $\restrict{\bs F_2}{\oball{R}{o}}$.
		\item \label{lem:usf-coupling-delta} $\card{\bs F_1\Delta\bs F_2}\leq 2\card{\restrict{E}{\oball{R}{o}}}$ a.s., where $E$ is the edge set of $G$.
	\end{enumerate}
\end{lemma}

It should be noted that we do not require the distribution of $(\bs F_1, \bs F_2)$ in this lemma to be a measurable and equivariant function of $(G,o)$. Having the coupling on any single $(G,o)$ is sufficient.

\begin{proof}
	For every simple path $P$ in $G$, let $\mu_P$ be the distribution of the random forest constructed by Wilson's algorithm with initial configuration $P$.
	
	We construct the coupling step by step. In the beginning, the status of every edge is undetermined.  
	Construct $\bs F$ distributed as $\owusf(G)$. Include $\bs P:=\anc_{\bs F}(o)$ in both $\bs F_1$ and $\bs F_2$. 
	Then, given $\bs P$, construct $\restrict{\bs F_1}{\oball{R}{o}}$ distributed as the conditional law of $\restrict{\bs F}{\oball{R}{o}}$ conditionally on $\anc_{\bs F}(o)$ (i.e., using the distribution $\mu_{\bs P}$ defined above). Construct $\restrict{\bs F_2}{\oball{R}{o}}$ with the same distribution, but independently from $\restrict{\bs F_1}{\oball{R}{o}}$ (conditionally on $\bs P$). 
	
	For $i=1,2$, let $\nu_i$ be the distribution of $\owusf(G)$ conditioned on the {event} that its restriction to $\oball{R}{o}$ is identical to $\restrict{\bs F_i}{\oball{R}{o}}$, and that $\anc(o)=\bs P$ ($\bs P$ and $\restrict{\bs F_i}{\oball{R}{o}}$ are constructed in the previous step and are fixed here). Let $S$ be the set of edges in $\oball{R}{o}$. By \Cref{lem:usf-covering}, there exists a coupling of $\nu_1$ and $\nu_2$ such that their finite difference is at most $2\card{S}$ a.s. Construct the rest of $\bs F_1$ and $\bs F_2$ by this coupling.
	
	By construction, the conditions \ref{lem:usf-coupling-marginal}, \ref{lem:usf-coupling-anc} and \ref{lem:usf-coupling-delta} are satisfied. In the above coupling, note that the marginal distribution of ${\bs F}_1$ (given the previous steps; i.e., given $\anc(o)$) does not depend on $\restrict{\bs F_2}{\oball{R}{o}}$. This implies that, once $\anc(o)$ is given, $\bs F_1$ is independent from $\restrict{\bs F_2}{\oball{R}{o}}$. So, \ref{lem:usf-coupling-independent} holds and the claim is proved.
%
	%
\end{proof}

We now proceed towards proving \Cref{thm:usf-indistinguishability}. First, we provide a short proof for $\wusf(\mathbb Z^d)$ using the method of \Cref{sec:warmUp}.

\begin{proof}[Proof of \Cref{thm:usf-indistinguishability} (for the case $\mathbb Z^d$)]
	More general than $\mathbb Z^d$, we assume that $[\bs G, \bs o]$ is a unimodular random graph that satisfies the shift-coupling property a.s. (\Cref{lem:coupling}). The proof is similar to that of \Cref{prop:simpler}, and hence, we only sketch the proof. We may assume transience and ergodicity. Note that this implies a similar shift-coupling property for the loop-erased random walk. By \Cref{step1comp} (\Cref{thm:I in T 2}), it is enough to prove that every tail branch-property $A$ is trivial. Let $\bs F:=\owusf(\bs G)$. Similarly to the proof of \Cref{prop:simpler}, construct a copy $\bs F'$ of the $\bs F$ as follows: Given $n\in\mathbb N$ and a sample $(G,o)$, construct $\restrict{\bs F'}{\oball{n}{o}}$ independently from $\bs F$. Next, construct $\anc_{\bs F'}(o)$ coupled with $\bs F$ such that it shift-couples with $\anc_{\bs F}(o)$. Then, similarly to \Cref{lem:usf-coupling}, one can construct the rest of $\bs F'$ such that $\card{\bs F'\Delta \bs F}<\infty$. By this coupling, one can show that $A$ is independent from $[\bs G, \bs o; \restrict{\bs F}{\oball{n}{\bs o}}]$ for all $n$, and hence, $A$ is independent from itself conditionally on $[\bs G, \bs o]$. This proves that $h(G,o):=\probPalm{G}{[G,o,\bs F]\in A}\in\{0,1\}$ for almost every sample $(G,o)$ of $[\bs G, \bs o]$. It can be shown that either $h(G,\cdot)\equiv 0$ or $h(G,\cdot)\equiv 1$. Then, ergodicity implies that $A$ is trivial, and the claim is proved.
\end{proof}

For the general case, we prove \Cref{step2comp,step3comp} in the next lemmas. In what follows, since the components are one-ended a.s., the reader might replace $\sigfield{branch}{T}{}$ with $\sigfield{comp}{T}{}$. However, since some of the lemmas do not need unimodularity, we keep using $\sigfield{branch}{T}{}$ for having more general formulations.

\begin{lemma}[\Cref{step2comp-1}]
	\label{lem:usf-conditionalIndependence}
	Let $[\bs G, \bs o]$ be a transient random rooted graph 
	and $\bs F=\owusf(\bs G)$. Then, Every tail branch-property $A\in\sigfield{branch}{T}{}$ (resp. every tail level-set-property $A\in\sigfield{level}{T}{}$) is independent from itself conditionally on $[\bs G, \bs o; \anc_{\bs F}(\bs o)]$; i.e., $\probCond{A}{[\bs G, \bs o; \anc_{\bs F}(\bs o)]}\in\{0,1\}$ a.s.
\end{lemma}
Here, we do neither require unimodularity nor one-endedness.
\begin{proof}
	To prove the claim, it is enough to prove that for every $R\geq 0$, $A$ is independent from $[\bs G, \bs o; \restrict{\bs F}{\oball{R}{\bs o}}]$ conditionally on $[\bs G, \bs o; \anc(\bs o)]$. Equivalently, for every event $B$ that depends only on $[\bs G, \bs o; \restrict{\bs F}{\oball{R}{\bs o}}]$, we should prove
	\begin{equation*}
		\probCond{A\cap B}{\bs G, \bs o; \anc(\bs o)} = \probCond{A}{\bs G, \bs o; \anc(\bs o)}\cdot \probCond{B}{\bs G, \bs o; \anc(\bs o)}, \text{ a.s.}
	\end{equation*}
	By \Cref{lem:usf-anc}, a version of the conditional distribution of $[\bs G, \bs o; \bs F]$ conditionally on $[\bs G, \bs o; \anc(\bs o)]$ can be given by a version of Wilson's algorithm. Denote the latter by $\mathbb P_{(G,o;P)}$ for every sample $(G,o; P)$ of $[\bs G, \bs o; \anc(\bs o)]$; i.e., $\mathbb P_{(G,o;P)}$ is the distribution of $\owusf(G)$ conditioned on $\anc(o)=P$. Hence, it is enough to prove that, for almost every sample $(G,o;P)$, 
	\begin{equation}
		\label{eq:lem:usf-conditionalIndependence-1}
		\probPalm{(G,o;P)}{A\cap B} = \probPalm{(G,o;P)}{A}\cdot \probPalm{(G,o;P)}{B}.
	\end{equation}
	\Cref{lem:usf-coupling} gives a coupling $(\bs F_1,\bs F_2)$ of two copies of $\owusf(G)$. Let $A_i$ (resp. $B_i$) be the event that $[G,o; \bs F_i]\in A$ (resp. $\in B$). By part~\ref{lem:usf-coupling-independent} of \Cref{lem:usf-coupling}, $A_1$ is independent from $B_2$ under $\mathbb P_{(G,o;P)}$. Parts~\ref{lem:usf-coupling-anc} and~\ref{lem:usf-coupling-delta} of \Cref{lem:usf-coupling} imply that $C_{\bs F_1}(\bs o)$ has a common ancestral branch with $C_{\bs F_2}(\bs o)$ and $\card{L_{\bs F_1}(\bs o)\Delta L_{\bs F_2}(\bs o)}<\infty$. Therefore, since $A$ is a tail branch-property (resp. tail level-set-property), $\probPalm{(G,o;P)}{A_1\Delta A_2}=0$. This implies that $A_2$ is independent from $B_2$. Finally,~\Cref{eq:lem:usf-conditionalIndependence-1} is implied by part~\ref{lem:usf-coupling-marginal} of \Cref{lem:usf-coupling}. So, the claim is proved.
\end{proof}

\begin{lemma}[\Cref{step2comp-2}]
	\label{lem:usf-measurable}
	In the setting of \Cref{lem:usf-conditionalIndependence},
	\begin{enumerate}[label=(\roman*)]
		\item \label{lem:usf-measurable-comp} If $A\in\sigfield{branch}{T}{}$ is a tail branch-property, then  $\probCond{A}{\bs G, \bs o; \anc_{\bs F}(\bs o)}$ is measurable with respect to the null-event-augmentation of $\sigfield{path}{I}{}$.
		\item \label{lem:usf-measurable-foil} If, in addition, the level-set of $\bs o$ in $\owusf(\bs G)$ is infinite a.s. and $A\in\sigfield{level}{T}{}$ is a tail level-set-property, then $\probCond{A}{\bs G, \bs o; \anc_{\bs F}(\bs o)}$ is measurable with respect to the null-event-augmentation of $\sigfield{path}{T}{}$.
	\end{enumerate}
\end{lemma}


\begin{proof}
	Assume $(G,o;(x_i)_{i\geq 0})$ and $k\geq 0$ are given such that $G$ is transient, $x_0=o$ and $(x_i)_{i\geq k}$ is a simple path in $G$ (i.e., only traverses the edges and visits every vertex at most once).  Let $\mathbb P_{(G,o;(x_i)_i,k)}$ be the distribution of $\owusf(G)$ conditioned on $\forall i\geq k: (x_i,x_{i+1})\in \bs F$, as described in \Cref{lem:usf-anc} by a version of Wilson's algorithm.
	
	\ref{lem:usf-measurable-foil}.
	Define $g$ as follows. 
	Let $M'$ be the set of $[G,x_0;(x_i)_{i\geq 0}]$ such that $G$ is transient, $(x_i)_{i\geq k}$ is a simple path in $G$ for some $k$, and that $\mathbb P_{(G,x_0;(x_i)_i,k)}$-almost surely, the level-set of $x_0$ in the forest is infinite.
	By \Cref{lem:usf-coupling}, one can see that the last condition does not depend on $k$, and hence, $M'\in\sigfield{path}{T}{}$.
	Let $g$ be zero outside $M'$.
	Given $[G,x_0;(x_i)_i]\in M'$, let $k$ be the smallest number such that $(x_i)_{i\geq k}$ is a simple path and define $g[G,x_0;(x_i)_i]:=\probPalm{(G,x_0;(x_i)_i,k)}{A}$.
	The construction shows that $g$ is a version of $\probCond{A}{\bs G, \bs o; \anc_{\bs F}(\bs o)}$. We should prove that $g$ is $\sigfield{path}{\hat T}{}$-measurable (In fact, since $g$ is zero outside $M'$, the next arguments imply ghat $g$ is $\sigfield{path}{T}{}$-measurable).
	
	Let $G$ be deterministic and $(x_i)_i$ and $(x'_i)_i$ be two sequences in $G$ such that $[G,x_0;(x_i)_i]\in M'$ and $[G,x'_0;(x'_i)_i]\in M'$. Assume also that $\forall i\geq N: x_i=x'_i$, for some $N$ (but $x_0$ might be different from $x'_0$). By \Cref{lem:cber,lem:completion-I(R)}, it is enough to prove that $g[G,x_0;(x_i)_i] = g[G,x'_0;(x'_i)_i]$. We prove this assuming that $(x_i)_i$ and $(x'_i)_i$ are simple paths for simplicity. The general case is similar.
	Let $\bs F$ (resp. $\bs F'$) denote $\owusf(G)$ conditioned on $(x_i,x_{i+1})\in\bs F, \forall i$ (resp. $(x'_i,x'_{i+1})\in\bs F', \forall i$).
	We couple $\bs F$ and $\bs F'$ as follows. 
	Let $S$ be the set of edges $\{(x_i,x_{i+1}): 0\leq i<N\}\cup\{(x'_i,x'_{i+1}): 0\leq i<N\}$ and let $o:=x_N=x'_N$. By \Cref{lem:usf-coupling}, there exists a coupling of $\bs F$ and $\bs F'$ such that $\card{\bs F\Delta\bs F'}<\infty$ a.s. In this case, since $\anc_{\bs F}(x_0)$ and $\anc_{\bs F'}(x'_0)$ eventually coincide, one also has $L_{\bs F}(x_0)\Delta L_{\bs F'}(x'_0)<\infty$. Choose $z\in L_{\bs F}(x_0)\cap L_{\bs F'}(x'_0)$ far enough such that $\anc_{\bs F}(z)=\anc_{\bs F'}(z)$ (which exists a.s. since the level-set of $x_0$ is infinite a.s. by assumption). Since $A$ is a level-set-property, one has $\identity{A}[G,x_0;\bs F]=\identity{A}[G,z;\bs F]$ and $\identity{A}[G,x'_0,\bs F'] = \identity{A}[G,z;\bs F']$. Also, since $A$ is a tail level-set-property, one gets $\identity{A}[G,z;\bs F]=\identity{A}[G,z;\bs F']$. This implies that $\identity{A}[G,x_0;\bs F] = \identity{A}[G,x'_0;\bs F']$. Therefore, $g[G,x_0;(x_i)_i] = g[G,x'_0;(x'_i)_i]$ as desired. So, we have proved that $g$ is $\sigfield{path}{\hat T}{}$-measurable.
	
	\ref{lem:usf-measurable-comp}.
	Define $M'$ and $g$ similarly to the previous case, but without requiring that the level-set of $o$ is infinite.
	One can prove similarly that, in almost every sample $(G,o)$, if $(x_i)_i$ and $(x'_i)_i$ are simple paths that shift-couple (but maybe $x_0\neq x'_0$), then $g[G,x_0;(x_i)_i]=g[G,x'_0;(x'_i)_i]$. So, \Cref{lem:cber,lem:completion-I(R)} imply that $g$ is $\sigfield{path}{\hat I}{}$-measurable. 
\end{proof}

\begin{corollary}[\Cref{step2comp,step2foil}]
	\label{cor:usf-sigfield}
	In the setting of \Cref{lem:usf-conditionalIndependence}, one has
	$
		\sigfield{branch}{T}{} = \sigfield{path}{I}{} \text{ mod } \mathbb P.
	$
	If, in addition, the level-set of $\bs o$ in $\owusf(\bs G)$ is infinite a.s., then 
	$
		\sigfield{level}{T}{} = \sigfield{path}{T}{} \text{ mod } \mathbb P.
	$
\end{corollary}
\begin{proof}
	The inclusions $\sigfield{path}{I}{}\subseteq\sigfield{comp}{T}{}$ and $\sigfield{path}{T}{}\subseteq\sigfield{level}{T}{}$ mod $\mathbb P$ are proved in \Cref{lem:basicInclusion}. By \Cref{lem:completion-cond}, the converses are implies by \Cref{lem:usf-conditionalIndependence,lem:usf-measurable}.
\end{proof}


\begin{lemma}[\Cref{step3comp}]
	\label{lem:usf-trivial I}
	Let $[\bs G, \bs o]$ be an ergodic transient unimodular graph and $(\bs X_n)_n$ be the loop-erased random walk starting from $\bs o$. If $\omid{\mathrm{deg}(\bs o)}<\infty$, then $[\bs G, \bs o; (\bs X_n)_n]$ has trivial invariant sigma-field; i.e., $\sigfield{path}{I}{}$ is trivial.
\end{lemma}
It should be noted that, despite \Cref{lem:stationary}, $[\bs G, \bs o; (\bs X_n)_n]$ is not necessarily stationary here, even after any biasing. In fact, its distribution might be non-absolutely-continuous with respect to that of $[\bs G, \bs X_1; (\bs X_{n+1})_n]$ (see \Cref{rem:quasi}).
\begin{proof}
	Let $(\bs Y_n)_n$ be the simple random walk starting from $\bs o$. 
	Let $A\in\sigfield{path}{I}{}$ be arbitrary.
	Let $B_0$ be the set of $[G,o,(x_i)_i]$ such that $(x_i)_i$ is a sequence starting from $o$ and $d(o,x_i)\to\infty$. 
	Let $B\subseteq B_0$ be the event that $B_0$ holds and $\mathrm{LE}(x_i)_i\in A$, where $\mathrm{LE}(\cdot)$ denotes the loop-erasure procedure. One has $\myprob{[\bs G, \bs o; (\bs Y_i)_i]\in B} = \myprob{[\bs G, \bs o; (\bs X_n)_n]\in A}$. We claim that $B\in \sigfield{path}{I}{}$. 
	For this, since $B_0\in\sigfield{path}{I}{}$, it is enough the prove that if $(y_i)_i$ and $(y'_i)_i$ are sequences in $G$ that shift-couple and $d(y_0,y_i)\to\infty$, then $\identity{B}[G,y_0;(y_i)_i] = \identity{B}[G,y'_0;(y'_i)_i]$. The latter means $\identity{A}[G,y_0;\mathrm{LE}(y_i)_i] = \identity{A}[G,y'_0;\mathrm LE(y'_i)_i]$. 
	The reader can verify that $\mathrm{LE}(y_i)_i$ shift-couples with $\mathrm{LE}(y'_i)_i$, and so, the last claim follows from the assumption $A\in\sigfield{path}{I}{}$.
	So, we have shown that $B\in\sigfield{path}{I}{}$. By \Cref{thm:pathErgodic}, one obtains $\myprob{[\bs G, \bs o; (\bs Y_i)_i]\in B}\in\{0,1\}$. So, we have proved that $\myprob{[\bs G, \bs o; \anc(o)]\in A}\in \{0,1\}$ and the claim is proved.

\end{proof}

By combining what we have proved, one can complete the proofs of \Cref{thm:usf-indistinguishability} and the $\wusf$ case of \Cref{thm:fixed-comp} as follows.

\begin{proof}[Proof of \Cref{thm:usf-indistinguishability} (General Case)]
	We may assume that $[\bs G, \bs o]$ is ergodic.
	As mentioned, in the recurrent case, $\wusf(\bs G)$ is connected and there is nothing to prove. So, assume that $\bs G$ is transient a.s. In particular, $\anc(\bs o)$ is well defined and has the same distribution as the loop-erased random walk (given a realization of the rooted graph).
	\Cref{cor:usf-sigfield} and \Cref{lem:usf-trivial I} proved \Cref{step2comp,step3comp} respectively. So, the claim is implied by \Cref{thm:general}.

\end{proof}

\begin{proof}[Proof of \Cref{thm:fixed-comp} (WUSF Case)]
	Similarly to the case of CMTs (see the proof of \Cref{prop:tail-fixed}), the proof is the same as that of \Cref{step2comp} in the proof of \Cref{thm:usf-indistinguishability}. We only highlight the differences here (see the proof of \Cref{prop:tail-fixed} for more details). Given an event $A\in\sigfield{branch}{T}{G}$ and $o\in G$, the same proof as that of \Cref{cor:usf-sigfield} constructs an event $A'\in\sigfield{path}{I}{G}$ such that $\probPalm{(G,o)}{A\Delta A'}=0$, where $\mathbb P_{(G,o)}:=\delta_o\otimes \mathbb P_G$ and $\mathbb P_G$ is the distribution of $\wusf(G)$. It can be seen that the construction of $A'$ does not depend on $o$. This implies that $Q_G[A\Delta A']=0$, which proves the claim. The claim for $k$-tuples of components can also be proved similarly.
\end{proof}

\section{Models on Bernoulli and Poisson Point Processes}
\label{sec:BernoulliPoisson}


In this section, we generalize \Cref{model1,model2} to cover the examples based on the Bernoulli or Poisson point process, as mentioned in \Cref{main:cmt}. The main goal is to prove indistinguishability in these models (\Cref{thm:cmt-all}) and their one-endedness.
The main examples are provided first in \Cref{subsec:BernoulliExamples}. Then, indistinguishability is proved for two specific discrete models in \Cref{subsec:howard}. Based on the ideas in this proof, the generalizations of \Cref{model1,model2} are provided in \Cref{subsec:bernoulli,subsec:poisson}.

\subsection{Main Examples}
\label{subsec:BernoulliExamples}

As alluded to in \Cref{intro:models},
there are several point-maps studied in the literature, which are built on the points of a Bernoulli or Poisson point process. We generalize these models in the following examples, which are the main aims of this section.

\begin{example}[Generalized Howard's Model]
	\label{ex:howard}
	Let $[\bs G, \bs o]$ be an infinite unimodular graph such that $\bs G$‌ is a transitive graph a.s. Let $\Phi$ be a Bernoulli point process on $\bs G':=\mathbb Z\times \bs G$. For each point $(t,x)\in\Phi$, let $\bs F'(t,x)$ be the point of the form $(t+1,y)\in \Phi$ that is closest to $(t,x)$. If there are ties, choose one of them randomly and independently using i.i.d. marks on $\bs G'\times \bs G'$ (consider the mark of $((t,x),(t+1,y))$). This model was introduced in the case $\bs G=\mathbb Z^d$ in~\cite{GaRoSa04howardModel}, based on river models, with the goal of studying their empirical/scaling limit properties. It was also shown that $\bs F'$ is one-ended a.s. for every $d$, but connected only for $d=2,3$. Note that this model does not satisfy the weak aperiodicity condition.
\end{example}
\begin{example}[Discrete Strip Point-Map]
	\label{ex:strip-discrete}
	Let $[\bs G, \bs o]$‌ be an infinite unimodular graph such that $\omid{\mathrm{deg}(\bs o)}<\infty$. Let $\Phi$ be a Bernoulli point process on $\bs G':=\mathbb Z\times \bs G$. For $(t,x)\in \Phi$, let $\bs F'(t,x)$ be the point of $\Phi$ in the strip $[t+1,\infty)\times B_1(x)$ with the minimum first coordinate (choose among the ties randomly using i.i.d. marks on $\bs G'\times \bs G'$). As a variant, one may choose $\bs F'(x)$ as the first point in the cone $\{(t',x'): \norm{x'-x}<{t'-t}\}$. 
\end{example}

\begin{example}[Generalized Strip Point-Map]
	\label{ex:strip}
	Let $[\bs G, \bs o, \bs \mu]$ be either the Euclidean $\mathbb R^d$ or the hyperbolic space $\mathbb H^d$, equipped with the volume measure (or more generally, any \textit{unimodular random measured metric space} with some conditions, see \Cref{ex:strip2}). Let $\Phi$ be a Poisson point process in $\mathbb R\times \bs G$. For each $(t,x)\in \Phi$, let $\bs F_{\Phi}(t,x)$ be the point of $\Phi$ in the strip $(t,\infty)\times B_1(x)$ with the minimum first coordinate. One may also replace the strip by a cone as in \Cref{ex:strip-discrete}. This model is a space-time version of a particle process, which in the case $\bs G=\mathbb R$ was introduced in~\cite{FeLaTh04}. It was also shown that $\bs F_{\Phi}$ is connected and one-ended in this case.
\end{example}

The models mentioned above are point-maps on $\Phi$ which are clearly cycle-free. Hence, an oriented forest is obtained in each model and one can consider its connected components and level-sets. The one-endedness of these models will be proved in \Cref{subsec:howard} and the rest of the required conditions (see \Cref{model3,model4,model5,model6} below) will be verified in \Cref{ex:strip1.5,ex:strip2}.

%

In the last examples, the forest is constructed without additional randomness, except for breaking the ties (but we will allow additional randomness in the next models). Hence, given $\Phi$, the balance condition of \Cref{model1} is violated. Nevertheless, we will leverage the randomness of $\Phi$ in order to prove the three steps of \Cref{intro:method-comp,intro:method-foil} and to prove indistinguishability of components and level-sets.

\subsection{Indistinguishability in Howard's Model and the Discrete Strip Point-Map}
\label{subsec:howard}

Before introducing general models, we present the proof of indistinguishability in the generalized Howard's model and the discrete strip point-map (\Cref{ex:howard,ex:strip-discrete}). The proof shows the main ideas for defining the general models (\Cref{model3,model4,model5,model6}). The general Bernoulli case is proved with exactly the same proof.

The notion of indistinguishability is defined similarly to \Cref{def:indistinguishability} as follows:
Let $\tilde G_*$ be the space of all $[G,o; \varphi, m]$, where $(G,o)$ is a rooted graph or discrete space (possibly having marks if needed), $\varphi$ is a subset of $G$‌ and $m:G\times G\to [0,1]$ is a marking of pairs of points (one may consider more general types of marks in similar models). This space is Polish with a topology similar to that of $\mathcal G_*$‌ (see~\cite{Kh19generalization}). In the above examples, a random element $[\bs G, \bs o; \Phi, \bs m]$ of $\tilde G_*$ is given and a function $\bs F':\varphi\to\varphi$‌ is constructed on almost all samples $(G,o;\varphi,m)$ without additional randomness. The notion of \defstyle{component-properties} and \defstyle{level-set-properties} can be defined similarly to \Cref{def:invariantDrainage} for events $A\subseteq\tilde G_*$. This leads to the notion of indistinguishability as in \Cref{def:indistinguishability}.

\begin{proof}[Proof of \Cref{thm:cmt-all} (for Howard's Model and the Discrete Strip Point-Map)]
	The proof of indistinguishability in these models is similar to the analogous proof for CMTs (\Cref{thm:indistinguishability}) and we only highlight the differences. 
	Note that in these examples, the graph under study is $\bs G':=\mathbb Z\times \bs G$‌ rooted at $\bs o':=(0,\bs o)$. 
	
	First, we prove one-endedness. If the claim is false, one can prove similarly to \Cref{thm:one-ended-or-connected} (using update-tolerance) that, with positive probability, there is a single connected component and is two-ended. Let $P$ be the unique bi-infinite path in the graph of $\bs F'$. Now, considering the partition $\{\{t\}\times \bs G :t\in\mathbb Z\}$‌ of $\bs G'$, the path $P$ selects at most one point from each element of the partition. This contradicts the mass transport principle, see the \textit{no infinite/finite inclusion lemma} (Lemma~2.9) of~\cite{eft}. So, one-endedness is proved.
	
	We start by Steps~\Cref{step3comp} and~\ref{step3foil}. Consider a sample $G$ and $x':=(t,x)\in G':=\mathbb Z\times G$. Let $\mathbb P_{G}$ denote the distribution of $(\Phi,\bs m)$ when $G$‌ is fixed. Given that $x'\in \Phi$, we construct $\bs F'(x')$ as follows ‌by revealing $\Phi$ and $\bs m$ partially until a stopping time. For the discrete strip point-map, reveal $\Phi$ in the vertical sections $\{t+u\}\times B_1(x)$, $u=1,2,\ldots$ one by one until finding a point of $\Phi$. Then, reveal $\bs m(x',\cdot)$ on ties. For Howard's model, reveal $\Phi$ on $\{t+1\}\times B_n(x)$, $n=0,1,\ldots$ one by one until finding a point of $\Phi$ and then reveal $\bs m$ on the ties. The important property is that, conditionally on the revealed set $S(x')$, the distribution of the unrevealed data is unchanged ($S(x')$‌ can be regarded as a subset of the disjoint union of $G'$‌ and $G'\times G'$). This is a property of \textit{stopping sets}, which is an important notion in probability theory and is discussed in \Cref{ap:stopping} (see \Cref{thm:stopping}). In addition, $S(\bs F'(x'))$ (i.e., the data needed to be revealed for constructing $\bs F'^{(2)}(x')$) is completely disjoint from $S(x')$. Continuing by induction, one obtains that $\anc(x'):=\anc_{\bs F'}(x')$ is a Markov chain on $\mathbb Z\times G$ (under $\mathbb P_G$). The kernel of this Markov chain is 
	\[
		K(x',B):=K(G,x',B):=\probPalmC{G}{\bs F'(x')\in B}{x'\in{\Phi}},
	\]
	for $x'=(t,x)\in G'$‌ and $B\subseteq G'$. 
	In the discrete strip point-map, one can check that $K$ balances the function $b(t,x):=\mathrm{deg}_G(x)+1$, where $\mathrm{deg}_G(x)$ is the degree of $x$‌ in $G$ 
		(see~\Cref{ex:strip1.5} below for the proof of this claim). In the generalized Howard model, leveraging the fact that $\bs G$‌ is transitive a.s., $K$‌ balances the function $b\equiv 1$ (see~\Cref{ex:strip1.5} again for the proof).
	Therefore, in both models, \Cref{thm:pathErgodic} implies that $[\bs G', \bs o'; \anc(\bs o')]$‌ (conditionally on $\bs o'\in \Phi$) is ergodic. In addition, for the discrete strip point-map, $K$ satisfies the weak aperiodicity condition of \Cref{model2}. Therefore, \Cref{thm:pathTail} implies that $[\bs G', \bs o'; \anc(\bs o')]$‌ (conditionally on $\bs o'\in \Phi$) is tail-trivial. For Howard's model, the weak aperiodicity condition is not satisfied, but the trick in \Cref{ex:crw} proves the tail-triviality. This establishes Steps~\ref{step3comp} and~\ref{step3foil} in both models.
	
	For Steps~\ref{step1comp} and~\ref{step1foil}, we need to defined tail branch/level-set properties. This can be done similarly to \Cref{def:tailDrainage,def:super-tail-comp} leveraging finite modifications of $(\varphi,m)$ that result in a finite modification of $\bs F'$. 
	Indeed, by the mass transport principle, one can show that every data lies in at most finitely many sets $S(x)$ a.s. (this will be proved rigorously in \Cref{lem:shadow} below).
	This implies that a finite modification of $(\Phi,\bs m)$ results in a finite modification of $\bs F'$ a.s. 
	In the proof of \Cref{thm:I in T 2}, in the definitions of $M'$ and $M''$, add the condition that the sets $(S(x))_{x\in G}$ have no infinite overlap.
	Now, the proof of \Cref{thm:I in T 2} works word by word. 
	So, Steps~\ref{step1comp} and~\ref{step1foil} are established.
	
	For Steps~\ref{step2comp} and~\ref{step2foil}, we modify the proof of \Cref{thm:tailDrainage} as follows. Assume $A$ is a tail branch-property (resp. a tail level-set-property). To prove that $A$ is equivalent to a shift-invariant (resp. tail) path property (of $[\bs G', \bs o'; \anc(\bs o')]$), we will construct a version of the conditional distribution of $(\Phi,\bs m)$ given $[\bs G',\bs o';\anc(\bs o')]$. 
	Assume $(G',o';(x'_i)_i)$ is given, where $G'=\mathbb Z\times G$‌ and $o'=(0,o)$. 
	Let $N$ be the smallest integer such that $(x'_i)_{i\geq N}$ are distinct and $(S(x'_i))_{i\geq N}$ are pairwise disjoint conditionally on $\bs F'(x'_i)=x'_{i+1}$, $\mathbb P_G$-a.s., assuming in addition that $\probPalm{G}{\bs F'(x'_i)=x'_{i+1}}>0$ for all $i\geq N$.
	Assume $N<\infty$.
	Let $\Psi_i$ be an independent copy of $\restrict{(\Phi,\bs m)}{S(x'_i)}$ conditionally on $\bs F'(x'_i)=x'_{i+1}$. Do this construction such that $(\Phi,\bs m), \Psi_1,\Psi_2,\ldots$ are independent.
	Let $(\Phi_1, \bs m_1)$ be obtained from $(\Phi,\bs m)$ by replacing $(\Phi,\bs m)$ with $\Psi_i$ inside $S'_i$ for all $i\geq N$, where $S'_i:=S(x'_i, \Psi_i)$ is the analogue of $S(x'_i)$ when $\Phi$ is replaced by $\Psi_i$. The stopping set properties guarantee that $S(x'_i,(\Phi_1,\bs m_1)) = S'_i$ and $\bs F'(x'_i,(\Phi_1,\bs m_1))=x'_{i+1}$ for $i\geq N$ a.s. In addition, one can see that $(\Phi_1,\bs m_1)$ is a regular conditional version of $(\Phi,\bs m)$ conditioned on $\anc(o')=(x'_i)_i$ (for this, one should use \Cref{thm:stopping} for the nested sequence of stopping sets $\bigcup_{i=1}^n S(\bs F'^{(i)}(o'))$, $n=1,2,\ldots$ and take the limit). In addition, $(\Phi_1,\bs m_1)$ is coupled with $(\Phi,\bs m)$ in such a way that they are identical outside $\cup_i S'_i$.
	
	Let $M'$ be the set of all $[G',o';(x'_i)_i]$ that satisfy the conditions mentioned in the previous paragraph, plus the condition that the sets $(S(y'))_{y'\in G'}$ have no infinite overlap. Define $M''$ similarly by the additional assumption that both $\bs F'(\Phi,\bs m):=\bs F'$ and $\bs F'(\Phi_1,\bs m_1)$ are locally-finite and their level-sets containing $o'$ are infinite, $\mathbb P_{G'}$-a.s. Note that $M'\in\sigfield{path}{I}{}$, $M''\in\sigfield{path}{T}{}$ and $\myprob{[\bs G', \bs o'; \anc(\bs o')]\in M'}=\myprob{[\bs G', \bs o'; \anc(\bs o')]\in M''}=1$. 
	
	Now, assume $[G',o';(x'_i)_i]$‌ is in $M'$ (resp. $M''$) and $(\tilde x'_i)_i$ is another sequence in $G'$ that shift-couples (resp. eventually coincides) with $(x'_i)_i$. Define $(\tilde\Phi_1,\tilde{\bs m}_1)$ similarly to $(\Phi_1,\bs m_1)$ using the same source of randomness as that used in $\Psi_i$ (as soon as the sequences couple). In this construction, $(\tilde\Phi_1,\tilde{\bs m}_1)$ is a finite modification of $(\Phi_1,\bs m_1)$ a.s. and one can continue the proof of \Cref{thm:tailDrainage} to show that $g\in\sigfield{path}{\hat I}{}$ (resp. $g\in\sigfield{path}{\hat T}{}$), where 
	$g[G',x'_0;(x'_i)_i]:= \probPalm{G'}{[G',x'_0;\Phi_1,\bs m_1]\in A}$.

	Finally, to prove that $g\in\{0,1\}$ in the proof of \Cref{thm:tailDrainage}, we construct a copy $(\hat{\Phi}_1,\hat{\bs m}_1)$‌ of $(\Phi_1, \bs m_1)$, coupled with itself as follows: Given $n>0$, resample the first $n$‌ steps of the construction of $(\Phi_1,\bs m_1)$ independently from the previous construction, but leave the steps $n+1,n+2,\ldots$ intact. Call the resulting sets $(\hat S'_i)_i$, which satisfy $\hat S'_{n+i}=S'_{n+i}, \forall i$. Then, construct the rest of $(\hat{\Phi}_1,\hat{\bs m}_1)$ in $B_n(o')$ (i.e., except the already constructed parts of $(\hat{\Phi}_1,\hat{\bs m}_1)$ in $\cup_i \hat S'_i$) independently from $(\Phi_1,\bs m_1)$‌ and from the previous steps. Let the rest of $(\hat{\Phi},\hat{\bs m})$ (i.e., outside $B_n(o')$ and outside $\cup_i\hat S'_i$) be identical to $(\Phi_1,\bs m_1)$‌. In this coupling, $(\Phi_1,\bs m_1)$ is resampled in a large neighborhood of $o'$ (conditionally on $\anc(o')=(x'_i)_i$). Also, $(\hat{\Phi}_1,\hat{\bs m}_1)$ is a finite modification of $(\Phi_1, \bs m_1)$ a.s. One can use these facts to continue the proof of \Cref{thm:tailDrainage} to show that $A$ is independent from itself conditionally on $[\bs G', \bs o'; \anc(\bs o')]$. This completes Steps~\ref{step2comp} and~\ref{step2foil}. So, the claim is proved.
\end{proof}

\subsection{The General Bernoulli Case}
\label{subsec:bernoulli}

Now, we describe our general model on Bernoulli point process. In fact, everything is already mentioned in the proof given in the last subsection, and we only need to single out the minimum requirements of the proof.

Let $[\bs G, \bs o]$ be a unimodular graph or discrete space. Given $0< p\leq 1$, let $\Phi$ be the Bernoulli point process on $\bs G$ with parameter $p$. That is, given $\bs G$, the points are retained independently with probability $p$. In \Cref{ex:howard,ex:strip-discrete}, an additional i.i.d. marking was used to break the ties. In general, we may also use the i.i.d. marks as a source of randomness to define a point-map $\bs F'$ as follows (in particular, if some measurability conditions hold, then $\bs F'$‌ is a factor of i.i.d. here). It should be noted that $\bs G$‌ might be a marked graph from the beginning. For instance, in \Cref{ex:howard,ex:strip-discrete}, a marking is in fact needed to distinguish the time direction. In this case, given $\bs G$‌ (which might already has some marking), the new i.i.d. marks will be used as a source of randomness.

Assume that $\bs m$‌ is an i.i.d. marking of $\bs G^d$ for some $d\in\mathbb N$ (note that $d=2$ in \Cref{ex:howard,ex:strip-discrete}). More generally, we may assume that $k$ and $d_1,\ldots, d_k\in\mathbb N$ are given and $\bs m_i$ is an i.i.d. marking of $\bs G^{d_i}$, with mark space $\Xi_i$ and with distribution $\nu_i$, independently from each other. But for simplicity of notation, we assume $k=1$ in what follows. 

Recall the space $\tilde G_*$‌ of all $[G,o;\varphi, m]$ from \Cref{subsec:howard}. Assume that for all samples $(G,o;\varphi, m)$, a function $\bs F':\varphi\to\varphi$‌ is given ($\bs F'$ uses $m$ as a source of randomness and there is no other extra randomness). We may extend $\bs F'$‌ to $G$ by $\bs F'(x):=x$ if $x\not\in {\varphi}$. We no longer require $\bs F'$ be a CMT‌; i.e., given $(G,{\varphi})$, the jumps $\bs F'(x), x\in {\varphi}$, are not necessarily independent. Even if $\bs F'$ is a CMT, it does not necessarily satisfy the balance condition of \Cref{model1}; e.g., in \Cref{ex:howard,ex:strip-discrete}. 
Instead, roughly speaking, we assume that, to construct $\bs F'$, it is enough to reveal only a part $Z$ of $(\varphi,m)$ (see \Cref{as:stopping} below for the precise definition). This is based on the notion of \textit{stopping pair of sets}, see \Cref{def:stoppingPair}, which heuristically means that $Z$ is determined by the revealed information of $(\varphi,m)$ (this generalizes stopping times). In fact, in all of the examples of interest in this work, the stopping set condition is guaranteed by constructing $\bs F'(o)$ via a \defstyle{decision tree} (see~\cite{LaPeYo23stopping}):‌ Heuristically, an observer starts from $o$ and reveals the evaluation of $\identity{\Phi}(\cdot)$‌ or $\bs m(\cdot)$‌ at a sequence of points one by one. At each step, the observer decides whether to stop or not, and decides the next information to reveal, based measurably of the currently observed data. We will also assume that the sets $Z(\bs F'^{(i)}(o)), i\geq 0$ are disjoint a.s. This condition can usually be proved using a Lyapunov cocycle $h$ (see \Cref{rem:model1}) that is non-increasing on the order of the revealed points.

\begin{assumption}
	\label{as:stopping}
	Assume that, to every $(G,o; \varphi,m)$, an element $\bs F'(o):=\bs F'(G,o;\varphi,m)\in G$ and a stopping pair of sets $Z=(Z_1,Z_2)$ are assigned (see \Cref{def:stoppingPair} by letting $\mu$ be the counting measure on $G$). Assume that:
	\begin{enumerate}[label=(\roman*)]
		\item 
		$\bs F'(o)\in \varphi$ if $o\in \varphi$ and $\bs F'(o)=o$ if $o\not\in\varphi$. Also, $\bs F'(o)\in Z_1$.
		\item For every $(G,o)$, $\bs F'(o)$ is a measurable function of $(\varphi\cap Z_1, \restrict{m}{Z_2})$.
		\item The function $[G,o,x;\varphi,m]\mapsto (\identity{\{x\in Z_1\}},\identity{\{F'(o)=x\}})$ is well defined and measurable.
	\end{enumerate}
\end{assumption}
Note that, for $o\not\in\varphi$, one may assume $Z(o)=(\{o\},\emptyset)$ since $\bs F'(o)=o$.

Despite the fact that the jumps are not necessarily independent, the last assumptions and \Cref{thm:stopping} imply the following Markov property, as in the proof of \Cref{subsec:howard} (note that we forget $\Phi$ and $\bs m$ and keep only $\bs F'$): 
\begin{lemma}
	\label{lem:Markov}
	Given $(G,o)$, if for every $x\in \Phi$, the sets $Z(\bs F'^{(i)}(x)), i\geq 0$ are disjoint, then
	the ancestor line of $o$ is a Markov chain on $G$‌ with kernel defined by 
	\[
	K(x,B):=K(G,x,B):=\probPalmC{G}{\bs F'(x)\in B}{x\in{\Phi}},
	\]
	for $x\in G$ and $B\subseteq G$, where $\mathbb P_{G}$ is the distribution of $(\Phi,\bs m)$ given $G$. 
\end{lemma}

We are now ready to describe the main models of this subsection.
\begin{model}
	\label{model3}
	Let $[\bs G, \bs o; \Phi, \bs m]$ be as above. Assume that $\bs F'$ and the stopping pair of sets $Z=(Z_1,Z_2)$ satisfy \Cref{as:stopping}, and also satisfy
	$\omid{\card{Z_1(\bs o)}}<\infty$. In addition, assume that almost all samples $(G,o)$ satisfy the following conditions:
	\begin{enumerate}[label=(\roman*)]
		\item (Cycle-Free) There is no cycle in the graph of $\bs F'$ on $\Phi$,  $\mathbb P_G$-a.s.
		\item (Disjoint Stopping Sets) The sets $Z(\bs F'^{(i)}(o)), i\geq 0$ are disjoint when $o\in \Phi$, $\mathbb P_G$-a.s.
		\item (Balance) One has {$b(G,\cdot)>0$} and the Markov kernel $K$ on $G$, defined above, balances the measure $b(G,\cdot)$.
		\item (Weak Irreducibility) The graph with vertex set $G$ and edge set $\{(x,y): K(x,y)+K(y,x)>0 \}$ is connected.
	\end{enumerate}
\end{model}

See \Cref{prob:directional} for an example which does not satisfy the condition of disjoint stopping sets.

\begin{model}
	\label{model4}
	In \Cref{model3}, assume in addition that almost all samples $(G,o)$ satisfy the following:
	\begin{enumerate}[label=(\roman*)]
		\setcounter{enumi}{4}
		\item (One-Endedness) All components of $\bs F'$ are one-ended with infinite level-sets, $\mathbb P_G$-a.s.
		\item (Weak Aperiodicity) The elements of the set 
		\[
		\{d\in\mathbb N: \exists n\in\mathbb N, \exists x\in G: K^n(x,\cdot)\wedge K^{n+d}(x,\cdot)\not\equiv 0\}
		\] 
		have no common divisor larger than 1, where $\cdot\wedge \cdot$ denotes the minimum of two measures.
	\end{enumerate}
\end{model}

Note that \Cref{model3,model4} generalize \Cref{model1,model2} respectively. This will be shown in \Cref{ex:model1VSmodel3} below.

The measurability assumption of $Z_1$ in \Cref{as:stopping} and the assumption $\omid{\norm{Z_1(\bs o)}}<\infty$ imply the following:

\begin{lemma}
	\label{lem:shadow}
	The set $\{y\in \bs G: x\in Z_1(y) \}$ is finite for all $x\in \bs G$, a.s. As a result, any finite modification of $(\Phi, \bs m)$ results in a finite modification of $\bs F'$,‌ a.s.
\end{lemma}
\begin{proof}
	By the mass transport principle, one has 
	$\omid{\norm{\{y\in \bs G: \bs o\in  Z_1(y) \}}} = \omid{\norm{Z_1(\bs o)}}<\infty$.  Hence, $\{y: \bs o\in  Z_1(y) \}$ is finite a.s. 
	The claim is then implied by \Cref{lem:happensAtRoot}.
\end{proof}

\begin{proof}[Proof of \Cref{thm:cmt-all} (for \Cref{model3,model4})]
	Let $\bs X_n:=\bs F'^{(n)}(\bs o)$. \Cref{lem:Markov} shows that $[\bs G, \bs X_n; \anc(\bs X_n)]$ is a Markov chain. The measurability assumptions in \Cref{as:stopping} and the assumptions of \Cref{model3,model4} directly enable one to use the stationarity, ergodicity and tail triviality results given by \Cref{lem:stationary,thm:pathErgodic,thm:pathTail}. Also, the stopping set condition enables one to describe the conditional distribution of $\Phi$ conditionally to $\anc(o)$ as in the proof of \Cref{subsec:howard}. The rest of the proof is the same as that of \Cref{subsec:howard}.
\end{proof}

\begin{example}
	\label{ex:strip1.5}
	In this example, we verify the conditions of \Cref{model3,model4} for the discrete strip point-map and the generalized Howard model (\Cref{ex:strip-discrete,ex:howard}). 
	If $b(x):=\norm{B_1(x)}=\mathrm{deg}(x)+1$ and $q:=1-p$, then 
	\[K((t,x),(t',x'))=q^{b(x)(t'-t-1)}(1-q^{b(x)})/b(x),\] for $t'>t$‌ and $x'\in B_1(x)$. One can check that $K$ balances $b$. Denoting $(s,y):=\bs F'(t,x)$, one can let $S(t,x):=\{(t,x)\}\cup \left([t+1,s]\times B_1(x)\right)$ as discussed in \Cref{subsec:howard}. 
	Now, it is easy to see that all of the conditions of \Cref{model3,model4} are satisfied.\\
	For the generalized Howard model, the formula for $K$ is 
	\[
		K((t,x),(t+1,x')) = q^{\norm{B_{n-1}(x)}}(1-q^{\norm{\partial B_n(x)}})/\norm{\partial B_n(x)},
	\]
	where $n:=d(x,x')$‌ and $\partial B_n(x):= B_n(x)\setminus B_{n-1}(x)$. Using the fact that $\norm{\partial B_n(x)}$ does not depend on $x$ due to transitivity, one can check that $K$‌ balances the function $b\equiv 1$. One can check that all conditions of \Cref{model3,model4} are satisfied except weak aperiodicity (recall that indistinguishability is proved in \Cref{subsec:howard} even without weak aperiodicity).
\end{example}



\begin{example}[CMTs]
	\label{ex:model1VSmodel3}
	In this example, we show that \Cref{model1,model2} are special cases of \Cref{model3,model4} respectively. By letting $p=1$, one has $\Phi=\bs G$‌ and it remains only to check the stopping set conditions. Since $\mathcal G_{*}$ is Polish, one can find a Borel isomorphism $\iota:\mathcal G_{*}\to [0,1]$.
	Fix a typical sample $(G,o)$. Recall that we have assumed that the points of $G$‌ might have some marks from the beginning (other than those marks used to define $\bs F'$). So, we may assume from the beginning that $[G,x]\neq [G,y], \forall x\neq y$.
	This implies that $\iota[G,\cdot]$‌ is injective. In this case, we construct $\bs F(o)$ using the (new) mark $\bs m(o)$ of $o$, and hence, one can let $S(o):=\{o\}$. For this, sort the vertices of $G$‌ by increasing distance from $o$‌ and break the ties by $\iota[G,\cdot]$. If $\preceq$ denotes this order on $G$, one can let $\bs F(o)$ be the smallest $x\in G$‌ such that $\sum_{y\preceq x} K(o,y)\geq \bs m(o)$. This satisfies the desired conditions.
	%
	%
\end{example}

\begin{remark}
	\label{rem:extraRandomness}
	%
	Instead of the Bernoulli point process, one may assume that each point $x\in\bs G$ belongs to $\Phi$ with probability $0\leq \bs \mu(x)\leq 1$ independently from the other points, where $\bs \mu$ is a given equivariant random measure on $\bs G$. The arguments for indistinguishability are similar to those for the Bernoulli case $\bs \mu(x)\equiv p$.
\end{remark}

\subsection{The General Poisson Case}
\label{subsec:poisson}

In this subsection, we define models on the Poisson point process with ideas similar to those of the previous subsection. 
For this, we leverage the notion of unimodular random measured metric spaces \cite{Kh23unimodular} and modify \Cref{model3,model4} accordingly. The main example is the strip point-map (\Cref{ex:strip}).

\subsubsection{Defining the Model}

At first reading, the reader might assume that the base space $\bs G$ is $\mathbb R^d$, $\mathbb H^d$ or products of such spaces. More generally, in each sample, we let $G$ be a boundedly-compact metric space, $o\in G$ and $\mu$ is a boundedly-finite Borel measure on $G$ (in the discrete case, we usually let $\mu$ be the counting measure). We assume that $[\bs G, \bs o, \bs \mu]$ is a \defstyle{unimodular random measured metric space}, which is defined in~\cite{Kh23unimodular} by a natural generalization of the mass transport principle:
\begin{equation*}
	\omid{\int_{\bs G} g(\bs G, \bs o, x,\bs \mu) d\bs \mu(x)} = \omid{\int_{\bs G} g(\bs G, x, \bs o,\bs \mu)d\bs \mu(x)},
\end{equation*}
for all measurable and non negative functions $g$. More generally, $\bs G$‌ might have some other decorations; e.g., the direction of the axes when $\bs G=\mathbb R^d$ or the time direction when $\bs G=\bs G_0\times\mathbb R$ (see~\cite{Kh23unimodular}).

Unimodular random measured metric spaces share many ideas with unimodular graphs, but there are some differences. For instance, in \Cref{lem:happensAtRoot}, one should replace $\bs S\neq\emptyset$‌ and $\bs S=\bs G$ by $\bs \mu(\bs S)=0$ and $\bs\mu(\bs G\setminus\bs S)=0$ respectively (Lemma~4.3 of~\cite{Kh23unimodular}). This should be taken into account in every usage of \Cref{lem:happensAtRoot}. Also, one can relax invariant events to \defstyle{essentially invariant} events $A$ defined by the following condition: For $\mathbb P$-a.e. sample $(G,o,\mu)$, for $\mu$-a.e. $o'\in G$, one has $\identity{A}[G,o,\mu]=\identity{A}[G,o',\mu]$. Under unimodularity, every essentially invariant event is equivalent to some invariant event (Remark~4.28 of~\cite{Kh23unimodular}).

Let $\Phi$ be a Poisson point process on $\bs G$ with intensity measure $\lambda \bs \mu$, where $0<\lambda<\infty$ is fixed. For simplicity, we assume that $\bs \mu$ has no atoms a.s., and hence, $\Phi$ has no multiple points a.s. In most of the example, there is no tie a.s. in defining the point-map $\bs F'$, but we may still keep an i.i.d. marking $\bs m$ of $\Phi^d$ for some $d$ (or a tuple of such markings) as a source of randomness, similarly to \Cref{subsec:bernoulli}.

Let $\tilde G_*$ denote the set of all samples $[G,o,\mu;\varphi, m]$. Here, $\tilde G_*$ is a Borel subspace of some Polish space (obtained by letting $\varphi$‌ have multiplicities).
Similarly to the previous subsection, assume that for all samples $(G,o,\mu;\varphi, m)$, an element $\bs F'(o)\in G$‌ and a stopping pair $Z=(Z_1,Z_2)$ are assigned that satisfy \Cref{as:stopping}. 
Recall that ${\bs F'}(\varphi)\subseteq\varphi$, $m$ is regarded as a source of randomness for defining $\bs F'$ and we do not require the jumps $\bs F'(x)$, $x\in\varphi$ be independent (given $(G,\varphi)$). 
In fact, in most of the examples of interest, the stopping set condition is guaranteed by constructing $\bs F'(x)$ via a \textit{continuous-time decision tree}; see~\cite{LaPeYo23stopping}. 
Define the Markov kernel $K$‌ on $G$ by
\[
K(x,B):=K(G,\mu,x,B):=\probPalmC{(G,\mu)}{\bs F'(x)\in B}{x\in{\Phi}},
\]
for $x\in G$ and $B\subseteq G$, where $\mathbb P_{(G,\mu)}$ is the distribution of $(\Phi,\bs m)$ given $(G,\mu)$. In point process theory, conditioning on the zero probability event $x\in\Phi$ is formalized by Palm theory. However, in the special case of Poisson point processes, the situation is simpler and conditioning on $x\in\Phi$‌ is equivalent to replacing $\Phi$ by $\Phi\cup\{x\}$‌ (and extending $\bs m$ accordingly) by Slyvniak-Mecke theorem (Theorem~6.24 of~\cite{Kh23unimodular}). In what follows, we denote by $\mathbb P_{\Phi}$ the distribution obtained by conditioning on $\bs o\in\Phi$.

\begin{model}
	\label{model5}
	Let $[\bs G, \bs o; \Phi, \bs m]$ be as above. Assume that $\bs F'$ and the stopping pair of sets $Z=(Z_1,Z_2)$ satisfy \Cref{as:stopping}. Assume also that $\omidPalm{\Phi}{\card{\Phi\cap Z_1(\bs o)}}<\infty$ and $\omidPalm{\Phi}{\bs\mu(Z_1(\bs o))}<\infty$. Let $b(\bs o):=b[\bs G, \bs o,\bs \mu]\geq 0$ be a given measurable function of $[\bs G, \bs o, \bs \mu]$ such that $\omid{b(\bs o)}<\infty$.
	In addition, assume that	there exists a full-probability set $E$ of samples of $[\bs G, \bs o, \bs \mu]$ such that every $(G,o,\mu)\in E$ satisfies the following conditions:
	\begin{enumerate}[label=(\roman*)]
		\item (Cycle-Free) There is no cycle in the graph of $\bs F'$ on $G$,  $\mathbb P_{(G,\mu)}$-a.s.
		\item (Disjoint Stopping Sets) The subsets $S(\bs F'^{(i)}(o)), i\geq 0$ are disjoint a.s. conditionally on $o\in \Phi$.
		\item (Balance) One has $b(\cdot)=b[G,\cdot,\mu]>0$, $\mu$-a.e., and the Markov kernel $K$ on $G$, defined above, balances the measure $b\mu$ defined by $b\mu(dx):=b(x)\mu(dx)$; i.e., for all Borel sets $B\subseteq G$,
		\[
			\int_{G} b(x)K(x,B)d\mu(x)= \int_B b(x)d\mu(x).
		\]
		\item (Weak Irreducibility) 
		By letting $K'$ be the time-reversal of $K$ with respect to the invariant measure $b\mu$ and $\tilde K:=\frac 12(K+K')$, for every Borel $B\subseteq G$ in which $\mu(B)>0$, there exists $n$ in which $\tilde K^{(n)}(o,B)>0$; i.e., $\mu$ is mutually absolutely continuous with $\sum_n 2^{-n} \tilde K^{(n)}(o,\cdot)$. A simpler sufficient condition is $\left(\sum_n K^{(n)}(x,\cdot)\right)\wedge\left(\sum_n K^{(n)}(y,\cdot)\right)\neq 0$ for all $x,y\in G$.
	\end{enumerate}
\end{model}


\begin{model}
	\label{model6}
	In \Cref{model5}, assume in addition that every $(G,o,\mu)\in E$ satisfies the one-endedness conditions of \Cref{model4} and a similar weak aperiodicity condition obtained by replacing $\exists x\in G$ with \textit{for a $\mu$-non-null set of $x\in G$.}
\end{model}

\begin{example}[Generalized Strip Point-Map]
	\label{ex:strip2}
	In this example, we show that the generalized strip point-map (\Cref{ex:strip}) satisfies the conditions of \Cref{model5,model6}. We also generalize it beyond $\mathbb R^d$ and $\mathbb H^d$.\\
	One-endedness can be proved similarly to that of the discrete strip point-map (see \Cref{subsec:howard}) using the partition $\{[t+\bs u,t+1+\bs u]\times \bs G: t\in\mathbb Z\}$ of $\bs G'$, where $\bs u\in[0,1]$‌ is a uniform random number independent from everything else.
	Note that no i.i.d. marking is needed in this example. Denoting $(s,y):=\bs F'(t,x)$, one can let $S(t,x):=(t,s]\times B_1(x)$, which is a stopping set that determines $\bs F'(t,x)$. Also, if $l:=\bs \mu(B_1(x))$, then $K((t,x),d(t',x')) = \frac 1 l \exp(-l(t'-t)) d(t',x')$ for $t'>t$‌ and $x'\in B_1(x)$. One can see that $K$ balances the function $b(x):=\bs \mu(B_1(x))$. These definitions satisfy all conditions of \Cref{model5,model6}. \\
	Beyond $\mathbb R^d$ and $\mathbb H^d$, one can let $[\bs G, \bs o, \bs \mu]$ be an arbitrary unimodular random measured metric space with the following conditions. One needs that $\bs \mu$ has no atoms a.s., $\bs \mu(\bs G)=\infty$ a.s. and $\omid{\bs \mu(B_1(\bs o))}<\infty$. Also, for weak irreducibility, one needs that the Markov chain on $\bs G$ (not on $\bs G'$) with kernel $K_0(x,dy):=1/\bs \mu(B_1(x)) \identity{B_1(x)}(y) \bs \mu(dy)$ is irreducible; i.e., $\bs \mu$ is mutually absolutely continuous with respect to $\sum_n 2^{-n} K_0^{(n)}$. Under these conditions, the assumptions of \Cref{model5,model6} are satisfied.
\end{example}

\begin{example}[Continuous-time Voter Model]
	\label{ex:voter-continuous}
	Consider the multi-type version of the continuous-time voter model on a unimodular graph or discrete space $[\bs G, \bs o]$ with a given factor rate function $r:\bs G\times \bs G\to \mathbb R^+$. Similarly to \Cref{ex:voter-discrete}, at each time, we represent the vertices with the same opinion by a partition of $\bs G$ (the values of the opinions are irrelevant here). This model can be defined using a marked Poisson point process on $\bs G\times \mathbb R$. By Ligget's dual representation~\cite{bookLi85interacting}, there exists a finest stationary distribution, which is obtained by the horizontal sections of the dual coalescing Markov chain model. 
	This model is similar to \Cref{model6} except that the clusters are the level-sets of the dual \textit{real forest}, not those of the dual discrete forest. However, the proof given in the next subsection can be slightly modified to prove that the clusters are indistinguishable. For instance, the ancestry chain is a continuous-time Markov chain on $\bs G\times\mathbb R$ whose tail-triviality can be proved similarly to \Cref{thm:pathErgodic-poisson} below (and by using the trick of \Cref{prop:crw} for circumventing weak aperiodicity).
\end{example}

\subsubsection{Proving Indistinguishability}

To prove indistinguishability for \Cref{model5,model6}, we first establish Steps~\ref{step3comp} and~\ref{step3foil} in the following theorem, which extends \Cref{thm:pathIntro} and is of independent interest. The ergodicity result generalizes Theorem~5.6 of~\cite{Kh23unimodular}, which is basically the case where $\bs \mu$ and $K(x,\cdot)$ are mutually absolutely continuous.

\begin{theorem}[Ergodicity and Tail Triviality of Markov Chains on Unimodular Spaces]
	\label{thm:pathErgodic-poisson}
	Let $[\bs G, \bs o, \bs \mu]$ be an ergodic unimodular random measured metric space and let $K$ be a factor Markov kernel. Let $(\bs X_n)_n$ be the Markov chain on $\bs G$ with kernel $K$ started from $\bs o$.
	\begin{enumerate}[label=(\roman*)]
		\item If the balance and weak irreducibility conditions of \Cref{model5} hold, then the Markov chain $[\bs G, \bs X_n, \bs \mu; (\bs X_{n+i})_{i\geq 0}]$, $n\in\mathbb N$, is ergodic.
		\item If, in addition, the weak aperiodicity condition of \Cref{model6} holds, then the Markov chain $[\bs G, \bs X_n, \bs \mu; (\bs X_{n+i})_{i\geq 0}]$ is tail-trivial.
	\end{enumerate}
\end{theorem}

\begin{proof}
	The first part can be proved similarly to \Cref{thm:pathErgodic} (see also the proof of \Cref{subsec:howard}) with the following modification. 
	Consider the events $E_0$ and $E_1$ defined in the proof. 
	The proof shows that $\myprob{E_0\cup E_1}=1$ and $\identity{E_j}(\bs o) = \identity{E_j}(\bs X_1)$ a.s. Stationarity implies that $\identity{E_j}(\bs X_{-1})=\identity{E_j}(\bs o)$ a.s. Now, by the weak irreducibility condition of \Cref{model5}, one can deduce that $E_0$ and $E_1$ are essentially invariant events. The rest of the proof of \Cref{thm:pathErgodic} is then valid.
	
	
	The second part can also be proved similarly to \Cref{thm:pathTail} with the following modifications. First, similarly to \Cref{lem:aperiodicity}, one can show that almost surely, for $\bs \mu$-a.e. $x\in \bs G$ and every $d_0\in \mathbb N$, there exists $n\in\mathbb N$‌ such that $K^n(x,\cdot)\wedge K^{n+d_0}(x,\cdot)\not\equiv 0$.
	Then, to modify \Cref{lem:pathTV}, the assumption should be $K^l(o,\cdot)\wedge K^{l+k}(o,\cdot)\not\equiv 0$. Also, in the coupling mentioned in the proof of \Cref{lem:pathTV}, one should use the set of $x$ in which $K^l(x,\cdot)\wedge K^{l+k}(x,\cdot)$ has total mass larger than $\epsilon$. The rest of the proof of \Cref{thm:pathTail} remains the same.
\end{proof}

%


\begin{proof}[Proof of \Cref{thm:cmt-all} (for \Cref{model5,model6})]
	The proof of indistinguishability of these models is similar to that of \Cref{thm:indistinguishability} and that of the Bernoulli case (given in \Cref{subsec:howard}), and we only highlight the differences.
	Steps~\ref{step3comp} and~\ref{step3foil} are proved in \Cref{thm:pathErgodic-poisson}. 
	By the assumption $\omidPalm{\Phi}{\card{\Phi\cap Z_1(\bs o)}}<\infty$ and similarly to \Cref{lem:shadow}, the set $\{y\in \Phi: x\in Z_1(y) \}$ is finite for all $x\in \Phi$, a.s. This implies that deleting finitely many points from $\Phi$, or changing their marks, results in a finite modification of $\bs F'$. Similarly, by the assumption $\omidPalm{\Phi}{\bs\mu(Z_1(\bs o))}<\infty$ and the \textit{exchange formula} between $\Phi$ and $\bs \mu$ (see Proposition~5.18 of~\cite{Kh23unimodular}), one can show that, almost surely, for $\bs \mu$-a.e. $x\in \bs G$, the set $\{y\in \Phi: x\in Z_1(y) \}$ is finite. This implies that, for every $k\in\mathbb N$ and for $\bs \mu^{k}$-a.e. $x_1,\ldots, x_k\in \bs G$, adding $\{x_1,\ldots,x_k\}$‌ to $\Phi$ (with arbitrary marks) results in a finite modification of $\bs F'$. 
	
	By the above observation, we redefine the tail sigma-fields as follows. An invariant point-map property $A\subseteq\tilde{\mathcal G}_*$ is called a \defstyle{tail point-map property} if the following holds for all samples $[G,o,\mu;\varphi, m]$: For every $k\in\mathbb N$, for $\mu^{k}$-a.e. $x_1,\ldots,x_k\in G$ and for every finite set $B\subseteq \Phi$, updating $\varphi$ and $m$ by adding $x_1,\ldots,x_k$ (with arbitrary marks) and deleting $B$‌ (or changing the marks of the points of $B$) does not change $\identity{A}[G,o,\mu;\varphi,m]$. 
	In what follows, such updates are meant whenever the term \textit{almost all updates of $(\varphi,m)$} is used.
	We also fix the set $E'$ of all $[G,o,\mu;\varphi,m]$ in which almost all updates of $(\varphi,m)$ result in a finite modification of $\bs F'$. Note that $E'$‌ is an (invariant and) tail point-map-property, and $\myprob{E'}=1$ by the previous paragraph.
	Define \defstyle{tail level-set-properties} similarly to \Cref{def:tailDrainage} by considering only those updates (defined above) that result in a finite modification of both $\bs F'$ and the level-set of the origin. Define \defstyle{tail branch-properties} similarly.
	
	Now, \Cref{step1comp,step1foil} can be proved similarly to \Cref{thm:I in T 2} and the Bernoulli case, by just replacing $[G,o;f]$ by $[G,o,\mu;\varphi,m]$ everywhere and adding the condition $[G,o,\mu;\varphi,m]\in E'$ to the definitions of $M'$‌ and $M''$. The rest of the proof remains unchanged and \Cref{step1comp,step1foil} are proved.
	
	\Cref{step2comp,step2foil} can also be proved similarly to the Bernoulli case. For this, it is enough to replace conditioning on $\bs F'(x'_i)=(x'_{i+1})$ in \Cref{subsec:howard} by a regular conditional version of $(\Phi,\bs m)$‌ conditionally on $\bs F'(x'_i)$ (the condition $\probPalm{G}{\bs F'(x'_i)=x'_{i+1}}>0$ is no longer necessary). Also, in the definitions of the sets $M'$‌ and $M''$, one adds the condition $[G,o,\mu;\varphi,m]\in E'$. The rest of the proof remains unchanged and \Cref{step2comp,step2foil} are proved. So, the proof is completed.
\end{proof}

%
%

\section{Problems}
\label{sec:problems}

\begin{problem}
	\label{prob:lyapunov}
	In \Cref{model1}, does there necessarily exist a Lyapunov cocycle? See \Cref{rem:model1}. Or weaker, does there necessarily exit a factor preorder in such a way that $\bs F$ is decreasing with respect to the preorder?
\end{problem}

Recall that the stationarity lemma (\Cref{lem:stationary}) requires $\omid{b(\bs o)}<\infty$.
\begin{problem}
	\label{prob:bias}
	Are the indistinguishability results in \Cref{thm:indistinguishability,thm:usf-indistinguishability} valid without the assumptions $\omid{b(\bs o)}<\infty$ and $\omid{\deg(\bs o)}<\infty$ respectively?
\end{problem}

If the weak aperiodicity condition of \Cref{model2} is removed, then the indistinguishability of level-sets might (or might not) fail as shown in \Cref{ex:crw}. 
\Cref{rem:periodic} studies the periodic case; i.e., when the largest common divisor in \Cref{eq:aperiodic} is $d_0<\infty$. So, it remains to consider when $d_0=\infty$; i.e., when the set \Cref{eq:aperiodic} is equal to $\{0\}$. In this case, \Cref{lem:aperiodicity}, which is the basis of the proof of \Cref{thm:pathTail}, fails completely. Recall also that another technique was used in \Cref{ex:crw}.

\begin{problem}
	\label{prob:aperiodic}
	Without the weak aperiodicity condition of \Cref{model2}, are there other simple necessary/sufficient conditions for the tail-triviality of the Markov chain (as in \Cref{thm:pathIntro}) and the indistinguishability of level-sets?
\end{problem}

\begin{remark}
	\label{rem:periodic2}
	Without the weak aperiodicity condition, the indistinguishability of level-sets is equivalent to the tail-triviality of the Markov chain with kernel $K$. This is implied by \Cref{thm:tailDrainage}.
\end{remark}

\begin{problem}
	\label{prob:aperiodic2}
	In \Cref{prob:aperiodic}, are either of the following conditions sufficient?
	\begin{enumerate}[label=(\roman*)]
		\item For all cycles in the graph with edge set $K(\cdot,\cdot)>0$, consider their \textit{directed length}; i.e., the difference of the number of edges in the forward or backward direction. Is it enough that the largest common divisor of these directed lengths is 1? This would solve the problem in \Cref{ex:dl}.
		\item Assume there is a Lyapunov cocycle $h$ such that $h(\bs F(x))=h(x)-1$ a.s. Can one say that the level-sets included in a common level-set of $h$ are indistinguishable (which holds in \Cref{ex:crw} for non-bipartite graphs)? If not, is it enough in \Cref{prob:aperiodic} that the level-sets of $h$ are indistinguishable (e.g., in \Cref{ex:crw,ex:free})? This would solve the problem in \Cref{ex:free}.
	\end{enumerate}
\end{problem}

\begin{problem}
	\label{prob:crw}
	For coalescing random walks (\Cref{ex:crw}), when the degrees are unbounded, can $\bs F$ be two-ended? As an instance of the analogous problem for coalescing Markov chains (\Cref{ex:cmc}), assume $(\bs b(i))_{i\in\mathbb Z}$ is a stationary process with $\bs b(\cdot)> 1$ a.s. For all $i\in\mathbb Z$, let $p(i,i\pm 1):= 1/(2b(i))$ and $p(i,i):=1-1/(b(i))$. If $\bs b(0)$ has a sufficiently heavy tail under the condition $\omid{\bs b(0)}<\infty$, will the resulting coalescing Markov chain model (on $\mathbb Z^2$) be two-ended a.s.?
\end{problem}

\begin{problem}[LERW]
	\label{prob:lerw}
	On a unimodular non-bipartite random graph $[\bs G, \bs o]$,
	\begin{enumerate}[label=(\roman*)]
		\item If $(\bs X_n)_n$ is the loop-erased random walk from $\bs o$, is $[\bs G, \bs X_n; (\bs X_{n+i})_{i\geq 0}]$ tail-trivial (while it is not Markovian nor stationary)?
		\item Does $\wusf(\bs G)$ have indistinguishable level-sets? 
	\end{enumerate}
\end{problem}

In fact, the first statement is \Cref{step3foil} for proving the indistinguishability of level-sets (see \Cref{lem:usf-trivial I} for \Cref{step3comp}). In the transient case, all other steps are established in \Cref{sec:wusf}. So, in the transient case, the first statement implies the second. In the recurrent  case, \Cref{step2foil} should also be proved.

For the next problem, let $\Phi$ be a Poisson point process in $\mathbb R^d$. For each $x=(t,y)\in \Phi$, where $t\in \mathbb R$ and $y\in\mathbb R^{d-1}$, let $\bs F(x)$ be the closest point of $\Phi\cap (t,\infty)\times\mathbb R^{d-1}$ to $x$. This is called the \defstyle{directional point-map} \cite{BaBo07radial}.

\begin{problem}[Directional Point-Map]
	\label{prob:directional}
	Are the connected components (resp. level-sets) of the directional point-map indistinguishable?
\end{problem}
Note that this model does not satisfy the condition of disjoint stopping sets of \Cref{model5}. Indeed, $\anc(x)$ is not a Markov chain.

\appendix

\section{Measurability Lemmas}
\label{ap:measurability}

In this appendix, we provide some measure-theoretic lemmas regarding null-event-augmentation (\Cref{def:augmentation}), which are essential in the proofs of indistinguishability in this paper.
We need the concept of Borel equivalence relations. Let $M$ be a Polish space or a Borel subset of some Polish space. An equivalence relation $\mathcal R$ on $M$ is Borel if it is a Borel subset of $M\times M$. If so, the \defstyle{$\mathcal R$-invariant sigma-field} $\mathcal I(\mathcal R)$ is the set of Borel subsets $A$ of $M$ such that $\identity{A}(x)=\identity{A}(y)$ for all $(x,y)\in \mathcal R$. Also, $\mathcal R$ is called countable if every equivalence class in $\mathcal R$ is countable. 

\begin{lemma}
	\label{lem:cber}
	Each of the sigma-fields $\sigfield{pm}{I}{}$, $\sigfield{pm}{T}{}$, $\sigfield{comp}{I}{}$, $\sigfield{branch}{T}{}$, $\sigfield{level}{I}{}$ and $\sigfield{level}{T}{}$ is of the form $\mathcal I(\mathcal R)$ for some countable Borel equivalence relation $\mathcal R$ on $\mathcal G'_*$. The same holds for the sigma-fields $\sigfield{path}{I}{}$ and $\sigfield{path}{T}{}$ on $\mathcal G_{\infty}$.
\end{lemma}

In fact, similar claims hold for the sigma-field $\mathcal I$ on $\mathcal G_*$ and for analogous sigma-fields when a base graph is fixed.

\begin{proof}
	For $\sigfield{path}{I}{}$ (resp. $\sigfield{path}{T}{}$), let $([G,x_0;(x_i)_{i\geq 0}], [G, x'_0;(x'_i)_{i\geq 0}])\in\mathcal R$ if $(x_i)_i$ shift-couples (resp. successfully couples) with $(x'_i)_i$. For the other sigma-fields, in each of the following cases, we say $([G,o;f],[G,o';f'])\in \mathcal R$ if:
	\begin{itemize}
		\item Case $\sigfield{pm}{I}{}$: If $f=f'$.
		\item Case $\sigfield{comp}{I}{}$: If $f=f'$ and $o'\in C_f(o)$.
		\item Case $\sigfield{level}{I}{}$: If $f=f'$ and $o'\in L_f(o)$. 
		\item Case $\sigfield{pm}{T}{}$: If $\card{\{f\neq f'\}}<\infty$.
		\item Case $\sigfield{branch}{T}{}$: If either $f=f'$ and $o'\in C_f(o)$ or $o=o'$, $\card{\{f\neq f'\}}<\infty$ and $C_f(o)$ has a common ancestral branch with $C_{f'}(o)$.
		\item Case $\sigfield{level}{T}{}$: If either $f=f'$ and $o'\in L_f(o)$ or $o=o'$, $\card{\{f\neq f'\}}<\infty$ and $\card{L_f(o)\Delta L_{f'}(o)}<\infty$.
	\end{itemize}
	In all cases, $\mathcal R$ is an equivalence relation, except in the cases of $\sigfield{branch}{T}{}$ and $\sigfield{level}{T}{}$. In these cases, we let $\mathcal R$ be the smallest equivalence relation that contains the mentioned relation. Then $\mathcal I(R)$ is the desired sigma-field in each case. 
	

All these equivalence relations are finer than that of $\sigfield{pm}{T}{}$. So, all of them are countable. Checking that they are Borel is skipped for brevity.
\end{proof}



\begin{lemma}
	\label{lem:completion-I(R)}
	Let $(M,\mathbb P)$ be a probability space, $\mathcal R$ be a measurable equivalence relation on $M$,  $A\subseteq M$ be an event and $g$ be a measurable function on $M$. In order to show that $A$ and $g$ are measurable with respect to the null-event-augmentation of $\mathcal I(\mathcal R)$, it is enough to find a measurable subset $M'\in \mathcal I(\mathcal R)$ of full measure, such that $A\cap M'\in\mathcal I(\mathcal R)$ and $\restrict{g}{M'}\in\mathcal I(\mathcal R)$, or equivalently,
	\begin{equation*}
		\forall (x,y)\in (M'\times M')\cap \mathcal R:\  \identity{A}(x)=\identity{A}(y), \ g(x)=g(y).
	\end{equation*}
\end{lemma}

\begin{proof}
	The assumption $M'\in \mathcal I(\mathcal R)$ implies that the last formula is indeed equivalent to $A\cap M'\in\mathcal I(\mathcal R)$ and $\restrict{g}{M'}\in\mathcal I(\mathcal R)$. Also, since $\myprob{M'}=1$, $A$ is equivalent to $A\cap M'$ and $g=\restrict{g}{M'}$ a.s. This implies the claim.
\end{proof}

\begin{lemma}
	\label{lem:completion-g}
	In a probability space $(M,\mathcal F_0,\mathbb P)$, let $\mathcal F\subseteq\mathcal F_0$ be a sub-sigma-field. In order to prove that an event $A$ is in the null-event-augmentation of $\mathcal F$, it is enough to show that for every $\epsilon>0$, there exists some $\mathcal F$-measurable function $g$ such that $\omid{\norm{\identity{A}-g}}<\epsilon$.
\end{lemma}

\begin{proof}
	Let $g_n$ be such a function for $\epsilon:=1/n$. Since $(g_n)_n$ converges in $L_1$, there exists a subsequence which converges a.s. We may assume $g_n$ is convergent a.s. from the beginning. Let $g:=\lim_n g_n$ when the limit exists and $g=0$ otherwise. 
	Then, $g$ is $\mathcal F$-measurable and $g=\identity{A}$ a.s. Now, let $B$ be the event $g=1$. So, $B\in \mathcal F$ and $B$ is equivalent to $A$.
\end{proof}

\begin{lemma}
	\label{lem:completion-cond}
	In \Cref{lem:completion-g}, let $\hat{\mathcal F}$ be the null-event-augmentation of $\mathcal F$. Assume $\mathcal F'$ is another sigma-field such that $\mathcal F\subseteq\mathcal F'\subseteq\mathcal F_0$. In order to prove $A\in\hat{\mathcal F}$, it is enough to prove that $\probCond{A}{\mathcal F'}\in\{0,1\}$ a.s. and that $\probCond{A}{\mathcal F'}$ is $\hat{\mathcal F}$-measurable; i.e., there is a version of the conditional probability $\probCond{A}{\mathcal F'}$ that is $\mathcal F$-measurable.
\end{lemma}
\begin{proof}
	Let $g$ be any version of $\probCond{A}{\mathcal F'}$. So, $g(\cdot)\in\{0,1\}$ a.s. By letting $B$ be the event $g=1$, one has $B\in\hat{\mathcal F}$ and $\identity{B}$ is also a version of $\probCond{A}{\mathcal F'}$. One gets
	\begin{equation*}
		\myprob{A\cap B} = \omid{\omidCond{\identity{A}\identity{B}}{\mathcal F'}}= \omid{\identity{B}\probCond{A}{\mathcal F'}} = \omid{\identity{B}^2}=\myprob{B}.
	\end{equation*}
	Similarly, one gets $\myprob{A\cap B^c} = \omid{\identity{B}\identity{B^c}} = 0$. This implies that $\myprob{A\Delta B}=0$. So, $A$ is in the null-event-augmentation of $\mathcal F$.
\end{proof}

\section{Stopping Sets for Marked Poisson Point Processes}
\label{ap:stopping}

Fix $d\in\mathbb N$ and $(G,o,\mu)$, where $G$ is a boundedly-compact metric space, $o\in G$ and $\mu$ is a boundedly-finite measure on $G$. We assume that $\mu$ is diffuse and $\Phi$ is a Poisson point process with intensity measure $\lambda \mu$ (where $\lambda>0$ is fixed), and hence $\Phi$ is a simple point process a.s. Similarly, we may allow $G$ to be discrete and $\Phi$ to be a Bernoulli point process on $G$, where each point $x\in G$ belongs to $\Phi$ with probability $p\mu(x)$ (where $p$ is fixed, and assuming $p\mu(\cdot)\in [0,1]$).

In this appendix, we study stopping sets for the pair $(\Phi,\bs m)$, where, conditionally on $\Phi$, $\bs m$ is an i.i.d. marking of $\Phi^d$ with an arbitrary compact mark space $\Xi$. So, $\bs m$ is a marking of $d$-tuples of points of $\Phi$. This notion is used in \Cref{subsec:bernoulli,subsec:poisson}.

Let $\mathcal N$ be the set of discrete subsets of $G$. Let also $\mathcal N'$ be the set of pairs $(\varphi,m)$, where $\varphi\in\mathcal N$ and $m:\varphi^d\to\Xi$. These spaces are Borel subsets of some Polish spaces.

\begin{definition}
	\label{def:stoppingPair}
	A \defstyle{stopping pair of sets} for $(\Phi, \bs m)$ is a pair $Z=(Z_1,Z_2)$ which are functions of $(\Phi,\bs m)$ such that: 
	\begin{enumerate}[label=(\roman*)]
		\item $Z_1\subseteq G$ and $Z_2\subseteq Z_1^d$.
		\item If $(\varphi,m)\in\mathcal N'$ and $(\varphi',m')\in\mathcal N'$ satisfy
		$\varphi'\cap Z_1(\varphi,m) = \varphi \cap Z_1(\varphi,m)$ and $\restrict{m'}{Z_2(\varphi,m)^d}=\restrict{m}{Z_2(\varphi,m)^d}$,
		then $Z(\varphi',m')=Z(\varphi,m)$.
		\item The maps $(\varphi,m,y)\mapsto \identity{Z_1(\varphi,m)}(y)$ (where $y\in \varphi$) and $(\varphi,m,z)\mapsto \identity{Z_2(\varphi,m)}(z)$ (where $z\in \varphi^d$) are measurable.
	\end{enumerate}
	In fact, if $G$ is discrete, one can also let $\bs m$ be a marking of $G^d$ and assume $Z_2\subseteq G^d$. 
\end{definition}
As mentioned in \Cref{subsec:bernoulli,subsec:poisson}, in all of the examples of interest in this work, stopping sets are constructed by \textit{decision trees} (similarly to~\cite{LaPeYo23stopping}). In this case, $Z_1$ denotes the subset of $G$ revealed by the decision tree and $Z_2$ denotes the revealed information from $m$.

\begin{theorem}[Marked Poisson Outside a Stopping Pair of Set]
	\label{thm:stopping}
	Assume $Z_1=(Z_1,Z_2)$ is a stopping pair of sets for $(\Phi,\bs m)$. Then, conditionally on $(\Phi\cap Z_1, \restrict{\bs m}{Z_2})$ and on the event that $\Phi\cap Z_1$ is finite, the rest of $(\Phi,\bs m)$ is distributed as follows: $\Phi\setminus Z_1$ is a Poisson point process with distribution $\lambda \restrict{\mu}{G\setminus Z_1}$ and $\restrict{\bs m}{\Phi^d\setminus Z_2}$ is an i.i.d. marking of $\Phi^d\setminus Z_2$ with the same mark distribution.
\end{theorem}
This generalizes Theorem~A.3 of~\cite{LaPeYo23stopping}, which is an analogous result for the unmarked Poisson point process. In fact, when $G$ is discrete, the proof of~\cite{LaPeYo23stopping} works for i.i.d. markings of $G^d$ as well (where there is no $\Phi$). We use this in the following proof. We omit the measurability discussions for brevity.
\begin{proof}
	Let $\Psi:=\Phi\cap Z_1$, $\Psi':=\Phi\setminus Z_1$, $\bs m_1:=\restrict{\bs m}{\Psi^d}$ and $\bs m_2:= \restrict{\bs m}{\Phi^d\setminus \Psi^d}$.
	Note that $Z_1$ and $Z_2$ are determined by $(\Phi\cap Z_1, \restrict{\bs m}{Z_2})$, and hence, are determined by $(\Psi,\bs m_1)$. 
	Let $E$ be the event that $\Psi$ is finite.
	Conditionally on $\Phi$, the set $Z_2$ is a stopping set for $\bs m$. This implies that, conditionally on $(\Phi,\restrict{\bs m}{Z_2})$ and on the event $E$, the rest of $\bs m$ is an i.i.d. marking of $\Phi^d\setminus Z_2$. 
	So, it is enough to prove that, conditionally on $(\Psi,\bs m_1)$ and on $E$, $\Psi'$ is a Poisson point process with intensity measure $\lambda\restrict{\mu}{G\setminus Z_1}$ and $\bs m_2$ is an i.i.d. marking of $(\Psi\cup\Psi')^d\setminus\Psi^d$. This will be proved by the same method as Theorem~A.3 of~\cite{LaPeYo23stopping}. 
	
	In the next formula, $x:=(x_1,\ldots,x_k)\in G^k$ and $\Phi^{(k)}$ is the $k$-th factorial measure of $\Phi$; i.e., $\Phi^k$ restricted to the set of $x\in G^k$ with distinct coordinates. Also, $\bar x$ denotes the set $\{x_1,\ldots,x_k\}$, and $\mathbb E_x$ denotes conditional expectation conditionally on $\Phi\supseteq \bar x$. In addition, for a deterministic marking $m_1$ of $\bar x^d$, $\mathbb E_{(x,m_1)}$ denotes conditional expectation conditionally on $\Phi\supseteq \bar x$ and $\restrict{\bs m}{\bar x^d}=m_1$.
	
	Let $g$ be any nonnegative measurable function on $\mathcal N'\times\mathcal N'$. One can show that $(\Psi,\bs m_1,\Psi',\bs m_2)$ is a measurable function of $(\Phi,\bs m)$ (see Lemma~A.1 of~\cite{LaPeYo23stopping}). By the multivariate Mecke formula, one obtains
	\begin{eqnarray*}
		I&:=& \omid{g(\Psi,\bs m_1,\Psi',\bs m_2)\identity{\{\norm{\Psi}=k\}}}\\
		&=& \frac 1{k!} \omid{\int g(\bar x,\bs m_1 \Psi',\bs m_2)\identity{\{\Psi = \bar x\}}\Phi^{(k)}(dx)}\\
		&=& \frac 1{k!} \int \omidPalm{x}{g(\bar x,\bs m_1, \Psi',\bs m_2)\identity{\{\Psi = \bar x\}}}\mu^{(k)}(dx)\\
		&=& \frac 1{k!} \int\int_{\bar x^d} \omidPalm{(x,m_1)}{g(\bar x,m_1, \Psi',\bs m_2)\identity{\{\Psi = \bar x\}}}dm_1 \mu^{(k)}(dx).
	\end{eqnarray*}
	By Slivnyak's theorem, conditionally on $\Phi\supseteq\bar x$ and $\bs m_1=m_1$, $\Phi\setminus \bar x$ is a Poisson point process with intensity measure $\lambda\mu$ and the rest of $\bs m$ is an i.i.d. marking of $\Phi^d\setminus\bar x^d$. Then, by the stopping set condition, one can show that $\Psi = \bar x$ is equivalent to $\bar x\in Z_1(\bar x, m_1)$ and $(\Phi\setminus \bar x)\cap Z_1(\bar x, m_1)=\emptyset$. This implies that, if $\Psi''$ is a Poisson point process with intensity measure $\lambda\restrict{\mu}{G\setminus Z_1}$ and $\bs m_2'$ is an i.i.d. marking of $(\Psi\cup\Psi'')^d\setminus \Psi^d$ (given $\Psi$ and $\bs m_1$), then $\omidPalm{(x,m_1)}{g(\bar x,m_1, \Psi',\bs m_2)\identity{\{\Psi = \bar x\}}} = \omidPalm{(x,m_1)}{g(\Psi,\bs m_1, \Psi'',\bs m'_2)}$ (similarly to Lemma~A.1 of~\cite{LaPeYo23stopping}, one can show that $(\Psi,\bs m_1,\Psi'',\bs m_1)$ makes sense as a random element of $\mathcal N'\times\mathcal N'$). By reversing the procedure mentioned above, one gets $I= \omid{g(\Psi,\bs m_1,\Psi'',\bs m'_2)\identity{\{\norm{\Psi}=k\}}}$. Summation over $k$ gives $\omid{g(\Psi,\bs m_1,\Psi',\bs m_2)\identity{E}} = \omid{g(\Psi,\bs m_1,\Psi'',\bs m'_2)\identity{E}}$. This implies the claim.
\end{proof}

\begin{remark}
	If $\bs m_1,\ldots,\bs m_k$ are independent i.i.d. markings of $\Phi^{d_1},\ldots,\Phi^{d_k}$, one can similarly define \textit{stopping tuples of sets} for $(\Phi,\bs m_1,\ldots,\bs m_k)$ and extend \Cref{thm:stopping}. We omit this generalization for brevity.
\end{remark}

\section*{Acknowledgements}
We thank Russell Lyons for valuable comments on improving the presentation of the article.
We also thank Milad Barzegar for introducing to us the stochastic covering property of $\wusf$, which was essential in the proofs of \Cref{sec:wusf}.

\bibliography{bib} 
\bibliographystyle{plain}

\end{document}